\documentclass[12pt]{amsart}

\usepackage{amsmath,amsthm,amscd,amsfonts,wasysym,amssymb,epic,eepic,bbm,tikz-cd,mathtools,mathdots,old-arrows,stmaryrd,multicol}
\usepackage[pagebackref,colorlinks=true,linkcolor=blue,citecolor=blue]{hyperref}
\usepackage[noabbrev]{cleveref}
\usepackage{MnSymbol}
\usepackage{enumitem}



\allowdisplaybreaks
\newtheorem{thm}{Theorem}[section]
\newtheorem{cor}[thm]{Corollary}

\newtheorem{lem}[thm]{Lemma}

\newtheorem{prop}[thm]{Proposition}
\theoremstyle{remark}
\newtheorem*{rem}{Remark}

\newcounter{remarkscounter}

\numberwithin{equation}{section}
\newcommand{\A}{\mathbb{A}}
\newcommand{\GL}{\mathrm{GL}}

\newcommand{\ZZ}{\mathbb{Z}}

\newcommand{\Ind}{\mathrm{Ind}}
\newcommand{\QQ}{\mathbb{Q}}

\newcommand{\lto}{\longrightarrow}

\newcommand{\OO}{\mathcal{O}}
\newcommand{\CC}{\mathbb{C}}
\newcommand{\RR}{\mathbb{R}}
\newcommand{\GG}{\mathbb{G}}

\newcommand{\quash}[1]{}

\theoremstyle{definition}
\newtheorem{defn}[thm]{Definition}

\newenvironment{psmatrix}
  {\left(\begin{smallmatrix}}
  {\end{smallmatrix}\right)}

\renewcommand{\bar}{\overline}
\numberwithin{equation}{subsection}

\renewcommand{\hat}{\widehat}

\newcommand{\one}{\mathbbm{1}}

\newcommand{\norm}[1]{\lVert#1\rVert}

\renewcommand{\indexname}{Index of Symbols}\usepackage[columns=3]{idxlayout}

\allowdisplaybreaks

\makeindex

\begin{document}

\title[Plancherel and Poisson summation formulae for affine $\Psi$-bundles]{Plancherel and Poisson summation formulae for a family of affine $\Psi$-bundles}

\author{Jayce R. Getz}
\address{Department of Mathematics\\
Duke University\\
Durham, NC 27708}
\email{jgetz@math.duke.edu}

\author{Miao (Pam) Gu}
\address{Department of Mathematics\\
University of Michigan\\
Ann Arbor, MI 48109-1043}
\email{pmgu@umich.edu}

\author{Chun-Hsien Hsu}
\address{Department of Mathematics\\
University of Chicago\\
Chicago, IL 60637}
\email{chunhsien@uchicago.edu}


\thanks{The first author is thankful for partial support provided by NSF grant DMS-2400550 and NSF RTG grant DMS-2231514. The second author appreciates partial support provided by an AMS-Simons Travel Grant. Any opinions, findings, and conclusions or recommendations expressed in this material are those of the authors and do not necessarily reflect the views of the National Science Foundation. 
}
\begin{abstract}
We prove a Plancherel formula for a family of affine $\Psi$-bundles.  As an application, we construct Fourier transforms and asymptotic Schwartz spaces for the family.  We then prove the corresponding Poisson summation formula under suitable assumptions.  The choice of affine $\Psi$-bundles we consider is motivated by applications to triple product $L$-functions explored in another paper of the authors and Leslie.
\end{abstract}

\maketitle

\setcounter{tocdepth}{2}
{\small
\begin{multicols}{2}
\tableofcontents
\end{multicols}
}

\section{Introduction}

For  spherical varieties $X$ over a number field satisfying natural desiderata there should be a Schwartz space, a Fourier transform satisfying an appropriate twisted equivariance property under the group action, and a Poisson summation formula.   We refer to this collection of conjectures collectively as the Poisson summation conjecture for $X.$  It was originally posed 
by Braverman and Kazhdan \cite{BK-lifting,BK:normalized} in the case of reductive monoids, refined by Ng\^o \cite{NgoSums,Ngo:Hankel}, and partially extended to the setting of spherical varieties by Sakellaridis \cite{SakellaridisSph}.  We also mention related work of L.~Lafforgue  \cite{LafforgueJJM}.   

In this paper we prove a Plancherel formula for certain Hermitian line bundles over a family of spherical varieties.
 We then use it to prove the Poisson summation conjecture in this setting.  
 This is the first case of the Poisson summation conjecture proved for singular spaces whose Plancherel decomposition involves admissible representations of arbitrarily high rank. 

Our primary motivation for the particular choice of spaces in this paper is the work in 
\cite{GGHL}.  In loc.~cit.~the authors construct an integral representation related to triple product $L$-functions.  Then they propose a technique, the fiber bundle method, for proving the analytic continuation of these $L$-functions.  The global requirements for the technique are a pair of Poisson summation formulae.  One is the Poisson summation formula for a vector space (which is known) and the other is the Poisson summation formula proved under suitable assumptions in this paper.  We refer
 to loc.~cit.~for more details.

\subsection{The setting of this paper}

Let us formulate the Poisson summation conjecture more precisely in the case relevant to the current paper.  Let $r\ge 3$ be an integer and let $F$ be a number field. For an $F$-algebra $R,$ let \index{$N$}
\begin{align*}
N(R):=\left\{\begin{psmatrix} I_{r-2} & z \\ & I_2 \end{psmatrix},z \in  M_{r-2,2}(R)\right\},
\end{align*}
let $P \leq \GL_{r}$ be the standard parabolic subgroup with unipotent radical $N,$
and let \index{$X^\circ_r$}
\begin{align} \label{XP}
    X^\circ_r:=N \backslash \GL_{r}.
\end{align}
Let \index{$\mathcal{M}$}
\begin{align} \label{Mvee}  
    \mathcal{M}(R):=\{(u_1,u_2) \in M_{r-2,2}(R): u_1 \wedge u_2=0\}\quad\textrm{and}\quad   \mathcal{M}^{ \circ}:=\mathcal{M}-\{0\}. 
\end{align}
Thus $\mathcal{M}^\circ(F)$  is the set of elements in $M_{r-2,2}(F)$ of rank $1.$    To orient the reader, 
in this paper we often work with pairs $(Y,Y^{\circ})$ where
$Y$ is a scheme and $Y^{\circ}$ is an open dense orbit under a particular group action.
Let \index{$X^\circ$}
\begin{align*}
X^\circ:=X_r^\circ \times \mathcal{M}^\circ.
\end{align*}
Write
\begin{align} \label{GLr1r2}
\GL_{r_1,r_2}:=\GL_{r_1} \times \GL_{r_2}, \quad\textrm{and}\quad \GL_{r_1,r_2,r_3}:=\GL_{r_2} \times \GL_{r_2} \times \GL_{r_3}.
\end{align}
We have an action 
\begin{align*} \begin{split}
 \mathcal{R}:   X^\circ(R) \times \GL_{r,r-2,2}(R) &\lto X^\circ(R)\\
    ((x,m),g,g',h) &\longmapsto \left( \begin{psmatrix} g' & \\ & h \end{psmatrix}^{-1}xg,g'^tmh^{\iota}\right). \end{split}
\end{align*}
Here $\iota$ is defined as in \eqref{iota}.
There is a natural map
$$
 X^\circ \lto \mathbb{P}^1
$$
(see \eqref{Plperp}).
For $\ell \in \mathbb{P}^1(F)$ we let \index{$X^\circ_\ell$} $X^{\circ}_{\ell}$ be the fiber over $\ell.$
This paper is devoted to studying the Poisson summation conjecture for certain Hermitian line bundles $\mathcal{H}$ over the family of spaces $X^\circ_\ell(\A_F) $
as $\ell \in \mathbb{P}^1(F)$ varies.   The line bundles $\mathcal{H}$ are the affine $\Psi$-bundles mentioned in the title of this paper at the level of points (see \cite{GGHL} for the definition of an affine $\Psi$-bundle).  They are similar to the Whittaker inductions discussed in \cite{SV}.

Let $v$ be a place of $F.$ Our main technical achievement in this paper is a Plancherel formula for the Harish-Chandra space $\mathcal{C}(X_\ell^\circ(F_v),\mathcal{H}).$  Let $K_{r,r-2}$ be a maximal compact subgroup of $\GL_{r,r-2}(F_v)$ and let $\mathfrak{F}$ be a finite set of $K_{r,r-2}$-types.
\begin{thm} \label{Planch:intro}
There is a $\CC$-linear topological isomorphism
\begin{align*}
\mathrm{HP}_{h_\ell}:    \mathcal{C}(X^\circ_{\ell}(F_v),\mathcal{H})_{\mathfrak{F}} \,\tilde{\lto}\, \mathcal{C}(\mathrm{Tp}_{r,r-2,1},\mathcal{Z})_{\mathfrak{F}}.
\end{align*}
\end{thm}
Here the subscript $_{\mathfrak{F}}$ denotes functions transforming according $K_{r,r-2}$-types in $\mathfrak{F},$ and
$\mathcal{C}(\mathrm{Tp}_{r,r-2,1},\mathcal{Z})$ is a space of sections of a bundle $\mathcal{Z}$ over the tempered spectrum $\mathrm{Tp}_{r,r-2,1}$ of $\GL_{r,r-2,1}(F_v).$  It is defined in \eqref{bundle:Z}.  
\Cref{Planch:intro} is proved as 
 Theorem \ref{thm:PW:Z} below.  There are theorems similar to \Cref{Planch:intro} in the literature. In the $p$-adic case perhaps the closest is in \cite{Delorme:Harinck:Sakellaridis}.  However, we must prove a result for all places of the number field, and our setting is not a special case of loc.~cit., even over $p$-adic fields.  Moreover, our result is more precise; it does not involve an inexplicit description of the discrete spectrum of auxiliary spaces.

Let $\det \times \mathrm{Pl}_{X_r^\circ}: X_{r}^\circ \lto X_{r}$ be the open immersion of \eqref{embed0}, and let \index{$X_\ell$}
\begin{align*}
X_\ell:=X_{r} \times \mathcal{M}_\ell \supset X_\ell^\circ.
\end{align*}
Here $\mathcal{M}_\ell$ is defined as in \eqref{Mell}. Using Theorem \ref{Planch:intro} we construct asymptotic Schwartz spaces 
\begin{align*}
\mathcal{S}^{\mathrm{as}}(X_\ell(F_v),\mathcal{H}) <C^\infty(X^\circ_\ell(F_v),\mathcal{H})
\end{align*}
in \S \ref{sec:As}. One can then form the restricted direct product
\begin{align*}
    \mathcal{S}^{\mathrm{as}}(X_\ell(\A_F) ,\mathcal{H}) &=  \left(\widehat{\otimes}_{v|\infty}\mathcal{S}^{\mathrm{as}}(X_\ell(F_v) ,\mathcal{H})\right) \otimes \bigotimes_{v\nmid\infty}{}'\mathcal{S}^{\mathrm{as}}(X_{\ell}(F_v),\mathcal{H}) 
\end{align*}
with respect to the basic functions defined in \eqref{bell}.  Here and below $\widehat{\otimes}$ denotes the completed projective tensor product.   This is well-defined because
$\mathcal{S}^{\mathrm{as}}(X_\ell(F_v) ,\mathcal{H})$ is a Fr\'echet space for $v|\infty.$   We construct local Fourier transforms in Theorem \ref{thm:FT}. They induce a global Fourier transform
\begin{align*}
\mathcal{F}:    \mathcal{S}^{\mathrm{as}}(X_{\ell}(\A_F),\mathcal{H}) \tilde{\lto}  \mathcal{S}^{\mathrm{as}}(X_{\ell}(\A_F),\mathcal{H})
\end{align*}
by Theorem \ref{thm:As:FT} and Corollary \ref{cor:basic:fixed}.

We prove a weak Poisson summation formula for these Schwartz spaces.  
Let us prepare the notation necessary to state it.   
For an affine algebraic group $G$ over $F,$ we let $[G]:=G(F) \backslash G(\A_F)$ and let $A_G \leq G(F_\infty)$ be the usual central subgroup (see \eqref{AG}).
Let $\mathcal{S}(G(F_\infty)):=\widehat{\otimes}_{v|\infty}\mathcal{S}(G(F_v))$ (see \S \ref{sec:HC} for notation).
The adelic Schwartz space
\begin{align*}
    \mathcal{S}(G(\A_F))= \mathcal{S}(G(F_\infty))\otimes C^\infty_c(G(\A_F^\infty))
\end{align*}
is then defined. 

For $f \in \mathcal{S}(\GL_{r,r-2}(\A_F))$ we have the usual convolution operator
\begin{align*}
    \mathcal{R}(f):L^2([\GL_{r,r-2}]) &\lto L^2([\GL_{r,r-2}]).
\end{align*}
We say that $f$ has cuspidal image if $\mathcal{R}(f)$ has cuspidal image. For $(f_0,\Phi) \in \mathcal{S}^{\mathrm{as}}(X_{\ell}(\A_F),\mathcal{H}) \times  \mathcal{S}(\GL_{r,r-2}(\A_F)) $ let 
\begin{align*}
    \mathcal{R}_u(\Phi)f_0(x,m):=\int_{\GL_{r,r-2}(\A_F) }f_0\left(\begin{psmatrix} g' & \\ & I_2 \end{psmatrix}^{-1}xg,g'^{t}m \right)\Phi(g,g')|\det g'|^{3/2}dgdg'.
\end{align*}
Thus $\mathcal{R}_u(\Phi)f_0 \in \mathcal{S}^{\mathrm{as}}(X_{\ell}(\A_F),\mathcal{H}).$  The subscript $u$ is short for ``unitary.''

We consider the following assumptions on $f \in \mathcal{S}^{\mathrm{as}}(X_{\ell}(\A_F),\mathcal{H}):$ 
\begin{enumerate}[label=(A\text{\arabic*}), ref=A\text{\arabic*}]
    \item One has $f=\mathcal{R}_u(\Phi \otimes \Phi')f_0$ for some 
    $$
    (f_0,\Phi,\Phi') \in \mathcal{S}^{\mathrm{as}}(X_{\ell}(\A_F),\mathcal{H}) \times \mathcal{S}(\GL_{r-2}(\A_F)) \times  \mathcal{S}(\GL_{r}(\A_F)),
    $$
    where $\Phi \otimes \Phi'$ has cuspidal image and is finite under a maximal compact subgroup of $\GL_{r,r-2}(\A_F).$ 
    \label{a}
    \item \label{b} The functions $f$ and $\mathcal{F}(f)$ are $\Xi$-rapidly decreasing in the sense of Definition \ref{defn:Xi:rapid:loc}.  
\end{enumerate}  
\begin{thm} \label{thm:PS}
Assume that $f \in \mathcal{S}^{\mathrm{as}}(X_{\ell}(\A_F),\mathcal{H})$ satisfies \eqref{a} and \eqref{b}. Then
\begin{align*}
    &
    \sum_{x\in X^\circ_{\ell}(F)}f(x)=\sum_{x \in X^\circ_{\ell}(F)}\mathcal{F}(f)(x).
\end{align*}
\end{thm}

 For any cuspidal automorphic representation $\pi$ of $\GL_{r,r-2}(\A_F)$ one can always find $f \in \mathcal{S}(\GL_{r,r-2}(\A_F))$ with cuspidal image such that $\pi(f) \neq 0$ \cite{BPLZZ}.  Thus assumption \eqref{a}, though strong, is not any loss of generality for our purposes.  We use it because it simplifies the proof significantly.  

There should be many functions in $\mathcal{S}^{\mathrm{as}}(X_\ell(\A_F),\mathcal{H})$ that satisfy \eqref{b}.  In order to construct them, it would be helpful to develop a Paley-Weiner theorem for $\mathcal{S}(X_\ell^{\circ}(F_v),\mathcal{H}),$ the Schwartz space of the open orbit.  This requires a different set of ideas than those developed in the current paper. We hope to prove it on another occasion.  We do check that the basic function is $\Xi$-rapidly decreasing in \Cref{lem:basic:bound}.

We use \textbf{the spectral method} to prove Theorem \ref{thm:PS}.  
This is an umbrella term for any method that uses known properties of automorphic representations (spectral data) to deduce geometric results.  In the case at hand, we use known results on $L$-functions of automorphic representations  to deduce our Poisson summation formula.  At first glance this may seem circular, because one usually wants to deduce properties of $L$-functions from Poisson summation formulae.  However, the technique of passing back and forth between spectral and geometric data to ultimately deduce a result about spectral data result is ubiquitous in automorphic representation theory.  In the setting at hand, our intended application is discussed in \cite{GGHL}.  

We do not know how to construct the Fourier transform in our setting without using the Plancherel formula, even though the $\gamma$-factors that appear in its definition correspond to $L$-functions whose analytic properties are known.  Likewise we do not know a proof of the Poisson summation formula without using the spectral method.  It would be interesting if there was a more direct approach.

\subsection{Outline}  In \S \ref{sec:AG:prelim} we collect preliminaries on algebraic groups and norms.  To fix notation we recall the Harish-Chandra Plancherel formula in \S \ref{sec:HC}, and related theorems for Whittaker models in \S \ref{sec:gen}.  We then define the Harish-Chandra space $\mathcal{C}(X_\ell^\circ(F_v),\mathcal{H})$ of $(X_\ell^\circ(F_v),\mathcal{H})$
in \S \ref{sec:HC-X_ell}.  The definition is motivated by the work in \cite{Bernstein:Planch}, but it is slightly different from that of loc.~cit.

Theorem \ref{thm:PW:Z}, our matrical Harish-Chandra Plancherel theorem for $\mathcal{C}(X_\ell^\circ(F_v),\mathcal{H}),$ is the subject of \S \ref{sec:loc:zeta}, \S \ref{sec:Pl:Xcir}, and \S \ref{sec:Matric}.  We formulate the Harish-Chandra Plancherel formula in terms of zeta integrals built from Whittaker models.  This is to facilitate comparison with the global unfolding arguments we use to prove Theorem \ref{thm:PS}.  After defining local zeta integrals in \S \ref{sec:loc:zeta}, we prove a Plancherel inversion formula (Theorem \ref{thm:Plancherel}) for elements of $\mathcal{S}(X^{\circ}_{\ell}(F_v),\mathcal{H})$ in \S \ref{sec:Pl:Xcir}.  The notation $\mathcal{S}(X^{\circ}_{\ell}(F_v),\mathcal{H})$ denotes Schwartz functions on the open orbit; this is just compactly supported and smooth functions in the non-Archimedean case.  The matrical Harish-Chandra Plancherel theorem mentioned earlier then ultimately follows from a continuity argument executed in \S \ref{sec:Matric}.  All of this requires careful estimates.  After proving Theorem \ref{thm:PW:Z} we apply it to construct the Fourier transform in \S \ref{sec:fourier}.  The argument is similar to that employed in \cite{DRS:Schwartz}; we define the transform by describing its image under the matrical Harish-Chandra Plancherel formula in terms of $\gamma$-factors.

  Local asymptotic Schwartz spaces $\mathcal{S}^{\mathrm{as}}(X_\ell(F_v),\mathcal{H})$ are defined in \S \ref{sec:As}.  The definition is an adaptation of the definition in \cite{DRS:Schwartz}.   
The definition of the basic function in $\mathcal{S}^{\mathrm{as}}(X_\ell(F_v),\mathcal{H})$ is given in \S \ref{sec:unramified}.  We then turn to the derivation of the Poisson summation formula in \S \ref{sec:the:sum}.  As mentioned earlier, the proof involves using known properties of $L$-functions to study  $\mathcal{S}^{\mathrm{as}}(X_\ell(F_v),\mathcal{H}).$  

We have included an index of symbols for the reader's convenience.

\subsection*{Acknowledgments} The first author thanks P.~Humphries for pointing out \cite[Appendix A]{FLO} and H.~Hahn for her constant encouragement and help with editing.  This work was finalized at the Postech Mathematics Institute; the first author thanks the institute for their support under the auspices of the Global Research Scholar program.

\section{Preliminaries}
 \label{sec:AG:prelim}

\subsection{Analytic number theory conventions}
Let $f,g:X \to \CC$ be functions on a set $X$ and let $Q$ be a set of parameters.  We write $f \ll_{Q} g$ if there exists a $c_Q \in \RR_{>0}$ depending on $Q$ such that $|f(x)|\leq c_Q|g(x)|$ for all $x \in X.$  If $f \ll_Q g$ we say $g$ \textbf{dominates} $f.$  We write $f \asymp_Q g$ if $f \ll_Q g$ and $g \ll_Q f.$

\subsection{Algebraic groups}

Let $G$ be a connected affine algebraic group over a characteristic zero field $F$ with maximal split torus $A_0.$  Let $M_0$ be the centralizer of $A_0$ in $G.$ If $F$ is Archimedean, let $\mathfrak{g}:=\mathrm{Lie}\,\mathrm{Res_{F/\RR}}(G)$ and $U(\mathfrak{g})$ be the universal enveloping algebra of the complexification of $\mathfrak{g}.$

For a parabolic subgroup $P\leq G$ we let $N_P$ denote the unipotent radical of $P.$   As usual $P$ is said to be semistandard if it contains $A_0.$  A semistandard Levi subgroup is the unique Levi subgroup of a semistandard parabolic subgroup that contains $A_0.$  If we write ``let $P=MN_P$ be a semistandard parabolic subgroup of $G$'' we always assume that $M$ is the semistandard Levi subgroup of $P$.
Let $W(G,M)=N_{G}(M)(F)/M(F).$  

Let $F$ be a local field. We let $K \leq G(F)$ be a maximal compact subgroup such that $G(F)=P(F)K$ for any semistandard parabolic subgroup $P$. When $G=\GL_r,$ we take $A_0=T_r$ to be the torus of diagonal matrices, and let $K_r$ denote such a choice of maximal compact subgroup. When $F$ is non-Archimedean, we can take $K_r=\GL_r(\OO_F)$. Let $B_r$ be the group of upper triangular matrices. Set \index{$T_r(F)^+$}
\begin{align} \label{+}
    T_r(F)^+:=\left\{t \in T_r(F):|\alpha_i(t)| \leq 1 \textrm{ for all simple roots }\alpha_i \textrm{ with respect to }B_r\right\}.
\end{align} 
We let \index{$U_r$} $U_r$ be the unipotent radical 
of $B_r$. Let \index{$\mathcal{P}_r$}
\begin{align} \label{mira}
\mathcal{P}_r(R):=\left\{\begin{psmatrix} g & x\\ & 1 \end{psmatrix}: (g,x) \in \GL_{r-1}(R) \times R^{r-2}\right\}.
\end{align}
It is the usual mirabolic subgroup.  We let $N_{\mathcal{P}_r}$ be its unipotent radical. We write $U_{r_1,r_2}:=U_{r_1} \times U_{r_2},$ etc. as in \eqref{GLr1r2} to ease notation.

 Let $P=MN_P$ be a semistandard parabolic subgroup of $G.$  Define the map
\begin{align*}
    H_P:G(F) \lto \mathfrak{a}_M:=\mathrm{Hom}(X^*(M),\RR)
\end{align*}
by $\langle H_P(mnk),\lambda \rangle=\log |\lambda(m)|$ for $(m,n,k) \in M(F) \times N_P(F) \times K.$  
When restricted to $P(F)$ the map $H_P$ is a homomorphism. Let $
    M(F)^1:=\mathrm{ker} (H_P:M(F) \to \mathfrak{a}_M).$
Set
\begin{align} \label{A:gps} 
\mathfrak{a}_{M,F}:&=H_{P}(M(F)), 
    \quad \mathfrak{a}_M^*:=X^*(M) \otimes_{\ZZ} \RR, \quad
\mathfrak{a}_{M,F}^\vee:=\mathrm{Hom}(\mathfrak{a}_{M,F},2\pi \ZZ),\quad 
\mathfrak{a}_{M,F}^*:=\mathfrak{a}_{M}^*/\mathfrak{a}_{M,F}^\vee. 
\end{align}
Thus $\mathfrak{a}_{M,F}=\mathfrak{a}_M$ if $F$ is Archimedean, and $\mathfrak{a}_{M,F}$ is a lattice of full rank inside $\mathfrak{a}_{M}$ if $F$ is non-Archimedean.
Moreover $\mathfrak{a}_{M,F}^*=\mathfrak{a}_{M}^*$
 when $F$ is Archimedean.

\subsection{Norms and log norms}

Let $F$ be a local field of characteristic zero. For a finite dimensional $F$-vector space $V,$ let $\norm{\,\cdot\,}=\norm{\,\cdot\,}_V$ be the norm on $V(F)$ which is the Euclidean norm if $F=\RR$, the square of the Euclidean norm if $F=\CC$, and is the box norm if $F$ is non-Archimedean.  We use analogous conventions for operator norms of elements of $\GL_V(F).$

For any integral separated scheme $Z$ of finite type over $F,$ we will make use of the notion of norms $\norm{\,\cdot\,}_Z$ and log norms $\sigma_Z$ on $Z$ \cite[\S 18]{Kottwitz:Clay} \cite[\S 1.2]{BP:Ast}.  Our constructions will only depend on norms and log norms up to the standard notion of equivalence discussed in loc.~cit.  
We point out that if $Z_1$ and $Z_2$ are integral separated schemes of finite type over $F$ and $(z_1,z_2) \in Z_1(F) \times Z_2(F)$ then
\begin{align} \label{prod:log}
\sigma_{Z_1}(z_1)^{1/2}\sigma_{Z_2}(z_2)^{1/2} \ll \sigma_{Z_1 \times Z_2}(z_1,z_2) \ll \sigma_{Z_1}(z_1)\sigma_{Z_2}(z_2).
\end{align}
By \cite[Lemma 1.2.1]{BP:Ast}, if $\phi:Z_1 \to Z_2$ is a morphism then
\begin{align} \label{pullback}
\sigma_{Z_1}(z_1) \gg \sigma_{Z_2}(\phi(z_1)).
\end{align}
If, moreover, $\phi$ is finite, then
\begin{align} \label{iso:inv}
\sigma_{Z_1}(z_1) \asymp \sigma_{Z_2}(\phi(z_1)).
\end{align}

For convenience we make certain choices of log norms explicit.
For $g\in \GL_n(F),$
\begin{align*}
    \sigma_{\GL_n}(g)=1+\log(\max(\norm{\,g\,}_{M_n},\norm{g^{-1}}_{M_n})).
\end{align*}
For an affine algebraic group $G$ over $F$, choose a closed $F$-embedding $G\hookrightarrow \GL_n.$ Let $\sigma_G$ be the restriction of $\sigma_{\GL_n}.$ Two different choices of embeddings give rise to equivalent log norms. Therefore, we will drop the subscript $G$ and simply write \index{$\sigma$} $\sigma$ whenever we are using log norms on affine algebraic groups. One has that
\begin{align} \label{sig:ineq}
    \sigma(xy)\ll \sigma(x)+\sigma(y)\ll \sigma(x)\sigma(y).
\end{align}
Note that $\sigma(x)=\sigma(xyy^{-1}) \ll \sigma(xy)\sigma(y^{-1})=\sigma(xy)\sigma(y),$ or in other words
\begin{align} \label{rev:sig:ineq}
   \sigma(x)\sigma(y)^{-1} \ll \sigma(xy). 
\end{align}

We assume that the embedding $G \hookrightarrow \GL_n$ is chosen so that for $x\in G(F)$ and $k_1,k_2\in K$ we have $\sigma(k_1xk_2)=\sigma(x).$
We let $\norm{\,\cdot\,}:=\norm{\,\cdot\,}_G:=e^{\sigma(\cdot)}.$
\quash{
Let $p:X \to Y$ be a morphism of integral separated schemes of finite type over $F.$ For simplicity assume $p$ is surjective on $F$-points. One says that $p$ satisfies the \textbf{norm descent property} if
$$
\sigma_Y(y) \asymp \mathrm{inf}_{p^{-1}(y)}\sigma_X(x).
$$
By (\cite[Propostion 18.2]{Kottwitz:Clay} \cite[Lemma 1.2.2]{BP:Ast}), if $f$ admits a section then it satisfies the norm descent property.  The following lemma is an immediate consequence of this:
\begin{lem} \label{lem:norm:descent}
Let $X \times G \to X$ be an action of an affine algebraic group $G$ on an integral separated scheme $X$ of finite type over $F.$  Then
$\sigma_X(x)\sigma_G(g) \gg \sigma_{X}(xg)$ for $(x,g) \in X(F) \times G(F).$   \qed
\end{lem}}

We normalize $\sigma_{U_r \backslash \GL_r}$ so that for $n,a,k \in U_r(F) \times T_r(F) \times K_r$ one has 
\begin{align*}
    \sigma_{U_r\backslash \GL_r}(nak)=\sigma_{T_r}(a).
\end{align*}
Let $\norm{\,\cdot\,}_{U_r\backslash \GL_r}:=e^{\sigma_{U_r\backslash \GL_r}(\cdot)}.$


\subsection{$\Xi$-functions}\label{ssec:HS}

Let $G$ be a connected reductive group over $F.$ For $F=\CC,$ we always view $G(F)=\mathrm{Res}_{\CC/\RR} G(\RR)$ as a real Lie group.  Fix a minimal parabolic subgroup $P_0=M_0N_0$ of $G$ containing $A_0$. When $G=\GL_{r},$ let $P_0=B_{r}.$ A parabolic subgroup $P=MN_P$ of $G$ is said to be standard if $P\ge P_0.$ 

We let \index{$\Xi_G$}
$$
\Xi_G:G(F) \lto \RR_{>0}
$$
be the Harish-Chandra function defined with respect to $P_0$. 
For readers' convenience, we restate some properties of $\Xi_G$ from \cite[Proposition 1.5.1]{BP:Ast} that will be used later.

\begin{prop}\label{prop:HCfunction}
 One has the following: 
    \begin{enumerate}[label=(\roman*)]
\item \label{delta:asymp} There is a $d>0$ such that if $m \in M_0(F)$ satisfies $|\alpha(m)| \leq 1$ for all simple roots $\alpha$ of $M_0$ in $P_0$ then
$$
\delta_{P_0}^{1/2}(m) \ll \Xi_G(m) \ll \delta_{P_0}^{1/2}(m)\sigma(m)^d.
$$

    
        \item \label{descent} Let $P=MN_P$ be a semistandard parabolic subgroup of $G$. For any $d>0$, there exists $d'>0$ such that for all $m\in M(F)$
        \begin{align*}
            \delta_P^{1/2}(m)\int_{N_P(F)} \Xi_G(mu)\sigma(mu)^{-d'}du\ll \Xi_M(m)\sigma(m)^{-d}.
        \end{align*}
        
        \item \label{convergence}There exists $d>0$ such that $\int_{G(F)} \Xi_G(g)^2\sigma(g)^{-d}dg<\infty.$

        \item \label{doubling} One has $           \frac{1}{\mathrm{meas}_{dk}(K)}\int_{K} \Xi_G(g_1kg_2)dk=\Xi_G(g_1)\Xi_G(g_2).$
    \end{enumerate}\qed
\end{prop}

\begin{cor} \label{cor:xi:bounded}
    The function $\Xi_G$ is bounded above.
\end{cor}

\begin{proof}
  This follows from Proposition \ref{prop:HCfunction}\ref{delta:asymp} and the Cartan decomposition.
\end{proof}

\begin{lem} \label{lem:Levi}
Let $g\in \GL_{r-1}(F).$ Suppose $g\in K_{r-1}tK_{r-1},$ where $t\in T_{r-1}(F)^+.$ Write 
\begin{align*}
    t=\begin{psmatrix}
        t'\\
         & t''
    \end{psmatrix}=\left(\begin{matrix}\begin{psmatrix}
        t_1 & & \\&\ddots & \\ & & t_{j}
    \end{psmatrix}& \\
    & \begin{psmatrix}
        t_{j+1} & & \\&\ddots & \\ & & t_{r-1}
    \end{psmatrix}\end{matrix}\right),
\end{align*}
where
$$
j:=j(t):=\begin{cases} 0 &\textrm{if } |t_1| \geq 1,\\
\textrm{the unique integer $1 \leq i < r-1$ such that }|t_i|\leq  1 <|t_{i+1}| &\textrm{if }|t_{r-1}|>1>|t_1|,\\
r-1 &\textrm{if }|t_{r-1}|\leq 1.
\end{cases}
$$ 
There is a $d>0$ independent of $g$ such that
$$
\frac{\Xi_{\GL_{r-1}}(g)\sigma(g)^{d}|\det t'|}{|\det g|^{1/2}}\gg \Xi_{\GL_{r}}\begin{psmatrix} g & \\ & 1 \end{psmatrix} \gg \frac{\Xi_{\GL_{r-1}}(g)|\det t'|}{|\det g|^{1/2}\sigma(g)^{d}}.
$$
\end{lem}

\begin{proof} 
We may assume $g=t.$  Then using Proposition \ref{prop:HCfunction}\ref{delta:asymp} with $G=\GL_r$ we have
\begin{align*}
     \Xi_{\GL_r}\begin{psmatrix} g & \\ & 1 \end{psmatrix} = \Xi_{\GL_r}\begin{psmatrix} t' & &\\  & 1&\\& & t'' \end{psmatrix}\gg \delta_{B_r}^{1/2}\begin{psmatrix} t' & &\\  & 1&\\& & t'' \end{psmatrix}&=\delta^{1/2}_{B_{r-1}}(g)|\det t'|^{1/2}|\det t''|^{-1/2}\\&=\delta^{1/2}_{B_{r-1}}(g)|\det g|^{-1/2}|\det t'|.
\end{align*}
The lower bound now follows from an application of Proposition \ref{prop:HCfunction}\ref{delta:asymp} with $G=\GL_{r-1}.$ The upper bound can be proved similarly.\end{proof}

\begin{lem} \label{lem:parab} There is a $d>0$ such that for $(g,y) \in \GL_{r-1}(F) \times F^{r-1}$ one has 
\begin{align*}
     \Xi_{\GL_{r}}\begin{psmatrix}
         g & \\
         & 1
     \end{psmatrix}\sigma\begin{psmatrix}
         g & y \\
         & 1
     \end{psmatrix}^{d}\gg\Xi_{\GL_r}\begin{psmatrix} g& y \\ & 1 \end{psmatrix} \gg
     \frac{\Xi_{\GL_{r}}\begin{psmatrix} g&  \\ & 1 \end{psmatrix}}{\sigma\begin{psmatrix}
         g & y \\
         & 1
     \end{psmatrix}^{d} \max(1,\min(\norm{g^{-1}y},\norm{y}))^{r-1}}.
\end{align*}
\end{lem}

\begin{proof}
Assume that $F$ is non-Archimedean. The lemma is clear if $y,g^{-1}y\in \mathcal{O}_F^{r-1}.$ Assume first $y\not\in \mathcal{O}_F^{r-1}.$   By Proposition \ref{prop:HCfunction}\ref{doubling} we have
\begin{align*}
\Xi_{\GL_r}\begin{psmatrix} I_{r-1} & -y \\ &1\end{psmatrix}\Xi_{\GL_r}\begin{psmatrix} g & y \\ & 1 \end{psmatrix}=\frac{1}{\mathrm{meas}_{dk}(K_r)}\int_{K_r}\Xi_{\GL_r}\left(\begin{psmatrix} I_{r-1} & -y \\& 1\end{psmatrix}k\begin{psmatrix} g & y \\ & 1\end{psmatrix}\right)dk.
\end{align*}
This is bounded below by $\Xi_{\GL_r}\begin{psmatrix} g & \\ & 1 \end{psmatrix}$ times \begin{align*}
&\mathrm{meas}_{dk}\left(\begin{psmatrix} I_{r-1} & y \\ & 1\end{psmatrix} K_r\begin{psmatrix} I_{r-1} & -y \\ & 1\end{psmatrix} \cap K_r\right)\\&\geq \mathrm{meas}_{dk}\left(\begin{psmatrix} I_{r-1} & y \\ & 1\end{psmatrix} \left(I_r+\varpi\mathfrak{gl}_r(\OO_F)\right)\begin{psmatrix} I_{r-1} & -y \\ & 1\end{psmatrix} \cap \left(I_r +\varpi \mathfrak{gl}_r(\OO_F) \right) \right)\\
&\gg \mathrm{meas}_{dX}\left(\begin{psmatrix} I_{r-1} & y \\ & 1\end{psmatrix} \mathfrak{gl}_r(\OO_F)\begin{psmatrix} I_{r-1} & -y \\ & 1\end{psmatrix} \cap \mathfrak{gl}_r(\OO_F) \right).
\end{align*}
Here $dX$ is the Haar measure on the additive group $\mathfrak{gl}_r(F)$ giving $\mathfrak{gl}_r(\OO_F)$ measure $1.$ 

Let $a\in F^\times$ such that $\norm{y}=|a|>1.$ Then
\begin{align} \label{Cartan:comp}
    K_r\begin{psmatrix} I_{r-1} & y \\ & 1 \end{psmatrix}K_r=K_r \begin{psmatrix} 1 & & a \\ & I_{r-2} & \\ & & 1 \end{psmatrix}K_r=K_r \begin{psmatrix} a^{-1} & \\ & I_{r-2} & \\ & & a\end{psmatrix}K_r.
\end{align}
Here we have used the observation that 
$$\begin{psmatrix}  a^{-1} & & -1\\&I_{r-2} & \\ 1 & &  \end{psmatrix} \begin{psmatrix} 1 & & a \\ & I_{r-2} & \\ & & 1 \end{psmatrix}\begin{psmatrix}  1 & & \\ &  I_{r-2} &  \\  -a^{-1} & &  1\end{psmatrix}=\begin{psmatrix} a^{-1} & & \\ & I_{r-2} & \\ & & a \end{psmatrix}.
$$
By \eqref{Cartan:comp}, the measure above is $|a|^{-2(r-1)}=\norm{y}^{-2(r-1)}.$  Thus
\begin{align*}
\Xi_{\GL_r}\begin{psmatrix} g & y \\&1\end{psmatrix} &\gg \frac{\Xi_{\GL_r}\begin{psmatrix} g &  \\ & 1\end{psmatrix}}{\norm{y}^{2(r-1)}\Xi_{\GL_r}\begin{psmatrix} I_{r-1} & -y \\ & 1\end{psmatrix}}=\frac{\Xi_{\GL_r}\begin{psmatrix} g &  \\ & 1\end{psmatrix}}{\norm{y}^{2(r-1)}\Xi_{\GL_r}\begin{psmatrix} a^{-1} & & \\ &I_{r-1} &  \\ & & a\end{psmatrix}}\gg \frac{\Xi_{\GL_r}\begin{psmatrix} g &  \\ & 1\end{psmatrix}}{\sigma(y)^{d}\norm{y}^{r-1}}
\end{align*}
for some $d>0$ by Proposition \ref{prop:HCfunction}\ref{delta:asymp}. 
On the other hand, if $g^{-1}y \not \in \OO_F^{r-1}$ then we use the identities
\begin{align} \label{inverses}
    \Xi_{\GL_r}\begin{psmatrix} g & y \\&1\end{psmatrix}= \Xi_{\GL_r}\begin{psmatrix} g^{-1} & -g^{-1}y \\&1\end{psmatrix}\gg \frac{\Xi_{\GL_r}\begin{psmatrix} g^{-1} &  \\ & 1\end{psmatrix}}{\sigma(g^{-1}y)^{d}\norm{g^{-1}y}^{r-1}}=\frac{\Xi_{\GL_r}\begin{psmatrix} g &  \\ & 1\end{psmatrix}}{\sigma(g^{-1}y)^{d}\norm{g^{-1}y}^{r-1}}.
\end{align}
This completes the proof of the lower bound in the non-Archimedean case.

For the upper bound, we may assume $g=\begin{psmatrix}t' & \\ & t''\end{psmatrix}\in T_{r-1}(F)^+$ as in Lemma \ref{lem:Levi}. By Proposition \ref{prop:HCfunction}\ref{descent} there is a $d>0$ such that
\begin{align*}
    \int_{F^{r-1}}\Xi_{\GL_r}\begin{psmatrix}
        g & y \\
        & 1
    \end{psmatrix}\sigma\begin{psmatrix}
        g  & y \\
         & 1
    \end{psmatrix}^{-d}dy\ll |\det g|^{1/2}\Xi_{\GL_{r-1}}(g).
\end{align*}
Observe that the integral over $y$ is invariant under $\mathcal{O}_F^{j}\oplus \bigoplus_{i=j+1}^{r-1} t_i\mathcal{O}_F.$ Therefore,
\begin{align*}
    \Xi_{\GL_{r}}\begin{psmatrix}
        g & y \\
        & 1
    \end{psmatrix}\ll  \sigma\begin{psmatrix}
        g  & y \\
         & 1
    \end{psmatrix}^{d}|\det g|^{1/2}\Xi_{\GL_{r-1}}(g)|\det t''|^{-1}.
\end{align*}
Thus the desired bound follows from Lemma \ref{lem:Levi}.

Suppose $F$ is Archimedean. We may assume $g=\begin{psmatrix}t'& \\ &t''\end{psmatrix}\in T_{r-1}(F)^+.$  By \cite[Part II, Proposition 8.16(v)]{Varadarajan} it suffices to consider $y$ with  $|y_i|\ge \max(1,|t_i|)$ for all $i$. Let $\varsigma_1\le \ldots\le \varsigma_r$ (viewed as elements in $F$) be the singular values of $\begin{psmatrix}
    g & y\\
    & 1
\end{psmatrix}$. Then using Cauchy interlacing theorem, for $1\le i\le r-1$
\begin{align*}
    |\varsigma_i|\le |t_i|\le |\varsigma_{i+1}|.
\end{align*}
By our assumptions on $y,$ 
$\norm{y}\asymp \norm{\begin{psmatrix}g & y \\ & 1 \end{psmatrix}}_{M_r} \asymp |\varsigma_r|.$ Using Proposition \ref{prop:HCfunction}\ref{delta:asymp} we have
\begin{align*}
    \Xi_{\GL_{r}}\begin{psmatrix}
    g & y\\
    & 1
\end{psmatrix}&\gg \prod_{i=1}^{r} |\varsigma_i|^{(r+1-2i)/2}=|\det g|^{(r-1)/2} \prod_{i=1}^{r} |\varsigma_i|^{1-i}\\
&\gg\frac{|\det g|^{(r-1)/2} }{\norm{y}^{r-1}}\prod_{i=1}^{r-1} |t_i|^{1-i}=\frac{|\det g|^{1/2}}{\norm{y}^{r-1}} \delta_{B_{r-1}}^{1/2}(g).
\end{align*}
Applying Proposition \ref{prop:HCfunction}\ref{delta:asymp} and Lemma \ref{lem:Levi} we deduce that 
$$
\Xi_{\GL_r}\begin{psmatrix} g & y \\ & 1 \end{psmatrix} \gg \frac{\Xi_{\GL_r}\begin{psmatrix} g & \\ & 1 \end{psmatrix}}{\sigma(g)^{d}\norm{y}^{r-1}}
$$
for some $d>0$.  Taking inverses as in \eqref{inverses} we deduce the lower bound.

Recall the definition of $1 \leq j \leq r-1$ from the statement of Lemma \ref{lem:Levi}. To prove the upper bound, we observe that Proposition \ref{prop:HCfunction}\ref{delta:asymp} implies that
there is a $d>0$ such that
\begin{align*}
    \Xi_{\GL_r}\begin{psmatrix}
    g & y\\
    & 1
\end{psmatrix}&\ll \sigma\begin{psmatrix} g & y\\ & 1 \end{psmatrix}^d|\det g|^{(r-1)/2}\prod_{i=2}^{r}|\varsigma_i|^{1-i}\\&=
\sigma\begin{psmatrix}
    g & y\\
    & 1
\end{psmatrix}^{d}|\varsigma_1|^{j}|\det g|^{(r-1)/2-j} \prod_{i=2}^{r} |\varsigma_i|^{1-i+j}\\
&\le \sigma\begin{psmatrix}
    g & y\\
    & 1
\end{psmatrix}^{d}|t_1|^{j}|\det g|^{(r-1)/2-j} \prod_{i=2}^{j} |t_i|^{1-i+j}\prod_{i=1+j}^{r-1} |t_i|^{-i+j}\\
&= \sigma\begin{psmatrix}
    g & y\\
    & 1
\end{psmatrix}^{d}\delta_{B_{r-1}}^{1/2}(g) |\det t'|^{1/2}|\det t''|^{-1/2}.
\end{align*}
Using Proposition \ref{prop:HCfunction}\ref{delta:asymp} and Lemma \ref{lem:Levi} again we deduce the upper bound. 
\end{proof}

\section{The Harish-Chandra Plancherel formulae} \label{sec:HC}

Let $G$ be a connected reductive group over a local field $F.$ For  $f\in C^\infty(G(F))$ and $d\in \RR$, we let
\begin{align*}
    p_d(f):=\sup_{g\in G(F)} |f(g)|\Xi(g)^{-1}\sigma(g)^d.
\end{align*}
Assume $F$ is non-Archimedean.  Let \index{$\mathcal{C}_{d}( G(F))$}
\begin{align*}
\mathcal{C}_{d}( G(F)):=\left\{\textrm{bi-}K\textrm{-finite }f \in C^\infty( G(F)):p_{-d}(f)<\infty\right\}.
\end{align*} 
When $F$ is Archimedean, we set 
\begin{align*}
\mathcal{C}_d(G(F)):=\left\{f \in C^\infty(G(F)): p_{-d}(u*f*v)<\infty \textrm{ for all }u,v \in U(\mathfrak{g})\right\}.
\end{align*}
We then define \index{$\mathcal{C}(G(F))$}
$$
\mathcal{C}^w(G(F)):=\bigcup_{d>0}\mathcal{C}_d(G(F)) \quad\text{and}\quad \mathcal{C}(G(F))=\bigcap_d \mathcal{C}_d(G(F)).
$$
The space $\mathcal{C}^w(G(F))$ is an LF-space. In the non-Archimedean case $\mathcal{C}(G(F))$ is an LF-space and in the Archimedean case it is a Fr\'echet space.  We refer the reader to \cite[\S 1.5]{BP:Ast} and \cite[\S 2.4]{BP:GLn} for details.  The space $\mathcal{C}(G(F))$ is the \textbf{Harish-Chandra space} and $\mathcal{C}^w(G(F))$ is the \textbf{weak Harish-Chandra space}.

For a quasi-projective scheme $X$ of finite type over $F,$ let $\mathcal{S}_{\mathrm{ES}}(X(F)):=C_c^\infty(X(F))$ when $F$ is non-Archimedean.  If $F$ is Archimedean, let $\mathcal{S}_{\mathrm{ES}}(X(F))$ be the Schwartz space of $\mathrm{Res}_{F/\RR} X(\RR)$ introduced in \cite{Elazar:Shaviv}. It is a (nuclear) Fr\'echet space. We refer to $\mathcal{S}_{\mathrm{ES}}(X(F))$ as the \textbf{Elazar-Shaviv Schwartz space}. When $X$ is smooth, we write $\mathcal{S}(X(F)):=\mathcal{S}_{\mathrm{ES}}(X(F)).$

We have natural inclusions
$$
\mathcal{S}(G(F)) \hookrightarrow  \mathcal{C}(G(F)) \hookrightarrow \mathcal{C}^w(G(F)),
$$
where the right map is continuous and the left map is continuous when $F$ is Archimedean.

Let $\mathrm{Irr}(G)$ be the set of equivalence classes of irreducible admissible representations of $G(F).$   
We often abuse notation and identify a representation with its class. Let 
$$
\mathrm{Irr}_2(G) \subseteq \mathrm{Tp}(G) \subseteq \mathrm{Irr}(G)
$$
be the subsets of square-integrable representations and tempered representations of $G(F)$.   
After this section we will write
\begin{align*}
\mathrm{Tp}_r:=\mathrm{Tp}(\GL_r), \quad \mathrm{Tp}_{r_1,r_2}:=\mathrm{Tp}(\GL_{r_1,r_2}), \quad \mathrm{Tp}_{r_1,r_2,r_3}:=\mathrm{Tp}(\GL_{r_1,r_2,r_3})
\end{align*}
to ease notation.

Assume for the moment that $F$ is Archimedean. Let $\mathfrak{t} \leq \mathfrak{g}_{\CC}$ be a Cartan subalgebra.  If $\pi$ is an irreducible admissible representation of $G(F)$ then the infinitisimal character $\chi_{\pi}$ may be identified with an element of  $\mathfrak{t}^*:=\mathrm{Hom}_{\CC}(\mathfrak{t},\CC),$ unique up to the action of the Weyl group $W(\mathfrak{g}_\CC,\mathfrak{t})$ of $\mathfrak{t}$ in $\mathfrak{g}_{\CC}.$  
Choose a $W(\mathfrak{g}_{\CC},\mathfrak{t})$-invariant norm $\norm{\,\cdot\,}$ on $\mathfrak{t}^*$ and define
\begin{align} \label{norm:pi}
    \norm{\pi}:=\norm{\chi_{\pi}}.
\end{align}

For any semistandard parabolic subgroup $P=MN_P$ of $G,$ we have an action of $i\mathfrak{a}_{M}^*$ on $\mathrm{Irr}_2(M)$ given by 
\begin{align*}  \begin{split}
 \mathrm{Irr}_2(M) \times i\mathfrak{a}_{M}^* &\lto\mathrm{Irr}_2(M)\\
(\sigma,\lambda) &\longmapsto  \sigma \otimes e^{\langle H_P(\cdot),\lambda \rangle}. \end{split}
\end{align*}
The action depends only on $M,$ not the choice of semistandard parabolic subgroup containing $M.$
We let $i\mathfrak{a}_{M,\sigma}^\vee$ be the stabilizer of $\sigma$ under this action; thus $
    i\mathfrak{a}_{M,F}^{\vee}  \leq i\mathfrak{a}_{M,\sigma}^\vee.$
    Here we have used notation from \eqref{A:gps}.

We now assume that $G$ is a finite product of general linear groups. This is so that we can make use of the following fact particular to this situation:  Unitary inductions of irreducible square-integrable representations are irreducible and generic.  We refer to \cite[\S 8.4]{Getz:Hahn} for references in the non-Archimedean case.  In the Archimedean setting see  \cite[\S 2]{JacquetPerfectRS} for genericity and \cite[Theorem 1]{Tadic} for irreducibility.

For $F$-algebras $R,$ let
\begin{align*}
\begin{split}
\Psi_r:U_r(R) &\lto R\\
(u_{ij}) &\longmapsto \sum_{i=1}^{r-1}u_{i,i+1}.
\end{split}
\end{align*} 
We continue to write $\psi:U_r(F) \to \CC^\times$ for $\psi \circ \Psi_r.$ It is a generic character of $U_r(F).$

Let $\sigma \in \mathrm{Irr}_2(M).$
Assume $M=\prod_{i=1}^k\GL_{r_i}$
and let $\mathcal{W}(\sigma,\psi)$ be the Whittaker model of $\sigma$ with respect to $\psi \circ \prod_{i=1}^k\Psi_{r_i}.$ 
Let \index{$I(\sigma_\lambda)$}
\begin{align*} \begin{split}
I(\sigma_\lambda):&=\Ind_P^G(\sigma_\lambda):=\left\{\textrm{smooth }\varphi :G(F) \to \mathcal{W}(\sigma,\psi):\varphi(nmg)=e^{\langle H_P(m),\lambda +\rho_P \rangle}\sigma(m)\varphi(g)\right\}. \end{split}
\end{align*}
Thus we realize these inductions as functions to the Whittaker model.  When $\lambda=0$ we omit it from notation, writing $I(\sigma):=I(\sigma_0).$

   By the Iwasawa decomposition, restriction to $K$ induces an isomorphism of $K$-representations
\begin{align*}
    I(\sigma_\lambda) \tilde{\lto} I(\sigma).
\end{align*}
We will often identify the underlying spaces using this isomorphism.  In particular, we regard the underlying space of $I(\sigma_\lambda)$ as independent of $\lambda.$

The representations $I(\sigma_\lambda)$ are all admissible of finite length, and they are irreducible if $\lambda \in i\mathfrak{a}_M^*.$
For each $\sigma \in \mathrm{Irr}_2(M)$ we have a map
\begin{align} \label{into:Temp} \begin{split}
    i\mathfrak{a}_M^* &\lto \mathrm{Tp}(G)\\
    \lambda &\longmapsto I(\sigma_\lambda). \end{split}
\end{align}
We often write an element of $\mathrm{Tp}_{r_1,r_2}$ as $
I(\underline{\sigma_\lambda})=I(\sigma_\lambda) \otimes I(\sigma'_{\lambda'}).$ The map \eqref{into:Temp} factors through the action of $i\mathfrak{a}_{M,\sigma}^\vee,$ and induces an injection
\begin{align} \label{loc:isom}
    i\mathfrak{a}_M^*/i\mathfrak{a}_{M,\sigma}^\vee \lto \mathrm{Tp}(G).
\end{align}
We endow $\mathrm{Tp}(G)$ with the unique topology so that \eqref{loc:isom} is a homeomorphism onto its image for all $\sigma$ and $M.$  The space $\mathrm{Tp}(G)$ comes equipped with a Plancherel measure $d\mu_G,$ defined using the formal degree and intertwining operators.  We refer to \cite[\S 2.4, \S 2.6]{BP:Ast} and the references therein for the definition.  We often write $d\mu$ for $d\mu_G$ to ease notation.

If $V$ is a locally convex Hausdorff topological $\CC$-vector space, 
$C^\infty(\mathrm{Tp}(G),V)$ denotes the space of functions $T:\mathrm{Tp}(G) \to V$ such that the pullback to $i\mathfrak{a}_M^*$ along \eqref{into:Temp} is smooth.  We make the analogous definition if $V$ is replaced by a locally trivial sheaf of locally convex Hausdorff topological $\CC$-vector spaces (in the category of sheaves of locally convex Hausdorff topological vector spaces over topological spaces).  Finally, if $V$ is a locally trivial sheaf of locally convex quasi-complete Haudorff topological $\CC$-vector spaces over $\mathrm{Tp}(G),$ we say that a section $\mathrm{Tp}(G) \to V$
is holomorphic (resp.~meromorphic) if its pullback along \eqref{into:Temp} is the restriction of a holomorphic (resp.~meromorphic) function on $\mathfrak{a}_{M\CC}^*$ for all $\sigma \in \mathrm{Irr}_2(M).$ 

In the special case $V=\CC,$ we define \index{$\mathcal{C}(\mathrm{Tp}(G))$}
\begin{align} \label{CT:spaces}
    \mathcal{C}(\mathrm{Tp}(G)) <C^\infty(\mathrm{Tp}(G)):=C^\infty(\mathrm{Tp}(G),\CC)
\end{align}
as follows.  If $F$ is non-Archimedean, $\mathcal{C}(\mathrm{Tp}(G))$ is the subspace of functions that vanish outside a finite set of connected components of $\mathrm{Tp}(G).$ 
If $F$ is Archimedean, then $\mathcal{C}(\mathrm{Tp}(G))$ is the subspace of functions $T$ such that for every semistandard parabolic subgroup $P=MN_P,$ every differential operator $D$ on $i\mathfrak{a}_M^\ast$ with constant coefficients, and every $d>0$ one has
\begin{align*} 
p_{D,d}(T):=\mathrm{sup}_{\sigma \in \mathrm{Irr}_2(M)}|(DT)_{\sigma}|(1+\norm{\sigma})^d<\infty,
\end{align*}
where 
\begin{align} \label{D:def}
    (DT)_{\sigma}:=D(\lambda\mapsto T_{I(\sigma_\lambda)})|_{\lambda=0}.
\end{align}

Let us discuss topologies on these spaces.  When $F$ is non-Archimedean, choose a sequence of subsets $U_i \subset \mathrm{Tp}(G),$ $i=1,2,\dots$ such that each $U_i$ is a finite union of connected components of $\mathrm{Tp}(G)$ and $\bigcup_{i=1}^\infty U_i=\mathrm{Tp}(G).$  Then the subspace $\mathcal{C}_i \leq \mathcal{C}(\mathrm{Tp}(G))$ consisting of functions with support on $U_i$ is naturally a Fr\'echet space, and thus $\mathcal{C}(\mathrm{Tp}(G))=\bigcup_{i=1}^\infty \mathcal{C}_i$ is an LF-space.  In the Archimedean case, the seminorms $p_{D,d}$ give $\mathcal{C}(\mathrm{Tp}(G))$ the structure of a Fr\'echet space.

The map \eqref{into:Temp} factors through $i\mathfrak{a}_{M,F}^*=i\mathfrak{a}_{M}^\ast/i\mathfrak{a}_{M,F}^\vee.$  When $F$ is non-Archimedean the quotient $\mathfrak{a}_{M\CC}^\ast/i\mathfrak{a}_{M,F}^\vee$ may be canonically identified with the complex points of an algebraic torus $\mathcal{A}$ \cite[\S V.2.4]{Renard}.  Since $V$ is assumed to be locally trivial, the restriction of an element $T \in C^\infty(\mathrm{Tp}(G),V)$ to the $i\mathfrak{a}_{M}^\ast$-orbit of $\pi \in \mathrm{Tp}(G)$ may be regarded as a map
\begin{align} \label{induced}
i\mathfrak{a}_{M,F}^* \lto V(\pi).
\end{align}
We say that $T$ is algebraic or polynomial if the image of \eqref{induced} is contained in a finite dimensional subspace $W$ and the map $i\mathfrak{a}_{M,F}^* \to W$ 
is the restriction to $\mathcal{A}(\RR)$ of the $\CC$-points of a morphism of schemes of finite type over $\CC.$  Rational maps are defined analogously.     

For $f \in \mathcal{S}(G(F))$ and $\pi\in \mathrm{Tp}(G),$ let
$
    f_{\pi}(g):=\mathrm{tr}\,\pi(g)\pi(f^\vee).$
Here $f^\vee(g):=f(g^{-1}).$  By \cite[\S 2.13]{BP:GLn} we have
\begin{align} \label{fpi:weak}
    f_{\pi} \in \mathcal{C}^w(G(F)).
\end{align}

   For a Hilbert space $V,$ write
\begin{align*}
\mathcal{HS}(V):=V^\vee \widehat{\otimes }V
\end{align*}
for the space of Hilbert-Schmidt operators on $V.$
If $I(\sigma_\lambda) \cong I(\sigma'_{\lambda'})$ for some $(\sigma,\lambda), (\sigma',\lambda'),$ then the theory of intertwining operators yields compatible families of isomorphisms
\begin{align*}
\mathcal{HS}(I(\sigma_\lambda)) \tilde{\lto} \mathcal{HS}(I(\sigma'_{\lambda'})).
\end{align*}
The Archimedean case is discussed in \cite[\S II.(2.2)]{Arthur:HP:realreductive} and the non-Archimedean case is discussed in \cite[\S V.3]{Waldspurger:plancherel}.  
Thus we may unambiguously define $\mathcal{HS}(\pi)$ for $\pi \in \mathrm{Tp}(G).$

In the Archimedean case, the space $\mathcal{HS}(\pi)^{\mathrm{sm}}$ of smooth functions inherits a natural topology, and in the non-Archimedean case we endow it with the finest locally convex topology. Let \index{$\mathcal{HS}$} $\mathcal{HS}$ be the locally constant sheaf of locally convex topological vector spaces over $\mathrm{Tp}(G)$ whose fiber over $\pi$ is $\mathrm{End}(\pi).$
We can then consider the space
\begin{align*}
    \mathcal{C}(\mathrm{Tp}(G),\mathcal{HS}) <C^\infty(\mathrm{Tp}(G),\mathcal{HS})
\end{align*}
defined in \cite[\S 2.6]{BP:Ast}; it is denoted $\mathcal{C}(\mathcal{X}_{\mathrm{Temp}}(G),\mathcal{H}(G))$ in loc.~cit.  
In the non-Archimedean case, $\mathcal{C}(\mathrm{Tp}(G),\mathcal{HS})$ is the space of smooth sections of $\mathcal{HS}^{\mathrm{sm}}$ whose support lies in a finite set of connected components of $\mathrm{Tp}(G).$  In the Archimedean case, $\mathcal{C}(\mathrm{Tp}(G),\mathcal{HS})$ consists of sections $T$ such that for every semistandard parabolic subgroup $P=MN_P,$ every differential operator $D$ on $i\mathfrak{a}_M^\ast$ with constant coefficients, every $u,v \in U(\mathrm{Lie}(K)),$ and every $N \in \ZZ_{\geq 0}$ one has
\begin{align*} 
p_{D,u,v,N}(T):=\mathrm{sup}_{\sigma \in \mathrm{Irr}_2(M)}\norm{(DT)_{\sigma}}_{u,v  }(1+\norm{\sigma})^N<\infty,
\end{align*}
where $\norm{T}_{u,v}:=\norm{u\ast T\ast v}$ and $(DT)_\sigma$ is defined as in \eqref{D:def}. In any case, the space $\mathcal{C}(\mathrm{Tp}(G),\mathcal{HS})$ comes equipped with a locally convex topology. We refer to \cite[\S 2.6]{BP:Ast} for details.

The following is the matrical Plancherel theorem of Harish-Chandra \cite[Theorem 2.6.1]{BP:Ast}:

\begin{thm} \label{thm:HC}
One has an isomorphism of topological vector spaces \index{$\mathrm{HP}_G$}
\begin{align*}
 \mathrm{HP}_G: \mathcal{C}(G(F)) &\tilde{\lto}  \mathcal{C}(\mathrm{Tp}(G),\mathcal{HS})\\
 f &\longmapsto (\pi \mapsto \pi(f)).
\end{align*}\qed
\end{thm}

\section{Whittaker functions} \label{sec:gen}

\subsection{Some function spaces on $U_r\backslash \GL_r$}
  Consider 
$$
C^\infty(U_r(F) \backslash \GL_r(F),\psi):=\left\{f \in C^\infty(\GL_r(F)): f(ng)=\psi(n)f(g) \textrm{ for }(n,g) \in U_r(F) \times \GL_r(F)\right\}.
$$
In this subsection we define various subspaces of $C^\infty(U_r(F) \backslash \GL_r(F),\psi)$. Our presentation  is modeled on \cite[\S 2.4]{BP:GLn}.

For $x=ntk$ with $(n,t,k) \in U_r(F) \times T_r(F) \times K_r,$ we define 
\begin{align*} 
    \Xi_{U_r \backslash \GL_r}(x):=\delta_{B_r}^{1/2}(t).
\end{align*}
For $f\in C^\infty(U_r(F) \backslash \GL_r(F),\psi)$ and $d\in \RR$, let \index{$\nu_{d}$}
\begin{align*} 
    \nu_{d}(f):=\sup_{x\in U_r(F)\backslash \GL_r(F)} |f(x)| \Xi_{U_r \backslash \GL_r}(x)^{-1} \sigma_{U_r\backslash \GL_r}(x)^d.
\end{align*}
 For non-Archimedean $F,$ the space $\mathcal{C}_d(U_r(F)\backslash \GL_r(F),\psi)$ consists of right $K_r$-finite functions $f$ such that $\nu_{-d}(f)<\infty$. For a compact open subgroup $K,$ the norm $\nu_{-d}$ makes $\mathcal{C}_d(U_r(F)\backslash \GL_r(F),\psi)^K$ a Banach space, and thus $\mathcal{C}_d(U_r(F)\backslash \GL_r(F),\psi)$ is an LF-space. For Archimedean $F,$ the space $\mathcal{C}_d(U_r(F)\backslash \GL_r(F),\psi)$ consists of functions $f$ such that for $u\in U(\mathfrak{g})$ \index{$\nu_{u,-d}$}
 \begin{align*}
      \nu_{u,-d}(f):=\nu_{-d}(\mathcal{R}(u)f)
 \end{align*}
 is finite.  Here $\mathcal{R}$ is the regular action of $U(\mathfrak{g})$ on $C^\infty(U_r(F) \backslash \GL_r(F),\psi).$ The seminorms $\{\nu_{u,-d}\}$ make $\mathcal{C}_d(U_r(F)\backslash \GL_r(F),\psi)$ a Fr\'echet space. In any case, we let \index{$\mathcal{C}(U_r(F) \backslash \GL_r(F),\psi)$}
\begin{align*} \begin{split}
\mathcal{C}^w(U_r(F) \backslash \GL_r(F),\psi):=\bigcup_{d>0}\,\mathcal{C}_d(U_r(F) \backslash \GL_r(F),\psi),\\
\mathcal{C}(U_r(F) \backslash \GL_r(F),\psi):=\bigcap_{d}\,\mathcal{C}_d(U_r(F) \backslash \GL_r(F),\psi). \end{split}
\end{align*}
  We define
$\mathcal{C}_d(U_{r,r}(F) \backslash \GL_{r,r}(F),\overline{\psi} \otimes \psi),$ etc.~analogously. These are LF-spaces. When $F$ is Archimedean, $\mathcal{C}(U_r(F) \backslash \GL_r(F),\psi)$ is a Fr\'echet space.

\begin{lem} \label{lem:cont}
Let $d>0$. There is a separately continuous map
\begin{align*}
\mathcal{C}_d(U_r(F) \backslash \GL_r(F),\psi) \times \mathcal{C}(\GL_r(F)) &\lto \mathcal{C}_d(U_r(F) \backslash \GL_r(F),\psi)\\
(W,f) &\longmapsto \Big(x \mapsto \int_{\GL_r(F)}f(g)W(xg)dg\Big).
\end{align*}
When $F$ is Archimedean, it is continuous.  If $F$ is non-Archimedean and $K <\GL_r(F)$ is a compact open subgroup, then the restriction of the map to $\mathcal{C}_d(U_r(F) \backslash \GL_r(F),\psi)^K \times \mathcal{C}(\GL_r(F))^{K \times K}$ is continuous.  
\end{lem}
\begin{proof}
Write $x=u't'k'$ with $(u',t',k') \in U_r(F) \times T_r(F) \times K_r.$  There are continuous seminorms $\nu_d$ on $\mathcal{C}_d(U_r(F) \backslash \GL_r(F),\psi)$ and $\nu_{d'}$ on $\mathcal{C}(\GL_r(F))$ such that 
\begin{align*}
    &\int_{\GL_r(F)}|f(g)W(xg)|dg\\
    &=\int_{\GL_r(F)}|f(k^{\prime-1}g)W(t'g)|dg\\
    &\leq \nu_d(W) \int_{\GL_r(F)}|f(k^{\prime-1}g)|\Xi_{U_r \backslash \GL_r} (t'g)\sigma^d_{U_r\backslash \GL_r}(t'g)dg\\
    &\ll \nu_d(W)\Xi_{U_r \backslash \GL_r} (t')\sigma^d_{U_r\backslash \GL_r}(t')\int_{\GL_r(F)}|f(k^{\prime-1}g)|\Xi_{U_r \backslash \GL_r} (g)\sigma_{U_r\backslash \GL_r}^d(g)dg\\
    &=\nu_d(W)\Xi_{U_r \backslash \GL_r} (x)\sigma_{U_r\backslash \GL_r}^d(x)\int_{\GL_r(F)}|f(k'^{-1}g)|\Xi_{U_r \backslash \GL_r} (g)\sigma_{U_r\backslash \GL_r}^d(g)dg\\
    &\leq \nu_d(W)\nu_{d'}(f) \Xi_{U_r \backslash \GL_r} (x)\sigma_{U_r\backslash \GL_r}^d(x)\int_{\GL_r(F)}\Xi_{\GL_r}(g)\sigma(g)^{-d'}\Xi_{U_r \backslash \GL_r} (g)\sigma_{U_r\backslash \GL_r}^d(g) dg
\end{align*}
for any $d'>0.$ Therefore, to prove separate continuity, it suffices to show for a suitable $d'>0$
\begin{align*}
    \int_{\GL_r(F)}\Xi(g)\sigma(g)^{-d'}\Xi_{U_r \backslash \GL_r} (g)\sigma^d_{U_r\backslash \GL_r}(g)dg<\infty.
\end{align*}
By the Iwasawa decomposition, this is
\begin{align*}
     \int_{T_r(F)}\Big(\int_{U_r(F)}\Xi(tu)\sigma(tu)^{-d'} du\Big)\delta_{B_r}^{1/2} (t)\sigma^d_{U_r\backslash \GL_r}(t) dt.
\end{align*}
By Proposition \ref{prop:HCfunction}\ref{descent}, given $d''>0,$ there is a $d'>0$ such that the above is dominated by
\begin{align*}
     \int_{T_r(F)} \sigma^{-d''}(t)\sigma_{U_r\backslash \GL_r}^d(t) dt.
\end{align*}
This integral is finite for $d''$ large.

Separately continuous bilinear forms on Fr\'echet spaces are jointly continuous \cite[Corollary to Theorem 34.1]{Treves}.  This immediately implies the continuity assertions.
\end{proof}

Let $\widetilde{w}_r=\begin{psmatrix} & & &1\\ & & -1 & \\   & \reflectbox{$\ddots$} & &\\ (-1)^{r-1} & &\end{psmatrix} \in \GL_r(\ZZ).$ 
\begin{lem} \label{lem:tilde:iso}
    For each $d \in \RR,$ one has a continuous automorphism \index{$\widetilde{W}$}
    \begin{align*}
        \widetilde{\cdot}:\mathcal{C}_d(U_r(F) \backslash \GL_r(F),\psi) &\lto \mathcal{C}_d(U_r(F) \backslash \GL_r(F),\psi)\\
        W &\longmapsto \left(g \mapsto \widetilde{W}(g):=W\left( \widetilde{w}_r g^{-t}\right)\right).
    \end{align*}
\end{lem}
\begin{proof}
We may assume $\widetilde{w}_r \in K_r.$  
For $x=ntk$ with $(n,t,k) \in U_r(F) \times T_r(F) \times K_r$ we have 
$$
\Xi_{U_r \backslash \GL_r}(\widetilde{w}_rg^{-t})=\Xi_{U_r \backslash \GL_r}(\widetilde{w}_rg^{-t}\widetilde{w}^{-1}_r)=\Xi_{U_r \backslash \GL_r}(\widetilde{w}_r t^{-t}\widetilde{w}^{-1}_r)=\Xi_{U_r \backslash \GL_r}(t)=\Xi_{U_r \backslash \GL_r(F)}(g).
$$
Similarly $\sigma_{U_r \backslash \GL_r}(\widetilde{w}_rg^{-t})=\sigma_{U_r \backslash \GL_r}(g).$  The lemma follows.  
\end{proof}

\subsection{Generic representations} 

Let $\pi$ be a generic irreducible admissible representation of $\GL_r(F)$. Let $\mathcal{W}(\pi,\psi)=\mathcal{W}(\pi,\psi \circ \Psi)$ be the $\psi \circ \Psi$-Whittaker model. When $F$ is non-Archimedean, we endow $\mathcal{W}(\pi,\psi)$ with the finest locally convex topology; this makes it into an LF-space \cite[\S 2.3]{BP:GLn}.
When $\pi$ is unitary and $F$ is Archimedean, $\mathcal{W}(\pi,\psi)$ is the space of smooth vectors in its completion with respect to an invariant inner product.  Hence it is naturally a smooth Fr\'echet representation of moderate growth \cite[Lemma 11.5.1]{Wallach:RGII}.

We require many bounds on Whittaker functions in this work.  Most of them are very similar in structure, but will be applied in settings where different hypotheses and uniformity assertions are required.

\begin{lem}\label{lem:whittakerweak}
For $\pi \in \mathrm{Tp}_r,$ there is a $d>0$ such that 
\begin{align} \label{containment}
\mathcal{W}(\pi,\psi) <\mathcal{C}_d(U_r(F) \backslash \GL_r(F),\psi).
\end{align}
Moreover, if $V$ is any model of $\pi$ with underlying space of smooth vectors $V^{\mathrm{sm}}$ and $V^{\mathrm{sm}} \to \mathcal{W}(\pi,\psi)$ is an intertwining map, then the induced map $V^{\mathrm{sm}} \to \mathcal{C}^w(U_r(F) \backslash \GL_r(F),\psi)$ is continuous.
\end{lem}
\noindent Here we endow $V^{\mathrm{sm}}$ with the usual Fr\'echet topology in the Archimedean case and the finest locally convex topology in the non-Archimedean case.

\begin{proof}
This is \cite[Lemma 2.8.1]{BP:GLn}, except that loc.~cit.~asserts \eqref{containment} with \sloppy $\mathcal{C}_d(U_r(F) \backslash \GL_r(F),\psi)$ replaced by $\mathcal{C}^w(U_r(F) \backslash \GL_r(F),\psi).$  Happily the references cited (\cite[\S 15.7.3]{Wallach:RGII} and \cite[Theorem 3.1]{Lapid:Mao:Asymp}) imply the stronger containment asserted in the lemma.
\end{proof}

  It follows from \cite[Proposition 2.2]{JPSSGL3I} and \cite[(8.3)]{JPSSGL3II} that every $W\in \mathcal{W}(\pi,\psi)$ is dominated by a gauge $\mathcal{G} \in C^\infty(\GL_r(F))$. To be more precise, there exists an $\ell\in \RR_{\ge 0}$ such that for each $W\in \mathcal{W}(\pi,\psi)$ there is a $\phi\in \mathcal{S}(F^{r-1})$ satisfying
\begin{align} \label{gauge}
    |W|(nak)\le \mathcal{G}(nak):=\left|\prod_{j=1}^{r-1} \alpha_j(a) \right|^{-\ell}\delta_{B_r}^{1/2}(a)\phi(\alpha_1(a),\cdots, \alpha_{r-1}(a))
\end{align}
for all $(n,a,k)\in U_r(F)\times T_r(F)\times K_r$. Here $\alpha_i$ is the $i$th simple root of $T_r$ in $U_r.$ We call $\ell$ the \textbf{exponent} of the gauge.  We let $\ell(\pi) \in \RR_{\geq 0}$ be the infimum of the set of $\ell$ such that an estimate of the form \eqref{gauge} holds for all $W \in \mathcal{W}(\pi,\psi)$, where $\phi$ is allowed to vary as $W$ and $\ell$ vary.
We refer to $\ell(\pi)$ as the \textbf{exponent} of $\pi.$  When $F$ is non-Archimedean, we say that \textbf{the gauge is unramified} if $\phi=\one_{\OO_F^{r-1}}.$ 

By \cite[Proposition 2.5]{JacquetShalika:Whittaker} and \cite[Proposition 3.5]{JacquetPerfectRS} we have
\begin{thm}  \label{thm:exponent}
If $\pi$ is tempered, then $\ell(\pi)=0.$   \qed
\end{thm}

\begin{lem} \label{lem:refine:gauge}
Assume $\pi$ is unitary.  When $F$ is Archimedean, we can refine the bound in \eqref{gauge} as follows: for every $N>0$ and $\ell>\ell(\pi)$ there exists a continuous seminorm $\nu_{\ell,N}$ on $\mathcal{W}(\pi,\psi)$ such that 
$$
|W|(nak)\le\nu_{N,\ell}(W)\left|\prod_{j=1}^{r-1} \alpha_j(a) \right|^{-\ell}\delta_{B_r}^{1/2}(a)\prod_{i=1}^{r-1}(1+|\alpha_i(a)|)^{-N}.
$$
\end{lem}
\begin{proof}
By the gauge estimates \eqref{gauge}, the family of linear functionals 
$$
W \longmapsto W(nak)\left(\left|\prod_{j=1}^{r-1} \alpha_j(a) \right|^{-\ell}\delta_{B_r}^{1/2}(a)\prod_{i=1}^{r-1}(1+|\alpha_i(a)|)^{-N}\right)^{-1}
$$
is pointwise bounded as $n,a,k$ vary, so we can apply the uniform boundedness principle to deduce the lemma.  
\end{proof}

\begin{lem} \label{lem:C:gauge}
 For all $N>0,$ there is a continuous seminorm $\nu_N$ on $\mathcal{C}_d(U_r(F) \backslash \GL_r(F),\psi)$ such that 
$$
|W|(nak)\le\nu_N(W)\sigma(a)^d\delta_{B_r}^{1/2}(a)\prod_{i=1}^{r-1}(1+|\alpha_i(a)|)^{-N}
$$
for all $W \in \mathcal{C}_d(U_r(F) \backslash \GL_r(F),\psi).$
\end{lem}
\begin{proof}
The space $\mathcal{C}_d(U_r(F) \backslash \GL_r(F),\psi)$ is an LF-space \cite[\S 2.4]{BP:GLn}, and hence the uniform boundedness principle is valid for it \cite[\S 33]{Treves}. Thus the lemma follows from the pointwise estimate proved in the same way as \cite[Lemma 2.5.1]{BP:Asai}.
\end{proof}

The map 
 \begin{align} \label{Whitt} \begin{split}
     \mathcal{C}^w(\GL_r(F))  &\lto  \mathcal{C}^w(U_{r,r}(F) \backslash \GL_{r,r}(F),\overline{\psi} \otimes \psi)\\
     f &\longmapsto W_f(g_1,g_2):=\int_{U_r(F)}^\ast f(g_1^{-1}ng_2)\overline{\psi}(n)dn \end{split}
 \end{align}
is a well-defined continuous map by \cite[Lemma 2.14.1]{BP:GLn}.  We refer to  loc.~cit.~for the definition of the integral on the right. For $f \in \mathcal{S}(\GL_r(F))$ and $\pi \in \mathrm{Tp}_r,$ since $f_{\pi}\in  \mathcal{C}^w(\GL_r(F))$ by \eqref{fpi:weak},
one has $W_{f_{\pi}} \in \mathcal{C}^w(U_{r,r}(F) \backslash \GL_{r,r}(F),\overline{\psi} \otimes \psi).$

There is an absolutely convergent Hermitian inner product on $\mathcal{C}^w(U_r(F) \backslash \GL_r(F),\psi)$ given by 
\begin{align}\label{eq:Whitt}
    \left(W_1,W_2 \right)^{\mathrm{Whitt}}:=\zeta(r)\int_{U_r(F) \backslash \mathcal{P}_r(F)}W_1(p)\overline{W}_2(p)d_rp,
\end{align}
where $\mathcal{P}_r$ is the mirabolic subgroup defined in \eqref{mira}
(see the end of \cite[\S 2.8]{BP:GLn}).  The factor $\zeta(r)$ is included for ease of comparison with \cite[Proposition A.2]{FLO}.
The restriction of $(\cdot,\cdot)^{\mathrm{Whitt}}$ to $\mathcal{W}(\pi,\psi)$ is $\GL_n(F)$-invariant by
\cite{Baruch,Bernstein:P,Sahi}.
 Let $\mathcal{B}_{\mathcal{W}(\pi,\psi)}$ be an orthonormal basis of $\mathcal{W}(\pi,\psi)$ with respect to this inner product.  We assume that the orthonormal basis consists of vectors finite under $K_r.$ 
One has
\begin{align} \label{Whitt:exp}
    W_{f_{\pi}}(g_1,g_2)=\sum_{W \in \mathcal{B}_{\mathcal{W}(\pi,\psi)}}\overline{W}(g_1)\pi(f^\vee)W(g_2),
\end{align}
where the sum on the right converges in $\mathcal{C}^w(U_{r,r}(F) \backslash \GL_{r,r}(F),\overline{\psi} \otimes \psi)$
(see the end of \cite[\S 2.14]{BP:GLn}).
The Whittaker Plancherel formula \cite[Proposition 2.14.2]{BP:GLn} asserts that for $f \in \mathcal{S}(\GL_r(F)),$ one has
\begin{align} \label{Whitt:Planch}
    W_f(g_1,g_2)=\int_{\mathrm{Tp}_r}W_{f_\pi}(g_1,g_2) d\mu(\pi).
\end{align}


Let $P=MN_P$ be a standard parabolic subgroup of $\GL_r,$ let $\bar{P}$ be its opposite and let $\overleftarrow{P}$ be the standard parabolic subgroup that is $\GL_r(F)$-conjugate to $\bar{P}.$
Let $\sigma \in \mathrm{Irr}_2(M)$ and let $\varphi \in I(\mathcal{W}(\sigma,\psi))$ be a smooth vector.  Write $M=\prod_{i=1}^k\GL_{r_i}$ and
\begin{align} \label{wP}
w_P:=\begin{psmatrix} & & I_{r_k}\\ & \textrm{\reflectbox{$\ddots$}} & \\ I_{r_1} & & \end{psmatrix}.
\end{align}
Thus $w_P^{-1}\overleftarrow{P}w_P=\overline{P}.$ Let \index{$W_{\varphi_{\lambda}}$}
\begin{align} \begin{split} \label{Whitt:fam}
W_{\varphi_{\lambda}}(g):&=\int_{N_{\overleftarrow{P}}(F)}I(\sigma_\lambda)\left(w_P^{-1}ng\right)\varphi(I_r) \overline{\psi}(n)dn,\\
W_{\varphi_{\lambda}}'(g):&=\int_{N_{\overleftarrow{P}}(F)}I(\sigma_\lambda)\left(w_P^{-1}ng\right)\varphi(I_r) \psi(n)dn \end{split}
\end{align}
whenever these integrals are absolutely convergent or obtained by analytic continuation from a cone of absolute convergence.  When $\lambda=0$ we drop it from notation, writing, e.g.,~$W_{\varphi}(g):=W_{\varphi_0}(g).$
By \cite[Theorem 15.4.1]{Wallach} and
 \cite[Proposition 3.1]{Shahidi:On:certain} we have the following:

\begin{thm} \label{thm:Whitt:analyfam}
The integrals in \eqref{Whitt:fam} are absolutely convergent for $\mathrm{Re}(\lambda)$ in a suitable cone and may be analytically continued to all $\lambda \in \mathfrak{a}_{M\CC}^\ast.$  \qed
\end{thm}

  Despite the fact that $I(\sigma_\lambda)$ may not be irreducible, by \cite[\S 2]{JacquetPerfectRS} and \cite[Th\'eor\`emes 2 and 4]{Rodier:Whitt} it admits a unique $\psi$-Whittaker functional (assumed to be continuous in the Archimedean case) and hence we have well-defined Whittaker models
$$
\mathcal{W}(I(\sigma_\lambda),\psi) \quad \textrm{and}\quad \mathcal{W}(I(\sigma_\lambda),\overline{\psi}).
$$

\begin{lem} \label{lem:Arch:Whitt:fam:bound}
Assume $F$ is Archimedean.  Let $\varphi \in \mathcal{W}(I(\sigma),\psi)$ and let $\mathcal{K} \subset \mathfrak{a}_{M\CC}^{*}$ be a compact set.  There is an integer $N_1>0$ depending on $\varphi$ and $\mathcal{K}$ such that 
for $(a,k) \in  T_r(F) \times K_r,$ $N_2\in \ZZ_{\geq 0},$  and $D \in U(\mathfrak{gl}_r)$ 
one has 
\begin{align*}
|\mathcal{R}(D)W_{\varphi_\lambda}(ak)| \ll_{\varphi,\mathcal{K},D,N_2} \norm{a}^{N_1}\prod_{i=1}^{r-1}(1+|\alpha_i(a)|)^{-N_2}
\end{align*}
for $\lambda \in \mathcal{K}.$   
\end{lem}
\begin{proof}
Every discrete series representation of $\GL_2(\RR)$ may be realized as a subrepresentation of a suitable principal series representation.  Thus the lemma follows from the estimate for principal series representations given in \cite[Proposition 3.3]{JacquetPerfectRS}.   
\end{proof}

\begin{lem} \label{lem:nonArch:Whitt:fam:bound}
Assume $F$ is non-Archimedean.  Let $\varphi \in \mathcal{W}(I(\sigma),\psi)$ and let $\mathcal{K} \subset \mathfrak{a}_{M\RR}^{*}$ be a compact set.  There is an integer $N>0$ depending on $\mathcal{K},$ and a compact set $\Omega \subset F^{r-1}$ depending on  $\psi$ and $\varphi,$ such that 
for $(a,k) \in  T_r(F) \times K_r$ and $\mathrm{Re}(\lambda) \in \mathcal{K}$
one has 
\begin{align*}
|W_{\varphi_\lambda}(ak)| \ll_{\varphi,\mathcal{K},m'} \norm{a}^{N}
\one_{\Omega}(\alpha_1(a),\dots,\alpha_{r-1}(a)).
\end{align*}
\end{lem}
\begin{proof}
 The Whittaker functions $W_{\varphi_\lambda}$ are fixed by a compact open subgroup $K$ independent of $\lambda.$  Thus we can assume $k=I_r.$  

  There is a compact set $\Omega \subset F^{r-1},$ depending only on $\psi$ and $U_r(F) \cap K,$  such that if $(\alpha_1(a),\dots,\alpha_{r-1}(a)) \not \in \Omega$ then  $\psi(ana^{-1}) \neq 1$  for some $n \in U_{r}(F) \cap K.$  On the other hand, for $n \in U_r(F) \cap K$ one has 
$W_{\varphi_{\lambda}}(a)=W_{\varphi_{\lambda}}(an)=\psi(ana^{-1})W_{\varphi_{\lambda}}(a).$  These two observations imply the bound on the support of $W_{\varphi_{\lambda}}(ak)$ asserted in the lemma.

We now bound the magnitude.  We may embed $\sigma$ into the induction of a (possibly nonunitary) supercuspidal representation \cite[Theorem 8.3.5]{Getz:Hahn}.  Thus at the expense of enlarging $\mathcal{K}$ we may assume $\sigma$ is supercuspidal.  
 
For all $\lambda$ we have the identity
\begin{align*}
W_{\varphi_{\lambda}}(a)&=\int_{N_{\overleftarrow{P}}(F)}^{\mathrm{st}}I(\sigma_\lambda)\left(w_Pna\right)\varphi(I_r) \overline{\psi}(n)dn,
\end{align*}
where the integral should be understood as a stable integral; in other words, the integral is taken over a sufficiently large compact open subgroup $K_{N_{\overleftarrow{P}}}$ of $N_{\overleftarrow{P}}(F).$  If $K_{N_{\overleftarrow{P}}}$ is large enough, then the integral is independent of the choice of $K_{N_{\overleftarrow{P}}},$ and $K_{N_{\overleftarrow{P}}}$ may be chosen independently of $\lambda$ \cite[\S 3.1]{CPS:derivatives}. We change variables $n \mapsto ana^{-1}$ to see that the above is 
\begin{align*}
\delta_{\overleftarrow{P}}(a)\int_{N_{\overleftarrow{P}}(F)}^{\mathrm{st}}I(\sigma_\lambda)\left(w_Pan\right)\varphi(I_r) \overline{\psi}(ana^{-1})dn.
\end{align*}
Since $\varphi$ is realized in the Whittaker model, it is compactly supported as a function of $M(F)^1$ modulo the action of $U_r(F) \cap M(F)$ \cite[Corollary 6.5]{CasselmanShalika}.  The desired bound follows.  
\end{proof}

Let $
    \mathcal{B}_{\sigma}$
be an orthonormal basis of $I(\sigma)$ with respect to the pairing
\begin{align} \label{Bsiginnerpro}
    \left(\varphi_1,\varphi_2 \right):= \int_{P(F) \backslash \GL_r(F)} (\varphi_1(g),\overline{\varphi}_2(g))^{\mathrm{Whitt}} d\dot{g}.
\end{align}
Here $(\varphi_1(g),\overline{\varphi}_2(g))^{\mathrm{Whitt}}$ is defined by extending \eqref{eq:Whitt} to products of general linear groups in the natural manner.
We assume that the basis consists of $K_r$-finite vectors.  
We point out that 
\begin{align} \label{base:twist}
    \mathcal{B}_{\sigma}(\lambda):=\{\varphi_\lambda:\varphi \in \mathcal{B}_{\sigma}\}
\end{align}
is an orthonormal basis of $I(\sigma_\lambda)$ for all $\lambda \in i \mathfrak{a}_M.$

\begin{prop} \label{prop:Whitt:exp}
For $f \in \mathcal{S}(\GL_r(F)),$  $\sigma \in \mathrm{Irr}_2(M)$ and $\lambda \in i\mathfrak{a}_M^*,$ one has
\begin{align*} 
    W_{f_{I(\sigma_\lambda)}}(g_1,g_2)=\sum_{\varphi \in \mathcal{B}_{\sigma}}\overline{W}_{\varphi_{\lambda}}(g_1)I(\sigma_\lambda)(f^\vee)W_{\varphi_{\lambda}}(g_2).
\end{align*}
The sum on the right is convergent in $\mathcal{C}^w(U_{r,r}(F) \backslash \GL_{r,r}(F),\overline{\psi} \otimes \psi).$
\end{prop}
\begin{proof}  In view of \eqref{Whitt:exp},
the proposition follows from the fact that 
\begin{align*} 
(W_{\varphi_{1\lambda}},W_{\varphi_{2\lambda}})^{\mathrm{Whitt}}=(\varphi_1,\varphi_2)
\end{align*}
by \cite[Proposition A.2]{FLO}. 
\end{proof}
The difference between the expression in  \eqref{Whitt:exp} and Proposition \ref{prop:Whitt:exp} is that in the latter expression the space $\mathcal{B}_{\sigma}$ is fixed and now the functions $W_{\varphi_{\lambda}}$ vary in a family.  We point out that though we assume that $\mathcal{B}_{\sigma}$  consists of $K_r$-finite vectors, in the Archimedean case the function $f$ is not assumed to be $K_r$-finite on the right or left.

We point out that $w_P^{-1}\widetilde{w}_rw_P\overline{P}=Pw_P^{-1}\widetilde{w}_rw_P.$  Thus for $\varphi \in \mathrm{Ind}_{P}^{\GL_r}(\sigma)$ we can define
$$
\widetilde{\varphi}(g):=\varphi(w_P^{-1} \widetilde{w}_rw_Pg^{-t}) \in \mathrm{Ind}_P^{\GL_r}(m \mapsto \sigma((w_P^{-1} \widetilde{w}_rw_P)m^{-t}(w_P^{-1} \widetilde{w}_rw_P)^{-1}) \cong \mathrm{Ind}_P^G(\sigma^\vee).
$$

\begin{lem} \label{lem:duals}
    The map $\widetilde{\,\cdot\, }:\mathrm{Ind}_{P}^{\GL_r}(\sigma) \to \mathrm{Ind}_P^{\GL_r}(m \mapsto \sigma((w_P^{-1} \widetilde{w}_rw_P)m^{-t}(w_P^{-1} \widetilde{w}_rw_P)^{-1})$ is an isometry with respect to \eqref{Bsiginnerpro}, and 
    $$
    \widetilde{W}_{\varphi}(g)=W_{\widetilde{\varphi}}(g).
    $$
\end{lem}
\begin{proof}
The first assertion is clear. 
 For the second, we compute
\begin{align*}
    \widetilde{W}_{\varphi}(g)&=\int_{N_{\overleftarrow{P}}(F)}I(\sigma_\lambda)\left(w_P^{-1}n\widetilde{w}_rg^{-t}\right)\varphi(I_r) \overline{\psi}(n)dn\\
    &=\int_{N_{\overline{P}}(F)}I(\sigma_\lambda)\left(w_P^{-1}\widetilde{w}_rn^{-t}g^{-t}\right)\varphi(I_r) \overline{\psi}(n)dn
    =W_{\widetilde{\varphi}}(g).
\end{align*}
\end{proof}

For the remaining assertions of this section, let $M$ be a semistandard Levi subgroup of $\GL_r,$ let $\sigma \in \mathrm{Irr}_2(M),$ and let $\varphi$ be a $K_r$-finite vector in $I(\sigma).$

\begin{thm} \label{thm:Whitt:analy} Let $D$ be a constant coefficient differential operator on $i\mathfrak{a}_M^*.$  
For $d \gg_{D,\varphi} 1$ the following hold:
\begin{enumerate}
    \item[(a)] If $F$ is non-Archimedean, then $\nu_{-d}(DW_{\varphi_\lambda}) \ll_{D,\varphi} 1.$
    \item[(b)] If $F$ is Archimedean and $u \in U(\mathfrak{gl}_r),$  there is a $d'>0$ such that  
    $$
    \nu_{u,-d}(DW_{\varphi_\lambda}) \ll_{u,D,\varphi} (1+|\lambda|)^{d'} .
    $$
\end{enumerate}
In either case for $d \gg_{D,\varphi}1$  one has a continuous map
\begin{align*}
DW_{\varphi_{(\cdot)}}:i\mathfrak{a}_M^\ast \lto \mathcal{C}_d(U_r(F) \backslash \GL_r(F),\psi).
\end{align*}
\end{thm}
\begin{proof}
In the non-Archimedean case, the bound is \cite[Lemme 5.7]{Delorme:Whitt}. For Archimedean $F,$ the bound is \cite[Corollary 16.3]{vandenBan:uniformtempered}. We remark that in loc. cit. $d$ depends on  $u,$ but the dependence on $u$ can be absorbed into the dependence on $\varphi.$ \quash{\textcolor{red}{Can we make this continuous as a function of $\varphi$ in any sense?}\textcolor{blue}{\cite[Corollary 15.6]{vandenBan:uniformtempered} has a continuous seminorm in $\varphi$}}

By \cite[(4.1)]{Delorme:Whitt} and \cite[Proposition 16.1]{vandenBan:uniformtempered} the map
\begin{align*}
    W_{\varphi_{(\cdot)}}:\mathfrak{a}_{M\CC}^* \lto C^\infty(U_{r}(F) \backslash \GL_r(F),\psi)
\end{align*}
is holomorphic. 
  Our previous bounds imply $DW_{\varphi_{(\cdot)}}|_{i\mathfrak{a}_M^*}$ has image in $\mathcal{C}_d(U_{r}(F)\backslash \GL_r(F),\psi).$ To see the induced map is continuous, let $S$ be a $\CC$-basis of constant coefficient differential operators of first order on $i\mathfrak{a}_M^\ast.$  Let $\mathcal{K} \subseteq i\mathfrak{a}_M^*$ be a compact set. If $\lambda,\lambda' \in \mathcal{K},$ then by the mean value theorem
\begin{align*}
    \nu_{-d}(DW_{\varphi_{\lambda}}-DW_{\varphi_{\lambda'}})\ll \norm{\lambda-\lambda'} \sup_{D'\in S}\sup_{\lambda'' \in \mathcal{K}}\nu_{-d}(D'DW_{\varphi_{\lambda''}}).
\end{align*}
When $F$ is Archimedean, the same argument yields an estimate with $\nu_{-d}$ replaced by $\nu_{u,-d}$ for any $u \in U(\mathfrak{gl}_r).$ 
This implies continuity.
\end{proof}

\begin{prop} \label{prop:unif}
For $d>0$ large, there is a continuous map
\begin{align*}
i\mathfrak{a}_M^* \times \mathcal{C}(\GL_r(F)) &\lto \mathcal{C}_d(U_{r,r}(F) \backslash \GL_{r,r}(F),\overline{\psi} \otimes \psi) \\
(\lambda,f) &\longmapsto \overline{W}_{\varphi_{\lambda}}I(\sigma_\lambda)(f^{\vee})W_{\varphi_{\lambda}}.
\end{align*}
Moreover, if $f$ is right $K_r$-finite, then for any continuous seminorm $\nu$ on $\mathcal{C}_d(U_{r,r}(F) \backslash \GL_{r,r}(F),\overline{\psi} \otimes \psi)$ and any $N \geq 0$,
\begin{align} \label{cont:bound}
\nu(\overline{W}_{\varphi_{\lambda}}I(\sigma_\lambda)(f^{\vee})W_{\varphi_{\lambda}}) \ll_{\varphi,N,f} (1+|\lambda|)^{-N}.
\end{align}
\end{prop}

\begin{proof}
The first assertion follows from Lemma \ref{lem:cont} and Theorem \ref{thm:Whitt:analy}. Thus to complete the proof it suffices to prove the bound \eqref{cont:bound}. It is trivial for non-Archimedean $F$ since $i\mathfrak{a}_{M,F}^*$  is compact. Suppose $F$ is Archimedean. We claim for any $N\ge 0$, $\nu_{-d}(I(\sigma_\lambda)(f^{\vee})W_{\varphi_{\lambda}})\ll_{\varphi,f,N} (1+|\lambda|)^{-N}.$  This combined with Theorem \ref{thm:Whitt:analy} implies \eqref{cont:bound}.

Let $Z$ be an element in the center of $U(\mathfrak{gl}_r).$ Let $\chi_{I(\sigma_\lambda)}$ be the infinitesimal character of $I(\sigma_\lambda).$  Then  we have
$$
\chi_{I(\sigma_\lambda)}(Z)I(\sigma_\lambda)(f^\vee)W_{\varphi_{\lambda}}= I(\sigma_\lambda)(f^\vee)\mathcal{R}(Z)W_{\varphi_{\lambda}}= I(\sigma_\lambda)(f^\vee*Z)W_{\varphi_{\lambda}}.
$$
By Theorem \ref{thm:Whitt:analy} $\nu_{-d}(I(\sigma_\lambda)(f^\vee*Z)W_{\varphi_\lambda}) \ll_{\varphi,f,Z}(1+|\lambda|)^{d'}$ where $d'$ is independent of $\varphi,f,Z$ (this is where we use the assumption that $f$ is right $K_r$-finite). Thus it suffices to show for any given $N\ge 0,$ we can choose $Z$ such that $\chi_{I(\sigma_\lambda)}(Z)\gg_{\sigma} (1+|\lambda|)^N$. This follows from the Harish-Chandra isomorphism.
\end{proof}

Arguing as in the previous proof, one obtains the following corollary:
\begin{cor} \label{cor:unif} For any constant 
coefficient differential operator $D$ on $i\mathfrak{a}_M^*$ and $d \gg_{D,\varphi} 1,$ there is a continuous map
   \begin{align*}
   i\mathfrak{a}_M^\ast \times \mathcal{C}(\GL_r(F)) &\lto \mathcal{C}_d(U_r(F) \backslash \GL_r(F),\psi)\\
   (\lambda,f) &\longmapsto I(\sigma_{\lambda})(f)DW_{\varphi_\lambda}.
   \end{align*}
Moreover, for any continuous seminorm $\nu$ on $\mathcal{C}_d(U_r(F) \backslash \GL_r(F),\overline{\psi} \otimes \psi)$ and any $N \geq 0$
$$
\nu(I(\sigma_{\lambda})(f)DW_{\varphi_\lambda}) \ll_{\varphi,N,f}(1+|\lambda|)^{-N}.
$$ \qed
\end{cor}

\section{The Harish-Chandra space of $(X_\ell(F),\mathcal{H})$}\label{sec:HC-X_ell}

\subsection{Some spherical varieties} Recall that we
defined $X_r^\circ:=N \backslash \GL_r$ in \eqref{XP}.  As in \eqref{Mvee} we denote by $\mathcal{M}^\circ$ the space of rank $1$ matrices in $M_{r-2,2}.$ 
There is a morphism
\begin{align*} 
  \mathrm{pr}: \mathcal{M}^{\circ} \lto \mathbb{P}^1
\end{align*}
that assigns to $(m_1,m_2) \in \mathcal{M}^{\circ}(R)$ the  morphism 
\begin{align*}
R^2 &\lto Rm_1+Rm_2 \\
   (v_1,v_2) &\longmapsto v_1m_1+v_2m_2.
\end{align*}
The image is a projective rank $1$-submodule of $R^{r-2}$ by 
\cite[Lemma 16.17]{Gortz_Wedhorn}.
Over $F,$ we can describe this morphism more concretely as follows: given $(m_1,m_2) \in \mathcal{M}^{\circ}(F)$ there is an $x \in F^{r-2}$ and $(c_1,c_2) \in F^2-\{(0,0)\}$ such that $(m_1,m_2)=(c_1x,c_2x).$  Then $\mathrm{pr}(m_1,m_2)$ is the line spanned by $(c_1,c_2).$  The fiber over $\ell \in \mathbb{P}^1(F)$ is denoted $\mathcal{M}_{\ell}^\circ,$ and we denote by \index{$\mathcal{M}_{\ell}$}
\begin{align}\label{Mell}
\mathcal{M}_{\ell} \cong \GG_a^{r-2}
\end{align}
the closure of $\mathcal{M}_{\ell}^\circ$ in $\mathcal{M}.$  
Let
\begin{align} \label{Plperp}
X^\circ:=X_r^\circ \times \mathcal{M}^\circ \lto \mathbb{P}^1
\end{align}
be the map given by projection to the second factor followed by $\mathrm{pr}.$  For $\ell \in \mathbb{P}^1(F)$ we denote by $X_{\ell}^\circ=X_r^\circ \times \mathcal{M}_{\ell}^\circ$ the fiber over $\ell$.

Fix $\beta \in F^\times.$  For $h\in \GL_2(F)$, let \index{$h^{\iota}$}
\begin{align} \label{iota}
    h^{\iota}:=h^{\iota_{\beta}}:=\begin{psmatrix} 1 & \\& \beta \end{psmatrix}h^{-t}\begin{psmatrix} 1 & \\& \beta \end{psmatrix}^{-1}.
\end{align}
The motivation for this definition comes from \cite{GGHL}.

For $\ell \in F^2-\{0\},$  
choose $h_\ell \in \mathrm{GL}_2(F)$ such that \index{$h_{\ell}$}
\begin{align} \label{hl} 
(1,0)h_{\ell}^\iota=\ell .
\end{align}
Let \index{$m_{\ell}$}
\begin{align}
m_{\ell}:=(e_{r-2},0)h_\ell^{\iota}.
\end{align}\quash{
We note that while $h_\ell$ is only unique up to left multiplication by $\begin{psmatrix}
    1 & \\
     & -1
\end{psmatrix}$, $m_{\ell}$ only depends on $\ell$.}

As in the introduction, we have an action \index{$\mathcal{R}$}
\begin{align*}
 \mathcal{R}:   X^\circ(R) \times \GL_{r,r-2,2}(R) &\lto X^\circ(R)\\
    ((x,m),(g,g',h)) &\longmapsto \left( \begin{psmatrix} g' & \\ & h \end{psmatrix}^{-1}xg,g'^tmh^{\iota}\right).
\end{align*}
Let 
\begin{align*}
    (\GL_2)_{\ell}(R):=\left\{h \in \GL_2(R): m_\ell h^{\iota}=m_\ell\right\}.
\end{align*}
Then the action  $\mathcal{R}$ on $X$ induces an action (still denoted $\mathcal{R}$) of $\GL_{r,r-2}\times (\GL_2)_{\ell}$ on $X_{\ell}^\circ.$
We let \index{$\langle\,,\,\rangle$}
\begin{align*}\begin{split}
\langle\,,\,\rangle:M_{r-2,2}(R) \times M_{r-2,2}(R) &\lto R\\
((u_1,u_2),(u_1',u_2')) &\longmapsto u_1^tu_1'+\beta u_2^tu_2'. \end{split}
\end{align*}
Thus 
\begin{align} \label{M:pair:equiv}
\langle g'^tm_1h^{\iota},m_2\rangle=\langle m_1,g'm_2h^{-1} \rangle
\end{align}
for $(g',h,m_1,m_2) \in \GL_{r-2}(R) \times \GL_2(R) \times M_{r-2,2}(R)^2.$

\subsection{A family of affine $\Psi$-bundles}\label{ssec:fs}

There is a unique Hermitian line bundle $\mathcal{H}:=\mathcal{H}_{\overline{\psi}}$ over $X^\circ(F)$ whose space of smooth sections $C^\infty(X^\circ(F),\mathcal{H})$ is
$$
\left\{f \in C^\infty(\GL_r(F) \times \mathcal{M}^{\circ}(F)): f\left(\begin{psmatrix} I_{r-2} & z\\ & I_2 \end{psmatrix}g,m\right)=\overline{\psi}\left( \langle m,z \rangle \right)f(g,m)\right\}.
$$
This is the space of functions underlying an affine $\Psi$-bundle in the sense of \cite{GGHL}.

When $F$ is non-Archimedean, we define \index{$\mathcal{S}(X^\circ(F),\mathcal{H})$}
$\mathcal{S}(X^\circ(F),\mathcal{H}):=C_c^\infty(X^\circ(F),\mathcal{H}).$ 
For Archimedean $F$, we let $\mathcal{S}(X^\circ(F),\mathcal{H})$ be the subspace of $C^\infty(X^\circ(F),\mathcal{H})$ consisting of functions $f$ such that
\begin{align*}
\left|\mathcal{R}(D)f(x)\right| \ll_{N,D}    \norm{x}_{X}^{-N}
\end{align*}
for all $N\in \RR$ and $D \in U(\mathfrak{gl}_{r,r-2,2}).$  This definition is modeled on one in \cite[\S 2.4]{BP:GLn}.  If $Z \subset \mathcal{M}^\circ$ is any smooth closed subscheme, we continue to write $\mathcal{H}$ for the pullback of $\mathcal{H}$ along $X^\circ_r \times Z \to X^\circ_r \times \mathcal{M}^\circ.$
We let
$$
\mathcal{S}(X^{\circ}_r(F) \times Z(F),\mathcal{H})<C^\infty(X^\circ_r(F) \times Z(F),\mathcal{H})
$$
be the image of $\mathcal{S}(X^\circ(F),\mathcal{H})$ under the restriction map.

By the Iwasawa decomposition, any $x\in X_\ell^\circ(F)$  may be written as
\begin{align}\label{coord}
    x=\left(\begin{psmatrix}k'^{-1}\begin{psmatrix}t_{r-2}g_1 & t_{r-2}y_1&\\ & t_{r-2} \end{psmatrix}& \\ & h_\ell^{-1}\begin{psmatrix} 1 & y_2\\ & t_{r}^{-1}\end{psmatrix} \end{psmatrix}zk,k'^tam_\ell \right),
\end{align}
where
\begin{align*} 
(z,g_1,y_1,t_{r-2},t_r,y_2,k,k',a) \in  F^\times \times \GL_{r-3}(F) \times F^{r-3} \times F^\times \times F^\times \times F \times K_r \times K_{r-2} \times F^\times.
\end{align*} 
Using the coordinates \eqref{coord}, we define \index{$\Xi_{X_\ell^\circ}$}
\begin{align}  \label{XiXP} \begin{split}
\Xi_{X_\ell^\circ}:X^\circ_\ell(F)  &\lto \RR_{>0}\\  x&\longmapsto \frac{|t_{r-2}|^{r-2}|t_r|^{(r-2)/2}|\det g_1|^{5/4}\Xi^{1/2}_{\GL_{r-3}}(g_1)\Xi_{\GL_{r-2}}^{1/2}\begin{psmatrix}
        g_1 &y_1 \\
          & 1
    \end{psmatrix}}{|a|^{(r-2)/2}\max(1,|y_2|)^{1/2}}\end{split}
\end{align}
and \index{$\sigma_{X_\ell^\circ}$}
\begin{align} \label{sigX} \begin{split}
\sigma_{X_\ell^\circ}(x):&=\sigma\Big(z,\begin{psmatrix}t_{r-2}g_1 & t_{r-2}y_1\\ & t_{r-2} \end{psmatrix},\begin{psmatrix} 1 & y_2 \\& t_{r}^{-1}\end{psmatrix},a\Big).
\end{split}
\end{align} 
Note that both functions are well-defined.

We define $L^2(X_{\ell}^\circ(F),\mathcal{H})$ using the norm
\begin{align*} 
\norm{f}_2^2:=\int_{N(F) \backslash \GL_r(F) \times \mathcal{P}_{r-2}(F) \backslash \GL_{r-2}(F)}|f|^2(g,g'^tm_{\ell}) |\det g'|d\dot{g}d\dot{g}'.
\end{align*}
Here $d\dot{g}'$ is defined as in \cite[\S 2.5]{BP:GLn}.  Under the usual $\GL_{r-2}(F)$-equivariant isomorphism $\mathcal{P}_{r-2}(F) \backslash \GL_{r-2}(F)\tilde{\to} F^{r-2}-\{0\}$ the measure $|\det g'|d\dot{g}'$ corresponds to the restriction of a Haar measure on $F^{r-2}.$

\begin{lem} \label{lem:Xi:bound}
There is a $d>0$ such that $\Xi_{X_\ell^\circ}\sigma_{X_\ell^\circ}^{-d} \in L^2(X_\ell^\circ(F),\mathcal{H}).$
\end{lem}
\begin{proof}
By the Iwasawa decomposition it suffices to show that for $d$ large
\begin{align*}
&\int_{\GL_{r-3}(F)\times F^{r-3}} |\det g_1|^{-1/2}\Xi_{\GL_{r-3}}(g_1)\Xi_{\GL_{r-2}}\begin{psmatrix}
    g_1 & y_1 \\
    & 1
\end{psmatrix}\sigma\begin{psmatrix} g_1 & y_1\\ & 1 \end{psmatrix}^{-d} dg_1 dy_1\\
&\times \int_{(F^\times)^4\times F} \sigma(z,t_{r-2},a)^{-d}\sigma\begin{psmatrix} 1 & y_2 \\
& t_{r}^{-1}\end{psmatrix}^{-d}\max(1,|y_2|)^{-1} d^\times z d^\times t_{r-2} d^\times a d^\times t_r dy_2<\infty.
\end{align*}
The lower integral converges for $d$ large. By Proposition \ref{prop:HCfunction}\ref{descent} and a change of variables $y_1\mapsto g_1y_1$, for any $d'>0,$ there is a $d>0$ such that the integral over $y_1$ is convergent, and it is dominated by
\begin{align*}
&\int_{\GL_{r-3}(F)} \Xi_{\GL_{r-3}}^2(g_1)\sigma(g_1)^{-d'}dg_1.
\end{align*}
This is finite for $d'$ large by  Proposition \ref{prop:HCfunction}\ref{convergence}. 
\end{proof}

For $f \in C^\infty(X^\circ_\ell(F),\mathcal{H})$ and $d\in \RR,$ let \index{$\mu_d$}
\begin{align*}\begin{split}
    \mu_d(f):&=\mathrm{sup}_{x \in X^\circ(F)}|f(x)|\Xi_{X_\ell^\circ}(x)^{-1}\sigma_{X_\ell^\circ}^d(x).\\
   \quash{\mu_d'(f):&=\mathrm{sup}_{x \in X^\circ(F)}|f(x)|\Xi_{X_\ell^\circ}'(x)^{-1}\sigma_{X_\ell^\circ}'(x)}\end{split}
\end{align*} 
When $F$ is Archimedean, let
\begin{align*} 
\mathcal{C}_{d}(X_\ell^\circ(F),\mathcal{H})=\left\{ f \in C^\infty(X_{\ell}^\circ(F),\mathcal{H}):\mu_{-d}(\mathcal{R}(u)f)<\infty \textrm{ for all }u \in U(\mathfrak{gl}_{r,r-2} \oplus \mathrm{Lie}\, (\GL_{2})_{\ell})\right\}.
\end{align*}
When $F$ is non-Archimedean, let
\begin{align*}
\mathcal{C}_d(X_\ell^\circ(F),\mathcal{H})=\bigcup_{K}\left\{ f \in C^\infty(X_{\ell}^\circ(F),\mathcal{H})^K:\mu_{-d}(f)<\infty\right\},
\end{align*}
where the union is over compact open subgroups $K < \GL_{r,r-2}(F) \times (\GL_2)_\ell(F).$
We then set \index{$\mathcal{C}(X_\ell^\circ(F),\mathcal{H})$}
\begin{align*} 
\mathcal{C}(X_\ell^\circ(F),\mathcal{H}):=\bigcap_{d}\mathcal{C}_d(X_\ell^\circ(F),\mathcal{H}).
\end{align*}
This is an LF-space, and in fact a Fr\'echet space if $F$ is Archimedean.

The following is a corollary of Lemma \ref{lem:Xi:bound}:

\begin{cor}\label{lem:L2}
We have a continuous inclusion $\mathcal{C}( X_\ell^\circ(F),\mathcal{H}) \hookrightarrow L^2(X_\ell^\circ(F),\mathcal{H}).$ \qed
\end{cor}

We define \index{$\mathcal{R}_u$}
\begin{align} \label{Ru} \begin{split}
\mathcal{R}_u:\GL_{r,r-2,2}(F) \times L^2(X_{\ell h^{\iota}}^{\circ}(F),\mathcal{H}) &\lto L^2(X_{\ell}^\circ(F),\mathcal{H})\\
((g,g',h),f) &\longmapsto \frac{|\det g'|^{3/2}}{|\det h|^{(r-2)/2}}\mathcal{R}(g,g',h)f.
\end{split}
\end{align}
One checks that $\mathcal{R}_u(g,g',h)$ sends $\mathcal{C}(X_{\ell h^{\iota}}^\circ(F),\mathcal{H})$ to $\mathcal{C}(X^\circ_{\ell}(F),\mathcal{H}).$  It is not an isometry in general.  However,  $\mathcal{R}_u|_{\GL_{r,r-2}(F) \times (\GL_{2})_\ell(F)}$ is a unitary representation of $\GL_{r,r-2}(F) \times (\GL_{2})_\ell(F).$

For $f \in \mathcal{S}(\GL_{r,r-2}(F) )$ and $x=(g,g'^tm_\ell) \in X_\ell^\circ(F)$  define \index{$I_f$}
\begin{align} \label{If} \begin{split}
I_f(x)&:=I_{f,\overline{\psi}}(x)\\
&:=\int_{\mathcal{P}_{r-2}(F) \times M_{r-2,2}(F)}\overline{\psi}(\langle m_{\ell} ,z \rangle)f^\vee\left(g^{-1}\begin{psmatrix} g' & \\ & I_2 \end{psmatrix}^{-1}\begin{psmatrix}p & z\\ & I_2 \end{psmatrix},g'^{-1}p\right) \frac{d_\ell p dz}{|\det p g'^3|^{1/2}}
.\end{split}
\end{align} 
For $h\in \mathcal{P}_{r-2}(F),$ by changing variables $(p,z)\mapsto (hp,hz)$ we see that
\begin{align*}
    &I_f\left(g,(hg')^tm_{\ell}\right)\\
    &=\int_{\mathcal{P}_{r-2}(F) \times M_{r-2,2}(F)}\overline{\psi}(\langle m_{\ell} ,z \rangle)f^\vee\left(g^{-1}\begin{psmatrix} g' & \\ & I_2 \end{psmatrix}^{-1}\begin{psmatrix}h^{-1}p & h^{-1}z\\ & I_2 \end{psmatrix},(hg')^{-1}p\right) \frac{d_\ell p dz}{|\det p|^{1/2}|\det hg'|^{3/2}}\\
    &=I_f\left(g,g'^tm_{\ell}\right).
\end{align*}
We therefore obtain a well-defined surjection
 \begin{align} \label{I1}
     I_{(\cdot)}:\mathcal{S}(\GL_{r,r-2}(F)) &\lto \mathcal{S}(X^\circ_\ell(F) ,\mathcal{H})
 \end{align}
that is continuous in the Archimedean case.  
\quash{ \textcolor{red}{This isn't correct as the change of variables doesn't make sense.}
Suppose $\ell'=a\ell$ for some $a\in F^\times.$ Then by changing variables $(p,z)\mapsto (a^{-1}p,a^{-1}z)$, we have
\begin{align*}
        &I_f(g,g'^tm_{\ell'})\\
        &=\int_{\mathcal{P}_{r-2}(F) \times M_{r-2,2}(F)}\overline{\psi}(\langle m_{\ell'} ,z \rangle)f^\vee\left(g^{-1}\begin{psmatrix} g'  & \\ & I_2 \end{psmatrix}^{-1}\begin{psmatrix}p & z\\ & I_2 \end{psmatrix},g'^{-1}p\right) \frac{d_\ell p dz}{|\det p|^\frac{1}{2}}|\det g'|^{-\frac{3}{2}}\\
         &=\int_{\mathcal{P}_{r-2}(F) \times M_{r-2,2}(F)}\overline{\psi}(\langle m_{\ell} ,z \rangle)f^\vee\left(g^{-1}\begin{psmatrix} ag' & \\ & I_2 \end{psmatrix}^{-1}\begin{psmatrix}p & z\\ & I_2 \end{psmatrix},(ag')^{-1}p\right) \frac{d_\ell p dz}{|\det p|^\frac{1}{2}}|\det ag'|^{-\frac{3}{2}}\\
        &=I_f(g,(ag')^tm_{\ell}).
\end{align*}
Thus the map $I_f$ only depends on the image of $\ell$ in $\mathbb{P}^1(F)$. }

\begin{lem}\label{lem:extension}
Let $f\in \mathcal{C}(\GL_{r,r-2}(F)).$ Then the integral defining $I_f$ is absolutely convergent and $I_f\in \mathcal{C}(X_\ell^\circ(F),\mathcal{H}).$  The map 
$I_{(\cdot)}:\mathcal{C}(\GL_{r,r-2}(F)) \to\mathcal{C}(X_\ell^\circ(F),\mathcal{H})$ is continuous.
\end{lem}

\begin{proof}
Write $x=\left(\begin{psmatrix}
    k'^{-1}g_1 & \\
     & g_2
\end{psmatrix}k, k'^tam_\ell\right).$ Then for any $d_1,d_2>0,$
\begin{align*}
    |a|^{3(r-2)/2}|I_f(x)|&\le \int_{\mathcal{P}_{r-2}(F) \times M_{r-2,2}(F)}|f^\vee|\left(k^{-1}\begin{psmatrix}
        (ag_1)^{-1} &\\
         & g_2^{-1}
    \end{psmatrix}\begin{psmatrix}p & z\\ & I_2 \end{psmatrix},k'^{-1}a^{-1}p\right) \frac{d_\ell p dz}{|\det p|^{1/2}}\\
    &\leq \nu_{d_1,d_2}(f) \int_{\mathcal{P}_{r-2}(F) \times M_{r-2,2}(F)}\Xi_{\GL_{r,r-2}}\left(\begin{psmatrix}
       (ag_1)^{-1} &\\
         & g_2^{-1}
    \end{psmatrix}\begin{psmatrix}p & z\\ & I_2 \end{psmatrix},p\right)\\
    &\hspace{1in} \times \sigma\left(\begin{psmatrix}
        (ag_1)^{-1} &\\
         & g_2^{-1}
    \end{psmatrix}\begin{psmatrix}p & z\\ & I_2 \end{psmatrix}\right)^{-d_1}\sigma(a^{-1}p)^{-d_2} \frac{d_\ell p dz}{|\det p|^{1/2}},
\end{align*}
where $\nu_{d_1,d_2}(f)$ is a continuous seminorm on $\mathcal{C}(\GL_{r,r-2}(F)).$
Changing variables $z\mapsto pz$ and applying Proposition \ref{prop:HCfunction}\ref{descent}, for any $d,d_1'>0$ we can find $d_1$ large so that the above is dominated by 
\begin{align}\label{eq:intP}
   \frac{\Xi_{\GL_{2}}(g_2)|\det g_1||a|^{r-2}}{\sigma(g_2)^d|\det g_2|^{(r-2)/2}}\int_{\mathcal{P}_{r-2}(F)}\frac{\Xi_{\GL_{r-2}}(g_1^{-1}p)\Xi_{\GL_{r-2}}(p)|\det p|^{1/2}}{\sigma
        ((ag_1)^{-1}p)^{d_1'}\sigma(a^{-1}p)^{d_2}}d_\ell p.
\end{align}
Observe that $\sigma(a^{-1}p)\ge \sigma(a)$ and 
\begin{align*}
    \sigma((ag_1)^{-1}p)\sigma(a^{-1}p)=\sigma(a^{-1}p)\sigma(p^{-1}ag_1 )\gg \sigma(g_1)
\end{align*} by \eqref{sig:ineq}.
Therefore, choosing $d_1'$ so that $d_1'':=d_1'-d>0$ and $d_2=2d,$ we have that the integral in
 \eqref{eq:intP} is dominated by $\frac{\Xi_{\GL_{2}}(g_2)|\det g_1||a|^{r-2}}{\sigma(g_2)^d|\det g_2|^{(r-2)/2}\sigma(g_1)^{d}\sigma(a)^{d}}$ times
\begin{align} \label{eq:Pbound}
    \nonumber&\int_{\mathcal{P}_{r-2}(F)}\frac{\Xi_{\GL_{r-2}}(g_1^{-1}p)\Xi_{\GL_{r-2}}(p)}{\sigma
        ((ag_1)^{-1}p)^{d_1''}}|\det p|^{1/2}d_\ell p\\
    \nonumber&=\left(\int_{|\det p|\le 1}+\int_{|\det p|> 1}\right)\frac{\Xi_{\GL_{r-2}}(g_1^{-1}p)\Xi_{\GL_{r-2}}(p)}{\sigma
        ((ag_1)^{-1}p)^{d_1''}}|\det p|^{1/2}d_\ell p\\
    \nonumber&\le \int_{|\det p|\le 1} \left(\frac{\Xi_{\GL_{r-2}}(g_1^{-1}p)}{\sigma
        ((ag_1)^{-1}p)^{d_1''}}+\frac{\Xi_{\GL_{r-2}}(g_1p)}{\sigma
        (ag_1p)^{d_1''}}\right)\Xi_{\GL_{r-2}}(p)|\det p|^{1/2}d_\ell p\\  
    &\le \int_{\mathcal{P}_{r-2}(F)}\left(\frac{\Xi_{\GL_{r-2}}(g_1^{-1}p)}{\sigma
        ((ag_1)^{-1}p)^{d_1''}}+\frac{\Xi_{\GL_{r-2}}(g_1p)}{\sigma
        (ag_1p)^{d_1''}}\right)\Xi_{\GL_{r-2}}(p)|\det p|^{1/2}d_\ell p.
\end{align}
Note that for any $h\in \GL_{r-2}(F)$
\begin{align*}
    \sigma^{-d_1''}(h)\ll \int_{F^\times} \sigma^{-d_1''}(ah) d^\times a.
\end{align*}
Therefore, \eqref{eq:Pbound} is dominated by
\begin{align*}
\sum_{j \in \{\pm 1\}}\int_{F^\times \times \mathcal{P}_{r-2}(F)}\frac{\Xi_{\GL_{r-2}}(ag_1^{j}p)\Xi_{\GL_{r-2}}(ap)}{\sigma(ag_1^{j}p)^{d_1''}}d^\times ad_\ell p
\ll &\sum_{j\in\{\pm 1\}}\int_{\GL_{r-2}(F)}\frac{\Xi_{\GL_{r-2}}(g_1^{j}g)\Xi_{\GL_{r-2}}(g)}{\sigma(g_1^{j}g)^{d_1''}} dg\\
=&\sum_{j\in\{\pm 1\}}\int_{\GL_{r-2}(F)}\frac{\Xi_{\GL_{r-2}}(g)\Xi_{\GL_{r-2}}(g_1^{j}g)}{\sigma(g)^{d_1''}} dg
\end{align*}
by the Iwasawa decomposition. For any $k \in K_{r-2}$ this is 
$$\sum_{j\in\{\pm 1\}}\int_{\GL_{r-2}(F)}\frac{\Xi_{\GL_{r-2}}(g)\Xi_{\GL_{r-2}}(g_1^{j}kg)}{\sigma(g)^{d_1''}} dg.
$$
Integrating over $K_{r-2}$ and applying  Proposition \ref{prop:HCfunction}\ref{convergence} and \ref{doubling} the above is
\begin{align*}
    &\frac{1}{\mathrm{meas}_{dk}(K_{r-2})}\sum_{j \in \{\pm 1\}}\int_{\GL_{r-2}(F)}\frac{\Xi_{\GL_{r-2}}(g)}{\sigma(g)^{d_1''}}\int_{K_{r-2}}\Xi_{\GL_{r-2}}(g_1^jkg)dkdg\\&=\frac{2}{\mathrm{meas}_{dk}(K_{r-2})}\Xi_{\GL_{r-2}}(g_1)\int_{\GL_{r-2}(F)}\frac{\Xi_{\GL_{r-2}}^2(g)}{\sigma(g)^{d_1''}}dg \ll \Xi_{\GL_{r-2}}(g_1)
\end{align*}
 for $d_1''$ sufficiently large.

In conclusion, for any $d>0$
\begin{align} \label{If:bound}
    I_{f}(x)\leq \frac{\nu_{d}(f) \Xi_{\GL_{r-2}}(g_1)\Xi_{\GL_2}(g_2)|\det g_1|}{(\sigma(g_1)\sigma(a)\sigma(g_2))^{d}|\det g_2|^{(r-2)/2}|a|^{(r-2)/2}},
\end{align}
where $\nu_{d}(F)$ is a continuous seminorm on $\mathcal{C}(\GL_{r,r-2}(F)).$
Since $\mathcal{C}(\GL_{r,r-2}(F))$ is a smooth $\GL_{r,r-2}(F)$-module, we deduce that $I_f \in C^\infty(X^\circ_{\ell}(F),\mathcal{H}).$

To show that $I_f\in \mathcal{C}(X_\ell^\circ(F),\mathcal{H}),$ by \eqref{If:bound} it suffices to show there is a $d>0$ such that
\begin{align*}
    \Xi_{\GL_{r-2}}\begin{psmatrix}
       g & y\\
        & 1
    \end{psmatrix}&\ll |\det g|^{1/4}\Xi_{\GL_{r-3}}^{1/2}(g)\Xi_{\GL_{r-2}}^{1/2}\begin{psmatrix}
       g & y\\
        & 1
    \end{psmatrix}\sigma\begin{psmatrix}
       g & y\\
        & 1
    \end{psmatrix}^d, \\
    \Xi_{\GL_2}\begin{psmatrix}
        1 & y_2\\
          & t_{r}^{-1}
    \end{psmatrix}&\ll \max(1,|y_2|)^{-1/2}\sigma\begin{psmatrix}
        1 & y_2\\
          & t_{r}^{-1}
    \end{psmatrix}^d.
\end{align*}
Both inequalities follow from Lemma \ref{lem:Levi} and Lemma \ref{lem:parab} except for the second statement when $|y_2|\ge 2\max(|t_r|,|t_r|^{-1}).$

Assume $|y_2|\geq 2\max(|t_r|,|t_r|^{-1}).$  When $F$ is non-Archimedean,
\begin{align*}
    \begin{psmatrix}
        1 & y_2\\
          & t_r^{-1}
    \end{psmatrix}\in K_{2}\begin{psmatrix}
         y_2^{-1}t_r^{-1}& \\
         & y_2
    \end{psmatrix}K_2.
\end{align*}
When $F$ is Archimedean let $\varsigma_1\le \varsigma_2$ be the singular values (viewed as elements in $F$) of $ \begin{psmatrix}
        1 & y_2\\
          & t_r^{-1}
    \end{psmatrix}$. Then direct computation gives $|\varsigma_2|\asymp |y_2|$ and $|\varsigma_1|\asymp |t_r|^{-1}|y_2|^{-1}$.
Therefore, in any case by Proposition \ref{prop:HCfunction}\ref{delta:asymp} 
\begin{align*}
    \Xi_{\GL_2}\begin{psmatrix}
        1 & y_2\\
          & t_{r}^{-1}
    \end{psmatrix}\ll |y_2|^{-1}|t_r|^{-1/2}\sigma\begin{psmatrix}
        1 & y_2\\
          & t_{r}^{-1}
    \end{psmatrix}^d\ll \max(1,|y_2|)^{-1/2}\sigma\begin{psmatrix}
        1 & y_2\\
          & t_{r}^{-1}
    \end{psmatrix}^d.
\end{align*}
This completes the proof.
\end{proof}
\quash{
Let $d'>0.$  By \eqref{If:bound} and the Iwasawa decomposition we have
\begin{align}
\norm{I_f(x)\sigma^{d'}_{X^\circ_{\ell}}}_2 \leq \nu_{2d'}(f)^2\int_{N(F) \backslash \GL_r(F) \times F^\times }\frac{\sigma_{X_\ell^\circ}^{d'}(\begin{psmatrix}g_1 & \\ & g_2 \end{psmatrix},a'm_\ell)\Xi^2(g_1)\Xi^2(g_2)}{(\sigma(g_1)\sigma(a')\sigma(g_1))^{2d}}dg_1dg_2d^\times a
\end{align}
By \eqref{lem:products} and Proposition \ref{prop:HCfunction}\ref{convergence} the integral converges for $d'$ sufficiently large.  Thus $\norm{I_f(x)\sigma^{d'}_{X^\circ_{\ell}}}_2 \ll \nu_{2d'}(f)^2.$  Since $\mathcal{C}(\GL_{r,r-2}(F))$ is a smooth $\GL_{r,r-2}(F)$-module we we can now apply Lemma \ref{lem:equiv:nA} and Lemma \ref{lem:equiv:A} to deduce that $I_f \in \mathcal{C}(X_\ell^{\circ}(F),\mathcal{H})$ and that the map $I_{(\cdot)}$ is continuous.}

\section{Local zeta integrals}

\label{sec:loc:zeta}

In this section we define and study various local zeta integrals.  The convergence and continuity assertions will be applied later in the paper. The proofs in this section are quite technical.  \textbf{We recommend that the reader skip them and return to them later if necessary.  }

\subsection{Preliminary bounds}

Let
\begin{align}  \label{mr}
    m_r:\GL_r(F)\lto T_r(F)
\end{align}
be defined as follows: If $F$ is Archimedean, let $m_r(g)$ be the unique element of $A_{T_r}$ such that $g \in U_r(F)m_r(g)K_r.$ Here $A_{T_r}$ is the subgroup of $T_r(F)$ whose diagonal entries are positive real numbers. If $F$ is non-Archimedean, choose a uniformizer $\varpi$ of $F$ and let $m_r(g)$ be the unique element of $\{\nu(\varpi): \nu \in X_*(T_r)\}$ such that 
$g \in U_r(F)m_r(g)K_r.$

Let $w_r:=\begin{psmatrix} & & 1 \\ & \textrm{\reflectbox{$\ddots$}} & \\ 1 & & \end{psmatrix} \in \GL_{r}(\ZZ)$. We observe that there is an  automorphism 
 \begin{align*} 
 \jmath :U_r \backslash \GL_r \lto U_r \backslash \GL_r
 \end{align*}
 given on points by $\jmath(g):=w_r g^{-t}w_r^{-1}.$
Let $\overline{U}_r$ be the unipotent radical of the Borel subgroup $\overline{B}_r$ of lower triangular matrices. 

\begin{lem} \label{lem:upper}
    For $\overline{u} \in \overline{U}_{r}(F)$ and $t=\begin{psmatrix}t' & \\ & I_{r-j} \end{psmatrix} \in T_r(F),$ one has $
        |\prod_{i=1}^jm_{r}(t\overline{u})_i| \ll |\det t|.$
\end{lem}
\begin{proof}
Upon replacing $K_r$ by a conjugate if necessary, we assume without loss of generality that $\jmath$ preserves $K_r.$ 
It suffices to show that 
\begin{align} \label{STS}
\left|\prod_{i=1}^jm_{r}(t\overline{u})_i\right|^{-1}=\left|\prod_{i=1}^{j}\jmath(m_{r}(t\overline{u}))_{r-i+1}\right|
\end{align}
dominates $|\det t|^{-1}.$  Since
$$
\jmath(m_{r}(t\overline{u}))=m_{r}(\jmath(t\overline{u})))=m_r(\jmath(t)\jmath(\overline{u})),
$$ the right hand side of \eqref{STS} is asymptotic to the norm of the wedge product of rows $r-j+1$ to $r$ of $\jmath(t)\jmath(\overline{u})$ with respect to a $K_r$-invariant metric on $\wedge^{j}F^r.$  This norm is $\gg|\det t|^{-1}.$
\end{proof}

\begin{lem}\label{lem:decompositioncompare}  There is a constant $\gamma>0$ such that for any
    $f\in C(B_r(F)\backslash\GL_r(F))$ and $g_1\in \GL_r(F),$  
    \begin{align*}
\int_{K_r}f(kg_1)\Xi_{U_{r} \backslash \GL_r}(kg_1)dk
        =\gamma\int_{\overline{U}_{r}(F)} f(\bar{u}g_1)\Xi_{U_{r} \backslash \GL_r}(\bar{u}g_1)\Xi_{U_{r} \backslash \GL_r}(\bar{u})d\bar{u}.
    \end{align*}
\end{lem}
\begin{proof}
Choose $\widetilde{f} \in C_c(U_r(F) \backslash \GL_r(F))$ such that $\int_{T_r(F)}\widetilde{f}(tg)dt=f(g).$  Then by the Iwasawa and Bruhat decompositions of $\GL_r(F),$ up to positive constants both integrals are equal to 
\begin{align*}
\int_{U_{r}(F)\backslash \GL_{r}(F)} \widetilde{f}(gg_1)\Xi_{U_{r} \backslash \GL_r}(gg_1)\Xi_{U_{r} \backslash \GL_r}(g) d\dot{g}.
\end{align*}
\end{proof}

\begin{lem} \label{lem:sig:bound:stuff}
For $(\kappa,h) \in F^\times \times \GL_r(F) $ with $|\kappa| \geq 1$ and $g=\begin{psmatrix} t' & \\ & t'' \end{psmatrix} \in T_{r}(F)^+$
 as in Lemma \ref{lem:Levi},
one has
\begin{align*}
    \sigma_{U_{r} \backslash \GL_{r}}(t'\kappa^{-1}h )\sigma_{U_{r} \backslash \GL_{r}}(t'\kappa^{-1} hw_rg^{-1}) \gg \frac{\sigma(g)\sigma(\kappa)}{\sigma(\det h)^{2}}.
\end{align*}
\end{lem}

\begin{proof}
By \eqref{pullback} and \eqref{rev:sig:ineq} one has
\begin{align*}\sigma_{U_{r-3} \backslash \GL_{r-3}}(t'\kappa^{-1}h )\sigma_{U_{r-3} \backslash \GL_{r-3}}(t'\kappa^{-1} hw_rg^{-1}) &\gg \sigma (\det(t'\kappa^{-1}h))\sigma (\det (t''^{-1}\kappa^{-1} h))\\&\gg \frac{\sigma (\det (t'\kappa^{-1}) )\sigma (\det (t''^{-1}\kappa^{-1} ))}{\sigma(\det h)^2}.
\end{align*}
We point out that 
 $|\det t'|\leq 1$ and $|\det t''| \geq 1.$
As $|\kappa|\ge 1,$ we have
\begin{align*}
    \sigma (\det t'\kappa^{-1} )\sigma (\det t''^{-1}\kappa^{-1} )&=(1-\log |\det t'|+r\log |\kappa|)(1+\log |\det t''|+r\log |\kappa|)\\
     &\ge (1-\log |\det t'|+\log |\det t''| )(1+\log |\kappa|)\\
     &\ge \sigma(g)\sigma(\kappa).
\end{align*}

\end{proof}

For $0 \leq j \leq r,$ let
\begin{align*} \begin{split}
T'(R):&=\{ \begin{psmatrix} t' & \\ & I_{r-j}\end{psmatrix}:t' \in T_{j}(R)\},\\
T''(R):&=\{t \in T_r(R): \alpha_i(t)=1\textrm{ for }1 \leq i \leq j\}. \end{split}
\end{align*}
Thus $T_r=T'T''$ and the product is direct.
\begin{lem} \label{lem:COV}
For $f \in C(T''(F)U_r(F) \backslash \GL_r(F))$ and   $g \in T_r(F),$ 
\begin{align*}
\int_{T'(F) \times  K_r}f(t_1kg^{-1})\Xi_{U_r \backslash \GL_r}(kg^{-1})dt_1dk=\int_{T'(F) \times K_r}f(g^{-1}t_1k)\Xi_{U_r \backslash \GL_r}(kg)dt_1dk
\end{align*}
whenever the integrals converge absolutely.  
\end{lem}

\begin{proof}
Choose $\widetilde{f} \in C(U_r(F) \backslash \GL_r(F))$ such that $\int_{T''(F)}\widetilde{f}(t_2h)dt_2=f(h)$ and $\int_{T''(F)}|\widetilde{f}|(t_2h)dt_2=|f|(h)$ for all $h \in \GL_r(F).$   Then by the Iwasawa decomposition
\begin{align*}
    &\int_{T'(F) \times  K_r}f(t_1kg^{-1})\Xi_{U_r \backslash \GL_r}(kg^{-1})dt_1dk\\&=\int_{T'(F) \times T''(F) \times  K_r}\widetilde{f}(t_1t_2kg^{-1})\Xi_{U_r \backslash \GL_r}(k)\Xi_{U_r \backslash \GL_r}(kg^{-1})dt_1dt_2dk\\
    &=\int_{U_r(F) \backslash \GL_r(F)}\widetilde{f}(g'g^{-1})\Xi_{U_r \backslash \GL_r}(g')\Xi_{U_r \backslash \GL_r}(g'g^{-1})d\dot{g}'.
\end{align*}
Changing variables $g' \mapsto g^{-1}g'g$ this is 
\begin{align*}
    &\int_{U_r(F) \backslash \GL_r(F)}\widetilde{f}(g^{-1}g')\Xi_{U_r \backslash \GL_r}(g'g)\Xi_{U_r \backslash \GL_r}(g')d\dot{g}'\\
    &=\int_{T'(F) \times T''(F) \times  K_r}\widetilde{f}(t_1t_2g^{-1}k)\Xi_{U_r \backslash \GL_r}(kg)\Xi_{U_r \backslash \GL_r}(k)dt_1dt_2dk\\
    &=\int_{T'(F) \times K_r}f(g^{-1}t_1k)\Xi_{U_r \backslash \GL_r}(kg)dt_1dk.
\end{align*}
\end{proof}

For $W \in \mathcal{C}_{-d}(U_{r,r-2}(F) \backslash \GL_{r,r-2}(F),\psi \otimes \overline{\psi}),$ $a,t_{r-2},t_r,z \in F^\times$ and $h_1,h_2 \in \GL_{r-3}(F),$ write
\begin{align} \label{fW}
f_{W}(h_1,h_2):=W\left(\begin{psmatrix}at_{r-2} h_1 & & &\\ & at_{r-2}  & \\ & & 1 & \\ & & &t_r^{-1}  \end{psmatrix}z,\begin{psmatrix} ah_2 & \\ & a\end{psmatrix}  \right)\frac{1}{|a|^{3(r-2)/2}}.
\end{align}
\begin{lem} \label{lem:fW}
Let $N\ge 0$. There is a continuous seminorm $\nu_{d,N}$ on $\mathcal{C}_{-d}(U_{r,r-2}(F) \backslash \GL_{r,r-2}(F),\psi \otimes \overline{\psi})$ such that for $g_1,g_2 \in \GL_{r-3}(F),$ one has that
$|f_W(g_2,g_2g_1)|$ is bounded above by 
\begin{align} \label{no:sigs}
\frac{\nu_{d,N}(W)|t_{r-2}|^{r-2}|t_r|^{(r-1)/2}}{|a|^{(r-2)/2}(1+|t_r|)^{N}(1+|at_{r-2}|)^N}\frac{
|\det g_2|^{2}|\det g_1|^{1/2}\Xi_{U_{r-3} \backslash \GL_{r-3}}(g_2)\Xi_{U_{r-3} \backslash \GL_{r-3}}(g_2g_1)}
{\prod_{i=1}^{r-3}(1+|\alpha_i|\begin{psmatrix} m_{r-3}(g_2) & \\ & 1 \end{psmatrix})^{N}(1+|\alpha_i|\begin{psmatrix} m_{r-3}(g_2g_1) & \\ & 1 \end{psmatrix})^{N}}
\end{align}
times 
\begin{align}
\label{log:terms}\begin{cases} \sigma(t_{r-2},z,t_r,a)^{-d/2}\sigma_{U_{r-3} \backslash \GL_{r-3}}(g_2)^{-d/2}\sigma_{U_{r-3} \backslash \GL_{r-3}}(g_2g_1)^{-d/2} & \textrm{ if }d \geq 0,\\
\sigma(t_{r-2},z,t_r,a)^{-2d}\sigma_{U_{r-3} \backslash \GL_{r-3}}(g_2)^{-d}\sigma_{U_{r-3}\backslash \GL_{r-3}}(g_2g_1)^{-d} & \textrm{ if }d \leq 0. \end{cases}
\end{align}
\end{lem}
\begin{proof}
Using Lemma \ref{lem:C:gauge}, we see that there is a continuous seminorm $\nu_{N,d}$ such that $
|f_{W}(g_2,g_2g_1)|$ is bounded by
$\nu_{N,d}\left(W \right)$ times \eqref{no:sigs}
times
\begin{align} \label{before:t}  \begin{split}
    &\sigma_{U_{r} \backslash \GL_r}\begin{psmatrix}zat_{r-2} g_2 & & &\\ & zat_{r-2}  & \\ & & z & \\ & & &zt_r^{-1}  \end{psmatrix}^{-d}\sigma_{U_{r-2} \backslash \GL_{r-2}}\begin{psmatrix} ag_2g_1 & \\ & a\end{psmatrix}^{-d}\end{split}.
\end{align}
It follows from \eqref{prod:log} and \eqref{iso:inv} that 
\begin{align*}
 \sigma_{U_{r-2} \backslash \GL_{r-2}}\begin{psmatrix} ag_2g_1& \\ & a\end{psmatrix}\gg \sigma(a)^{1/2}\sigma_{U_{r-3} \backslash \GL_{r-3}}(g_2g_1)^{1/2},
\end{align*}
\begin{align*}
    &\sigma_{U_{r} \backslash \GL_r}\begin{psmatrix}zat_{r-2} g_2 & & &\\ & zat_{r-2}  & \\ & & z & \\ & & &zt_r^{-1}  \end{psmatrix}\gg 
    \sigma(z,t_r,at_{r-2})^{1/2}\sigma_{U_{r-3} \backslash \GL_{r-3}}(g_2)^{1/2}.
\end{align*}
Thus, using \eqref{prod:log} and \eqref{iso:inv} again, 
\begin{align*}
    \sigma(a)^{1/2}\sigma(z,t_r,at_{r-2})^{1/2}\gg \sigma(t_{r-2},z,t_r,a)^{1/2}.
\end{align*}
For $d\ge 0,$ this shows \eqref{before:t} is dominated by \eqref{log:terms}. The proof for $d\le 0$ is similar using \eqref{prod:log} and \eqref{iso:inv}. 
\end{proof}

The following is the key technical estimate in this paper: 
\begin{prop}
 \label{prop:J} 
Assume $d\in \RR_{>0}$. The integral \index{$J(W)$}
    \begin{align}\label{eq:JW}
   J(W)(g,g'^tm_\ell):=\int_{U_{r-2}(F) \backslash \mathcal{P}_{r-2}(F)}W\left(\begin{psmatrix}pg' & \\ &  I_2 \end{psmatrix}g,pg'  \right)\frac{d_r p}{|\det pg'|^{3/2}}
    \end{align}
    is absolutely convergent for all $W \in \mathcal{C}_{-d}(U_{r,r-2}(F) \backslash \GL_{r,r-2}(F),\psi \otimes \overline{\psi}).$ There are absolute constants $e,d_0 \in \RR_{>0}$ such that for any $N \in \RR_{>0}$, there is a continuous seminorm $\mu_{d,N}$ on $ \mathcal{C}_{-d}(U_{r,r-2}(F) \backslash \GL_{r,r-2}(F),\psi \otimes \overline{\psi})$ satisfying
    \begin{align}
&|J(W)|\left(\begin{psmatrix}k'^{-1}t_{r-2}\begin{psmatrix} g_1 & y_1 & \\ & 1 \end{psmatrix} & \\ & &  \begin{psmatrix} 1  & y_2 \\ & t_r^{-1}\end{psmatrix}\end{psmatrix}zk, k'^tam_\ell\right) \label{J:eval} \\ \nonumber
    \leq &\mu_{d,N}(W)\frac{|\det g_1|^{5/4}\Xi_{\GL_{r-3}}(g_1)^{1/2}\Xi_{\GL_{r-2}}\begin{psmatrix}
        g_1 &y_1 \\
          & 1
    \end{psmatrix}^{1/2}|t_{r-2}|^{r-2}|t_r|^{(r-1)/2}}{\sigma(z,t_r,a)^{d/e-d_0}\sigma\left(t_{r-2}\begin{psmatrix}
        g_1 &y_1 \\
          & 1
    \end{psmatrix}\right)^{d/e-d_0}|a|^{(r-2)/2}(1+|t_r|)^{N}(1+|at_{r-2}|)^N}.
\end{align}   
Here $(a,z,t_{r-2},t_r,g_1,y_1,y_2,k,k') \in (F^\times)^4 \times \GL_{r-3}(F) \times F^{r-3} \times F \times K_r \times K_{r-2}.$
\end{prop}

\begin{proof} 
The absolute convergence statement is a consequence of Lemma \ref{lem:C:gauge} and the Iwasawa decomposition. 

The quantity \eqref{J:eval} is the norm of $\psi(t_ry_2)$ times
\begin{align} \label{before:bounds}\begin{split}
&\int_{U_{r-3}(F) \backslash \GL_{r-3}(F)}W\left(\begin{psmatrix}at_{r-2} g_2g_1 & & &\\ & at_{r-2}  & \\ & & 1 & \\ & & &t_r^{-1}  \end{psmatrix}zk,\begin{psmatrix} ag_2 & \\ & a\end{psmatrix} k' \right)\frac{\psi((g_2y_1)_{r-3})d\dot{g_2}}{|\det g_2|^{3/2}| a|^{3(r-2)/2}} \end{split}.
\end{align}   
At the expense of replacing $y_1$ by $k_0y_1$ for some $k_0 \in K_{r-3},$ we can assume $g_1\in T_{r-3}(F)^+.$  
By replacing $W$ with $\mathcal{R}(k,k')W$, we may assume $k,k'$ are the identities. 
Write 
\begin{align}\label{eq:g1t't''}
    g_1=t't''=\begin{psmatrix}
        t_1'& & & \\
        & \ddots & & \\
        & & t_j' & \\
        & &  & I_{r-3-j}
    \end{psmatrix}\begin{psmatrix}
        I_{j}& & & \\
        & t_{j+1}'' & & \\
        & & \ddots& \\
        & &  & t_{r-3}''
    \end{psmatrix},
\end{align}
with $j:=j(g_1)$ as in Lemma \ref{lem:Levi}.

We now use the $\psi$-function to deduce bounds of \eqref{before:bounds}.  For ease of notation, we write $w:=w_{r-3}=\begin{psmatrix} & & 1 \\ & \textrm{\reflectbox{$\ddots$}} & \\ 1 & & \end{psmatrix}.$ Using the notation \eqref{fW} and changing variables $g_2\mapsto t'g_2wg_1^{-1},$ we see that \eqref{before:bounds} is 
\begin{align} \label{before:psi:bounds}
\int_{U_{r-3}(F) \backslash \GL_{r-3}(F) }\psi((t'g_2wg_1^{-1}y_1)_{r-3})f_{W}(t'g_2w,t'g_2wg_1^{-1})\frac{d\dot{g_2}}{\delta_{B_{r-3}}(t')|\det g_2t''^{-1}|^{3/2}}.
\end{align}
For $1\le n\le r-3,$ let $e_{r-3,n}:=e_{r-3}e_n^{^t}\in M_{r-3}(F).$ Assume that $F$ is non-Archimedean.
If $x \in F^\times$ 
 satisfies $
    |x| \ll_{W} 1,$ then since $g_1 \in T_{r-3}(F)^+$ we have
that
\begin{align*}
&f_{W}(t'g_2(I+xe_{r-3,n})w,t'g_2(I+xe_{r-3,n})wg_1^{-1})\\&=f_{W}(t'g_2w(w^{-1}(I+xe_{r-3,n})w),t'g_2wg_1^{-1} g_1(w^{-1}(I+xe_{r-3,n})w)g_1^{-1})\\
&=f_W(t'g_2w,t'g_2wg_1^{-1})
\end{align*}
for any $1 \leq n\leq r-3.$
By a change of variables, it follows that in \eqref{before:psi:bounds}, the integral over $g_2$ is supported in 
 the set of $g_2$ such that 
\begin{align*}
 |(t'g_2e_{r-3,n} wg^{-1}_1y_1)_{r-3}|\ll_{W}  1.
\end{align*}
We choose $n$ so that 
\begin{align} \label{choose:n} \begin{split}
\norm{g_1^{-1}y_1} \asymp |(wg_1^{-1}y_1)_{n}| \textrm{ if } |t'_{r-3}| = 1 \quad\textrm{and}\quad
\norm{y_1} \asymp |(wy_1)_{n}| \textrm{ if } |t'_{r-3}| < 1. \end{split}
\end{align}
Then we have 
\begin{align}\label{eq:choosenresult}
 |(t'g_2e_{r-3,n} wg^{-1}_1y_1)_{r-3}| \gg |(g_2)_{r-3,r-3}|\min(\norm{y_1},\norm{g_1^{-1}y_1}).
\end{align}

Now assume $F$ is Archimedean. 
Write $\psi(z)=e^{\sqrt{-1}\mathrm{tr}_{F/\RR}(\alpha z)}$ for some $\alpha \in \CC^\times.$  
Let $\mathcal{V}_0,\mathcal{V}$ be a smooth partition of unity of $\RR$ such that $\mathcal{V}_0$ is supported on $|x|\le 2$ and is $1$ on $|x|\le 1$. 
Choose $n$ as in \eqref{choose:n}. 
Write 
\begin{align}\begin{split}\label{eq:beforebyparts}
    &\int_{ U_{r-3}(F) \backslash \GL_{r-3}(F)}  \psi((t'g_2w g_1^{-1}y_1)_{r-3})f_{W}(t'g_2w,t'g_2wg_1^{-1})\frac{\mathcal{V}\left((t'g_2e_{r-3,n} wg_1^{-1}y_1)_{r-3}\right)}{\delta_{B_{r-3}}(t')|\det g_2t''^{-1}|^{3/2}}d\dot{g_2}\\
    &=\int \frac{d}{dt_0} \frac{\psi((t'g_2e^{t_0e_{r-3,n}}wg_1^{-1}y_1)_{r-3})}{\sqrt{-1}\alpha(t'g_2e_{r-3,n} wg_1^{-1}y_1)_{r-3}}\Bigg|_{t_0=0}f_{W}(t'g_2w,t'g_2wg_1^{-1})\frac{\mathcal{V}\left((t'g_2e_{r-3,n}w g_1^{-1}y_1)_{r-3}\right)}{\delta_{B_{r-3}}(t')|\det g_2t''^{-1}|^{3/2}}d\dot{g_2}. \end{split}
\end{align}
Using Lemma \ref{lem:fW} one checks that it is permissible to apply the 
 Leibniz integral rule to the following integral:
\begin{align*}
     \int \frac{d}{dt_0} &\Bigg(
\frac{\psi((t'g_2e^{t_0e_{r-3,n}}wg_1^{-1}y_1)_{r-3})}{\sqrt{-1}\alpha (t'g_2e^{t_0e_{r-3,n}}e_{r-3,n} wg_1^{-1}y_1)_{r-3}}
     f_{W}(t'g_2 e^{t_0e_{r-3,n}}w,t'g_2e^{t_0e_{r-3,n}}wg_1^{-1})\\& \times \frac{\mathcal{V}\left((t'g_2e^{t_0e_{r-3,n}}e_{r-3,n} wg_1^{-1}y_1)_{r-3}\right)}{\delta_{B_{r-3}}(t')|\det g_2e^{t_0e_{r-3,n}}t''^{-1}|^{3/2}}\Bigg)\Bigg|_{t_0=0}d\dot{g_2}.
\end{align*}
By changing variables $g_2 \mapsto g_2e^{-t_0e_{r-3,n}},$ we see that the above quantity is $0.$  Hence 
 \eqref{eq:beforebyparts} is 
\begin{align*}
&\delta_{r-3,n}\int
\frac{\psi((t'g_2wg_1^{-1}y_1)_{r-3})}{\sqrt{-1}\alpha (t'g_2e_{r-3,n}wg_1^{-1} y_1)_{r-3}}
     f_{W}(t'g_2w ,t'g_2wg_1^{-1})\frac{\mathcal{V}\left((t'g_2e_{r-3,n} wg_1^{-1}y_1)_{r-3}\right)}{\delta_{B_{r-3}}(t')|\det g_2t''^{-1}|^{3/2}}d\dot{g_2}\\
    &-\int\frac{\psi((t'g_2wg_1^{-1}y_1)_{r-3})}{\sqrt{-1}\alpha(t'g_2e_{r-3,n} wg_1^{-1}y_1)_{r-3}}(e_{1,r-2-n}\otimes \mathrm{Id}).f_{W}(t'g_2w,t'g_2wg_1^{-1})\frac{\mathcal{V}\left((t'g_2e_{r-3,n} wg_1^{-1}y_1)_{r-3}\right)}{\delta_{B_{r-3}}(t')|\det g_2t''^{-1}|^{3/2}}d\dot{g_2}\\
    &-\int  \frac{\psi((t'g_2wg_1^{-1}y_1)_{r-3})}{\sqrt{-1}\alpha(t'g_2e_{r-3,n} wg_1^{-1}y_1)_{r-3}}(\mathrm{Id}\otimes g_1e_{1,r-2-n}g_1^{-1}).f_{W}(t'g_2w,t'g_2wg_1^{-1})\frac{\mathcal{V}((t'g_2e_{r-3,n} wg_1^{-1}y_1)_{r-3})}{\delta_{B_{r-3}}(t')|\det g_2t''^{-1}|^{3/2}}d\dot{g_2}\\
    &-\delta_{r-3,n}\int \frac{\psi((t'g_2wg_1^{-1}y_1)_{r-3})}{\sqrt{-1}\alpha}f_{W}(t'g_2w,t'g_2wg_1^{-1})\frac{\mathcal{V}'\left((t'g_2e_{r-3,n} wg_1^{-1}y_1)_{r-3}\right)}{\delta_{B_{r-3}}(t')|\det g_2t''^{-1}|^{3/2}}d\dot{g_2}\\
    &+\frac{3\delta_{r-3,n}}{2}\int  \frac{\psi((t'g_2wg_1^{-1}y_1)_{r-3})}{\sqrt{-1}\alpha(t'g_2e_{r-3,n}wg_1^{-1}y_1)_{r-3}}f_{W}(t'g_2w,t'g_2wg_1^{-1})\frac{\mathcal{V}\left((t'g_2e_{r-3,n} wg_1^{-1}y_1)_{r-3}\right)}{\delta_{B_{r-3}}(t')|\det g_2t''^{-1}|^{3/2}}d\dot{g_2},
\end{align*}
where $\delta_{r-3,n}$ is the Kronecker $\delta$-function and
where $\mathcal{V}'$ is the derivative of $\mathcal{V}.$
 To orient the reader, we point out that the equality of \eqref{eq:beforebyparts} and the expression above amounts to a version of integration by parts.

We bound the derivative occurring in the third term of the previous displayed equation using an analogue of the argument in the non-Archimedean case. Namely, $g_1\in T_{r-3}(F)^+$ so $\{g_1e_{1,r-2-n}g_1^{-1}:g_1\in T_{r-3}(F)^+\}$ lies in a compact subset of $\mathfrak{gl}_{r-3}(F)$.  We also point out that $\mathcal{V}'$ is supported in a compact subset of $\RR.$  Therefore, we can iterate the argument above to see that 
for any $N\in \ZZ_{\ge 0}$ one has
\begin{align*}
&\bigg|\int_{U_{r-3}(F) \backslash \GL_{r-3}(F)}\psi((t'g_2wg_1^{-1}y_1)_{r-3})f_{W}(t'g_2w,t'g_2wg_1^{-1})\frac{d\dot{g_2}}{\delta_{B_{r-3}}(t')|\det g_2t''^{-1}|^{3/2}}\bigg|\\&
\leq 
\left|\int_{U_{r-3}(F) \backslash \GL_{r-3}(F)}\left(1+|(t'g_2e_{r-3,n} wg_1^{-1}y_1)_{r-3}|\right)^{-N}\frac{|f_{W_N}|(t'g_2w,t'g_2wg_1^{-1})d\dot{g_2}}{\delta_{B_{r-3}}(t')|\det g_2t''^{-1}|^{3/2}}\right|
\end{align*}
for some $W_N \in \mathcal{C}_{-d}\left(U_{r,r-2}(F) \backslash \GL_{r,r-2}(F),\psi \otimes \overline{\psi} \right)$ chosen continuously in $W.$

Now let $F$ be arbitrary. By Lemma \ref{lem:fW} and the argument above together with \eqref{eq:choosenresult}, we see that
there is a continuous seminorm $\nu_{N,d}$ such that \eqref{before:psi:bounds} is bounded by $
\nu_{N,d}\left(W \right)$ 
times
the product of 
\begin{align} \label{first:factor1}
\frac{\sigma(t_{r-2},z,t_r,a)^{-d/2}|t_{r-2}|^{r-2}|t_r|^{(r-1)/2}}{|a|^{(r-2)/2}(1+|t_r|)^{N}(1+|at_{r-2}|)^N}
\end{align} 
and
\begin{align*}
\int_{U_{r-3}(F) \backslash \GL_{r-3}(F)}&\frac{|\det g_1|^{3/2}|\det g_2t''^{-1}|^{1/2}\Xi_{U_{r-3} \backslash \GL_{r-3}}(g_2)\Xi_{U_{r-3} \backslash \GL_{r-3}}(g_2wg_1^{-1})}{\sigma_{U_{r-3} \backslash \GL_{r-3}}(t'g_2)^{d/2}\sigma_{U_{r-3} \backslash \GL_{r-3}}(t'g_2wg_1^{-1})^{d/2}\left(1+|(g_2)_{r-3,r-3}|\min(\norm{y_1},\norm{g_1^{-1}y_1})\right)^{N}}\\& \times
  \frac{d\dot{g_2}}{ \prod_{i=1}^{r-3}(1+|\alpha_i\begin{psmatrix} m_{r-3}(t'g_2) & \\ & 1 \end{psmatrix})^{N}(1+|\alpha_i\begin{psmatrix} m_{r-3}(t'g_2wg_1^{-1}) & \\ & 1 \end{psmatrix})^{N}} .
\end{align*}

We use the Bruhat decomposition to write the above integral as (up to a positive constant)
\begin{align*}
\int_{F^\times \times F^{r-4} \times \overline{B}_{r-4}(F) }&\frac{|\det g_1|^{3/2}|\det t''^{-1}|^{1/2}|c|^{1/2}|\det \overline{b}|^{-1/2}\Xi_{U_{r-3} \backslash \GL_{r-3}}\begin{psmatrix} \overline{b} & \\ x & c \end{psmatrix}\Xi_{U_{r-3} \backslash \GL_{r-3}}\left(\begin{psmatrix} \overline{b} & \\ x & c \end{psmatrix}wg_1^{-1}\right)}{\sigma_{U_{r-3} \backslash \GL_{r-3}}(t'\begin{psmatrix} \overline{b} & \\ x & c \end{psmatrix})^{d/2}\sigma_{U_{r-3} \backslash \GL_{r-3}}(t'\begin{psmatrix} \overline{b} & \\ x & c \end{psmatrix}wg_1^{-1})^{d/2}\left(1+|c|\min(\norm{y_1},\norm{g_1^{-1}y_1})\right)^{N}}\\& \times
  \frac{d_r\overline{b}dx d^\times c}{ \prod_{i=1}^{r-3}(1+|\alpha_i\begin{psmatrix} m_{r-3}\left(t'\begin{psmatrix} \overline{b} & \\ x & c \end{psmatrix}\right) & \\ & 1 \end{psmatrix})^{N}(1+|\alpha_i\begin{psmatrix} m_{r-3}\left(t'\begin{psmatrix} \overline{b} & \\ x & c \end{psmatrix}wg_1^{-1}\right) & \\ & 1 \end{psmatrix})^{N}} .
\end{align*}
Choose $\kappa \in F^\times$ such that 
\begin{align*}
|\kappa|=\max(1,\min(\norm{y_1},\norm{g_1^{-1}y_1})).
\end{align*} 
We change variables
$(\overline{b},x) \mapsto (c\overline{b},cx)$ and then $c \mapsto \kappa^{-1}c$ in the integral above to see that it is 
\begin{align} \label{before:sig:lemma} \begin{split}
\int&\frac{|\det g_1|^{3/2}|\det t''^{-1}|^{1/2}|\det \overline{b}|^{-1/2}\Xi_{U_{r-3} \backslash \GL_{r-3}}\begin{psmatrix} \overline{b} & \\ x & 1 \end{psmatrix}\Xi_{U_{r-3} \backslash \GL_{r-3}}\left(\begin{psmatrix} \overline{b} & \\ x & 1\end{psmatrix}wg_1^{-1}\right)}{\sigma_{U_{r-3} \backslash \GL_{r-3}}(t'c\kappa^{-1}\begin{psmatrix} \overline{b} & \\ x & 1 \end{psmatrix})^{d/2}\sigma_{U_{r-3} \backslash \GL_{r-3}}(t'c\kappa^{-1}\begin{psmatrix} \overline{b} & \\ x & 1 \end{psmatrix}wg_1^{-1})^{d/2}\left(1+|c\kappa^{-1}|\min(\norm{y_1},\norm{g_1^{-1}y_1})\right)^{N}}\\& \times
  \frac{d_r\overline{b}dx |\kappa^{-1}c|^{(r-3)/2} d^\times c}{ \prod_{i=1}^{r-3}\left(1+|\alpha_i\begin{psmatrix} m_{r-3}\left(t'c\kappa^{-1}\begin{psmatrix} \overline{b} & \\ x & 1 \end{psmatrix}\right) & \\ & 1 \end{psmatrix}\right)^{N}\left(1+|\alpha_i\begin{psmatrix} m_{r-3}\left(t'c\kappa^{-1}\begin{psmatrix} \overline{b} & \\ x & 1 \end{psmatrix}wg_1^{-1}\right) & \\ & 1 \end{psmatrix}\right)^{N}} .\end{split}
\end{align}

By Lemma \ref{lem:sig:bound:stuff} and \eqref{sig:ineq}, the above is dominated by the product of  
\begin{align}\label{after:sig:lem1}
    \frac{|\det g_1||\det t'|^{1/2}}{|\kappa|^{(r-3)/2}\sigma(g_1)^{d/2}\sigma(\kappa)^{d/2}}
\end{align}
and 
\begin{align*}
\begin{split}
&\int\frac{|\det \overline{b}|^{-1/2}|c|^{(r-3)/2}\sigma(c)^{d}\sigma(\det \overline{b})^{d}\Xi_{U_{r-3} \backslash \GL_{r-3}}\begin{psmatrix} \overline{b} & \\ x & 1\end{psmatrix}\Xi_{U_{r-3} \backslash \GL_{r-3}}\left(\begin{psmatrix} \overline{b} & \\ x & 1 \end{psmatrix}wg_1^{-1}\right)d_r\overline{b}dxd^\times c}{\left(1+\left|\frac{c}{\kappa}\right|\min(\norm{y_1},\norm{g_1^{-1}y_1}\right)^{N}
  \prod_{i=1}^{r-3}(1+|\alpha_i\begin{psmatrix} m_{r-3}(c\kappa^{-1}t'\begin{psmatrix}\overline{b} & \\ x & 1 \end{psmatrix}) & \\ & 1 \end{psmatrix}|)^{N}(1+|\alpha_i\begin{psmatrix} m_{r-3}(c\kappa^{-1}t'\begin{psmatrix}\overline{b}& \\ x & 1 \end{psmatrix}wg_1^{-1}) & \\ & 1 \end{psmatrix}|)^{N}}.  
\end{split}
\end{align*}
Executing the integral over $c,$ the integral above is dominated by 
\begin{align*}
\begin{split}
& \int_{ F^{r-4}\times \overline{B}_{r-4}(F)} \frac{|\det \overline{b}|^{-1/2}\sigma(\det \overline{b})^{d}\Xi_{U_{r-3} \backslash \GL_{r-3}}\begin{psmatrix} \overline{b} & \\ x & 1\end{psmatrix}\Xi_{U_{r-3} \backslash \GL_{r-3}}\left(\begin{psmatrix} \overline{b} & \\ x & 1 \end{psmatrix}wg_1^{-1}\right)d_r\overline{b}dx}{
  \prod_{i=1}^{r-4}(1+|\alpha_i\left(t'm_{r-3}\begin{psmatrix}\overline{b} & \\ x & 1\end{psmatrix}\right)|)^{N}(1+|\alpha_i\left(t' m_{r-3}\left(\begin{psmatrix} \overline{b} & \\ x & 1 \end{psmatrix} wg_1^{-1}\right)\right) |)^{N}}.\end{split}
\end{align*}

For $\tfrac{1}{2}>\epsilon>0,$ temporarily write for $c\in F^\times$
$$
\phi_\epsilon(c):=|c|^{1/2}\max(|c|,|c|^{-1})^{\epsilon}.
$$
By Lemma \ref{lem:decompositioncompare} for any $\tfrac{1}{2}>\epsilon>0,$ the integral above is dominated by
\begin{align*}
    &\int_{ T_{r-4}(F) \times K_{r-3}} \frac{\phi_\epsilon(\det t)\Xi_{U_{r-3} \backslash \GL_{r-3}}\left(kwg_1^{-1}\right)dtdk}{
  \prod_{i=1}^{r-4}(1+|\alpha_i(t'\begin{psmatrix}t & \\ & 1\end{psmatrix})|)^{N}(1+|\alpha_i(t'\begin{psmatrix} t & \\ & 1 \end{psmatrix}m_{r-3}(k wg_1^{-1}))|)^{N}}\\
  & \le \int_{ T_{r-4}(F) \times K_{r-3}} \frac{\phi_\epsilon(\det t)\Xi_{U_{r-3} \backslash \GL_{r-3}}\left(kg_1^{-1}\right)dtdk}{
  \prod_{i=j+1}^{r-4}(1+|\alpha_i\begin{psmatrix}
      t & \\
      & 1
  \end{psmatrix}|)^{N} \prod_{i=1}^{\min(j,r-4)}(1+|\alpha_i(t'\begin{psmatrix} t & \\ & 1 \end{psmatrix}m_{r-3}(kg_1^{-1}))|)^{N}}.
\end{align*}
Here we have changed variables $k \mapsto kw^{-1}.$
We write the integral over $T_{r-4}(F)$ in coordinates as 
$\begin{psmatrix} t_1t_2\dots t_{r-4} & & \\ & \ddots & \\ & & t_{r-4} \end{psmatrix}.$ View $T_{\min(j,r-4)}$ as a subgroup of $T_{r-3}$. Then we can execute the integral over $(t_{j+1},\dots,t_{r-4})$ to see that the above is dominated by 
\begin{align*}
&\int_{T_{\min(j,r-4)}(F) \times K_{r-3}}\frac{\phi_\epsilon\left(\det t\right)\Xi_{U_{r-3} \backslash \GL_{r-3}}(kg_1^{-1})dtdk}{\prod_{i=1}^{\min(j,r-4)}\left(1+\left|\alpha_i\left(t'tm_{r-3}\left(k g_1^{-1}\right)\right)\right|\right)^{N}}\\
&=\int_{T_{\min(j,r-4)}(F) \times K_{r-3}}\frac{\phi_\epsilon\left(\prod_{i=1}^{\min(j,r-4)}m_{r-3}(t(kg_1^{-1})g_1)_i\right)\Xi_{U_{r-3} \backslash \GL_{r-3}}(kg_1^{-1})dtdk}{\prod_{i=1}^{\min(j,r-4)}\left(1+\left|\alpha_i\left(t'tm_{r-3}\left(k g_1^{-1}\right)\right)\right|\right)^{N}}.
\end{align*}

Using Lemma \ref{lem:COV} this is 
\begin{align*}
&\int_{K_{r-3}}\int_{T_{\min(j,r-4)}(F)}\frac{\phi_\epsilon(\prod_{i=1}^{\min(j,r-4)}m_{r-3}(tg_1^{-1}kg_1)_i)\Xi_{U_{r-3} \backslash \GL_{r-3}}(kg_1)}{
   \prod_{i=1}^{\min(j,r-4)}\left(1+\left|\alpha_i\left(t''^{-1}t\right)\right|\right)^{N}} dtdk.\\
   &\le \int_{K_{r-3}}\int_{T_{\min(j,r-4)}(F)}\frac{\phi_\epsilon(\prod_{i=1}^{\min(j,r-4)}m_{r-3}(tg_1^{-1}kg_1)_i)\Xi_{U_{r-3} \backslash \GL_{r-3}}(kg_1)}{
   \prod_{i=1}^{\min(j,r-4)}\left(1+\left|\alpha_i(t)\right|\right)^{N}} dtdk.
\end{align*}
Applying \Cref{lem:decompositioncompare} this is
\begin{align*}
&\gamma \int_{\overline{U}_{r-3}(F)}\int_{T_{\min(j,r-4)}(F)}\frac{\phi_{\epsilon}(\prod_{i=1}^{\min(j,r-4)}m_{r-3}(tg_1^{-1}\overline{u}g_1)_i)\Xi_{U_{r-3} \backslash \GL_{r-3}}(\overline{u})\Xi_{U_{r-3} \backslash \GL_{r-3}}(\overline{u}g_1)}{\prod_{i=1}^{\min(j,r-4)}\left(1+\left|\alpha_i(t)\right|\right)^{N}} dtd\overline{u}.
\end{align*}
Using Lemma \ref{lem:upper} and  \Cref{lem:decompositioncompare} again, the integral above is dominated by 
\begin{align*}
&\gamma \int_{\overline{U}_{r-3}(F)}\int_{T_{\min(j,r-4)}(F)}\frac{\phi_\epsilon(\det t)\Xi_{U_{r-3} \backslash \GL_{r-3}}(\overline{u})\Xi_{U_{r-3} \backslash \GL_{r-3}}(\overline{u}g_1)}{
 \prod_{i=1}^{\min(j,r-4)}\left(1+\left|\alpha_i(t)\right|\right)^{N}} dtd\overline{u}\\
   &\ll \gamma \int_{\overline{U}_{r-3}(F)}\Xi_{U_{r-3} \backslash \GL_{r-3}}(\overline{u})\Xi_{U_{r-3} \backslash \GL_{r-3}}(\overline{u}g_1)d\overline{u}=\Xi_{\GL_{r-3}}(g_1).
\end{align*}

After renormalizing $\nu_{N,d}(W)$ if necessary, and keeping in mind  \eqref{first:factor1} and \eqref{after:sig:lem1}, we see that \eqref{J:eval} is bounded by $\nu_{N,d}(W)$ times
\begin{align} \label{J:eval2}
\begin{split}
&\frac{\sigma(t_{r-2},z,t_r,a)^{-d/2}|t_{r-2}|^{r-2}|t_r|^{(r-1)/2}}{|a|^{(r-2)/2}(1+|t_r|)^{N}(1+|at_{r-2}|)^N}
 \frac{|\det g_1||\det t'|^{1/2}\Xi_{\GL_{r-3}}(g_1)}{|\kappa|^{(r-3)/2}\sigma(g_1)^{d/2}\sigma(\kappa)^{d/2}}.
\end{split}
\end{align}

By Lemma \ref{lem:Levi}, Lemma \ref{lem:parab}, \eqref{prod:log} and \eqref{iso:inv}, there are $e>0$ and $d_0\ge 0$ such that \eqref{J:eval2} is dominated by 
\begin{align*}
    &|\det g_1|^{5/4}\Xi_{\GL_{r-3}}(g_1)^{1/2}\Xi_{\GL_{r-2}}\begin{psmatrix}
        g_1 &y_1 \\
          & 1
    \end{psmatrix}^{1/2}\\
    &\times\sigma(z,t_r,a)^{d_0-d/e}\sigma\left(t_{r-2}\begin{psmatrix}
        g_1 &y_1 \\
          & 1
    \end{psmatrix}\right)^{d_0-d/e}\frac{|t_{r-2}|^{r-2}|t_r|^{(r-1)/2}}{|a|^{(r-2)/2}(1+|t_r|)^{N}(1+|at_{r-2}|)^N}.
\end{align*}
This completes the proof.
\end{proof}

For $s_0,s' \in \CC$ and $W \in C^\infty(U_{r,r-2}(F) \backslash \GL_{r,r-2}(F),\psi \otimes \overline{\psi}),$ let
\begin{align*}
    W_{s_0,s'}(g,g'):=W(g,g')|\det g|^{s_0}|\det g'|^{s'}.
\end{align*}
If $s_0=0,$ we often omit it from notation.

For the next lemma, assume $g_1\in T_{r-3}(F)^+$ and use the notation \eqref{eq:g1t't''}.
\begin{lem} \label{prop:Jd>0} 
    Let $d\ge 0$ and $N' \geq \mathrm{Re}(s') \geq 0.$ For $W\in \mathcal{C}_{d}(U_{r,r-2}(F) \backslash \GL_{r,r-2}(F),\psi \otimes \overline{\psi})$, the integral defining $J(W_{s'})$ is absolutely convergent. Furthermore, there is a $d'>0$ such that for any $N\ge 0,$ there is a continuous seminorm $\mu_{N,N',d}$ on $ \mathcal{C}_{d}(U_{r,r-2}(F) \backslash \GL_{r,r-2}(F),\psi \otimes \overline{\psi})$ such that
        \begin{align*}
    &|J(W_{s'})|\left(\begin{psmatrix}k'^{-1}t_{r-2}\begin{psmatrix} g_1 & y_1 & \\ & 1 \end{psmatrix} & \\ & &  \begin{psmatrix} 1  & y_2 \\ & t_r^{-1}\end{psmatrix}\end{psmatrix}zk, k'^tam_\ell\right)  \\
    \leq & \mu_{N,N',d}(W)|\det g_1|^{5/4}\Xi_{\GL_{r-3}}(g_1)^{1/2}\Xi_{\GL_{r-2}}\begin{psmatrix}
       g_1 &y_1 \\
          &1
    \end{psmatrix}^{1/2}\\
    &\times\frac{\sigma(z,t_r,a)^{d'}\sigma\left(t_{r-2}\begin{psmatrix}
        g_1 &y_1 \\
          & 1
    \end{psmatrix}\right)^{d'}|t_{r-2}|^{r-2}|t_r|^{(r-1)/2}|\det t''^{-1}|^{\mathrm{Re}(s')}}{|a|^{(r-2)(1/2-\mathrm{Re}(s'))}(1+|t_r|)^{N}(1+|at_{r-2}|)^N}.
\end{align*}   
\end{lem}

\begin{proof}
Arguing as in the proof of Proposition \ref{prop:J}, there is a continuous seminorm $\nu_{N,d}$ on  $\mathcal{C}_{d}(U_{r,r-2}(F) \backslash \GL_{r,r-2}(F),\psi \otimes \overline{\psi})$ such that the quantity in the lemma is dominated by 
$
\nu_{N,d}\left(W \right)$ 
times
the product of 
\begin{align*}
\frac{\sigma(t_{r-2},z,t_r,a)^{2d}|t_{r-2}|^{r-2}|t_r|^{(r-1)/2}}{|a|^{(r-2)(1/2-\mathrm{Re}(s'))}(1+|t_r|)^{N}(1+|at_{r-2}|)^N}
\end{align*} 
and
\begin{align*}\int&\frac{|\det g_1|^{3/2}|\det t''^{-1}|^{1/2+\mathrm{Re}(s')}|\det \overline{b}|^{-1/2+\mathrm{Re}(s')}\Xi_{U_{r-3} \backslash \GL_{r-3}}\begin{psmatrix} \overline{b} & \\ x & 1 \end{psmatrix}\Xi_{U_{r-3} \backslash \GL_{r-3}}\left(\begin{psmatrix} \overline{b} & \\ x & 1\end{psmatrix}wg_1^{-1}\right)}{\sigma_{U_{r-3} \backslash \GL_{r-3}}(t'c\kappa^{-1}\begin{psmatrix} \overline{b} & \\ x & 1 \end{psmatrix})^{-d}\sigma_{U_{r-3} \backslash \GL_{r-3}}(t'c\kappa^{-1}\begin{psmatrix} \overline{b} & \\ x & 1 \end{psmatrix}wg_1^{-1})^{-d}\left(1+|c\kappa^{-1}|\min(\norm{y_1},\norm{g_1^{-1}y_1})\right)^{N}}\\& \times
  \frac{d_r\overline{b}dx |\kappa^{-1}c|^{(r-3)/2(1+\mathrm{Re}(s'))} d^\times c}{ \prod_{i=1}^{r-3}(1+|\alpha_i\begin{psmatrix} m_{r-3}\left(t'c\kappa^{-1}\begin{psmatrix} \overline{b} & \\ x & 1 \end{psmatrix}\right) & \\ & 1 \end{psmatrix})^{N}(1+|\alpha_i\begin{psmatrix} m_{r-3}\left(t'c\kappa^{-1}\begin{psmatrix} \overline{b} & \\ x & 1 \end{psmatrix}wg_1^{-1}\right) & \\ & 1 \end{psmatrix})^{N}}.
\end{align*}
These are the analogues of \eqref{first:factor1} and \eqref{before:sig:lemma}, respectively.  
By \eqref{pullback} we have
\begin{align*}
&\sigma_{U_{r-3} \backslash \GL_{r-3}}(t'c\kappa^{-1}\begin{psmatrix} \overline{b} & \\ x & 1 \end{psmatrix})^{d}\sigma_{U_{r-3} \backslash \GL_{r-3}}(t'c\kappa^{-1}\begin{psmatrix} \overline{b} & \\ x & 1 \end{psmatrix}wg_1^{-1})^{d} \\&\ll \sigma(t')^{2d}\sigma(c)^{2d}\sigma(\kappa)^{2d}\sigma\begin{psmatrix} \overline{b}& \\ x & 1 \end{psmatrix}^d\sigma(g_1)^{d}.
\end{align*}
A small modification of the end of the proof of 
 Proposition \ref{prop:J} now yields the lemma.
\end{proof}

For the purposes of the following lemma, when $F$ is non-Archimedean, let $\mathfrak{F}_1$ be a finite set of $\OO_F^\times$-types and let $\mathcal{C}(\mathrm{Tp}_1)_{\mathfrak{F}}$ be the space of functions supported in characters having $\OO_F^\times$-types in $\mathfrak{F}_1.$
\begin{lem} \label{lem:map2X}
One has a separately continuous map
\begin{align*} 
\mathcal{C}(\mathrm{Tp}_1) \times \mathcal{C}(U_{r,r-2}(F) \backslash \GL_{r,r-2}(F),\psi \otimes \overline{\psi}) &\lto \mathcal{C}(X_\ell^\circ(F),\mathcal{H})
\end{align*}
given by sending $(\phi,W)$ to 
\begin{align*}
\left((g,g'^tm_\ell) \mapsto \int_{\mathrm{Tp}_1}\phi(\chi)\int_{F^\times}
J(W)\left(\begin{psmatrix}I_{r-2} & \\  & \begin{psmatrix} 1 & \\ & c^{-1}\end{psmatrix}h_\ell  \end{psmatrix}g,g'^tm_\ell \right)
\frac{\chi(c) d^\times c}{|c|^{(r-2)/2}}d\mu( \chi)\right).
\end{align*}
Here the integrals converge absolutely. It is jointly continuous when $F$ is Archimedean.  In the non-Archimedean case, if $K \leq \GL_{r,r-2}(F)$ is a compact open subgroup then the restriction to $\mathcal{C}(\mathrm{Tp}_1)_{\mathfrak{F}_1} \times \mathcal{C}(U_{r,r-2}(F) \backslash \GL_{r,r-2}(F),\psi \otimes \overline{\psi})^K$ is continuous.
\end{lem}

\begin{proof}  Using notation as in \eqref{coord}, consider
\begin{align*}  
\int_{\mathrm{Tp}_1}\phi(\chi)\int_{F^\times}
J(W)\left(\begin{psmatrix}k'^{-1}t_{r-2}\begin{psmatrix}  g_1 &  y_1 \\ & 1\end{psmatrix}& \\ &  \begin{psmatrix} 1 & y_2 \\ & c^{-1}t_r^{-1} \end{psmatrix} \end{psmatrix}zk,ak'^tm_\ell \right)
\frac{\chi(c) d^\times c}{|c|^{(r-2)/2}}d\mu( \chi).
\end{align*}
  The inner integral converges absolutely by Proposition \ref{prop:J}, and it is bounded by a constant independent of $\chi.$  Thus the double integral over $\mathrm{Tp}_1 \times F^\times$ converges absolutely by the definition of $\mathcal{C}(\mathrm{Tp}_1).$
By a change of variables $c \mapsto ct_r^{-1},$ the expression is 
\begin{align} \label{to:bound2}
\int_{\mathrm{Tp}_1}\frac{\phi(\chi)|t_r|^{(r-2)/2}}{\chi(t_r)}\int_{F^\times}\psi(cy_2)
J(W)\left(\begin{psmatrix}k'^{-1}t_{r-2}\begin{psmatrix}  g_1 &  y_1 \\ &1\end{psmatrix}& \\ &  \begin{psmatrix} 1 &  \\ & c^{-1} \end{psmatrix} \end{psmatrix}zk,ak'^tm_\ell \right)
\frac{\chi(c) d^\times c}{|c|^{(r-2)/2}}d\mu( \chi).
\end{align}

To obtain a suitable estimate for this expression, we require cancellation in the integral over $\mathrm{Tp}_1$ and in the integral over $c.$
One has that
$$
\mathrm{Tp}_1=\left\{|\cdot|^{it}\chi: (t,\chi) \in \mathfrak{a}_{\GL_1,F}^* \times \mathrm{Tp}_1/i\mathfrak{a}_{\GL_1}^*\right\}.
$$ 
Thus there is a linear invariant differential operator on $\frac{d}{d\chi}$ on $\mathrm{Tp}_1$ that acts by the operator $\frac{d}{dt}$ on each connected component.  Assume $|t_r| \neq 1.$ Applying integration by parts to \eqref{to:bound2}, we see that it is equal to $(\sqrt{-1} \log |t_r|)^{-1}$ times
\begin{align*}
&\int_{\mathrm{Tp}_1}\frac{d}{d \chi}\phi(\chi)\frac{|t_r|^{(r-2)/2}}{\chi(t_r)}\int_{F^\times}\psi(cy_2)
J(W)\left(\begin{psmatrix}k'^{-1}t_{r-2}\begin{psmatrix}  g_1 &  y_1 \\ &1\end{psmatrix}& \\ &  \begin{psmatrix} 1 & \\ & c^{-1} \end{psmatrix} \end{psmatrix}zk,ak'^tm_\ell \right)
\frac{ \chi(c) d^\times c}{|c|^{(r-2)/2}}d\mu( \chi)\\
&+\int_{\mathrm{Tp}_1}\phi(\chi)\frac{|t_r|^{(r-2)/2}}{\chi(t_r)}\int_{F^\times}\psi(cy_2)
J(W)\left(\begin{psmatrix}k'^{-1}t_{r-2}\begin{psmatrix}  g_1 &  y_1 \\ &1\end{psmatrix}& \\ &  \begin{psmatrix} 1 & \\ & c^{-1} \end{psmatrix} \end{psmatrix}zk,ak'^tm_\ell \right)
\frac{ \sqrt{-1}\log|c|\chi(c) d^\times c}{|c|^{(r-2)/2}}d\mu( \chi).
\end{align*}
We conclude that for any $\ell \in \ZZ_{\geq 0}$ there are $\phi_{i,\ell} \in \mathcal{C}(\mathrm{Tp}_1)$ with $1 \leq i \leq \ell,$ continuous in $\phi$, such that \eqref{to:bound2} is a finite sum of the form
\begin{align} \label{to:bound3}
\int_{\mathrm{Tp}_1}\frac{\phi_{i,\ell}(\chi)|t_r|^{(r-2)/2}}{(\log|t_r|)^\ell\chi(t_r)}\int_{F^\times}\psi(cy_2)
J(W)\left(\begin{psmatrix}k'^{-1}t_{r-2}\begin{psmatrix}  g_1 &  y_1 \\ &1\end{psmatrix}& \\ &  \begin{psmatrix} 1 & \\ & c^{-1} \end{psmatrix} \end{psmatrix}zk,ak'^tm_\ell \right)
\frac{(\log|c|)^i\chi(c) d^\times c}{|c|^{(r-2)/2}}d\mu( \chi).
\end{align}

We now use the $\psi$-function to deduce bounds on \eqref{to:bound3} when $F$ is non-Archimedean. Let $K\le \GL_{r,r-2}(F)$ and $K_0\le F^\times$ be compact open subgroups that fix $W$ and $\chi$ respectively.  Then for $n$ sufficiently large in a sense depending on $K \times K_0$, we have 
\begin{align*} \begin{split}
    &\int_{F^\times}\psi(cy_2)
J(W)\left(\begin{psmatrix}k'^{-1}t_{r-2}\begin{psmatrix}  g_1 &  y_1 \\ & 1 \end{psmatrix}& \\ &  \begin{psmatrix} 1 &  \\ & c^{-1} \end{psmatrix} \end{psmatrix}zk,ak'^tm_\ell \right)
\frac{(\log|c|)^\ell\chi(c) d^\times c}{|c|^{(r-2)/2}}\\
&=\int_{F^\times}\psi(c(1+\varpi^n x)y_2)
J(W)\left(\begin{psmatrix}k'^{-1}t_{r-2}\begin{psmatrix}  g_1 &  y_1 \\ & 1\end{psmatrix}& \\ &  \begin{psmatrix} 1 &  \\ & c^{-1} \end{psmatrix} \end{psmatrix}zk,ak'^tm_\ell \right)
\frac{(\log|c|)^\ell\chi(c) d^\times c}{|c|^{(r-2)/2}} \end{split}
\end{align*}
for all $x \in \OO_F.$  It follows that the integral over $c$ in \eqref{to:bound3} is supported in $|cy_2| \ll 1.$

Using Proposition \ref{prop:J} and the considerations above, for any $N,d \geq 0$ large and any $\ell \in \ZZ$ with $d>\ell\geq 0$ there is a continuous seminorm $\nu_{d,\ell,N}$ on $\mathcal{C}(U_{r,r-2}(F) \backslash \GL_{r,r-2}(F),\psi \otimes \overline{\psi})$ such that  the integral \eqref{to:bound2} is bounded by $\mu_{d,\ell,N}(W)$ times 
\begin{align*}
& \frac{|t_{r-2}|^{r-2}|t_r|^{(r-2)/2}|\det g_1|^{5/4}\Xi_{\GL_{r-3}}(g_1)^{1/2}\Xi_{\GL_{r-2}}\begin{psmatrix}
        g_1 &y_1 \\
          & 1
    \end{psmatrix}^{1/2}}{|a|^{(r-2)/2}\sigma(t_r)^\ell \sigma\begin{psmatrix}
        t_{r-2}g_1 &t_{r-2}y_1 \\
          & t_{r-2}
    \end{psmatrix}^{d}\sigma(z,a)^d}\int_{F^\times}\frac{\sigma(c)^{\ell}|c|^{1/2}d^\times c}{\sigma(c)^{d}(1+|cy_2|)^{N}(1+|c|)^{N}}\\
    &\ll \frac{|t_{r-2}|^{r-2}|t_r|^{(r-2)/2}|\det g_1|^{5/4}\Xi^{1/2}_{\GL_{r-3}}(g_1)\Xi_{\GL_{r-2}}^{1/2}\begin{psmatrix}
        g_1 &y_1 \\
          & 1
    \end{psmatrix}}{|a|^{(r-2)/2}\sigma(z,a)^{d}\sigma\begin{psmatrix}
        t_{r-2}g_1 &t_{r-2}y_1 \\
          & t_{r-2}
    \end{psmatrix}^{d}\sigma(t_r)^\ell \sigma(y_2)^{d-\ell}\max(1,|y_2|)^{1/2}}.
\end{align*}
For any $d'>0,$ we can choose $d,\ell,N$ so that this is dominated by  $\Xi_{X_\ell^\circ}(g,g'^tm_\ell)\sigma_{X_\ell^\circ}(g,g'^tm_\ell)^{-d'}.
$
This completes the proof when $F$ is non-Archimedean. For Archimedean $F,$ an integration by parts argument as in the proof of Proposition \ref{prop:J} yields the same bound.  In fact, the integration by parts argument required here is much simpler.

This proves that the map is well-defined and separately continuous.  The remaining assertions follow from the fact that a separately continuous bilinear map from a product of Fr\'echet spaces to a locally convex topological vector space is jointly continuous \cite[Corollary to Theorem 34.1]{Treves}. 
\end{proof}

\quash{
See \cite[Theorem 1 and Theorem 3]{Queiro:Sa} for the following lemma:
\begin{lem}\label{lem:Cartanproduct}
    For $h\in T_r(F)^+$, let $\lambda_1(h)\ge \ldots\ge \lambda_r(h) $ be real numbers such that 
    \begin{align*}
        h\in K_r\begin{psmatrix}
        \varpi^{\lambda_1(h)} & & &\\
        &\varpi^{\lambda_2(h)} & &\\
        & & \ddots &\\
        & & & \varpi^{\lambda_r(h)}
    \end{psmatrix}K_r &\textrm{ if }F\textrm{ is non-Archimedean, and}\\
     h \in K_r\begin{psmatrix}
        e^{-\lambda_1(h)} & & &\\
        &e^{-\lambda_2(h)} & &\\
        & & \ddots &\\
        & & & e^{-\lambda_r(h)}
    \end{psmatrix}K_r &\textrm{ if }F \textrm{ is Archimedean.}
    \end{align*}
    Let $h,h_1,h_2\in T_r(F)^+$. If $h\in K_rh_1K_rh_2K_r$ then
    \begin{align*}
        \max_i(\lambda_i(h_1)+\lambda_{r+1-i}(h_2))\le &\lambda_1(h)\le \lambda_1(h_1)+\lambda_1(h_2),\\
        \lambda_r(h_1)+\lambda_r(h_2)\le &\lambda_r(h)\le \min_i(\lambda_i(h_1)+\lambda_{r+1-i}(h_2)).
   \end{align*}\qed
\end{lem}}

 Let $M$ be a standard Levi subgroup of $\GL_{r,r-2},$ let $\underline{\sigma}\in \mathrm{Irr}_2(M),$ and let $\underline{\varphi}$ be a smooth vector in the space of $\underline{\sigma}.$
\begin{lem} \label{lem:ZWlambda}
 Let $\Omega \subset \mathfrak{a}_{M\CC}^*$ be a compact set and let $N_1 \in \RR_{>0}.$  There is a constant $N'\ge 0$ depending on $\Omega$ such that if $\underline{\lambda}\in \Omega$ and $N_1 \ge \mathrm{Re}(s')\ge N'+r,$ then  $J((W_{\underline{\varphi_\lambda}})_{s'})$ is absolutely convergent, and for any $N\ge 0$
\begin{align*} 
&|J((W_{\underline{\varphi_\lambda}})_{s'})|\left(\begin{psmatrix}k'^{-1}t_{r-2}\begin{psmatrix} g_1 & y_1 & \\ & 1 \end{psmatrix} & \\ & &  \begin{psmatrix} 1  & y_2 \\ & t_r^{-1}\end{psmatrix}\end{psmatrix}zk, k'^tam_\ell\right)  \\ \nonumber
    &\ll_{N_1,N',N,\varphi} \frac{\left(\norm{z}\norm{a}\norm{t_{r-2}}\norm{t_r}\norm{g_1}\right)^{N'}}{|a|^{(3/2-\mathrm{Re}(s'))(r-2)}(1+|at_{r-2}|)^N(1+|t_r|)^N}.
\end{align*}
If in addition $N_1\ge \mathrm{Re}(\chi)\ge N'+(r-2)/2,$ then there is an absolute constant $N_0$ such that
\begin{align*}
    &\int_{F^\times} |J((W_{\underline{\varphi_\lambda}})_{s'})|\left(\begin{psmatrix}k'^{-1}t_{r-2}\begin{psmatrix} g_1 & y_1 & \\ & 1 \end{psmatrix} & \\ & &  \begin{psmatrix} 1  & y_2 \\ & c^{-1}t_r^{-1}\end{psmatrix}\end{psmatrix}zk, k'^tam_\ell \right)
\frac{|\chi|(c) d^\times c}{|c|^{(r-2)/2}}\\
&\ll_{N_1,N',N,\varphi} \frac{\left(\norm{z}\norm{a}\norm{t_{r-2}}\norm{g_1}\right)^{N'}\max(1,|y_2|)^{-\mathrm{Re}(\chi)+N'+(r-2)/2}}{|a|^{(3/2-\mathrm{Re}(s'))(r-2)}|t_r|^{\mathrm{Re}(\chi)-(r-2)/2}(1+|at_{r-2}|)^N}(1+\norm{\chi})^{N_0}.
\end{align*}
\end{lem}
Here $\norm{\chi}$ is defined as in \eqref{norm:pi} if $F$ is Archimedean, and $\norm{\chi}:=1$ if $F$ is non-Archimedean.
\begin{proof}
   Without loss of generality we can assume  $s'\ge 0$ is real and $g_1\in T_{r-3}(F)^+$. Using Lemma \ref{lem:Arch:Whitt:fam:bound} and Lemma \ref{lem:nonArch:Whitt:fam:bound}, there is $N''\ge 0$ such that for any $N\ge 0$ the quantity is dominated by
    \begin{align*}
    \frac{\left(\norm{z}\norm{a}\norm{t_{r-2}}\norm{t_r}\right)^{N''}}{|a|^{(3/2-s')(r-2)}(1+|at_{r-2}|)^N(1+|t_r|)^N}
   \int_{U_{r-3}(F) \backslash \GL_{r-3}(F)}\frac{\norm{m_{r-3}(g_2)}^{N''}\norm{m_{r-3}(g_2g_1)}^{N''}|\det g_2|^{s'} d\dot{g_2}}
{\prod_{i=1}^{r-3}(1+|\alpha_i\begin{psmatrix} m_{r-3}(g_2) & \\ & 1 \end{psmatrix}|)^N}.
\end{align*}

Now by \eqref{prod:log} and \eqref{pullback} one has 
$$
\sigma_{U_{r-3} \backslash \GL_{r-3}}(g_2)\sigma_{\GL_{r-3}}(g_1)\gg\sigma_{U_{r-3}\backslash \GL_{r-3} \times \GL_{r-3}}(g_2,g_1)  \gg  \sigma_{U_{r-3} \backslash \GL_{r-3}}(g_2g_1) .
$$
Taking exponentials, we see that there is an $A>1$ such that $\norm{m_{r-3}(g_2)}^A\norm{g_1}^A \gg \norm{m_{r-3}(g_2g_1)}.$  Thus the integral above is dominated by
 \begin{align*} 
   & \norm{g_1}^{AN''}\int_{T_{r-3}(F)}\frac{\norm{t}^{(1+A)N''}|\det t|^{s'}}
{\prod_{i=1}^{r-3}(1+|\alpha_i\begin{psmatrix} t & \\ & 1 \end{psmatrix}|)^{N}}\delta_{B_{r-3}}^{-1}(t)dt \ll \norm{g_1}^{AN''}.
\end{align*}
Taking $N'=(1+A)N'',$ we deduce the first statement.
The second statement can be proved as in Lemma \ref{lem:map2X}. In the Archimedean case, perform integration by parts in $c.$ We leave the details to the reader.
\end{proof}

\begin{rem} We have not optimized the lower bounds of $\mathrm{Re}(s')$ and $\mathrm{Re}(\chi)$ in Lemma \ref{lem:ZWlambda}; they are chosen so that the integrals over $t$ and $c$ converge.   
\end{rem}

\subsection{Zeta integrals}

\begin{lem} \label{lem:for:zeta}
Let $d>0$ and $A>A'>-\tfrac{1}{2}.$ For any  $N \in \ZZ_{ \geq 0},$ there is a continuous seminorm $\nu_{d,N,A,A'}$ on  
$\mathcal{C}_{d}(U_{r,r-2}(F) \backslash \GL_{r,r-2}(F),\psi \otimes \overline{\psi})
$
such that for $W$ in this space and $s_0,s' \in \CC$ with $A\ge \mathrm{Re}(s_0+s')\ge A'$ one has 
\begin{align*}
    \int_{U_{r-2}(F) \backslash \mathcal{P}_{r-2}(F)}|W_{s_0,s'}|\left(\begin{psmatrix}p & &\\ &  1 & \\ & &c^{-1}\end{psmatrix},p  \right)\frac{d_r p}{|\det p|^{3/2}|c|^{(r-2)/2}} \leq \nu_{d,N,A,A'}(W)\frac{|c|^{1/2-\mathrm{Re}(s_0)}\sigma(c)^{d}}{(1+|c|)^{N}}.
\end{align*}
\end{lem}
\begin{proof}
One can apply the Iwasawa decomposition and Lemma \ref{lem:C:gauge} to deduce the bound.
\end{proof}

  Let $\chi:F^\times \to \CC^\times$. For $W\in C^\infty(U_{r,r-2}(F)\backslash \GL_{r,r-2}(F),\psi \otimes \bar{\psi}),$ define \index{$Z(W,\chi)$}
\begin{align} \label{ZWchi}
Z(W,\chi):=\int_{U_{r-2}(F) \backslash \mathcal{P}_{r-2}(F) \times F^\times} W\left(\begin{psmatrix}p &  \\ &   \begin{psmatrix} 1 & \\ & c^{-1}\end{psmatrix}h_\ell \end{psmatrix},p  \right)\frac{d_r p\chi(c)d^\times c}{|\det p|^{3/2}|c|^{(r-2)/2}}
\end{align}
whenever this is absolute convergent.

  For $(g,g') \in \GL_{r,r-2}(F),$ define
\begin{align} \label{ZWchi:func}
Z(W,\chi)(g,g'^tm_{\ell}):=|\det g'|^{-3/2}Z(\mathcal{R}\left(\begin{psmatrix}
        g' & \\
         & I_2
    \end{psmatrix}g,g'\right)W,\chi). 
\end{align}
This descends to a function in $C^\infty(X^{\circ}_{\ell}(F),\mathcal{H})$ whenever $Z(\mathcal{R}\left(\begin{psmatrix}
        g' & \\
         & I_2
    \end{psmatrix}g,g'\right)W,\chi)$ is absolutely convergent. Similarly, for $f \in  C^\infty(X^\circ_{\ell}(F),\mathcal{H}),$ we define \index{$Z(f,W,\chi)$}
\begin{align} \begin{split} \label{Zint:2}
&Z(f,W,\chi)\\
&:=\int f\left(g,g'^tm_{\ell}\right) \left(\int_{F^\times} J(W)\left(\begin{psmatrix}
        I_{r-2} & \\
           & \begin{psmatrix} 1 & \\ & c^{-1}\end{psmatrix}h_\ell
\end{psmatrix}g,g'^tm_\ell\right)\frac{\chi(c)d^\times c}{|c|^{(r-2)/2}}\right)\frac{|\det g'|d\dot{g}'d\dot{g}}{|\det h_\ell|^{2-r}}
    \end{split}
\end{align}
whenever these integrals are defined.  Here $J(W)$ is defined as in \eqref{eq:JW} and the outer integral is over $\mathcal{P}_{r-2}(F) \backslash \GL_{r-2}(F) \times N(F) \backslash \GL_r(F)$. Note that 
\begin{align} \label{Z:ids} \begin{split}
    Z(W,\chi)(g,g'^tm_\ell)&=\int_{F^\times} J(W)\left(\begin{psmatrix}
        I_{r-2} & \\
           & \begin{psmatrix} 1 & \\ & c^{-1}\end{psmatrix}h_\ell
\end{psmatrix}g,g'^tm_\ell\right)\frac{\chi(c)d^\times c}{|c|^{(r-2)/2}},\\
    Z(f,W,\chi)&=|\det h_\ell|^{r-2}\int  f(g,g'^tm_\ell)Z(W,\chi)(g,g'^tm_{\ell})|\det g'| d\dot{g}'d\dot{g}, \end{split}
\end{align}
where the lower integral is over $\mathcal{P}_{r-2}(F) \backslash \GL_{r-2}(F) \times N(F) \backslash \GL_r(F).$

\begin{lem}
\label{lem:Z:cont0}
Assume $d>0.$  There is a $d'>0$ such that for $\mathrm{Re}(\chi)=0,$ one has a continuous map
\begin{align*}    Z(\cdot,\chi):\mathcal{C}_d(U_{r,r-2}(F) \backslash \GL_{r,r-2}(F),\psi \otimes \overline{\psi}) \lto \mathcal{C}_{d'}(X_\ell^\circ(F),\mathcal{H}).
\end{align*}
Moreover, for any continuous seminorm $\nu$ on the codomain, there is a continuous seminorm $\nu_0$ on the domain such that 
$
\nu(Z(W,\chi)) \leq \nu_0(W)(1+\norm{\chi})^{d_0}
$
for some $d_0>0.$
\end{lem}
\begin{proof} This follows from the
 proof of Lemma \ref{lem:map2X}, with Proposition \ref{prop:J} replaced by Lemma \ref{prop:Jd>0}.
 \end{proof}

Combining \Cref{lem:Xi:bound}, \Cref{lem:Z:cont0}, and the last part of the proof of \Cref{lem:cont}, we deduce the following lemma:

\begin{lem} \label{lem:Z:cont} Assume $d>0.$ For $\mathrm{Re}(\chi) =0,$ one has a separately continuous map
\begin{align*}
\mathcal{C}(X_\ell^\circ(F),\mathcal{H}) \times \mathcal{C}_d(U_{r,r-2}(F) \backslash \GL_{r,r-2}(F),\psi \otimes \overline{\psi})  &\lto \CC\\
(f,W) &\longmapsto Z(f, W,\chi).
\end{align*} 
When $F$ is Archimedean, it is continuous.  If $F$ is non-Archimedean and $K <\GL_{r,r-2}(F)$ is a compact open subgroup, then the restriction of the map to 
$$
\mathcal{C}(X_\ell^\circ(F),\mathcal{H})^K \times \mathcal{C}_d(U_{r,r-2}(F) \backslash \GL_{r,r-2}(F),\psi \otimes \overline{\psi})^K
$$ is continuous.  
\qed
\end{lem}

\begin{lem} \label{lem:nonvanish}
For any generic unitary $\pi$ and $\pi'$ and quasi-character $\chi,$ we can choose a $W \in \mathcal{W}(\pi,\psi) \widehat{\otimes} \mathcal{W}(\pi',\overline{\psi})$ such that the integral defining $Z(W,\chi)$ is absolutely convergent and nonzero.  If $F$ is Archimedean, $\pi$ and $\pi'$ are tempered, and $\chi$ is unitary, then we may choose a $K_{r,r-2}$-finite $W$ such that $Z(W,\chi) \neq 0.$   
\end{lem}

\begin{proof}
Let
$C_c^\infty(U_{r,r-2}(F) \backslash \mathcal{P}_{r,r-2}(F),\psi)$ be the set of smooth functions on $\mathcal{P}_{r,r-2}(F)$ that are compactly supported modulo $U_{r,r-2}(F)$ and satisfy $f(n_1g_1,n_2g_2)=\psi(n_1)\overline{\psi}(n_2)f(g_1,g_2)$ for all $(n_1,n_2) \in U_{r,r-2}(F).$  Given any $f=f_1\otimes f_2\in C_c^\infty(U_{r,r-2}(F) \backslash \mathcal{P}_{r,r-2}(F),\psi \otimes \overline{\psi}),$ we can choose $W \otimes W' \in \mathcal{W}(\pi,\psi) \widehat{\otimes} \mathcal{W}(\pi',\overline{\psi})$ such that $W \otimes W'|_{\mathcal{P}_{r}(F) \times \mathcal{P}_{r-2}(F)}=f$ \cite[(2.2)]{JPSS:Conv} \cite[Theorem 1]{Kemarsky} \cite[Proposition 5]{Jacquet:quasi-split}. Then
\begin{align*}
    Z(W,\chi)&=\int_{U_{r-2}(F) \backslash \mathcal{P}_{r-2}(F) \times F^\times} \omega_\pi(c^{-1})W\left(\begin{psmatrix}cp & \\ & \begin{psmatrix} c& \\
&  1\end{psmatrix}h_\ell \end{psmatrix},p  \right)\frac{\chi(c)d_r pd^\times c}{|\det p|^{3/2}|c|^{(r-2)/2}}\\
&=\int_{U_{r-2}(F) \backslash \mathcal{P}_{r-2}(F) \times F^\times} f\left(\begin{psmatrix}cp & &\\ & c& \\
& & 1\end{psmatrix},p  \right)\frac{(\chi\omega_\pi^{-1})(c)d_r pd^\times c}{|\det p|^{3/2}|c|^{(r-2)/2}}.
\end{align*}
Clearly, the integral converges absolutely. 

We can choose $f=f_1\otimes f_2 \in C_c^\infty(U_r(F) \backslash \mathcal{P}_{r}(F),\psi) \otimes  C_c^\infty(U_{r-2}(F) \backslash \mathcal{P}_{r-2}(F),\overline{\psi})$ such that
\begin{align*}
    \int_{F^\times} f_1\begin{psmatrix}cp & &\\ & c& \\
& & 1\end{psmatrix}\frac{(\chi\omega_\pi^{-1})(c)d^\times c}{|c|^{(r-2)/2}} 
\end{align*}
is not identically zero as a function of $p$, so that there exists $f_2$ with $Z(W,\chi)\neq 0.$  This completes the proof in the non-Archimedean case.  

In the Archimedean case, we must show that we can choose $W$ to be $K_{r,r-2}$-finite.  This follows from the continuity assertion of Lemma \ref{lem:Z:cont0} and the fact that $K_{r,r-2}$-finite vectors are dense.
\end{proof}

Recall the definition of $I_f$ from \eqref{If}.

\begin{lem} \label{lem:Zh}
Assume $\pi$ and $\pi'$ are tempered. For $f \in \mathcal{S}(\GL_{r,r-2}(F) )$,  the integral over $F^\times \times N(F) \backslash \GL_{r}(F) \times U_{r-2}(F) \backslash \GL_{r-2}(F)$ in $Z(I_f,W,\chi)$ converges absolutely and 
$$Z(I_{f},W,\chi)=Z(\pi \otimes \pi'(f)W,\chi)$$
for $\mathrm{Re}(\chi)>-\tfrac{1}{2}.$ 
\end{lem}

\begin{proof} 
Taking a change of variables $z \mapsto g'z$ in the definition of $I_f$, we see that
 $Z(I_{f},W,\chi)$ is bounded by $|\det h_\ell|^{r-2}$ times
\begin{align*}
    &\int_{\mathcal{P}_{r-2}(F) \backslash \GL_{r-2}(F) \times   N(F) \backslash \GL_{r}(F)}\Bigg(\int_{\mathcal{P}_{r-2}(F) \times M_{r-2,2}(F)} |f^\vee|\left( g^{-1}\begin{psmatrix} g'^{-1}p & z\\ & I_2 \end{psmatrix}, g'^{-1}p\right)\frac{|\det g'|^{1/2}d_\ell p dz}{|\det p|^{1/2} } 
    \\& \quad \times\int_{U_{r-2}(F) \backslash \mathcal{P}_{r-2}(F) \times F^\times}|W|\left(\begin{psmatrix} p'g' &  \\ & \begin{psmatrix} 1 & \\ & c^{-1}\end{psmatrix} h_\ell \end{psmatrix}g,p'g'\right)\frac{d_r p' }{|\det p'g'|^{3/2}} \frac{|\chi|(c)d^\times c}{|c|^{(r-2)/2}}\Bigg)|\det g'|d\dot{g}'d\dot{g} \\
    &=\int_{\GL_{r-2}(F) \times    \GL_{r}(F)}\Bigg( |f^\vee|\left( g^{-1}\begin{psmatrix} g'^{-1} & \\ & I_2 \end{psmatrix}, g'^{-1}\right) \\& \quad \times\int_{U_{r-2}(F) \backslash \mathcal{P}_{r-2}(F) \times F^\times}|W|\left(\begin{psmatrix} p'g' &  \\ & \begin{psmatrix} 1 & \\ & c^{-1}\end{psmatrix} h_\ell \end{psmatrix}g,p'g'\right)\frac{d_r p' }{|\det p'|^{3/2}} \frac{|\chi|(c)d^\times c}{|c|^{(r-2)/2}}\Bigg)dg'dg\\
    &=\int_{U_{r-2}(F) \backslash \mathcal{P}_{r-2}(F) \times F^\times}\int_{\GL_{r,r-2}(F)}|f|(g,g')|W|\left(\begin{psmatrix} p' &  \\ & \begin{psmatrix} 1 & \\ & c^{-1}\end{psmatrix} h_\ell \end{psmatrix}g,p'g'\right)dgdg'\frac{d_r p' }{|\det p'|^{3/2}} \frac{|\chi|(c)d^\times c}{|c|^{(r-2)/2}}.
\end{align*}
This converges by  Lemma \ref{lem:whittakerweak} and Lemma \ref{lem:for:zeta}.
\quash{

In the non-Archimedean case, the Whittaker function $W$ is uniformly smooth under the action of $\GL_{r,r-2}(F),$ and hence Corollary \ref{cor:Z:conv} implies that the integral above is convergent for $\mathrm{Re}(\chi)>-\tfrac{1}{2}.$  The same is true in the Archimedean case, but one must use the fact that the estimate \eqref{gauge} may be taken to be uniform for $W$ in a compact subset of $\mathcal{W}(\pi,\psi) \widehat{\otimes} \mathcal{W}(\pi',\overline{\psi})$ by \cite[Proposition 3.5]{JacquetPerfectRS}.}

The identity  now follows from a change of variables. In fact it is essentially the change of variables executed in the displayed equation above but with the absolute values signs removed.  
 \end{proof}

\begin{cor}\label{cor:tempered:ZW=ZfW}
    Assume $\pi,\pi'$ and $\chi$ are tempered. For $f\in \mathcal{C}(\GL_{r,r-2}(F)),$ 
$$
Z(I_{f},W,\chi)=Z(\pi \otimes \pi'(f)W,\chi).$$
\end{cor}

\begin{proof}
    Assume $F$ is Archimedean.  Then both sides are continuous in $f$ by \Cref{lem:cont}, Lemma \ref{lem:whittakerweak}, \Cref{lem:extension}, Lemma \ref{lem:for:zeta} and \Cref{lem:Z:cont}. Therefore, the assertion follows from \Cref{lem:Zh} as $\mathcal{S}(\GL_{r,r-2}(F))$ is dense in $\mathcal{C}(\GL_{r,r-2}(F)).$  The same proof is valid in the non-Archimedean case after we restrict to $K \times K$-invariants for an appropriate compact open subgroup $K \leq \GL_{r,r-2}(F).$
\end{proof}

\subsection{Zeta integrals for nontempered representations} \label{zeta:nontemp}
The Pl\"ucker embedding provides an immersion
\begin{align} \label{embed0}
\det \times \mathrm{Pl}_r:X_r^\circ  &\lto \GG_m \times M_{r-2,r}\wedge \GG_a^r\wedge \GG_a^r\times M_{2,r}
\end{align}
given on points by 
$$
g=\begin{psmatrix}
        v_1\\
        v_2\\
        \vdots \\
         v_{r}
    \end{psmatrix}\longmapsto \bigg(\det g,\begin{psmatrix}
        v_1\\
        v_2\\
        \vdots \\
         v_{r-2}
    \end{psmatrix}\wedge v_{r-1}\wedge v_{r},\begin{psmatrix}
        v_{r-1}\\
         v_{r}
    \end{psmatrix}\bigg).
$$
Let $X_r$ be the closure of $\det \times \mathrm{Pl}_r(X_r^\circ)$ in the codomain.

As in the introduction, set
\begin{align*}
X_\ell:=X_{r} \times \mathcal{M}_\ell \supset X_\ell^\circ.
\end{align*} 
Writing 
$g=n\begin{psmatrix}
    g_1 & \\
    & g_2
\end{psmatrix}k$ for $(n,g_1,g_2,k) \in N(F) \times \GL_{r-2}(F) \times \GL_2(F) \times K_r,$ we define a norm on $X_r^\circ(F)$ via
\quash{\begin{align} \label{gnorm}
    \norm{g}_\ell:=\max\left(|\det g_1\det g_2|,|\det g_1\det g_2|^{-1},\norm{g_1}_{M_{r-2,r-2}}|\det g_2|,\norm{g_2}_{M_{2,2}}\right).
\end{align}}
\begin{align} \label{gnorm}
    \norm{g}_\ell:=\max\left(|\det g_1\det g_2|,|\det g_1\det g_2|^{-1},\norm{g_1}_{M_{r-2,r-2}}|\det g_2|,\norm{g_2}_{M_{2,2}},\frac{\norm{e_2^t h_\ell g_2}_{\GG_a^2}}{|\det g_2|}\right).
\end{align}

\begin{defn} \label{defn:Xi:rapid:loc}
    A function $f \in \mathcal{C}(X^{\circ}_\ell(F),\mathcal{H})$ is \textbf{$(\Xi,X_\ell)$-rapidly decreasing} or simply \textbf{$\Xi$-rapidly decreasing}, if for any $d>0$ and $N>0$ 
    \begin{align} \label{rd:bound}
        |f(g,g'^tm_\ell)|\ll_{d,N} |\Xi_{X_\ell^\circ}(g,g'^tm_\ell)|\max(1,\norm{g}_\ell, \norm{g'^tm_\ell})^{-N}\sigma_{X_\ell^\circ}(g,g'^tm_\ell)^{-d},
    \end{align}   
    and, when $F$ is Archimedean, \eqref{rd:bound} holds with $f$ replaced by $\mathcal{R}(u)f$ for any $u \in U(\mathfrak{gl}_{r,r-2} \oplus \mathrm{Lie}\,(\GL_2)_{\ell}).$
\end{defn}

The motivation for this definition comes from  the immersion $\det \times \mathrm{Pl}$ together with an examination of the support of the basic function $b_\ell,$ defined in \eqref{bell} below. We expect that elements of the Schwartz space of $(X_\ell(F),\mathcal{H})$ will be $\Xi$-rapidly decreasing.

\begin{lem} \label{lem:ZfW} 
Let $d>0$ and $N \ge 0.$ Suppose $f \in \mathcal{C}(X_\ell^\circ(F),\mathcal{H})$ is $\Xi$-rapidly decreasing. There is a continuous seminorm $\nu_{f,N}$ on $\mathcal{C}_{d}(U_{r,r-2}(F) \backslash \GL_{r,r-2}(F),\psi \otimes \overline{\psi})$ such that if
$N\ge \mathrm{Re}(\chi)\ge    0 $ and $N  \geq \mathrm{Re}(s') \ge 0
$
then
$$
\int_{\mathcal{P}_{r-2}(F) \backslash \GL_{r-2}(F) \times N(F) \backslash \GL_r(F)} |f|(g,g'^tm_\ell)|Z(W_{s'},\chi)(g,g'^tm_{\ell})||\det g'| d\dot{g}'d\dot{g} \leq \nu_{f,N}(W).
$$
\end{lem}

\begin{proof} 
Write $(g,g'^tm_\ell)$ using the coordinates in \eqref{coord}. 
By the proof of  Lemma \ref{lem:map2X}, with 
the application of Proposition \ref{prop:J} replaced by 
 Lemma \ref{prop:Jd>0},
there are $d'>0$ and a continuous seminorm $\nu$ on $\mathcal{C}_d(U_{r,r-2}(F) \backslash \GL_{r,r-2}(F),\psi \otimes \overline{\psi})$ such that 
\begin{align*}
    \left|Z(W_{s'},\chi)(g,g'^tm_{\ell})\right|\le \nu(W)\Xi_{X_\ell}^\circ(g,g'^tm_\ell)\sigma_{X_\ell^\circ}(g,g'^tm_\ell)^{d'}|a|^{\mathrm{Re}(s')(r-2)}|\det t''^{-1}|^{\mathrm{Re}(s')}|t_r|^{-\mathrm{Re}(\chi)}.
\end{align*} 
On the other hand, by definition, for any $d_0\ge 0$ and $N_0,N_1,N_2,N_3\ge 0$ 
\begin{align*}
|f(g,g'^tm_\ell)|&\ll_{d_0,N_0}\frac{\Xi_{X_\ell^\circ}(g,g'^tm_\ell)\sigma_{X_\ell^\circ}(g,g'^tm_\ell)^{-d_0}\max(1,|a|)^{-N_0}\max(|z|,|z|^{-1})^{-N_1}}{ \max(|\det g|,|\det g|^{-1})^{N_2}\max(1,|t_r^{-1}z|,\norm{\begin{psmatrix} zt_{r-2}g_1 &  \\ & zt_{r-2} \end{psmatrix}}_{M_{r-2,r-2}}|z^2t_r^{-1}|)^{N_3}}.
\end{align*}
By Lemma \ref{lem:Xi:bound}, it suffices to show there are $N_i$ such that
\begin{align*}
    \sup_{a,t_{r-2},t_r,z,g_1}\frac{\max(1,|a|)^{-N_0}|a|^{\mathrm{Re}(s')(r-2)}|\det t''^{-1}|^{\mathrm{Re}(s')}|t_r|^{-\mathrm{Re}(\chi)}\max(|z|,|z|^{-1})^{-N_1}}{ \max(|t_{r}^{-1}z^rt_{r-2}^{r-2}\det g_1|,|t_r^{-1}z^rt_{r-2}^{r-2}\det g_1|^{-1})^{N_2}\max(1,|t_r^{-1}z|,\norm{\begin{psmatrix} zt_{r-2}g_1 &  \\ & zt_{r-2} \end{psmatrix}}_{M_{r-2,r-2}}|z^2t_r^{-1}|)^{N_3}}
\end{align*}
is finite. 
We choose $N_0=N(r-2).$  Then for any $N_2,N_3$, we can choose $N_1$ so that the above is dominated by 
   \begin{align*}
    \sup_{t_{r-2},t_r,g_1}\frac{|\det t''^{-1}|^{\mathrm{Re}(s')}|t_r|^{-\mathrm{Re}(\chi)}}{ \max(|t_{r}^{-1}t_{r-2}^{r-2}\det g_1|,|t_r^{-1}t_{r-2}^{r-2}\det g_1|^{-1})^{N_2}\max(1,|t_r^{-1}|,\norm{\begin{psmatrix} t_{r-2}g_1 &  \\ & t_{r-2} \end{psmatrix}}_{M_{r-2,r-2}}|t_r^{-1}|)^{N_3}}.
\end{align*}
We now change variables $t_r \mapsto t_rt_{r-2}^{r-2}\det g_1$ and take $N_2$ large in a sense depending on $N$ and $N_3$ to see that the above is dominated by 
\begin{align*}
    \sup_{t_{r-2},g_1}\frac{|\det t''^{-1}|^{\mathrm{Re}(s')}|t_{r-2}^{r-2}\det g_1|^{-\mathrm{Re}(\chi)}}{\max(1,|t_{r-2}^{r-2}\det g_1|^{-1},\norm{\begin{psmatrix} t_{r-2}g_1 &  \\ & t_{r-2} \end{psmatrix}}_{M_{r-2,r-2}}|t_{r-2}^{r-2}\det g_1|^{-1})^{N_3}}.
\end{align*}
Since $|\det t''^{-1}|^{\mathrm{Re}(s')}\leq 1,$ this is bounded provided $N_3\geq N.$
\end{proof}

Let $M$ be a standard Levi subgroup of $\GL_{r,r-2},$ let $\underline{\sigma}$ be a square-integrable representation of $M(F),$ and let $\underline{\varphi}$ be a smooth vector in the space of $\underline{\sigma}.$

\begin{lem} \label{lem:ZfW2}
 Let $\Omega \subset \mathfrak{a}_{M\CC}^*$ be a compact set. Suppose $f \in \mathcal{C}(X_{\ell}^\circ(F),\mathcal{H})$ is $\Xi$-rapidly decreasing.
  There are constants $N_0,N_1>0$ such that for $\underline{\lambda}\in \Omega,$ $N_2,N_3 \geq N_1,$ and $N_2\ge \mathrm{Re}(s')\ge N_1,$ $N_3\ge \mathrm{Re}(\chi)  \ge N_1,$
\begin{align} \label{to:bound:619}
&\int_{\mathcal{P}_{r-2}(F) \backslash \GL_{r-2}(F) \times N(F) \backslash \GL_r(F)} |f|(g,g'^tm_\ell)|Z((W_{\underline{\varphi_{\lambda}}})_{s'},\chi)(g,g'^tm_{\ell})||\det g'| d\dot{g}'d\dot{g}
\end{align}
is $O_{f,\Omega,N_2,N_3, \varphi}\left((1+\norm{\chi})^{N_0}\right)$
\end{lem}

\begin{proof}   
Assume without loss of generality that $s'$ is real. We proceed as in the proof of Lemma \ref{lem:ZfW}, writing $(g,g'^tm_\ell)$ in the coordinates in \eqref{coord}.
 By Lemma \ref{lem:ZWlambda} 
\begin{align*}
    &|Z((W_{\underline{\varphi_{\lambda}}})_{s'},\chi)(g,g'^tm_{\ell})|\ll_{N_0,N_2,N_3,N} \frac{\left(\norm{z}\norm{a}\norm{t_{r-2}}\norm{g_1}\right)^{N'}\max(1,|y_2|)^{-\mathrm{Re}(\chi)+N'+(r-2)/2}}{|a|^{(3/2-s')(r-2)}|t_r|^{\mathrm{Re}(\chi)-(r-2)/2}(1+|at_{r-2}|)^N}(1+\norm{\chi})^{N_0}
\end{align*}
for some $N'>0$ depending on $\Omega$ and any $N_2\ge s'\ge 2N'+r, N_3\ge \mathrm{Re}(\chi)\ge N'+(r-2)/2.$  By the definition of a $\Xi$-rapidly decreasing function, the integral in the lemma is dominated by $(1+\norm{\chi})^{N_0}$ times
\begin{align*}
    &\int \Phi(z^rt_{r-2}^{r-2}t_r^{-1}\det g_1, z^3t_{r}^{-1}t_{r-2}\begin{psmatrix} g_1 & y_1\\ & 1\end{psmatrix},z,t_{r}^{-1}z,a)\\
    &\times \frac{\left(\norm{z}\norm{a}\norm{t_{r-2}}\norm{g_1}\right)^{N'}}{|a|^{(1-s')(r-2)}(1+|at_{r-2}|)^N\max(1,|y_2|)^{\mathrm{Re}(\chi)-N'+(3-r)/2}} \frac{ dy_1dy_2 d^\times a d^\times z d^\times t_{r-2} d^\times t_{r}dg_1}{|t_r|^{\mathrm{Re}(\chi)}|t_{r-2}|^{r-2}|\det g_1|^{3/4}}
\end{align*}
for some Schwartz function $\Phi \in \mathcal{S}(F^\times \times M_{r-2}(F)\times F^\times \times F^2)$. We have used Corollary \ref{cor:xi:bounded} to bound the $\Xi$ functions trivially.

Taking $N_1 \gg N'$ and executing the integrals over $z,y_1,y_2$ the above is dominated by 
 \begin{align*}
    &\int \Phi^{(0)}\left(\frac{t_{r-2}^{r-2}\det g_1}{t_r}, 
    \frac{t_{r-2}g_1}{t_r},\frac{t_{r-2}}{t_r},t_r^{-1},a\right)\frac{\left(\norm{a}\norm{t_{r-2}}\norm{g_1}\right)^{N'}}{|a|^{(1-s')(r-2)}(1+|at_{r-2}|)^N} \frac{d^\times a  d^\times t_{r-2} d^\times t_{r}dg_1}{|t_r|^{\mathrm{Re}(\chi)-(r-3)}|t_{r-2}|^{2r-5}|\det g_1|^{3/4}}
\end{align*}
for some $\Phi^{(0)} \in \mathcal{S}(F^\times \times M_{r-3}(F) \times F^3).$

Taking $N \gg \max(N',N_2),$ executing the integral over $a$, and using submultiplicativity of norms,  the quantity above is bounded by 
\begin{align*}
    &\int_{(F^\times)^2 \times \GL_{r-3}(F)}\Phi'\left(\frac{t_{r-2}^{r-2}\det g_1}{t_r}, \frac{t_{r-2} g_1}{t_r}, \frac{t_{r-2}}{t_r} ,t_r^{-1}\right) \frac{(\norm{t_{r-2}}^2\norm{g_1})^{N'} d^\times t_{r-2} d^\times t_{r}dg_1}{\max(1,|t_{r-2}|)^{(s'-1)(r-2)}|t_r|^{\mathrm{Re}(\chi)-(r-3)}|t_{r-2}|^{2r-5}|\det g_1|^{3/4}}
\end{align*}
for some $\Phi' \in \mathcal{S}(F^\times \times M_{r-3}(F) \times F^2).$ 

We change variables $g_1 \mapsto t_{r-2}^{-1}g_1$ then $t_{r-2} \mapsto t_{r-2} t_r \det g_1^{-1},$ and execute the integral over $t_{r-2}$ to see that the above is
\begin{align} \nonumber
    &\int_{(F^\times)^2 \times \GL_{r-3}(F)}\Phi''\left(\frac{g_1}{t_r}, \det g_1^{-1},t_r^{-1}\right) \frac{(\norm{t_r \det g_1^{-1}}^2\norm{t_r^{-1}g_1\det g_1})^{N'}  d^\times t_{r}dg_1}{\max(1,|t_r\det g_1^{-1}|)^{(s'-1)(r-2)}|t_r|^{\mathrm{Re}(\chi)}\chi'(t_r,\det g_1)}\\
    &=\int_{F^\times \times \GL_{r-3}(F)}\Phi''\left(g_1, t_r^{3-r}\det g_1^{-1},t_r^{-1}\right) \frac{(\norm{t_r^{4-r} \det g_1^{-1}}^2\norm{t_r^{r-3}g_1\det g_1})^{N'}  d^\times t_{r}dg_1}{\max(1,|t_r^{4-r}\det g_1^{-1}|)^{(s'-1)(r-2)}|t_r|^{\mathrm{Re}(\chi)}\chi'(t_r,t_{r}^{r-3}\det g_1)}  \label{before:split}
\end{align}
for some $\Phi'' \in \mathcal{S}(M_{r-3}(F) \times F^2)$ and some quasi-character $\chi':(F^\times)^2 \to \RR_{>0}.$ 
Here we have changed variables $g_1 \mapsto t_rg_1.$

Now
$$
\left(\norm{t_r^{4-r} \det g_1^{-1}}^2\norm{t_r^{r-3}g_1\det g_1}\right)^{N'} \ll \norm{g_1}^{A_1N'}\norm{t_r}^{A_2N'}
$$
for some $A_1,A_2 \in \RR_{>0}$ by  \eqref{pullback} and \eqref{sig:ineq}. Note that
\begin{align*}
    \norm{g_1}\asymp \max(\norm{g_1^{-1}}_{M_{r-3}},\norm{g_1}_{M_{r-3}})\ll \max(|\det g_1|^{-1}\norm{g_1}_{M_{r-3}}^{r-4},\norm{g_1}_{M_{r-3}}).
\end{align*}
For any $C\ge 0,$
\begin{align*}
    \Phi''( g_1, t_{r}^{3-r}\det g_1^{-1}  ,t_{r}^{-1})\le  \Phi_C( g_1,t_{r}^{-1})|t_r^{3-r}\det g_1^{-1}|^{-C}
\end{align*}
for some $\Phi_{C} \in \mathcal{S}(M_{r-3}(F) \times F)$ depending on $C.$ Combining these observations, we see that the integral above is dominated by 
\begin{align}
\int_{F^\times \times \GL_{r-3}(F)}\Phi_C\left(g_1,t_r^{-1}\right) \frac{ |t_r^{3-r}\det g_1^{-1}|^{-C}\norm{t_r}^{A_2N'}\norm{g_1}^{rA_1N'}\max(1,|\det g_1|^{-1})^{A_1N'}dg_1}{\max(1,|t_r^{4-r}\det g_1^{-1}|)^{(s'-1)(r-2)}|t_r|^{\mathrm{Re}(\chi)}\chi(t_r,t_{r}^{r-3}\det g_1)}
\end{align}
For $C N'$  (as a function of $N'$) the integral over $g_1$ converges, and for $\mathrm{Re}(\chi) \gg C$ the integral over $t_r$ converges. 
\end{proof}

\begin{lem} \label{lem:uniform:ZfW} Assume $F$ is Archimedean. Under the same hypothesis as \Cref{lem:ZfW2}, the integral \eqref{to:bound:619} is $O_{f,\Omega,N_2,N_3,N,\varphi}(\max\left(1,|\mathrm{Im}(s')|,\norm{\chi}\right)^{-N})
$
for any $N >0.$
\end{lem}

\begin{proof}
Let $\varphi_0$ be a smooth vector in the space of $\sigma$ and let $S \subset U(\mathfrak{gl}_{r,r-2})$ be any finite subset.
We claim that the implicit bound in \Cref{lem:ZfW2} may be taken to be uniform over $\varphi \in \{u.\varphi_0:u \in S\}.$  To see this we note that the dependence on $\varphi$ comes from \Cref{lem:ZWlambda}, and the dependence on $\varphi$ in \Cref{lem:ZWlambda} relies on \Cref{lem:Arch:Whitt:fam:bound}, which has the desired uniformity.  
With this in mind, the bound in terms of $s'$ follows from  
a standard integration by parts argument.  To obtain the bound in terms of $\norm{\chi},$ one also applies integration by parts, but this time differentiating $f$ instead of the Whittaker function. 
\end{proof}

\section{The Plancherel formula for $\mathcal{S}(X_\ell^\circ(F),\mathcal{H})$} \label{sec:Pl:Xcir}

We prove a Plancherel formula for $K_{r} \times K_{r-2}$-finite functions in $\mathcal{S}(X^{\circ}_{\ell}(F),\mathcal{H})$ in Theorem \ref{thm:Plancherel} below.  Before this we collect some preliminary results.

For $f\in \mathcal{S}(\GL_r(F)^2)$ we let 
\begin{align}
    W_f(g_1',g_2',g_1,g_2):=\int_{U_r(F)^2} f(g_1'^{-1}n_1g_2',g_1^{-1}n_2g_2)\psi(n_1)\overline{\psi}(n_2)dn_1dn_2.
\end{align}
The definition extends continuously to $\mathcal{C}^w(\GL_r(F)^2)$ by \cite[Lemma 2.14.1]{BP:GLn}.

The proof of the following lemma is inspired by the proof of \cite[Proposition 4.3.1]{BP:GLn}:

\begin{lem} \label{lem:RS}
For $f \in \mathcal{S}(\GL_{r}(F)^2)$ one has
\begin{align*}
    \int_{\mathcal{P}_{r}(F)}f(p,p)\frac{d_r p}{|\det p|^{1/2}}=\int_{(U_r(F) \backslash \mathcal{P}_{r}(F))^2}W_{f}(p_1,p_2,p_1,p_2)\frac{d_rp_1d_rp_2}{|\det p_1p_2|^{1/2}}.
\end{align*}
The integrals are absolutely convergent.
\end{lem}

\begin{proof}
The convergence claim in the lemma is the assertion that 
\begin{align} \label{convergence:claim}
\int_{\GL_{r-1}(F)^2}\left|W_{f}\left(\begin{psmatrix} g_1 & \\ & 1 \end{psmatrix},\begin{psmatrix} g_2 & \\ & 1 \end{psmatrix},\begin{psmatrix} g_1 & \\ & 1 \end{psmatrix},\begin{psmatrix} g_2 & \\ & 1 \end{psmatrix}\right)\right|\frac{dg_1dg_2}{|\det g_1g_2|^{1/2}}<\infty.
\end{align}
By \eqref{Whitt} and Lemma \ref{lem:C:gauge}, there is a $d>0$ such that for any $N>0,$ the integral \eqref{convergence:claim} is dominated by the square of
\begin{align*}
\int_{T_{r-1}(F)} \frac{\delta_{B_r}\begin{psmatrix} a & \\
 & 1 \end{psmatrix}\delta_{B_{r-1}}^{-1}(a) \sigma(a)^d}{(1+|a_{r-1}|)^{N}\prod_{i=1}^{r-2}(1+|\alpha_i(a)|)^{N}} \frac{d^\times a}{|\det a|^{1/2}}.
\end{align*}
 This integral converges for $N$ sufficiently large.

To prove the identity we proceed by induction on $r.$  The identity reduces to $f(1,1)=f(1,1)$ when $r=1.$  
Assume that the identity is true for $r-1.$  
We write
\begin{align*}
    &\int_{\mathcal{P}_{r}(F)}f(p,p) \frac{d_r p}{|\det p|^{1/2}}\\
    &=\int_{\GL_{r-1}(F)} \int_{F^{r-1}}f\left(\begin{psmatrix}I_{r-1} & x_1\\ & 1 \end{psmatrix} \begin{psmatrix} g & \\ & 1 \end{psmatrix} ,\begin{psmatrix}I_{r-1} & x_1\\ & 1 \end{psmatrix} \begin{psmatrix} g & \\ & 1 \end{psmatrix}\right) dx_1 \frac{dg}{|\det g|^{1/2}}\\
    &=\int_{\GL_{r-1}(F)} \Bigg(\int_{F^{r-1}}\Bigg(\int_{(F^{r-1})^2}\psi(-tx_2)f\left(\begin{psmatrix}I_{r-1} & x_1\\ & 1 \end{psmatrix} \begin{psmatrix} g & \\ & 1 \end{psmatrix} 
,\begin{psmatrix}I_{r-1} & x_1\\ & 1 \end{psmatrix}\begin{psmatrix}I_{r-1} & x_2\\ & 1 \end{psmatrix} \begin{psmatrix} g & \\ & 1 \end{psmatrix}\right)dx_1dx_2\Bigg)dt\Bigg)  \frac{dg}{|\det g|^{1/2}}.
\end{align*}
Here we have employed Fourier inversion.
We change variables $x_2 \mapsto x_2-x_1$ to obtain
\begin{align*}
 \int_{\GL_{r-1}(F)} \Bigg(\int_{F^{r-1}}\Bigg(\int_{(F^{r-1})^2}\psi(tx_1-tx_2)f\left(\begin{psmatrix}I_{r-1} & x_1\\ & 1 \end{psmatrix}\ \begin{psmatrix} g & \\ & 1 \end{psmatrix} 
,\begin{psmatrix}I_{r-1} & x_2\\ & 1 \end{psmatrix} \begin{psmatrix} g & \\ & 1 \end{psmatrix}\right)dx_1dx_2\Bigg)dt\Bigg) \frac{dg}{|\det g|^{1/2}}.
\end{align*}
Via the isomorphism
\begin{align*}
    \mathcal{P}_{r-1}(F) \backslash \GL_{r-1}(F) &\lto F^{r-1}-\{0\}\\
    g' &\longmapsto g^{\prime t}e_{r-1},
\end{align*}
we equip $\mathcal{P}_{r-1}(F) \backslash \GL_{r-1}(F)$ with  the unique quasi-invariant measure $d\dot{g}'$  such that 
$|\det g'|d\dot{g}'$ induces the Haar measure on $F^{r-1}$. Then we can rewrite the integral above as
\begin{align*}
    &\int_{\GL_{r-1}(F) \times \mathcal{P}_{r-1}(F) \backslash \GL_{r-1}(F)}\Bigg(\int_{(F^{r-1})^2}\psi(e_{r-1}^tg'x_1-e_{r-1}^tg'x_2)\\& \times f\left(\begin{psmatrix}I_{r-1} & x_1\\ & 1 \end{psmatrix} \begin{psmatrix} g & \\ & 1 \end{psmatrix} ,\begin{psmatrix}I_{r-1} & x_2\\ & 1 \end{psmatrix} \begin{psmatrix} g & \\ & 1 \end{psmatrix}\right)dx_1dx_2\Bigg)|\det g'|d\dot{g}' \frac{dg}{|\det g|^{1/2}}\\
    &=\int_{\GL_{r-1}(F)\times \mathcal{P}_{r-1}(F) \backslash \GL_{r-1}(F)}\Bigg(\int_{(F^{r-1})^2}\psi(e_{r-1}^tx_1-e_{r-1}^tx_2)\\& \times f\left(\begin{psmatrix} g' & \\ & 1 \end{psmatrix}^{-1}\begin{psmatrix}I_{r-1} & x_1\\ & 1 \end{psmatrix} \begin{psmatrix} g'g & \\ & 1 \end{psmatrix} ,\begin{psmatrix} g' & \\ & 1 \end{psmatrix}^{-1}\begin{psmatrix}I_{r-1} & x_2\\ & 1 \end{psmatrix} \begin{psmatrix} g'g & \\ & 1 \end{psmatrix}\right)dx_1 dx_2 \Bigg) \frac{d\dot{g}'}{|\det g'|}\frac{dg}{|\det g|^{1/2}}\\
    &=\int_{\GL_{r-1}(F)\times \mathcal{P}_{r-1}(F) \backslash \GL_{r-1}(F)}\Bigg(\int_{(F^{r-1})^2}\psi(e_{r-1}^tx_1-e_{r-1}^tx_2)\\& \times f\left(\begin{psmatrix} g' & \\ & 1 \end{psmatrix}^{-1}\begin{psmatrix}I_{r-1} & x_1\\ & 1 \end{psmatrix} \begin{psmatrix} g & \\ & 1 \end{psmatrix} ,\begin{psmatrix}  g'& \\ & 1 \end{psmatrix}^{-1}\begin{psmatrix}I_{r-1} & x_2\\ & 1 \end{psmatrix} \begin{psmatrix} g & \\ & 1 \end{psmatrix}\right)dx_1 dx_2 \Bigg) \frac{d\dot{g}'}{|\det g'|^{1/2}}\frac{dg}{|\det g|^{1/2}}.
\end{align*}

We rewrite the above slightly as
\begin{align}\label{eq:beforeinduction}
\begin{split}
&\int_{(\mathcal{P}_{r-1}(F) \backslash \GL_{r-1}(F))^2} \Bigg(\int_{(F^{r-1})^2}\psi(e_{r-1}^tx_1-e_{r-1}^tx_2)\\& \times \int_{\mathcal{P}_{r-1}(F)}f\left(\begin{psmatrix} g' & \\ & 1 \end{psmatrix}^{-1}\begin{psmatrix}I_{r-1} & x_1\\ & 1 \end{psmatrix} \begin{psmatrix} pg & \\ & 1 \end{psmatrix} ,\begin{psmatrix} g' & \\ & 1 \end{psmatrix}^{-1}\begin{psmatrix}I_{r-1} & x_2\\ & 1 \end{psmatrix} \begin{psmatrix} pg & \\ & 1 \end{psmatrix}\right)\frac{d_\ell p}{|\det p|^{1/2}}dx_1dx_2\Bigg) \frac{d\dot{g}'d\dot{g}}{|\det g'g|^{1/2}}.
\end{split}
\end{align}
Then we apply the induction hypothesis to the function
\begin{align*}
    \widetilde{f}(p_1,p_2):=f\left(\begin{psmatrix} g' & \\ & 1 \end{psmatrix}^{-1}\begin{psmatrix}I_{r-1} & x_1\\ & 1 \end{psmatrix} \begin{psmatrix} p_1g & \\ & 1 \end{psmatrix} ,\begin{psmatrix} g' & \\ & 1 \end{psmatrix}^{-1}\begin{psmatrix}I_{r-1} & x_2\\ & 1 \end{psmatrix} \begin{psmatrix} p_2g & \\ & 1 \end{psmatrix}\right)|\det p_1p_2|^{-1/2}.
\end{align*}
Note that 
\begin{align*}
    W_{\tilde{f}}(p_1,p_2,p_1,p_2)&=\int_{U_{r-1}(F)^2}f\left(\begin{psmatrix} g' & \\ & 1 \end{psmatrix}^{-1}\begin{psmatrix}I_{r-1} & x_1\\ & 1 \end{psmatrix} \begin{psmatrix} p_1^{-1}n_1p_2g & \\ & 1 \end{psmatrix} ,\begin{psmatrix} g' & \\ & 1 \end{psmatrix}^{-1}\begin{psmatrix}I_{r-1} & x_2\\ & 1 \end{psmatrix} \begin{psmatrix} p_1^{-1}n_2p_2g & \\ & 1 \end{psmatrix}\right)\\
    &\quad \quad \quad \quad \times |\det p_1^{-1}p_2|^{-1}\psi(n_1-n_2)dn_1dn_2.
\end{align*}
Therefore, by the induction hypothesis and the change of variables $(x_1,x_2)\mapsto (p_1^{-1}x_1, p_1^{-1}x_2),$ \eqref{eq:beforeinduction} is
\begin{align*}
&\int_{(\mathcal{P}_{r-1}(F) \backslash \GL_{r-1}(F))^2} \Bigg(\int_{(U_{r-1}(F) \backslash \mathcal{P}_{r-1}(F))^2}W_{f}\left(\begin{psmatrix} p_1g' & \\ & 1 \end{psmatrix}, \begin{psmatrix} p_2g & \\ & 1 \end{psmatrix},\begin{psmatrix} p_1g' & \\ & 1 \end{psmatrix}, \begin{psmatrix} p_2g & \\ & 1 \end{psmatrix}\right)\frac{d_rp_1 d_rp_2}{|\det p_1p_2|^{3/2}} \Bigg) \frac{d\dot{g} d\dot{g}'}{|\det gg'|^{1/2}}
\\&=\int_{(U_{r-1}(F) \backslash \GL_{r-1}(F))^2}W_{f}\left(\begin{psmatrix} g_1 & \\ & 1 \end{psmatrix},\begin{psmatrix} g_2 & \\ & 1 \end{psmatrix},\begin{psmatrix} g_1 & \\ & 1 \end{psmatrix},\begin{psmatrix} g_2 & \\ & 1 \end{psmatrix}\right) \frac{d\dot{g}_1d\dot{g}_2}{|\det g_1g_2|^{1/2}}.
\end{align*}
Here we have used \eqref{convergence:claim} to justify unfolding.   The lemma follows.
\end{proof}

\quash{ \textcolor{red}{This is using }
Let me take another stab at this.  Let
\begin{align}
    w:=\begin{psmatrix} & & 1 \\ & \textrm{\reflectbox{$\ddots$}} & \\ 1 & & \end{psmatrix}
\end{align}
Using Fourier inversion we write
\begin{align*}
    &\int_{\mathcal{P}_r(F)}f(p,p)d_rp\\&=\int_{\mathcal{P}_r(F) \backslash \GL_r(F)}\Bigg(\int_{\mathcal{P}_r(F) \backslash \GL_r(F)}\psi\left( \langle e_1wg_2w^{-1}, e_rg_1-e_r\rangle \right)\int_{\mathcal{P}_r(F)} f(pg_1,pg_1)d_rp|\det g_1|d\dot{g}_1\Bigg)|\det g_2|dg_2\\
    &=\int_{\mathcal{P}_r(F) \backslash \GL_r(F)}\int_{\GL_r(F)}\psi\left( \langle e_1wg_2w^{-1}, e_rg_1-e_r\rangle \right) f(g_1,g_1)|\det g_1|dg_1|\det g_2|dg_2\\
    &=\int_{\GL_r(F)}\int_{\mathcal{P}_r(F) \backslash \GL_r(F)}\psi\left( \langle e_1wg_2w^{-1}, e_rg_1-e_r\rangle \right) f(g_1,g_1)|\det g_2|dg_2|\det g_1|dg_1\\
    &=\int_{\GL_r(F)}\int_{\mathcal{P}_r(F) \backslash \GL_r(F)}\psi\left( \langle e_1, e_rg_1\r-e_r\rangle \right) f(g_1wg_2^{-t}w^{-1},g_1wg_2^{-t}w^{-1})dg_2|\det g_1|dg_1
\end{align*}
We employ the split version of \cite[Proposition 4.3.1]{BP:GLn} to see that this is 
\begin{align}
    \int_{\mathcal{P}_r(F) \backslash \GL_r(F)}\int_{U_r(F) \backslash \mathcal{P}_r(F)}\int_{U_r(F) \backslash \GL_r(F)}\psi\left( \langle e_rwg_2w^{-1}, e_rg_1-e_r\rangle \right) W_f(p,g_1)d_rp|\det g_1|d\dot{g}_1|\det g_2|dg_2
\end{align}
}

\begin{lem} \label{lem:unfold}
For $f \in \mathcal{C}(\GL_{r,r-2}(F))$ one has \begin{align*}
    &I_{f}(I_r,m_{\ell})=\int\Bigg(\int_{F^\times}W_{f^\vee}\left(\begin{psmatrix} p_1 & \\ & \begin{psmatrix} 1 & \\ & c_1^{-1}\end{psmatrix}h_\ell \end{psmatrix},\begin{psmatrix}p_2 & \\ & \begin{psmatrix} 1 & \\ & c_1^{-1}\end{psmatrix}h_\ell \end{psmatrix},p_1,p_2\right)\frac{d^\times c_1}{|c_1|^{r-2}\zeta(1)}\Bigg)\frac{|\det h_\ell|^{r-2}d_r p_1d_r p_2}{|\det p_1 p_2|^{3/2}},
\end{align*} 
where the integral is over $(U_{r-2}(F) \backslash \mathcal{P}_{r-2}(F))^2.$
\end{lem}

\begin{proof}
By Lemma \ref{lem:extension} and Lemma \ref{lem:for:zeta}, both sides are continuous in $f$, so it suffices to prove the lemma for $f\in \mathcal{S}(\GL_{r,r-2}(F)).$ By Lemma \ref{lem:RS} we have 
\begin{align*}
    &\int_{\mathcal{P}_{r-2}(F)}f^\vee\left(\begin{psmatrix} p & z\\ & I_2 \end{psmatrix},p \right)\frac{d_\ell p}{|\det p|^{1/2}}\\
    &=\int_{(U_{r-2}(F) \backslash \mathcal{P}_{r-2}(F))^2}\left(\int_{U_{r-2}(F)^2}\psi(n_1n_2^{-1})f^\vee\left(\begin{psmatrix}p_1^{-1} n_1 p_2 & z\\ & I_2 \end{psmatrix},p_1^{-1}n_2p_2\right)\frac{dn_1dn_2}{|\det p_1p_2|^{1/2}|\det p_1^{-1}p_2|}\right)d_rp_1d_rp_2.
\end{align*}
Thus using \eqref{If} 
\begin{align*}
    I_f(I_r,m_{\ell})=& \int_{\mathcal{P}_{r-2}(F) \times M_{r-2,2}(F)}\overline{\psi}(\langle m_{\ell} ,z \rangle)f^\vee\left(\begin{psmatrix}p & z\\ & I_2 \end{psmatrix},p\right)\frac{d_\ell p dz}{|\det p|^{1/2}}\\
    =&\int_{(U_{r-2}(F) \backslash \mathcal{P}_{r-2}(F))^2}\Bigg(\int_{U_{r-2}(F)^2 \times M_{r-2,2}(F)}\overline{\psi}(\langle  m_{\ell},z\rangle)\psi(n_1n_2^{-1})\\
    & \times f^\vee\left(\begin{psmatrix}p_1^{-1} n_1 p_2 & z\\ & I_2 \end{psmatrix},p_1^{-1}n_2p_2\right)dzdn_1dn_2\Bigg)\frac{d_r p_1d_r p_2}{|\det p_1|^{-1/2}|\det p_2|^{3/2}}\\
    =&\int_{(U_{r-2}(F) \backslash \mathcal{P}_{r-2}(F))^2}\Bigg(\int_{U_{r-2}(F)^2 \times M_{r-2,2}(F)}\overline{\psi}(\langle (e_{r-2},0),z\rangle)\psi(n_1n_2^{-1})\\
    & \times f^\vee\left(\begin{psmatrix} p_1^{-1} & \\ & h_\ell^{-1} \end{psmatrix}\begin{psmatrix} n_1  & z\\ & I_2 \end{psmatrix}\begin{psmatrix}p_2 & \\ & h_\ell \end{psmatrix},p_1^{-1}n_2p_2\right)dzdn_1dn_2\Bigg)\frac{ 
   |\det h_\ell|^{r-2}d_r p_1d_r p_2}{|\det p_1p_2|^{3/2}}.
\end{align*}
Here we have used \eqref{M:pair:equiv}.
We apply Fourier inversion to write this as 
\begin{align*}
   \int &\Bigg(\int_F\Bigg(\int_{U_{r-2}(F)^2 \times M_{r-2,2}(F) \times F}\overline{\psi}(\langle (e_{r-2},0),z\rangle)\psi(n_1n_2^{-1})\psi(c_1t)\\& \times f^\vee\left(\begin{psmatrix} p_1^{-1} & \\ & h_\ell^{-1} \end{psmatrix}\begin{psmatrix} n_1  & z_1 & z_2\\ & 1 & t\\ & & 1 \end{psmatrix}\begin{psmatrix}p_2 & \\ & h_\ell \end{psmatrix},p_1^{-1}n_2p_2\right)dzdn_1dn_2dt\Bigg)dc_1\Bigg)\frac{|\det h_\ell|^{r-2}d_r p_1d_r p_2}{|\det p_1p_2|^{3/2}}  
\\
   =&\int\Bigg(\int_F\Bigg(\int_{U_{r-2}(F)^2 \times M_{r-2,2}(F) \times F}\overline{\psi}(\langle (e_{r-2},0),z\rangle)\psi(n_1n_2^{-1})\psi(t)\\
   & \times f^\vee\left(\begin{psmatrix} p_1^{-1} & \\ & \left(\begin{psmatrix} 1 & \\ & c_1^{-1}\end{psmatrix}h_\ell\right)^{-1} \end{psmatrix}\begin{psmatrix} n_1  & z_1 & z_2\\ & 1 & t\\ & & 1 \end{psmatrix}\begin{psmatrix}p_2 & \\ & \begin{psmatrix} 1 & \\ & c_1^{-1}\end{psmatrix}h_\ell \end{psmatrix},p_1^{-1}n_2p_2\right)dzdn_1dn_2dt\Bigg)\frac{dc_1}{|c_1|^{r-1}}\Bigg)\frac{|\det h_\ell|^{r-2}d_r p_1d_r p_2}{|\det p_1p_2|^{3/2}}\\
   =&\int\Bigg(\int_{F^\times}W_{f^\vee}\left(\begin{psmatrix} p_1 & \\ & \begin{psmatrix} 1 & \\ & c_1^{-1}\end{psmatrix}h_\ell \end{psmatrix},\begin{psmatrix}p_2 & \\ & \begin{psmatrix} 1 & \\ & c_1^{-1}\end{psmatrix}h_\ell \end{psmatrix},p_1,p_2\right)\frac{d^\times c_1}{|c_1|^{r-2}\zeta(1)}\Bigg)\frac{
   |\det h_\ell|^{r-2}d_r p_1d_r p_2}{|\det p_1p_2|^{3/2}},
\end{align*}
where the outer integrals are over $(U_{r-2}(F) \backslash \mathcal{P}_{r-2}(F))^2.$
\end{proof}

Let $\mathfrak{F}$ be a finite set of $K_{r,r-2}$-types stable under taking contragredients. Let \index{$\mathcal{C}(\GL_{r,r-2}(F))_{\mathfrak{F}}$} \index{$\mathcal{S}(X_\ell^\circ(F),\mathcal{H})_{\mathfrak{F}}$}
\begin{align} \label{types} 
\mathcal{C}(\GL_{r,r-2}(F)){}_{\mathfrak{F}}<\mathcal{C}(\GL_{r,r-2}(F)),\quad 
\mathcal{S}(X_\ell^\circ(F),\mathcal{H}){}_{\mathfrak{F}}\leq \mathcal{S}(X_{\ell}^\circ(F),\mathcal{H}) 
\end{align}
be the subspace spanned by functions that transform according to these $K_{r,r-2}$-types on the right.  
This definition makes sense in the Archimedean or non-Archimedean settings, although it is more usual to work with functions fixed under suitable compact open subgroups in the non-Archimedean setting. We endow $\mathcal{C}(\GL_{r,r-2}(F)){}_{\mathfrak{F}}$ with the subspace topology.
We make use of analogous notation for similar spaces, e.g.,~$C_c^\infty(\GL_{r,r-2}(F))_{\mathfrak{F}} < C_c^\infty(\GL_{r,r-2}(F)).$ 

Let 
$
\underline{\sigma} \in \mathrm{Irr}_2(M)
$
for some semistandard Levi subgroup $M$ of $\GL_{r,r-2}.$
Let \index{$\mathcal{O}(\underline{\sigma})$}
\begin{align*}
    \mathcal{O}(\underline{\sigma}):=\{I(\underline{\sigma_{\lambda}}):\underline{\lambda} \in i \mathfrak{a}_M^\ast\} \subset \mathrm{Tp}_{r,r-2}.
\end{align*}
This is referred to as the \textbf{inertial orbit} of $I(\underline{\sigma}).$ 
Recall from  below \eqref{loc:isom} that $d\mu:=d\mu_{\GL_{r,r-2}}$ denotes the Plancherel measure. We will adopt a helpful notational convention.  
An expression of the form
\begin{align} \label{int}
\int_{\mathrm{Tp}_{r,r-2}}\sum_{\underline{\varphi} \in \mathcal{B}_{\underline{\sigma}}}(*) d\mu(I(\underline{\sigma}))
\end{align}
will have the following meaning.  
We always assume that for each inertial orbit in $\mathrm{Tp}_{r,r-2}$ we have chosen a base point $I(\underline{\sigma_b})
 \in \mathrm{Tp}_{r,r-2}$ and that $
\mathcal{B}_{\underline{\sigma_{b}\lambda}}=\mathcal{B}_{\underline{\sigma_{b}}}(\underline{\lambda})$ in the notation of \eqref{base:twist}.  In an expression of the form \eqref{int}, the terms abbreviated by $(*)$ may depend on the choice of basepoint, but all of our assertions will be independent of this choice.

The following is our Plancherel inversion formula for the Schwartz space of the open orbit:
\begin{thm} \label{thm:Plancherel} 
For $f \in \mathcal{S}(X^\circ_{\ell}(F),\mathcal{H})_{\mathfrak{F}},$ one has  that
\begin{align*}
    f(g,g'^tm_{\ell})
    &=\int_{\mathrm{Tp}_{r,r-2,1}} |\det h_\ell|^{r-2}\sum_{\underline{\varphi} \in \mathcal{B}_{\underline{\sigma}}}
        Z(f,W_{\underline{\varphi}},\chi)\overline{Z(W_{\underline{\varphi}},\chi)(g,g'^tm_{\ell})}\frac{d\mu(I(\underline{\sigma}) \otimes \chi)}{\zeta(1)}.
\end{align*}
\end{thm}
\begin{proof}
 Choose $\widetilde{f} \in \mathcal{S}(\GL_{r,r-2}(F))_{\mathfrak{F}}$ such that $I_{\widetilde{f}}=f.$  By the Whittaker Plancherel formula \eqref{Whitt:Planch}, we have that
\begin{align*}
    &W_{\widetilde{f}^\vee}\left(\begin{psmatrix} p_1 & \\ & \begin{psmatrix} 1 & \\ & c_1^{-1}\end{psmatrix}h_\ell \end{psmatrix},\begin{psmatrix}p_2 & \\ & \begin{psmatrix} 1 & \\ & c_1^{-1}\end{psmatrix}h_\ell \end{psmatrix},p_1,p_2\right)\\
    =&\int_{\mathrm{Tp}_{r,r-2}}W_{\widetilde{f}^\vee_{ I(\underline{\sigma})}}\left(\begin{psmatrix} p_1 & \\ & \begin{psmatrix} 1 & \\ & c_1^{-1}\end{psmatrix}h_\ell \end{psmatrix},\begin{psmatrix}p_2 & \\ & \begin{psmatrix} 1 & \\ & c_1^{-1}\end{psmatrix}h_\ell \end{psmatrix},p_1,p_2\right)
    d\mu(I(\underline{\sigma})).
\end{align*}  
Using  Proposition \ref{prop:Whitt:exp}, this is 
\begin{align*}
&\int_{\mathrm{Tp}_{r,r-2}}\sum_{\underline{\varphi} \in \mathcal{B}_{\underline{\sigma}}}I(\underline{\sigma})(\widetilde{f})W_{\underline{\varphi}}\left(\begin{psmatrix} p_2 & \\ & \begin{psmatrix} 1 & \\ & c_1^{-1}\end{psmatrix}h_\ell\end{psmatrix},p_2\right)
\overline{W}_{\underline{\varphi}}\left(\begin{psmatrix} p_1 & \\ & \begin{psmatrix} 1 & \\ & c_1^{-1}\end{psmatrix}h_\ell\end{psmatrix},p_1\right)d\mu(I(\underline{\sigma})).
\end{align*}
   Since we assumed $\widetilde{f}$ is right $K$-finite, the integral has support in a finite set of components. Indeed, in the non-Archimedean case this is a consequence of \cite[Theorem VIII.1.2]{Waldspurger:plancherel}, and in the Archimedean case it is a consequence of the classification of discrete series \cite[Theorem 9.20, Chapter XII]{KnappSS}.  Moreover
 the sum over $\varphi$ is finite for each $\underline{\sigma}$ because $I(\underline{\sigma})$ is admissible.
 
 Combining this with Lemma \ref{lem:unfold}, we see that $f(I_r,m_{\ell})=I_{\widetilde{f}}(I_r,m_{\ell})$ is
\begin{align} \label{before:d1} \begin{split}
    \int\bigg(&\int_{F^\times}  \int_{\mathrm{Tp}_{r,r-2}} \sum_{\underline{\varphi} \in \mathcal{B}_{\underline{\sigma}}}  
I(\underline{\sigma})(\widetilde{f})W_{\underline{\varphi}}\left(\begin{psmatrix} p_2 & \\ & \begin{psmatrix} 1 & \\ & c_1^{-1}\end{psmatrix}h_\ell\end{psmatrix},p_2\right)\\& \times 
\overline{W}_{\underline{\varphi}}\left(\begin{psmatrix} p_1 & \\ & \begin{psmatrix} 1 & \\ & c_1^{-1}\end{psmatrix}h_\ell\end{psmatrix},p_1\right)
d\mu(I(\underline{\sigma})) \frac{d^\times c_1}{|c_1|^{r-2}\zeta(1)}\bigg)\frac{|\det h_\ell|^{r-2}d_r p_1d_r p_2}{|\det p_1p_2|^{3/2}},\end{split}
\end{align}
where the outer integral is over ${(U_{r-2}(F) \backslash \mathcal{P}_{r-2}(F))^2}.$ Consider the function 
\begin{align*}
    V_{\underline{\varphi}}(c_2):=\int_{F^\times \times (U_{r-2}(F) \backslash \mathcal{P}_{r-2}(F))^2}&I(\underline{\sigma})(\widetilde{f})W_{\underline{\varphi}}\left(\begin{psmatrix} p_2 & \\ & \begin{psmatrix} 1 & \\ & c_1^{-1}c_2^{-1}\end{psmatrix}h_\ell\end{psmatrix},p_2\right)\\ \times & 
\overline{W}_{\underline{\varphi}}\left(\begin{psmatrix} p_1 & \\ & \begin{psmatrix} 1 & \\ & c_1^{-1}\end{psmatrix}h_\ell\end{psmatrix},p_1\right)
\frac{|\det h_\ell|^{r-2}d_r p_1d_r p_2}{|\det p_1p_2|^{3/2}}\frac{d^\times c_1}{|c_1|^{r-2}\zeta(1)}.
\end{align*}
The integral converges absolutely by Lemma \ref{lem:for:zeta}.
We have
\begin{align} \label{inV}
f(I_r,m_\ell)=\int_{\mathrm{Tp}_{r,r-2}} \sum_{\underline{\varphi}\in \mathcal{B}_{\underline{\sigma}}} V_{\underline{\varphi}}(1) d\mu(I(\underline{\sigma})).
\end{align}
Indeed, this follows from \eqref{before:d1} after rearranging the integral using the Fubini-Tonelli theorem, 
 Proposition \ref{prop:unif}, Lemma \ref{lem:for:zeta}, 
and the fact that the Plancherel measure has at most polynomial growth (see \cite[(2.13.2)]{BP:GLn} for a precise statement).

 Using Lemma \ref{lem:for:zeta}, one checks that $V_{\underline{\varphi}}$ is a smooth function on $F^\times$ such that  $|V_{\underline{\varphi}}(c_2)|\ll_{\epsilon,\underline{\varphi}} \min(|c_2|^{(r-3)/2+\epsilon},|c_2|^{(r-1)/2+\epsilon})$ for any $\epsilon>0$. Therefore, we can apply Mellin inversion to $V_{\underline{\varphi}}$ and write
\begin{align*}
    V_{\underline{\varphi}}(1)&=\int_{\mathrm{Tp}_1}\Big(\int_{F^\times} V_{\underline{\varphi}}(c_2)\frac{\chi(c_2)d^\times c_2}{|c_2|^{(r-2)/2}}\Big)d\mu(\chi).
\end{align*}
To check that this application of Mellin inversion is justified  \cite[(2.2)]{Blomer_Brumley_Ramanujan_Annals} is helpful.  
Changing variables $c_2 \mapsto c_2c_1^{-1}$ we see that 
\begin{align*}
V_{\underline{\varphi}}(1)&=
\int_{\mathrm{Tp}_1}\Bigg(\int_{(F^\times)^2 \times (U_{r-2}(F) \backslash \mathcal{P}_{r-2}(F))^2}I(\underline{\sigma})(\widetilde{f})W_{\underline{\varphi}}\left(\begin{psmatrix} p_2 & \\ & \begin{psmatrix} 1 & \\ & c_2^{-1}\end{psmatrix}h_\ell\end{psmatrix},p_2\right)\\& \times 
\overline{W}_{\underline{\varphi}}\left(\begin{psmatrix} p_1 & \\ & \begin{psmatrix} 1 & \\ & c_1^{-1}\end{psmatrix}h_\ell\end{psmatrix},p_1\right)
\frac{|\det h_\ell|^{r-2}d_r p_1d_r p_2}{|\det p_1p_2|^{3/2}}\frac{\overline{\chi}(c_1)\chi(c_2)d^\times c_1d^\times c_2}{|c_1c_2|^{(r-2)/2}\zeta(1)}\Bigg)d\mu(\chi).
\end{align*}
Here we use Lemma \ref{lem:for:zeta} to justify the change of variables.

Substituting this expression into \eqref{inV}, we have
 \begin{align*}
     f(I_r,m_\ell)&=\int_{\mathrm{Tp}_{r,r-2}}\left(\int_{\mathrm{Tp}_1} \sum_{\underline{\varphi}\in \mathcal{B}_{\underline{\sigma}}} Z(I(\underline{\sigma})(\widetilde{f})W_{\underline{\varphi}},\chi)\overline{Z(W_{\underline{\varphi}},\chi)}\frac{d\mu(I(\underline{\sigma}))}{|\det h_\ell|^{2-r}}\right)\frac{d\mu(\chi)}{\zeta(1)}\\
     &=\int_{\mathrm{Tp}_{r,r-2}}\left(\int_{\mathrm{Tp}_1} \sum_{\underline{\varphi}\in \mathcal{B}_{\underline{\sigma}}}
        Z(f,W_{\underline{\varphi}},\chi)\overline{Z(W_{\underline{\varphi}},\chi)}\frac{d\mu(I(\underline{\sigma}))}{|\det h_\ell|^{2-r}}\right)\frac{d\mu(\chi)}{\zeta(1)}.
 \end{align*}
Here the last identity follows from Lemma \ref{lem:Zh}. 
The integral over $\mathrm{Tp}_{r,r-2,1}$ is absolutely convergent by Proposition \ref{prop:unif}, Lemma \ref{lem:for:zeta}, and an integration by parts argument to control the integral over $\mathrm{Tp}_1$ in the Archimedean case.
  
This is the identity of the theorem in the special case $(g,g')=(I_r,I_{r-2})$ for arbitrary  choice of maximal compact subgroup $K_{r,r-2}$ in the Archimedean case.  To deduce it in general, we replace $K_{r,r-2}$ by $\left(\begin{psmatrix}
        g' & \\
          & I_2
    \end{psmatrix}g\right)^{-1}K_r\begin{psmatrix}
        g' & \\
          & I_2
    \end{psmatrix}g \times g'^{-1}K_{r-2}g',$ 
   $W$ by $\mathcal{R}\left(\begin{psmatrix}
        g' & \\
          & I_2
    \end{psmatrix}g,g'\right)W,$ and $f$ by 
\begin{align*}
   \mathcal{R}(\begin{psmatrix}
        g' & \\
          & I_2
    \end{psmatrix}g,g')f:=f\left(\begin{psmatrix}
        g'^{-1} & \\
         & I_2
    \end{psmatrix}(\cdot )\begin{psmatrix}
        g' & \\
          & I_2
    \end{psmatrix}g,g'^t(\cdot)^t m_{\ell}\right).
\end{align*}  By a change of variables we then have
\begin{align*}
    Z\big(\mathcal{R}\left(\begin{psmatrix}
        g' & \\
          & I_2
    \end{psmatrix}g,g' \right)f,\mathcal{R}\left(\begin{psmatrix}
        g' & \\
          & I_2
    \end{psmatrix}g,g'\right)W,\chi\big)&=|\det g'|^{-3/2}Z(f,W,\chi).
\end{align*}
\end{proof}

\section{A matrical Plancherel theorem for $(X^\circ_{\ell}(F),\mathcal{H})$}
\label{sec:Matric}

For $I(\underline{\sigma}) \otimes \chi=I(\sigma \otimes \sigma') \otimes \chi \in \mathrm{Tp}_{r,r-2,1},$ let
\begin{align} \label{bundle:Z}
    \mathcal{Z}(I(\underline{\sigma}),\chi)^{\mathrm{sm}}:=\bigg\{(g,g'^tm_\ell)\mapsto \overline{Z(W,\chi)(g,g'^tm_\ell)}: W \in \mathcal{W}(I(\sigma),\psi) \widehat{\otimes}\mathcal{W}(I(\sigma'),\overline{\psi}) \bigg\}.
\end{align}
The space $\mathcal{Z}(I(\underline{\sigma}),\chi)^{\mathrm{sm}}$ is naturally a smooth representation of $\GL_{r,r-2}(F)$ isomorphic to $\mathcal{W}(I(\sigma),\psi)^\vee \widehat{\otimes}\mathcal{W}(I(\sigma'),\overline{\psi})^\vee.$  In particular it is unitary.  We let $\mathcal{Z}(I(\underline{\sigma}),\chi)$ be its unitary completion.  There is then a canonical Hermitian vector bundle \index{$\mathcal{Z}$}
\begin{align}
\mathcal{Z} \lto \mathrm{Tp}_{r,r-2,1}
\end{align}
such that the fiber over $I(\underline{\sigma}) \otimes\chi$ is $\mathcal{Z}(I(\underline{\sigma}), \chi).$

\begin{rem}
It would be more canonical to define $\mathcal{Z}$ directly in terms of $(X_\ell(F),\mathcal{H})$ following \cite{Bernstein:Planch}.  However, we take the approach above because we ultimately want to relate our local zeta integrals to Whittaker functions.  It also allows us to avoid proving a local multiplicity one statement.
\end{rem}

As above, let $\mathfrak{F}$ denote a finite set of $K_{r,r-2}$-types stable under taking contragredients.  We define \index{$\mathcal{C}(\mathrm{Tp}_{r,r-2,1},\mathcal{Z})_{\mathfrak{F}}$} $\mathcal{C}(\mathrm{Tp}_{r,r-2,1},\mathcal{Z})_{\mathfrak{F}}$ to be the space of sections of $\mathcal{Z}$ satisfying the following condition:
\begin{enumerate}
\item[] The section is a finite sum of sections $E$ supported in $\mathcal{O}(\underline{\sigma}) \times \mathrm{Tp}_1$ satisfying
\begin{align*}
E(I(\underline{\sigma}) \otimes \chi)= T_0(I(\underline{\sigma}) \otimes \chi)Z(W_{\underline{\varphi}},\chi),
\end{align*}
where $\varphi \in \mathcal{B}_{\underline{\sigma}}$ has $K_{r,r-2}$-type in $\mathfrak{F}$ and $T_0 \in \mathcal{C}(\mathrm{Tp}_{r,r-2,1}).$ 
\end{enumerate}
We point out that $\mathcal{C}(\mathrm{Tp}_{r,r-2,1},\mathcal{Z})_{\mathfrak{F}}$ comes equipped with a locally convex topology inherited from the topology on $\mathcal{C}(\mathrm{Tp}_{r,r-2,1}).$

Recall the Harish-Chandra space $\mathcal{C}(X^\circ_{\ell}(F),\mathcal{H})$ defined in \S \ref{ssec:fs}.  
Let $M \leq \GL_{r,r-2}$ be a semistandard Levi subgroup, let $\underline{\sigma} \in \mathrm{Irr}_2(M),$ and let $\underline{\varphi} \in \mathcal{B}_{\underline{\sigma}}.$  The motivation for the definition of $\mathcal{C}(\mathrm{Tp}_{r,r-2,1},\mathcal{Z})_{\mathfrak{F}}$ comes from the following lemma:

\begin{lem} \label{lem:is:cont:for:HP2}
One has a continuous map
$
\Psi:\mathcal{C}(X^\circ_\ell(F),\mathcal{H}) \to \mathcal{C}(\mathrm{Tp}_{r,r-2,1})$
given by sending $f$ to the function $\Psi(f)$ supported on $\OO_{\underline{\sigma}} \times \mathrm{Tp}_1$ satisfying
$\Psi(f)(I(\underline{\sigma_{\lambda}}) \otimes \chi)=Z(f,W_{\underline{\varphi_\lambda}},\chi).$
\end{lem}

\begin{proof} 
 The function $\Psi(f)$ is well-defined by  Theorem \ref{thm:Whitt:analy} and Lemma \ref{lem:Z:cont}. 
Differentiating under the integral sign, we deduce smoothness in $\underline{\lambda};$ the required hypotheses can be checked using Theorem \ref{thm:Whitt:analy} and the proof of Lemma \ref{lem:Z:cont}.  
 As for smoothness in $\chi,$ we recall that 
 \begin{align*}
 Z(f,W,\chi)&=
     \int f\left(g,g'^tm_{\ell}\right)\left(\int_{F^\times} J(W)\left(\begin{psmatrix}
        I_{r-2} & \\
           & \begin{psmatrix} 1 & \\ & c^{-1}\end{psmatrix}h_\ell
\end{psmatrix}g,g'^tm_\ell\right)\frac{\chi(c)d^\times c}{|c|^{(r-2)/2}}\right)\frac{|\det g'|d\dot{g}'d\dot{g}}{|\det h_\ell|^{2-r}},
 \end{align*}
 where the outer integral is over $\mathcal{P}_{r-2}(F) \backslash \GL_{r-2}(F) \times N(F) \backslash \GL_r(F).$
Repeated differentiation in $\chi$ multiplies
 the inner integral by polynomials in $\log |c|,$ and the estimates used to prove the convergence of this integral in Lemma \ref{lem:Z:cont} are insensitive to this change. We conclude that $\Psi(f)$ is smooth as a function of $\chi$ as well.
 
 Thus $\Psi(f) \in \mathcal{C}(\mathrm{Tp}_{r,r-2,1})$ in the non-Archimedean case.  To prove this in the Archimedean case, it suffices to prove it for a function of the form
 $
 f=\mathcal{R}_u(f_2)f_1
 $
 for some $(f_1,f_2) \in C_c^\infty(\GL_{r,r-2}(F)) \times \mathcal{C}(X_\ell^{\circ}(F),\mathcal{H}).$  Indeed, every function in $\mathcal{C}(X_\ell^{\circ}(F),\mathcal{H})$ is a finite sum of functions of this form.  We have
 \begin{align*}
Z(\mathcal{R}_u(f_2)f_1,W_{\underline{\varphi}},\chi)=Z(f_1,I(\underline{\sigma})(f^\vee_2)W_{\underline{\varphi}},\chi).
\end{align*} 
We can now use the estimate in Corollary \ref{cor:unif} and the arguments above to conclude that $\Psi(f)$ is rapidly decreasing in $\underline{\lambda},$ and an integration by parts argument to prove that it is rapidly decreasing in $\chi.$ This proves that $\Psi(f) \in \mathcal{C}(\mathrm{Tp}_{r,r-2,1})$ in the Archimedean case as well.  We deduce $\Psi(f) \in \mathcal{C}(\mathrm{Tp}_{r,r-2,1})$ for all $f.$ 
Now the graph of $\Psi$ is
\begin{align*}
    \bigcap_{(\underline{\lambda},\chi) \in i\mathfrak{a}_M \times \mathrm{Tp}_1}\{(f,T):Z(f,W_{\underline{\varphi_\lambda}},\chi)=T(I(\underline{\sigma_{\lambda}}) \otimes \chi)\}
\end{align*}
and each of the sets in the intersection is closed since $f \mapsto Z(f,W_{\underline{\varphi_\lambda}},\chi)$ is continuous by \Cref{lem:Z:cont}.  Hence by the closed graph theorem \cite[Theorem 2]{Robertson:Robertson}  $\Psi$ is continuous.
\end{proof}

Assume we are given $T \in \mathcal{C}(\mathrm{Tp}_{r,r-2},\mathcal{HS})_{\mathfrak{F}}.$  Then  
$
T(I(\underline{\sigma}) \otimes \chi)=\sum_{i} \overline{\varphi}_{i} \otimes \varphi_i'
$
for some $\varphi_i, \varphi_i' \in I(\underline{\sigma}),$ and the sum is finite by our assumptions on $K_{r,r-2}$-types.   
We set \index{$\widehat{I}_{h_\ell}(T)$}
\begin{align} \label{Ih} \begin{split}
&\widehat{I}_{h_\ell}(T)(I(\underline{\sigma}) \otimes \chi)(g,g'^tm_\ell):=|\det h_\ell|^{r-2}\sum_{i} Z(W_{\varphi_i'},\chi)\overline{Z(W_{\varphi_i},\chi)(g,g'^tm_\ell)}. \end{split}
\end{align}
For $f \in \mathcal{C}(\GL_{r,r-2}(F))_{\mathfrak{F}},$ one has by Corollary \ref{cor:tempered:ZW=ZfW}
\begin{align} \label{Ihcirc}
\begin{split}
\widehat{I}_{h_\ell} \circ\mathrm{HP}(f)(I(\underline{\sigma}) \otimes \chi) &=|\det h_\ell|^{r-2} \sum_{\underline{\varphi} \in \mathcal{B}_{\underline{\sigma}}} Z(I(\underline{\sigma})(f)W_{\underline{\varphi}},\chi)\overline{Z(W_{\underline{\varphi}},\chi)(g,g'^tm_\ell)}\\
&=|\det h_\ell|^{r-2}\sum_{\underline{\varphi} \in \mathcal{B}_{\underline{\sigma}}} Z(I_f,W_{\underline{\varphi}},\chi)\overline{Z(W_{\underline{\varphi}},\chi)(g,g'^tm_\ell)}.
\end{split}
\end{align}

\begin{lem} \label{lem:dense:cont}
The operator $\widehat{I}_{h_\ell}$  defines a continuous map
\begin{align*}
\widehat{I}_{h_\ell}:\mathcal{C}(\mathrm{Tp}_{r,r-2},\mathcal{HS})_{\mathfrak{F}} \lto \mathcal{C}(\mathrm{Tp}_{r,r-2,1},\mathcal{Z})_{\mathfrak{F}}.
\end{align*}
\end{lem}

\begin{proof}
Let $M$ be a semistandard Levi subgroup of $\GL_{r,r-2},$ let $\underline{\sigma} \in \mathrm{Irr}_2(M),$ and let $\underline{\varphi} \in \mathcal{B}_{\underline{\sigma}}.$  Let $f \in \mathcal{C}(\GL_{r,r-2}(F)).$  By Lemma \ref{lem:is:cont:for:HP2}, the map
\begin{align*}
\Psi(I_f):\mathcal{O}(\underline{\sigma}) \times \mathrm{Tp}_1 &\lto \CC\\
I(\underline{\sigma_{\lambda}}) \otimes \chi &\longmapsto Z(I_f,W_{\underline{\varphi_{\lambda}}},\chi)
\end{align*}
 may be regarded as an element of $\mathcal{C}(\mathrm{Tp}_{r,r-2,1})$ by extending by zero.   
Since  
 $I_{(\cdot)}$ is continuous by \Cref{lem:extension} and $\Psi$ is continuous by \Cref{lem:is:cont:for:HP2}, the composite $\Psi \circ I_{(\cdot)}$
 is continuous.

In view of \eqref{Ihcirc}, we deduce that $\widehat{I}_{h_\ell} \circ \mathrm{HP}$ has image in $\mathcal{C}(\mathrm{Tp}_{r,r-2,1},\mathcal{Z})_{\mathfrak{F}}$ and is continuous.  Since $\mathrm{HP}$ is a homeomorphism by Theorem \ref{thm:HC}, we deduce the lemma. 
\end{proof}

\begin{lem}\label{lem:approx}
Let $T \in \mathcal{C}(\mathrm{Tp}_1).$  The closure of the linear span of $\left\{\chi\mapsto   \chi(a)T(\chi):  a\in F^\times\right\}$ in $\mathcal{C}(\mathrm{Tp}_1)$ contains $\{\phi T:\phi \in C_c^\infty(\mathrm{Tp}_1)\}.$
\end{lem}

\begin{proof}
Suppose $F$ is Archimedean. It suffices to show that given $f\in \mathcal{S}(\RR)$, the closure of the span of $
        \{ x\mapsto e^{isx}f(x): s\in \RR\}$
    in $\mathcal{S}(\RR)$ contains $\phi f$ for any $\phi\in C^\infty_c(\RR)$. To see this, by Fourier inversion, 
    \begin{align*}
        \phi(x)=\int_{\RR} \widehat{\phi}(s)e^{-isx}ds
    \end{align*}
    where $\widehat{\phi}$ is the Fourier transform of $\phi$. Then
    \begin{align*}
        \phi(x)f(x)=\int_{\RR} \widehat{\phi}(s)e^{-isx}dsf(x).
    \end{align*}
    Note that since $\phi$ is compactly supported, $\widehat{\phi}$ has exponential decay as $|s|\to \infty$. Thus the integral can be approximated by finite Riemann sums. The same discussion applies to  $x^m\frac{d^n(\phi f)}{dx^n}(x)$, so the convergence can be taken in $\mathcal{S}(\RR).$  The proof in the non-Archimedean case is similar.
\end{proof}

We use an idea from \cite{RB:quest} in the proof of the following lemma:
\begin{lem}\label{lem:dense}
The image of $\widehat{I}_{h_\ell}$ is dense in $\mathcal{C}(\mathrm{Tp}_{r,r-2,1},\mathcal{Z})_{\mathfrak{F}}.$
\end{lem}

\begin{proof}
To prove density, we first claim that the image of $\widehat{I}_{h_\ell}$ is stable under multiplication under a large subspace of functions in $\mathcal{C}(\mathrm{Tp}_{r,r-2,1})$. Since $\mathcal{C}(\mathrm{Tp}_{r,r-2},\mathcal{HS})$ is closed under multiplication by $\mathcal{C}(\mathrm{Tp}_{r,r-2}),$ the image of $\widehat{I}_{h_\ell}$ is stable under multiplication by functions of the form $I(\underline{\sigma}) \otimes \chi \mapsto \phi(I(\underline{\sigma}))$ for $\phi \in \mathcal{C}(\mathrm{Tp}_{r,r-2}).$  On the other hand, we have 
\begin{align*} 
\widehat{I}_{h_\ell}\left(\mathcal{R}\left(\left(\begin{psmatrix} I_{r-2} & \\& h_\ell^{-1}\begin{psmatrix} 1 & \\ & a \end{psmatrix}h_\ell\end{psmatrix},I_{r-2} \right),I_{r,r-2}\right)T\right)(I(\underline{\sigma}) \otimes \chi)=\frac{\chi(a)}{|a|^{(r-2)/2}}\widehat{I}_{h_\ell}(T)(I(\underline{\sigma}) \otimes \chi).
\end{align*}
Thus the image of $\widehat{I}_{h_\ell}$ is stable under multiplication by any linear combination of functions of the form
\begin{align} \label{of:the:form}
I(\underline{\sigma}) \otimes \chi \longmapsto f(I(\underline{\sigma}))\frac{\chi(a)}{|a|^{(r-2)/2}}
\end{align}
for $(f,a) \in \mathcal{C}(\mathrm{Tp}_{r,r-2}) \times F^\times.$

Write $\underline{\sigma}=(\sigma,\sigma').$ By Lemma \ref{lem:nonvanish}, for every $I(\underline{\sigma}) \otimes \chi \in \mathrm{Tp}_{r,r-2,1},$ there is a $K_{r,r-2}$-finite 
$$
W^\vee \otimes  W\in \left(\mathcal{W}(I(\sigma^\vee),\overline{\psi})\otimes\mathcal{W}(I(\sigma'^\vee),\psi)\right) \otimes \left(\mathcal{W}(I(\sigma),\psi)\otimes\mathcal{W}(I(\sigma'),\overline{\psi})\right)
$$
such that $Z(W^\vee,\bar{\chi})Z(W,\chi) \neq 0.$  Since $\mathcal{Z}(I(\underline{\sigma})\otimes \chi)$ is irreducible, this implies that for each $I(\underline{\sigma}) \otimes \chi$ the map on stalks 
$$
\widehat{I}_{h_\ell}:(\mathcal{HS}_{I(\underline{\sigma})})_{\mathfrak{F}} \lto \mathcal{Z}(I(\underline{\sigma}) \otimes \chi)_{\mathfrak{F}}
$$
is surjective.  Thus by a compactness argument, for any compact set $\mathcal{K} \subset \mathrm{Tp}_{r,r-2,1}$ there is a finite set of sections $T_1',\dots,T_n' \in \mathcal{C}(\mathrm{Tp}_{r,r-2},\mathcal{HS})_{\mathfrak{F}}$ such that $T_i:=\widehat{I}_{h_\ell}(T'_i)$ span $Z_{h_\ell}(I(\underline{\sigma}),\chi)_{\mathfrak{F}}$ for all $I(\underline{\sigma}) \otimes \chi \in \mathcal{K}.$

Now the image of $\widehat{I}_{h_\ell}$ contains all sections of the form $\phi_1 T_1,\dots,\phi_nT_n$ where each $\phi_i$ is of the form \eqref{of:the:form}.  The collection of these sets of sections, as $\mathcal{K}$ varies, spans a dense subset in $\mathcal{C}(\mathrm{Tp}_{r,r-2,1},\mathcal{Z})_{\mathfrak{F}}$ by Lemma \ref{lem:approx}.  
\end{proof}

\begin{lem} \label{lem:commute}
Let $\mathfrak{F}$ be a finite set of $K_{r,r-2}$-types.
There is a continuous map \index{$\mathrm{HP}_{h_\ell}$}
\begin{align*}
\mathrm{HP}_{h_\ell}:    \mathcal{C}(X^\circ_{\ell}(F),\mathcal{H})_{\mathfrak{F}} \lto \mathcal{C}(\mathrm{Tp}_{r,r-2,1},\mathcal{Z})_{\mathfrak{F}}
\end{align*}
given by
\begin{align*}
\mathrm{HP}_{h_\ell}(f)(I(\underline{\sigma}),\chi)(g,g'^tm_\ell)&=|\det h_\ell|^{r-2}\sum_{\underline{\varphi} \in \mathcal{B}_{\underline{\sigma}}}
        Z(f,W_{\underline{\varphi}},\chi)\overline{Z(W_{\underline{\varphi}},\chi)(g,g'^tm_{\ell})}.
\end{align*}
It fits into a commutative diagram
\begin{equation} \label{diag:for:proof}
\begin{tikzcd}
\mathcal{C}(\GL_{r,r-2}(F))_{\mathfrak{F}}\arrow[d,"I_{(\cdot)}"] \arrow[rrr,"\mathrm{HP}_{\GL_{r,r-2}}"] & &&\mathcal{C}(\mathrm{Tp}_{r,r-2},\mathcal{HS})_{\mathfrak{F}}\arrow[d,"\widehat{I}_{h_\ell}"]\\
\mathcal{C}(X_{\ell}^\circ(F),\mathcal{H})_{\mathfrak{F}}   \arrow[rrr,"\mathrm{HP}_{h_\ell}"] &&&\mathcal{C}(\mathrm{Tp}_{r,r-2,1},\mathcal{Z})_{\mathfrak{F}}.
\end{tikzcd}
\end{equation}
\end{lem}

\begin{proof}
The map $\mathrm{HP}_{h_\ell}$ is well-defined and continuous by Lemma \ref{lem:is:cont:for:HP2}. The diagram commutes by \eqref{Ihcirc}.
\end{proof}

The following is our matrical Harish-Chandra Plancherel theorem:

\begin{thm} \label{thm:PW:Z} 
The map $\mathrm{HP}_{h_\ell}$ of Lemma \ref{lem:commute} is a topological isomorphism.  
The inverse is given by 
\begin{align*}
 \mathrm{HP}^{-1}_{h_\ell}(T)(g,g'^tm_{\ell})&=\int_{\mathrm{Tp}_{r,r-2,1}} T(I(\underline{\sigma}) \otimes \chi)(g,g^{\prime t}m_{\ell})\frac{d\mu(I(\underline{\sigma})\otimes \chi)}{\zeta(1)}.
 \end{align*}
\end{thm}

\begin{proof}
We check that $\mathrm{HP}_{h_\ell}^{-1}$ is  well-defined.  
We may assume that there exists $I(\underline{\sigma}) \in \mathrm{Tp}_{r,r-2}$ and $\underline{\varphi} \in \mathcal{B}_{\underline{\sigma}}$ such that the support of $T$ is contained in $\OO_{\underline{\sigma}} \times \mathrm{Tp}_1$ and
\begin{align*}
T(I(\underline{\sigma_{\lambda}}) \otimes \chi)=\phi(I(\underline{\sigma_\lambda}),\chi)Z(W_{\underline{\varphi_\lambda}},\chi)
\end{align*}
for some $\phi \in \mathcal{C}(\mathrm{Tp}_{r,r-2,1}).$ Assume that $\phi=\phi_1 \otimes \phi_2$ where $(\phi_1,\phi_2) \in \mathcal{C}(\mathrm{Tp}_{r,r-2}) \times \mathcal{C}(\mathrm{Tp}_1).$  
Using Theorem \ref{thm:Whitt:analy}, choose $d>0$ such that $W_{\underline{\varphi_{\lambda}}}$
lies in $\mathcal{C}_d(U_{r,r-2}(F) \backslash \GL_{r,r-2}(F),\psi \otimes \overline{\psi})$
for all $\underline{\lambda}.$  
By Lemma \ref{lem:for:zeta} there is a continuous seminorm $\nu_{d}$ on this space such that 
\begin{align} \label{claim}\begin{split}
   & \int_{\mathcal{O}(\underline{\sigma}) \times \mathrm{Tp}_{1}}|\phi_1(I(\underline{\sigma_{\lambda}}))\phi_2(\chi)|\int\left|W_{\underline{\varphi_{\lambda}}}\left(\begin{psmatrix}p & \\ &  \begin{psmatrix} 1 & \\ & c^{-1} \end{psmatrix}h_\ell\end{psmatrix},p  \right)\right|\frac{d_r pd^\times cd\mu(I(\underline{\sigma})\otimes \chi)}{|\det p|^{3/2}|c|^{(r-2)/2}}\\
    & \leq \int_{\mathcal{O}(\underline{\sigma}) \times \mathrm{Tp}_{1}}|\phi_1(I(\underline{\sigma_{\lambda}}))\phi_2(\chi)|\nu_d(W_{\underline{\varphi_{\lambda}}})d\mu(I(\underline{\sigma_\lambda}) \otimes \chi), \end{split}
\end{align}
where the inner integral is over 
$U_{r-2}(F) \backslash \mathcal{P}_{r-2}(F) \times F^\times.$
Using Theorem \ref{thm:Whitt:analy}, we see that \eqref{claim} is finite.

For $(g_1,g_2)\in\GL_{r,r-2}(F),$ let
\begin{align*}
W_{\phi_1}(g,g'):=\int_{\OO_{\underline{\sigma}}}\phi_1(I(\underline{\sigma_{\lambda}}))W_{\underline{\varphi_\lambda}}\left(g,g'  \right)\frac{d\mu(I(\underline{\sigma_{\lambda}}))}{\zeta(1)}.
\end{align*}
We have a continuous map
\begin{align} \label{Whitt:cont} \begin{split}
\mathcal{C}(\mathrm{Tp}_{r,r-2}) &\lto \mathcal{C}(U_{r,r-2}(F) \backslash \GL_{r,r-2}(F),\psi \otimes \overline{\psi})\\
\phi_1 &\longmapsto W_{\phi_1}. \end{split}
\end{align}
Indeed, this follows from \cite[Theorem 14.3]{vandenBan:MS} in the Archimedean case and  \cite[Proposition 6.1]{Delorme:Whitt} in the non-Archimedean case. We point out that the normalizing factor in each of these references is built into the definition of $d\mu.$ In the Archimedean case, the normalizing factor in \cite{vandenBan:MS} is different from ours; however, one can prove an analogue of \cite[Theorem 14.3]{vandenBan:MS} for our normalization using \cite[Lemma 13.2]{vandenBan:MS} in place of \cite[Lemma 13.5 and Corollary 13.7]{vandenBan:MS}.


Since \eqref{claim} is finite, we can switch the order of integration to see that 
\begin{align} \label{after:unwind}
\mathrm{HP}_{h_\ell}^{-1}(T)(g,g'^tm_\ell)=\int_{\mathrm{Tp}_1}\phi_2(\chi)\int_{F^\times}
J(W_{\phi_1})\left(\begin{psmatrix}I_{r-2} & \\ & \begin{psmatrix} 1 & \\ & c^{-1} \end{psmatrix}h_\ell \end{psmatrix}g,g' \right)
\frac{\chi(c) d^\times c}{|c|^{(r-2)/2}}d\mu( \chi).
\end{align}
 Lemma \ref{lem:map2X} implies that \eqref{after:unwind} is an element of $\mathcal{C}(X_\ell^\circ(F),\mathcal{H}).$  Let $\mathcal{C}(\mathcal{O}_{\underline{\sigma}}) \subset \mathcal{C}(\mathrm{Tp}_{r,r-2})$ be the subset of functions supported on $\mathcal{O}_{\underline{\sigma}}.$
Now $\mathrm{HP}_{h_\ell}^{-1}(T)$ is the image of $(\phi_1, \phi_2)$ under the composite of the following morphisms:
\begin{align} \label{break:2} \begin{split}
\mathcal{C}(\mathcal{O}_{\underline{\sigma}}) \times \mathcal{C}(\mathrm{Tp}_1) &\lto \mathcal{C}(\mathrm{Tp}_1) \times \mathcal{C}(U_{r,r-2}(F) \backslash \GL_{r,r-2}(F),\psi \otimes \overline{\psi})_{\mathfrak{F}}  \lto  \mathcal{C}(X_\ell^\circ(F),\mathcal{H})_{\mathfrak{F}}\\
(\phi_1, \phi_2) &\longmapsto \left(\phi_2,W_{\phi_1}\right), \end{split}
\end{align}
where the right arrow is the (continuous) map of Lemma \ref{lem:map2X}. The left arrow is continuous because \eqref{Whitt:cont} is continuous.  Hence
  \eqref{break:2} is continuous.    We deduce that $\mathrm{HP}_{h_\ell}^{-1}(T)$ is well-defined and continuous.

Now the inversion formula
$\mathrm{HP}_{h_\ell}^{-1} \circ \mathrm{HP}_{h_\ell}(f)=f$ holds if $f \in \mathcal{S}(X_\ell^\circ(F),\mathcal{H})_{\mathfrak{F}}$ by Theorem \ref{thm:Plancherel}.  Since $\mathcal{S}(X_\ell^\circ(F),\mathcal{H})_{\mathfrak{F}}$ is dense in $\mathcal{C}(X_\ell^\circ(F),\mathcal{H})_{\mathfrak{F}}$ and both $\mathrm{HP}_{h_\ell}$ and $\mathrm{HP}_{\ell}^{-1}$ are continuous, it holds for all $f \in \mathcal{C}(X_\ell^\circ(F),\mathcal{H})_{\mathfrak{F}}.$  This implies in particular that $\mathrm{HP}_{h_\ell}^{-1}$ is a left inverse of $\mathrm{HP}_{h_\ell}$ and hence $\mathrm{HP}_{h_\ell}$ is injective.

To complete the proof, it suffices to show that $\mathrm{HP}_{h_\ell}$ is surjective, or equivalently that $\mathrm{HP}_{h_\ell}^{-1}$ is injective.  
By Theorem \ref{thm:HC}, $\mathrm{HP}_{\GL_{r,r-2}}$ is an isomorphism.  Thus Lemma \ref{lem:commute} implies that
$\mathrm{HP}_{h_\ell}(\mathcal{C}(X_{\ell}^\circ(F),\mathcal{H})_{\mathfrak{F}})$ contains $\widehat{I}_{h_\ell}(\mathcal{C}(\mathrm{Tp}_{r,r-2},\mathcal{HS})_{\mathfrak{F}}),$ which is dense in $\mathcal{C}(\mathrm{Tp}_{r,r-2,1},\mathcal{Z})_{\mathfrak{F}}$ by Lemma \ref{lem:dense}.
Now assume that $T \in \ker (\mathrm{HP}_{h_\ell}^{-1}).$
 Choose $f_i \in \mathcal{C}(X_{\ell}^\circ(F),\mathcal{H})_{\mathfrak{F}}$ such that $\mathrm{HP}_{h_\ell}(f_i) \to T.$  Then since $\mathrm{HP}_{h_\ell}^{-1}$ is continuous
 $$
 0=\mathrm{HP}_{h_\ell}^{-1}(T)= \mathrm{HP}_{h_\ell}^{-1} (\lim_{i\to\infty} \mathrm{HP}_{h_\ell}(f_i))=\lim_{i\to\infty} \mathrm{HP}_{h_\ell}^{-1} ( \mathrm{HP}_{h_\ell}(f_i))=\lim_{i\to\infty} f_i.
 $$
 Since $\mathrm{HP}_{h_\ell}$ is continuous, we deduce $T=\lim_{i\to\infty} \mathrm{HP}_{h_\ell}(f_i)= \mathrm{HP}_{h_\ell}(\lim_{i\to\infty} f_i)=0.$
\end{proof}

\section{The Fourier transform on the Harish-Chandra space}\label{sec:fourier}

\begin{lem} \label{lem:gamma:factors}
Let $\pi$ and $\pi'$ be generic irreducible tempered representations of $\GL_{r_1}(F)$ and $\GL_{r_2}(F),$ respectively.  Then 
\begin{align*}
\gamma(\tfrac{1}{2},\pi \times  \pi',\psi)\gamma(\tfrac{1}{2},\pi^\vee \times \pi'^\vee,\psi)&=\omega_{\pi}(-I_{r_1})\omega_{\pi'}(-I_{r_2}),\\
\gamma\left(\tfrac{1}{2},\pi^\vee \times \pi'^\vee,\psi \right)&=\overline{\gamma\left(\tfrac{1}{2},\pi \times \pi',\psi \right)}\omega_{\pi}(-I_{r_1})\omega_{\pi'}(-I_{r_2}).
\end{align*}
\end{lem}

\begin{proof}
By the local Langlands correspondence for general linear groups \cite{HT} it suffices to check the corresponding statement for representations of the Weil-Deligne group.  This follows from \cite[(3.6.8) and (3.6.9)]{Tate_NT}.
\end{proof}

Let $I(\underline{\sigma})=I(\sigma) \otimes I(\sigma') \in \mathrm{Tp}_{r,r-2}.$
To ease notation, in the following we will abbreviate
\begin{align} 
\gamma(I(\underline{\sigma})^\vee):=\gamma(\tfrac{1}{2},I(\sigma)^\vee \times I(\sigma')^\vee,\psi),\quad \gamma(I(\sigma) \times \chi):=\gamma\left( \tfrac{1}{2},I(\sigma) \times \chi,\psi\right).
\end{align}
We assume throughout this section that the finite set $\mathfrak{F}$ of $K_{r,r-2}$-types is stable under taking contragredients.

\begin{thm} \label{thm:FT}
There is a topological isomorphism
\begin{align*}
    \mathcal{F}: \mathcal{C}(X^\circ_{\ell}(F),\mathcal{H})_{\mathfrak{F}}\tilde{\lto}  \mathcal{C}(X^\circ_{\ell}(F),\mathcal{H})_{\mathfrak{F}}
\end{align*}
given by \index{$\mathcal{F}$}
\begin{align*}
&\mathcal{F}(f)(g,g'^tm_{\ell})\\
    &=\int_{\mathrm{Tp}_{r,r-2,1}}\gamma( I(\underline{\sigma})^\vee )\gamma(
    I(\sigma) \times \chi^{-1})
    \sum_{\varphi \in \mathcal{B}_{\underline{\sigma}}}
Z(f,\widetilde{W}_{\underline{\varphi}},\chi^{-1})\overline{Z(W_{\underline{\varphi}},\chi)}(g,g'^tm_{\ell})\frac{d\mu(I(\underline{\sigma}) \otimes \chi)}{\zeta(1)|\det h_\ell|^{2-r}}.
\end{align*}
One has that
\begin{align*}
    \mathcal{F} \circ \mathcal{F}(f)(g,g'^tm_{\ell})=
f\left(\begin{psmatrix} -I_{r-2}& \\ & h_\ell^{-1} \begin{psmatrix} (-1)^{r-1} &\\ & (-1)^{r}\end{psmatrix} h_\ell  \end{psmatrix}g,(-1)^{r}g'^tm_{\ell}\right).
\end{align*}
\end{thm}

\begin{proof} 
We first prove that $\mathcal{F}$ is well-defined.  By Theorem \ref{thm:PW:Z} it suffices to prove that the section given by sending $I(\underline{\sigma}) \otimes \chi$ to
\begin{align*} 
\gamma(I(\underline{\sigma})^\vee)\gamma(I(\sigma)\times \chi^{-1})
    |\det h_\ell|^{r-2}\sum_{\varphi \in \mathcal{B}_{\underline{\sigma}}} Z(f,\widetilde{W}_{\underline{\varphi}},\chi^{-1})\overline{Z\left(W_{\underline{\varphi}},\chi\right)}(\cdot,\cdot)
\end{align*}
    lies in $\mathcal{C}(\mathrm{Tp}_{r,r-2,1},\mathcal{Z})_\mathfrak{F}$. 

We claim that the section sending $I(\underline{\sigma}) \otimes \chi$ to 
\begin{align} \label{section:good}
|\det h_\ell|^{r-2}\sum_{\varphi \in \mathcal{B}_{\underline{\sigma}}} Z(f,\widetilde{W}_{\underline{\varphi}},\chi^{-1})\overline{Z\left(W_{\underline{\varphi}},\chi\right)}(\cdot,\cdot)
\end{align}
lies in $\mathcal{C}(\mathrm{Tp}_{r,r-2,1},\mathcal{Z})_{\mathfrak{F}}.$
This follows from Lemma \ref{lem:tilde:iso}
and a minor modification of the argument proving Lemma \ref{lem:is:cont:for:HP2}.  Thus to prove $\mathcal{F}$ is well-defined, it suffices to check that 
$$
T(I(\underline{\sigma}) \otimes \chi):= \gamma(I(\underline{\sigma})^{\vee})\gamma(I(\sigma) \times \chi)
$$
is in $C^\infty(\mathrm{Tp}_{r,r-2,1})$ and is of polynomial growth in the following sense: for every semistandard parabolic subgroup $P=MN_P$ and every differential operator $D$ on $i\mathfrak{a}_M^\ast$ with constant coefficients, there is an $N>0$ such that 
\begin{align} \label{poly:growth} 
\mathrm{sup}_{\sigma \in \mathrm{Irr}_2(M)}\norm{(DT)_{\sigma}}(1+\norm{\sigma})^{-N}<\infty.
\end{align}
Here we are using notation as in \eqref{D:def}.
  In
the non-Archimedean case this is clear.  In the Archimedean case it follows as in the proof of \cite[Lemma 4.2]{DRS:Schwartz}.
    
    Finally, to prove that $\mathcal{F}$ is an automorphism, observe that by the Plancherel formula (Theorem \ref{thm:PW:Z}) we have by the definition of $\mathcal{F}$
    \begin{align}\label{eq:functionaltempered}
    Z(\mathcal{F}(f),W_{\underline{\varphi}},\chi)=\gamma( I(\underline{\sigma})^\vee )\gamma(
    I(\sigma) \times \chi^{-1})Z(f,\widetilde{W}_{\underline{\varphi}},\chi^{-1}).
\end{align}
By Lemma \ref{lem:tilde:iso}, $\widetilde{\widetilde{W}}_{\underline{\varphi}}=\omega_{I(\underline{\sigma})}((-1)^{r-1}I_{\underline{r},\underline{r-2}})W_{\underline{\varphi}}$.
Therefore, we have
    \begin{align*}
        \mathcal{F} \circ \mathcal{F}(f)(g,g'^tm_{\ell})=&\int_{\mathrm{Tp}_{r,r-2,1}}\gamma( I(\underline{\sigma})^\vee)\gamma(I(\sigma) \times \chi^{-1})\\& \times 
      \sum_{\varphi}
    Z(\mathcal{F}(f),\widetilde{W}_{\underline{\varphi}},\chi^{-1})\overline{Z(W_{\underline{\varphi}},\chi)}(g,g'^tm_{\ell})\frac{d\mu(I(\underline{\sigma}) \otimes \chi)}{\zeta(1)|\det h_\ell|^{2-r}}\\
    =&\int_{\mathrm{Tp}_{r,r-2,1}}\gamma( I(\underline{\sigma})^\vee)\gamma(
    I(\sigma) \times \chi^{-1})\gamma(I(\underline{\sigma}))\gamma(
    I(\sigma) \times \chi)\\& \times 
    \omega_{I(\underline{\sigma})}((-1)^{r-1}I_{\underline{r},\underline{r-2}})\sum_{\varphi}
    Z(f,W_{\underline{\varphi}},\chi)\overline{Z(W_{\underline{\varphi}},\chi)}(g,g'^tm_{\ell})\frac{d\mu(I(\underline{\sigma}) \otimes \chi)}{\zeta(1)|\det h_\ell|^{2-r}}. 
    \end{align*}
    Here we have used \eqref{eq:functionaltempered} by replacing $(\underline{\sigma},\chi,W_{\underline{\varphi}})$ with $(\underline{\sigma}^\lor,\chi^{-1},\widetilde{W}_{\underline{\varphi}})$ and applying Lemma \ref{lem:duals}.
    We apply Lemma \ref{lem:gamma:factors} to conclude that this is
    \begin{align*}
       &\int_{\mathrm{Tp}_{r,r-2,1}}
      \sum_{\varphi \in \mathcal{B}_{\underline{\sigma}}}
    Z(f,W_{\underline{\varphi}},\chi)\overline{Z(W_{\underline{\varphi}},\chi)}(g,g'^tm_{\ell})\frac{\omega_{I(\underline{\sigma})}((-1)^{r-1}I_{\underline{r},\underline{r-2}})d\mu(I(\underline{\sigma}) \otimes \chi)}{\omega_{I(\sigma')}(-I_{r-2})\chi(-1)\zeta(1)|\det h_\ell|^{2-r}}\\
    &=\int_{\mathrm{Tp}_{r,r-2,1}} \sum_{\varphi \in \mathcal{B}_{\underline{\sigma}}}
    Z(f,W_{\underline{\varphi}},\chi)\overline{Z\left(W_{\underline{\varphi}},\chi\right)}\left(\begin{psmatrix} -I_{r-2}& \\ & h_\ell^{-1} \begin{psmatrix} (-1)^{r-1} &\\ & (-1)^{r}\end{psmatrix} h_\ell  \end{psmatrix}g,(-1)^{r}g'^tm_{\ell}\right)
\frac{d\mu(I(\underline{\sigma}) \otimes \chi)}{\zeta(1)|\det h_\ell|^{2-r}}.
\end{align*}
  By Theorem \ref{thm:PW:Z}, this is $
f\left(\begin{psmatrix} -I_{r-2}& \\ & h_\ell^{-1} \begin{psmatrix} (-1)^{r-1} &\\ & (-1)^{r}\end{psmatrix} h_\ell  \end{psmatrix}g,(-1)^{r}g'^tm_{\ell}\right).$
\end{proof}

For later use we restate \eqref{eq:functionaltempered} below.

\begin{cor} \label{cor:temp:func:eqn}
Let $\pi \otimes \pi'\otimes \chi\in \mathrm{Tp}_{r,r-2,1}$ and $W\in \mathcal{W}(\pi,\psi) \otimes \mathcal{W}(\pi',\overline{\psi}).$ For $f \in \mathcal{C}(X_\ell^\circ(F),\mathcal{H})_{\mathfrak{F}},$ one has that
\begin{align*}
    Z(\mathcal{F}(f),W,\chi)=\gamma(\tfrac{1}{2},\pi^\vee \times \pi'^\vee,\psi)\gamma(\tfrac{1}{2},\pi \otimes \chi^{-1},\psi)Z(f,\widetilde{W},\chi^{-1}).
\end{align*}
\qed
\end{cor}

 Recall the action $\mathcal{R}_u$ from \eqref{Ru}. The following lemma asserts, in particular, that the Fourier transform sends the scaling action of $F^\times$ on $\mathcal{M}_{\ell}(F)$ to its inverse:

\begin{lem} \label{lem:FT:scalar}
Let $a,b \in F^\times.$
For any $f \in \mathcal{C}(X^{\circ}_{\ell}(F),\mathcal{H})_{\mathfrak{F}}$ one has  
$$
\mathcal{F}_{}(\mathcal{R}_u(I_{r,r-2},h_\ell^{-1} \begin{psmatrix} a & \\ & b \end{psmatrix} h_\ell)f)=
|a|^{r-2}\mathcal{R}_u(I_{r,r-2},h_\ell^{-1} \begin{psmatrix} a^{-1} & \\ & b^{-1} \end{psmatrix} h_\ell )  \mathcal{F}_{}(f) \in \mathcal{C}(X^{\circ}_{\ell}(F),\mathcal{H})_{\mathfrak{F}}.
$$
\end{lem}

\begin{proof}
Let $\pi \otimes \pi'\otimes \chi\in \mathrm{Tp}_{r,r-2,1}$ and let $\omega_{\pi}\otimes \omega_{\pi}'$ be the central character of $\pi\otimes \pi'$. Let $W \in \mathcal{W}(\pi,\psi) \otimes \mathcal{W}(\pi',\psi).$ Changing variables $c \mapsto cba^{-1}$ in the integral \eqref{ZWchi} defining $Z(W,\chi),$ and using the definition \eqref{ZWchi:func}, we see that 
\begin{align*}
    Z_{}(W,\chi)\left(\begin{psmatrix} I_{r-2} & \\ & h_\ell^{-1} \begin{psmatrix} a & \\ & b \end{psmatrix} h_\ell \end{psmatrix}g,ag'^tm_{\ell} \right)=|ba^2|^{-(r-2)/2}\omega_\pi(a)\omega_{\pi'}(a)\bar{\chi}(ab^{-1}) Z_{}(W,\chi)\left(g,g'^tm_{\ell} \right).
\end{align*}
Thus, using \eqref{Z:ids} and the definition of $\mathcal{R}_u$ from \eqref{Ru}  we have
\begin{align*}
    &|\det h_\ell|^{-(r-2)}Z(\mathcal{R}_u(I_{r,r-2},h_\ell^{-1}\begin{psmatrix} a & \\ & b \end{psmatrix} h_\ell)f,\widetilde{W},\chi^{-1})\\
    &=|ab|^{-(r-2)/2}\int f\left(\begin{psmatrix}
        I_{r-2} & \\
        & h_\ell^{-1}\begin{psmatrix} a^{-1} & \\ & b^{-1} \end{psmatrix} h_\ell
    \end{psmatrix}g,a^{-1}g'^tm_\ell \right)Z(\widetilde{W},\chi^{-1})(g,g'^tm_\ell) |\det g'|d\dot{g}'d\dot{g}\\
    &=|b|^{(r-2)/2}|a|^{3(r-2)/2}\int f\left(g,g'^tm_\ell \right)Z(\widetilde{W},\chi^{-1})\left(\begin{psmatrix}
        I_{r-2} & \\
        & h_\ell^{-1}\begin{psmatrix} a & \\ & b \end{psmatrix} h_\ell
    \end{psmatrix}g,ag'^tm_\ell\right) |\det g'|d\dot{g}'d\dot{g}\\
    &=|a|^{(r-2)/2}\bar{\omega}_{\pi}(a)\bar{\omega}_{\pi}'(a)\chi(ab^{-1})\int f\left(g,g'^tm_\ell \right)Z(\widetilde{W},\chi^{-1})\left(g,g'^tm_\ell\right) |\det g'|d\dot{g}'d\dot{g}\\
    &=|\det h_\ell|^{-(r-2)}|a|^{(r-2)/2}\bar{\omega}_{\pi}(a)\bar{\omega}_{\pi}'(a)\chi(ab^{-1})Z(f,\widetilde{W},\chi^{-1}).
\end{align*}
Here the integrals are over $\mathcal{P}_{r-2}(F) \backslash \GL_{r-2}(F) \times N(F) \backslash \GL_r(F).$

Using these identities and the definition of $\mathcal{F}$ from Theorem \ref{thm:FT}, we have 
\begin{align*} 
&\mathcal{F}(\mathcal{R}_u(I_{r,r-2},h_\ell^{-1}\begin{psmatrix} a & \\ & b \end{psmatrix} h_\ell)f)(g_1,g'^t_1m_{\ell})\\
    &=|a|^{(r-2)/2}\int_{\mathrm{Tp}_{r,r-2,1}}\gamma(I(\underline{\sigma})^\vee)\gamma(
  I(\sigma)\times\chi^{-1})\\& \quad\times 
    \sum_{\underline{\varphi} \in \mathcal{B}_{\underline{\sigma}}}
    Z(f,\widetilde{W}_{\underline{\varphi}},\chi^{-1})\overline{\omega}_{I(\sigma)}(a)\overline{\omega}_{I(\sigma')}(a)\chi(ab^{-1})\overline{Z(W_{\underline{\varphi}},\chi)}(g,g'^tm_{\ell})\frac{d\mu(I(\underline{\sigma}) \otimes \chi)}{\zeta(1)|\det h_\ell|^{2-r}}\\
    &=|a|^{3(r-2)/2}|b|^{(r-2)/2}\int_{\mathrm{Tp}_{r,r-2,1}}\gamma(I(\underline{\sigma})^\vee)\gamma(
  I(\sigma)\times\chi^{-1})\\& \quad\times 
    \sum_{\underline{\varphi} \in \mathcal{B}_{\underline{\sigma}}}
    Z(f,\widetilde{W}_{\underline{\varphi}},\chi^{-1})\overline{Z(W_{\underline{\varphi}},\chi)}\left(\begin{psmatrix}
        I_{r-2} & \\
        & h_\ell^{-1}\begin{psmatrix} a & \\ & b \end{psmatrix} h_\ell
    \end{psmatrix}g,ag'^tm_{\ell}\right)\frac{d\mu(I(\underline{\sigma}) \otimes \chi)}{\zeta(1)|\det h_\ell|^{2-r}}\\
    &=|a|^{r-2}\mathcal{R}_u(I_{r,r-2},h_{\ell}^{-1}\begin{psmatrix} a^{-1} & \\ & b^{-1} \end{psmatrix}h_\ell)\mathcal{F}(f)(g_1,g_1'^tm_{\ell}).
\end{align*}
\end{proof}

The zeta integrals $Z(W,\chi)$, $Z(f,W,\chi)$ depend on the choice of $h_\ell,$ and hence so does $\mathcal{F}.$  We write $Z_{h_{\ell}}(W,\chi)$, $Z_{h_{\ell}}(f,W,\chi),$ $\mathcal{F}_{h_\ell}$ for the integrals and morphism when we need to indicate the dependence. 

In the special case $(g,g')=(I_r,I_{r-2}),$ the following lemma gives some information about the behavior of the Fourier transform as $h_\ell$ varies.

\begin{lem} \label{lem:FT}
Let $
(g,g',h) \in \GL_{r,r-2}(F) \times \mathrm{GL}_2(F).
$
For any $f \in \mathcal{C}(X^{\circ}_{\ell h^{\iota}}(F),\mathcal{H})_{\mathfrak{F}},$ one has 
$$
\mathcal{F}_{h_\ell }(\mathcal{R}_u(g,g',h)f)=|\det h|^{2-r}
\mathcal{R}_u(g^{-t},g'^{-t},h )  \mathcal{F}_{h_\ell h }(f) \in \mathcal{C}(X^{\circ}_{\ell}(F),\mathcal{H})_{\mathfrak{F}'}
$$
for an appropriate finite set $\mathfrak{F}'$  of $K_{r,r-2}$-types depending on $(g,g')$ and $\mathfrak{F}.$
\end{lem}

\begin{proof} Let $\pi \otimes \pi'\otimes \chi\in \mathrm{Tp}_{r,r-2,1}$ and $W \in \mathcal{W}(\pi,\psi) \otimes \mathcal{W}(\pi',\psi).$ We have
\begin{align*}
    &Z_{h_\ell}(W,\chi)\left(\begin{psmatrix}
        g' & \\
        & h\end{psmatrix}g_1g^{-1},g_1'^tm_\ell\right)\\
        &=\int_{F^\times\times U_{r-2}(F)\backslash \mathcal{P}_{r-2}(F)} W\left(\begin{psmatrix}
        pg_1'g' & \\
           & \begin{psmatrix} 1 & \\ & c^{-1}\end{psmatrix}h_\ell h
\end{psmatrix}g_1g^{-1},pg_1'\right)\frac{\chi(c)d_rpd^\times c}{|c|^{(r-2)/2}|\det pg_1'|^{3/2}}\\
        &=|\det g'|^{3/2}\int_{F^\times\times U_{r-2}(F)\backslash \mathcal{P}_{r-2}(F)} \mathcal{R}(g^{-1},g'^{-1})W\left(\begin{psmatrix}
        pg_1'g' & \\
           & \begin{psmatrix} 1 & \\ & c^{-1}\end{psmatrix}h_\ell h
\end{psmatrix}g_1,pg_1'g'\right)\frac{\chi(c)d_rpd^\times c}{|c|^{(r-2)/2}|\det pg_1'g'|^{3/2}}\\
        &=|\det g'|^{3/2}Z_{h_\ell h}(\mathcal{R}(g^{-1},g'^{-1}) W,\chi)\left(g_1,(g_1'g')^tm_\ell h^\iota\right).
\end{align*}

Therefore, using \eqref{Z:ids}
\begin{align*}
    &Z_{h_\ell}(\mathcal{R}_u(g,g',h)f,\widetilde{W},\chi^{-1})\\
    &=\frac{|\det g'|^{3/2}|\det h_\ell|^{r-2}}{|\det h|^{(r-2)/2}}\int f\left(\begin{psmatrix}
        g'^{-1} & \\
        & h^{-1}\end{psmatrix}g_1g,g'^tg_1'^tm_\ell h^{\iota}\right)Z_{h_\ell}(\widetilde{W},\chi^{-1})(g_1,g_1'^tm_\ell) |\det g_1'|d\dot{g_1}'d\dot{g_1}\\
    &= \frac{|\det g'|^{-1/2}|\det h_\ell|^{r-2}}{|\det h|^{-(r-2)/2}}\int f\left(g_1,g'^tg_1'^tm_\ell h^{\iota}\right)Z_{h_\ell}(\widetilde{W},\chi^{-1})\left(\begin{psmatrix}
        g' & \\
        & h\end{psmatrix}g_1g^{-1},g_1'^tm_\ell\right) |\det g_1'|d\dot{g_1}'d\dot{g_1}\\
        &=\frac{|\det g'||\det h_\ell|^{r-2}}{|\det h|^{-(r-2)/2}}\int f\left(g_1,g'^tg_1'^tm_\ell h^{\iota}\right)Z_{h_\ell h}((\mathcal{R}(g^{t},g'^{t}) W)^{\sim},\chi^{-1})\left(g_1,(g_1'g')^tm_\ell h^\iota\right) |\det g_1'|d\dot{g_1}'d\dot{g_1}\\
        &=|\det h|^{-(r-2)/2}Z_{h_\ell h}(f,(\mathcal{R}(g^{t},g'^{t}) W)^{\sim},\chi^{-1}).
\end{align*}
Here the integrals are over $\mathcal{P}_{r-2}(F) \backslash \GL_{r-2}(F) \times N(F) \backslash \GL_r(F).$
Using this identity and Theorem \ref{thm:FT}, we have 
\begin{align*}
&\mathcal{F}_{h_\ell}(\mathcal{R}_u(g,g',h)f)(g_1,g'^t_1m_{\ell})\\
    &=|\det h|^{-(r-2)/2}\int_{\mathrm{Tp}_{r,r-2,1}}\gamma(I(\underline{\sigma})^\vee)\gamma(
  I(\sigma)\times\chi^{-1})\\& \times 
    \sum_{\underline{\varphi} \in \mathcal{B}_{\underline{\sigma}}}
    Z_{h_\ell h}(f,(\mathcal{R}(g^t,g'^t)W_{\underline{\varphi}})^{\sim},\chi^{-1})\overline{Z_{h_\ell}(W_{\underline{\varphi}},\chi)}(g_1,g_1'^tm_{\ell})\frac{d\mu(I(\underline{\sigma}) \otimes \chi)}{\zeta(1)|\det h_\ell|^{2-r}}.
\end{align*}
This expression is independent of the choice of orthonormal basis.  Hence it is equal to
\begin{align} \label{Fh}\begin{split}
|&\det h|^{-(r-2)/2}\int_{\mathrm{Tp}_{r,r-2,1}}\gamma(I(\underline{\sigma})^\vee)\gamma(
   I(\sigma)\times\chi^{-1})\\ \times & 
    \sum
    Z_{h_\ell h}(f,\widetilde{W}_{\underline{\varphi}},\chi^{-1})\overline{Z_{h_\ell}(\mathcal{R}(g^{-t},g'^{-t})W_{\underline{\varphi}},\chi)}(g_1,g_1'^tm_{\ell})\frac{d\mu(I(\underline{\sigma}) \otimes \chi)}{\zeta(1)|\det h_\ell|^{2-r}}.
    \end{split}
\end{align}
We have
\begin{align*}
   & \overline{Z_{h_\ell}(\mathcal{R}(g^{-t},g'^{-t})W_{\underline{\varphi}},\chi)}(g_1,g_1'^tm_{\ell})\\
    &=\int_{U_{r-2}(F) \backslash \mathcal{P}_{r-2}(F) \times F^\times}\overline{W}_{\underline{\varphi}}\left(\begin{psmatrix}pg'_1 & \\ &  \begin{psmatrix} 1 & \\ & c^{-1} \end{psmatrix}h_\ell\end{psmatrix}g_1g^{-t},pg'_1g'^{-t}  \right)\frac{d_r p\overline{\chi}(c)d^\times c}{|\det pg_1'|^{3/2}|c|^{(r-2)/2}}\\
    &=\mathcal{R}_u(g^{-t},g'^{-t},I_{2})\overline{Z_{h_\ell h}(W_{\underline{\varphi}},\chi)}(\begin{psmatrix}
        I_{r-2} & \\
        & h^{-1}
    \end{psmatrix}g_1,g_1'^tm_{\ell} h^\iota).
\end{align*}
 Combining this computation with \eqref{Fh} we deduce that 
\begin{align*}
    &\mathcal{F}_{h_\ell}(\mathcal{R}_u(g,g',h)f)(g_1,g_1'^tm_{\ell})\\
    &=|\det h|^{(2-r)/2}\mathcal{R}_u(g^{-t},g'^{-t},I_2)\mathcal{F}_{h_{\ell}h}(f)(\begin{psmatrix}
        I_{r-2}&\\
            &h^{-1}
    \end{psmatrix}g_1,g_1'^tm_{\ell}h^\iota)\\
    &=|\det h|^{2-r}\mathcal{R}_u(g^{-t},g'^{-t},h)\mathcal{F}_{h_{\ell}h}(f)(g_1,g_1'^tm_{\ell})
\end{align*}
as claimed.
\end{proof}

\begin{thm}\label{thm:L2ext}
    The map $\mathcal{F}$ extends to an isometry of $L^2\left(X_{\ell}(F), |\det g'|d\dot{g}d\dot{g}'\right)$. For $f_1,f_2 \in L^2\left(X_{\ell}(F), |\det g'|d\dot{g}d\dot{g}'\right)$ one has that
    \begin{align*}
        &\int_{N(F)\backslash\GL_r(F)\times \mathcal{P}_{r-2}\backslash \GL_{r-2}(F)} \mathcal{F}(f_1)(g,g'^tm_{\ell})\overline{f_2}(g,g'^tm_{\ell})|\det g'|d\dot{g}d\dot{g}'\\
        &= \int_{N(F)\backslash\GL_r(F)\times \mathcal{P}_{r-2}\backslash \GL_{r-2}(F)} f_1(g,g'^tm_{\ell})\overline{\mathcal{F}(f_2)}\left(\begin{psmatrix} -I_{r-2}& \\ & h_\ell^{-1} \begin{psmatrix} (-1)^{r-1} &\\ & (-1)^{r}\end{psmatrix} h_\ell  \end{psmatrix}g,(-1)^{r}g'^tm_{\ell}\right)|\det g'|d\dot{g}d\dot{g}'.
    \end{align*}
\end{thm}

\begin{proof}
Let us first verify the identity on $\mathcal{C}(X^\circ_{\ell}(F),\mathcal{H})_{\mathfrak{F}}$. By definition
\begin{align*}
        &\int_{N(F)\backslash\GL_r(F)\times \mathcal{P}_{r-2}(F)\backslash \GL_{r-2}(F)} \mathcal{F}(f_1)(g,g'^tm_{\ell})\overline{f_2}(g,g'^tm_{\ell})|\det g'|d\dot{g}d\dot{g}'\\
        &=\int_{N(F)\backslash\GL_r(F)\times \mathcal{P}_{r-2}(F)\backslash \GL_{r-2}(F)}\int_{\mathrm{Tp}_{r,r-2,1}}\gamma(I(\underline{\sigma})^\vee)\gamma(I(\sigma) \times \chi^{-1})\\& \quad\times 
    \sum_{\varphi \in \mathcal{B}_{\underline{\sigma}}}
    Z(f_1,\widetilde{W}_{\underline{\varphi}},\chi^{-1})\overline{Z(W_{\underline{\varphi}},\chi)}(g,g'^tm_{\ell})\overline{f_2}(g,g'^tm_{\ell})\frac{d\mu(I(\underline{\sigma})\otimes \chi)}{\zeta(1)|\det h_\ell|^{2-r}}|\det g'|d\dot{g}d\dot{g}'.
\end{align*}
We claim that this expression converges absolutely and hence we can switch the order of integration.  First, the integral over $\mathrm{Tp}_{r,r-2}$ has support in a finite set of components and the the sum over $\varphi$ is finite on each component as explained in the proof of Theorem \ref{thm:Plancherel}.  The $\gamma$-factors have polynomial growth as explained in \eqref{poly:growth}.  The integral over $N(F)\backslash\GL_r(F)\times \mathcal{P}_{r-2}(F)\backslash \GL_{r-2}(F)$ is absolutely convergent and is dominated by $(1+|\lambda|)^{d'}(1+\norm{\chi})^{d'}$ for some $d'>0$ by Theorem \ref{thm:Whitt:analy}, Lemma \ref{lem:Xi:bound}, and Lemma \ref{lem:Z:cont0}.  Finally, $Z(f_1,\widetilde{W}_{\underline{\varphi}},\chi^{-1}) \in \mathcal{C}(\mathrm{Tp}_{r,r-2,1})$ as explained below \eqref{section:good}.  This implies the desired absolute convergence.

Thus the integral above is equal to 
\begin{align*}
    &\int_{\mathrm{Tp}_{r,r-2,1}} \sum_{\varphi \in \mathcal{B}_{\underline{\sigma}}}\int_{N(F)\backslash\GL_r(F)\times \mathcal{P}_{r-2}(F)\backslash \GL_{r-2}(F)}Z(\widetilde{W}_{\underline{\varphi}},\chi^{-1})(g,g'^tm_{\ell})f_1(g,g'^tm_{\ell})|\det g'|d\dot{g}d\dot{g}'\\& \quad\quad\times \gamma(I(\underline{\sigma})^\vee)\gamma(I(\sigma) \times \chi^{-1})\overline{Z(f_2,W_{\underline{\varphi}},\chi)}\frac{d\mu(I(\underline{\sigma})\otimes\chi)}{\zeta(1)|\det h_\ell|^{2-r}}.
\end{align*}
By Lemma \ref{lem:gamma:factors} and a change of variables $(\underline{\sigma},W_{\underline{\varphi}},\chi) \mapsto (\underline{\sigma}^\vee,\widetilde{W}_{\underline{\varphi}},\chi^{-1})$ in    \eqref{eq:functionaltempered}, we have
\begin{align*}
     &\gamma( I(\underline{\sigma})^\vee )\gamma(
    I(\sigma) \times \chi^{-1})\overline{Z(f_2,W_{\underline{\varphi}},\chi)}\\
    &=\omega_{I(\sigma')}(-I_{r-2})\chi(-1)\overline{\gamma( I(\underline{\sigma}) )\gamma(
    I(\sigma)^\vee \times \chi)Z(f_2,W_{\underline{\varphi}},\chi)}\\
    &=\omega_{I(\sigma)}((-1)^{r-1}I_{r})\omega_{I(\sigma')}((-1)^{r}I_{r-2})\chi(-1)\overline{Z(\mathcal{F}(f_2),\widetilde{W}_{\underline{\varphi}},\chi^{-1})}.
\end{align*}
Thus the integral becomes 
\begin{align*}
    &\int_{\mathrm{Tp}_{r,r-2,1}} \sum_{\varphi \in \mathcal{B}_{\underline{\sigma}}}\int_{N(F)\backslash\GL_r(F)\times \mathcal{P}_{r-2}(F)\backslash \GL_{r-2}(F)}Z(\widetilde{W}_{\underline{\varphi}},\chi^{-1})(g,g'^tm_{\ell})f_1(g,g'^tm_{\ell})|\det g'|d\dot{g}d\dot{g}'\\& \quad\quad\times \omega_{I(\sigma)}((-1)^{r-1}I_{r})\omega_{I(\sigma')}((-1)^{r}I_{r-2})\chi(-1)\overline{Z(\mathcal{F}(f_2),\widetilde{W}_{\underline{\varphi}},\chi^{-1})}\frac{d\mu(I(\underline{\sigma})\otimes\chi)}{\zeta(1)|\det h_\ell|^{2-r}}.
\end{align*}    
By another change of order of integration, which is justified by an argument symmetric to that used before, this is
\begin{align*}
     &\int_{}\int_{\mathrm{Tp}_{r,r-2,1}} \sum_{\varphi \in \mathcal{B}_{\underline{\sigma}}}Z(\widetilde{W}_{\underline{\varphi}},\chi^{-1})\left(\begin{psmatrix} -I_{r-2}& \\ & h_\ell^{-1} \begin{psmatrix} (-1)^{r-1} &\\ & (-1)^{r}\end{psmatrix} h_\ell  \end{psmatrix}g,(-1)^{r}g'^tm_{\ell}\right)\overline{Z(\mathcal{F}(f_2),\widetilde{W}_{\underline{\varphi}},\chi^{-1})}\\
     &\times \frac{d\mu(I(\underline{\sigma})\otimes\chi)}{\zeta(1)|\det h_\ell|^{2-r}}f_1(g,g'^tm_{\ell})|\det g'|d\dot{g}d\dot{g}',
\end{align*}
where the outer integral is over $N(F)\backslash\GL_r(F)\times \mathcal{P}_{r-2}(F)\backslash \GL_{r-2}(F)$. By changing variables, $(\underline{\sigma},\chi)\mapsto (\underline{\sigma}^\lor,\chi^{-1})$ and using \Cref{lem:duals}, one deduces the claimed identity  by Theorem \ref{thm:PW:Z}.

 Now for $f\in \mathcal{C}(X^\circ_{\ell}(F),\mathcal{H})_{\mathfrak{F}},$ using Theorem \ref{thm:FT}, Lemma \ref{lem:FT:scalar} and Lemma \ref{lem:FT}, we have 
\begin{align*}
    \norm{f}_2^2&=\int_{X^\circ_{\ell}(F)} f(g,g'^tm_{\ell})\overline{f}(g,g'^tm_{\ell}) |\det g'|d\dot{g}d\dot{g'}\\
    &=\int_{X^\circ_{\ell}(F)} \mathcal{F}^2(f)\left(\begin{psmatrix} -I_{r-2}& \\ & h_\ell^{-1} \begin{psmatrix} (-1)^{r-1} &\\ & (-1)^{r}\end{psmatrix} h_\ell  \end{psmatrix}g,(-1)^{r}g'^tm_{\ell}\right)\overline{f}(g,g'^tm_{\ell}) |\det g'|d\dot{g}d\dot{g'}\\
    &=\int_{X^\circ_{\ell}(F)} \mathcal{F}(f)\left(\begin{psmatrix} -I_{r-2}& \\ & h_\ell^{-1} \begin{psmatrix} (-1)^{r-1} &\\ & (-1)^{r}\end{psmatrix} h_\ell  \end{psmatrix}g,(-1)^{r}g'^tm_{\ell}\right)\\
    &\quad\quad\quad\times\overline{\mathcal{F}(f)}\left(\begin{psmatrix} -I_{r-2}& \\ & h_\ell^{-1} \begin{psmatrix} (-1)^{r-1} &\\ & (-1)^{r}\end{psmatrix} h_\ell  \end{psmatrix}g,(-1)^{r} g'^tm_{\ell}\right) |\det g'|d\dot{g}d\dot{g'}\\
    &=\norm{\mathcal{F}(f)}_2^2.
\end{align*}
Since $\bigcup_{\mathfrak{F}}\mathcal{C}(X_\ell^\circ(F),\mathcal{H})_{\mathfrak{F}}$ is dense in $L^2(X_\ell^\circ(F),\mathcal{H}),$ this proves $\mathcal{F}$ extends to an isometry.  We deduce moreover that the identity in the theorem holds for general $f_1,f_2 \in L^2(X_\ell^\circ(F),\mathcal{H}).$
\end{proof}

\section{The asymptotic Schwartz space} \label{sec:As}
We define an asymptotic Schwartz space \index{$\mathcal{S}^{\mathrm{as}}(X_\ell(F),\mathcal{H})$} $\mathcal{S}^{\mathrm{as}}(X_\ell(F),\mathcal{H}):=\mathcal{S}_{h_\ell}^{\mathrm{as}}(X_{\ell}(F),\mathcal{H})$ following \cite{DRS:Schwartz}. We consider the following data: 
\begin{enumerate}
\item \label{ss} A standard Levi subgroup $M=M_1\times M_2\leq \GL_{r,r-2},$
\item \label{underline:sig} $\underline{\sigma}=\sigma \otimes \sigma' \in \mathrm{Irr}_2(M),$
\item \label{varphi} A $K_\infty$-finite vector $\varphi \in I(\underline{\sigma}).$
\end{enumerate}
We let $\mathcal{O}(\underline{\sigma})$ be the inertial orbit of $I(\underline{\sigma}).$ For $s \in \CC$ and quasi-characters $\chi:F^\times \to \CC^\times,$ let $\chi_s:=\chi|\cdot|^s.$
We let 
\begin{align*}
\mathcal{O}(\underline{\sigma})_{\CC}:=\{I(\underline{\sigma_{\lambda}}):\underline{\lambda}=(\lambda,\lambda') \in \mathfrak{a}_{M\CC}^*\}\quad\textrm{and}\quad
\mathrm{Tp}_{1\CC}:=\{\chi_s:(\chi,s) \in \mathrm{Tp}_{1} \times \CC\}.
\end{align*}
The spaces $\OO(\underline{\sigma})_{\CC}$ and $\mathrm{Tp}_{1\CC}$ admit a canonical structure of a complex analytic space.  In the non-Archimedean case, they are in fact $\CC$-points of algebraic tori.  Thus it makes sense to speak of holomorphic and meromorphic functions on $\OO(\underline{\sigma})_{\CC}$ and $\mathrm{Tp}_{1\CC},$ and polynomial and rational functions in the non-Archimedean case.

 Suppose $F$ is non-Archimedean. We define $\mathcal{S}^{\mathrm{as}}(X_{\ell}(F),\mathcal{H}) <\mathcal{C}(X^\circ_\ell(F),\mathcal{H})$ to be the subspace consisting of $f$  such that 
 for all data \eqref{ss}, \eqref{underline:sig}, and \eqref{varphi}, 
 the function on $\mathcal{O}(\underline{\sigma})_{\CC} \times \mathrm{Tp}_{1\CC}$ given by 
\begin{align*}
(I(\underline{\sigma_\lambda}) , \chi) \longmapsto \frac{Z(f,W_{\underline{\varphi_{\lambda}}},\chi)}{L(\tfrac{1}{2},I(\sigma_\lambda) \times I(\sigma'_{\lambda'}))L(\tfrac{1}{2},I(\sigma_\lambda)^{\vee} \times \chi)}
\end{align*}
is a polynomial. 
If $\mathfrak{F}$ is a finite set of $K_{r,r-2}$-types, we write
$\mathcal{S}^{\mathrm{as}}(X_\ell(F),\mathcal{H})_{\mathfrak{F}}$ for the subspace of functions with $K_{r,r-2}$-type in $\mathfrak{F}.$

Now assume $F$ is Archimedean. For subsets $\Omega \subset \mathfrak{a}_M^* \times \RR$ let $$
V_{\Omega}:=\left\{(\lambda,s) \in \mathfrak{a}_{M \CC}^* \times \CC: (\mathrm{Re}(\lambda),\mathrm{Re}(s)) \in \Omega\right\}
$$
be the cylinder over $\Omega.$
We let $\mathcal{S}^{\mathrm{as}}(X_{\ell}(F),\mathcal{H})_{\mathfrak{F}}$ consist of functions $f\in \mathcal{C}(X^{\circ}_{\ell}(F),\mathcal{H})_{\mathfrak{F}}$ such that for all data \eqref{ss}, \eqref{underline:sig}, and \eqref{varphi}, any $D\in U(\mathfrak{gl}_{r,r-2}  \oplus \mathrm{Lie}\,(\GL_2)_\ell),$  all compact subsets $\mathcal{K} \subset \mathfrak{a}_M^* \times \RR,$ and all polynomials $p$ on $\mathfrak{a}_{M\CC}^* \times \CC$ such that
\begin{align*}p(\underline{\lambda},s)L(\tfrac{1}{2},I(\sigma_\lambda) \times I(\sigma'_{\lambda'}))L(\tfrac{1}{2},I(\sigma_\lambda)^{\vee} \times \chi_s)
    \end{align*}
    is holomorphic on $V_{\mathcal{K}},$
the seminorm
    \begin{align*}
       \nu_{D,p,\mathcal{K}}(f):= \sum_{\chi\in \widehat{K}_{\GG_m}}\mathrm{sup}_{(\underline{\lambda},s) \in V_{\mathcal{K}}}\left|p(\underline{\lambda},s)Z_{h_\ell}(Df,W_{\underline{\varphi_{\lambda}}},\chi_s)\right|
\end{align*}
is finite.  
We endow $\mathcal{S}^{\mathrm{as}}(X_{\ell}(F),\mathcal{H})_{\mathfrak{F}}$ with a topology using the seminorms $\nu_{D,p,\mathcal{K}}.$  This gives $\mathcal{S}^{\mathrm{as}}(X_\ell(F),\mathcal{H})_{\mathfrak{F}}$ the structure of a Fr\'echet space.  We let $\mathcal{S}^{\mathrm{as}}(X_\ell(F),\mathcal{H}):=\bigcup_{\mathfrak{F}}\mathcal{S}^{\mathrm{as}}(X_\ell(F),\mathcal{H})_{\mathfrak{F}}.$  It is naturally an LF space. 

A priori, the spaces $\mathcal{S}^{\mathrm{as}}(X_\ell(F),\mathcal{H})$ depend on the choice of $h_\ell.$  We write $\mathcal{S}_{h_\ell}^{\mathrm{as}}(X_\ell(F),\mathcal{H})$ for the space when we need to indicate the dependence.

For each $h \in \GL_2(F)$ we have a map
\begin{align} \label{Rh}
\begin{split}
\mathcal{R}(h):\mathcal{C}(X_{\ell h^{\iota}}^\circ(F),\mathcal{H}) &\lto \mathcal{C}(X_{\ell}^\circ(F),\mathcal{H})\\
f &\longmapsto \left((g,g') \mapsto f\left( \begin{psmatrix} I_{r-2} & \\ & h\end{psmatrix}^{-1}g,g'^tm_{\ell}h^{\iota} \right)\right).
\end{split}
\end{align}
One checks the following lemma:
\begin{lem} \label{lem:Rh}
The operator $\mathcal{R}(h)$ restricts to an isomorphism $$\mathcal{S}_{h_\ell h}^{\mathrm{as}}(X_{\ell h^{\iota}}(F),\mathcal{H}) \lto\mathcal{S}_{h_\ell }^{\mathrm{as}}(X_{\ell}(F),\mathcal{H}).$$ It is a homeomorphism when $F$ is Archimedean. \qed
\end{lem}

\begin{thm} \label{thm:As:FT}
The Fourier transform $\mathcal{F}_{h_\ell}$ induces an automorphism of $\mathcal{S}^{\mathrm{as}}(X_{\ell}(F),\mathcal{H})_{\mathfrak{F}}.$ It is continuous in the Archimedean case.
\end{thm}

\begin{proof}
By Corollary \ref{cor:temp:func:eqn}, we have 
\begin{align*}
    \frac{Z(\mathcal{F}(f),W,\chi)}{L(\tfrac{1}{2},\pi \times \pi')L(\tfrac{1}{2},\pi^{\vee} \times \chi)}&=\frac{\gamma\left(\tfrac{1}{2},\pi^\lor\times\pi'^\lor,\psi\right)\gamma\left(\tfrac{1}{2},\pi\times\chi^{-1},\psi\right)Z(f,\widetilde{W},\chi^{-1})}{L(\tfrac{1}{2},\pi \times \pi')L(\tfrac{1}{2},\pi^{\vee} \times \chi)}\\
    &=\varepsilon(\tfrac{1}{2},\pi^\vee \times \pi'^\vee,\psi)\varepsilon(\tfrac{1}{2},\pi \times \chi,\psi)\frac{ Z(f,\widetilde{W},\chi^{-1})}{L(\tfrac{1}{2},\pi^\lor \times \pi'^\lor)L(\tfrac{1}{2},\pi \times \chi^{-1})}.
\end{align*}
Using well-known properties of $\varepsilon$-factors \cite[Proposition 3.5]{CogdellPCMI}, this implies directly that $\mathcal{F}$ is an automorphism in the non-Archimedean case.  In the Archimedian case, one can apply the estimate in \cite[Proposition 3.23]{DRS:Schwartz} to deduce that $\mathcal{F}$ is a continuous automorphism with continuous inverse.
\end{proof}

\subsection{The local functional equation}

We begin by clarifying that the various possible definitions of the local zeta integrals all agree when they are defined.  
Let $\underline{\sigma} \in \mathrm{Irr}_2(M)$ and let $\underline{\varphi} \in \mathcal{B}_{\underline{\sigma}}.$
Assume $f \in \mathcal{S}^{\mathrm{as}}(X_\ell(F),\mathcal{H})$ is $\Xi$-rapidly decreasing. 

For $\lambda \in i\mathfrak{a}_M^*,$ the integral $Z(f,(W_{\underline{\varphi_{\lambda}}})_{s'},\chi_s)$ converges absolutely for  $\mathrm{Re}(s),\mathrm{Re}(s') \geq 0$ by Theorem \ref{thm:Whitt:analy} and Lemma \ref{lem:ZfW}.  
By the definition of the asymptotic Schwartz space, $Z(f,(W_{\underline{\varphi_{\lambda}}})_{s'},\chi_s)$ admits a meromorphic continuation to all $(\underline{\lambda},s',s) \in \mathfrak{a}_{M\CC}^* \times \CC \times \CC.$  This provides one definition of the local zeta integral.

On the other hand, by Lemma \ref{lem:ZfW2} for any compact set $\Omega \subset \mathfrak{a}_{M\CC}^*,$ there is a $N_1$ such that the integral defining 
 $Z(f,(W_{\underline{\varphi_{\lambda}}})_{s'},\chi_s)$ converges for $\lambda \in \Omega$ and for  $\mathrm{Re}(s') \geq N_1,$ $\mathrm{Re}(s) \geq N_1.$  This provides another definition of the $Z(f,(W_{\underline{\varphi_{\lambda}}})_{s'},\chi_s).$  Since we may take $\Omega$ to intersect $i\mathfrak{a}_M^*,$ the identity principle implies that these two possible definitions agree when both are defined.

\begin{prop}\label{zetafuncequa}
Assume that $f \in \mathcal{S}^{\mathrm{as}}(X_\ell^{\circ}(F),\mathcal{H})$ and $\mathcal{F}_{h_\ell}(f)$ are $\Xi$-rapidly decreasing.  Then for all generic unitary representations $\pi \otimes \pi'$ of $\GL_{r,r-2}(F),$ $\chi \in \mathrm{Tp}_1,$ and $W \in \mathcal{W}(\pi,\psi) \widehat{\otimes} \mathcal{W}(\pi',\overline{\psi}),$ one has that
\begin{align} \label{fe}
    &Z(\mathcal{F}_{h_\ell}(f),W,\chi_s)=
    \gamma(\tfrac{1}{2},\pi^\vee \times \pi'^\vee,\psi)\gamma(\tfrac{1}{2},
    \pi \otimes (\chi_s)^{-1},\psi)Z(f,\widetilde{W},\chi^{-1}_s)
\end{align}
as meromorphic functions in $s.$ 
\end{prop}

\begin{proof}
It is well-known that $\pi \cong I(\sigma_\lambda)$  and $\pi' \cong I(\sigma'_{\lambda'})$ for some $(\sigma,\sigma') \in \mathrm{Irr}_2(M_1) \times \mathrm{Irr}_2(M_2),$  $(\lambda,\lambda') \in \mathfrak{a}_{M_1\CC}^* \times \mathfrak{a}_{M_2\CC}^*.$ See \cite[Theorem 8.4.4]{Getz:Hahn} for the non-Archimedean setting and \cite[Lemma 2.5]{JacquetPerfectRS} for the Archimedean setting.  
The functional equation \eqref{fe}
is true by the definition of the Fourier transform for $(\lambda,\lambda',s) \in i\mathfrak{a}_M^* \times i\mathfrak{a}_{M'}^* \times i\RR.$ 
We deduce it in general by the identity principle.
\end{proof}

\section{The unramified computation}\label{sec:unramified}
Let $F$ be a non-Archimedean local field unramified over its prime field.  We assume that the additive character $\psi:F \to \CC^\times$  is unramified.  
For $$(\lambda=(\lambda_1,\dots,\lambda_r),\lambda'=(\lambda_1',\dots,\lambda_{r-2}')) \in \mathfrak{a}_{T_r\CC}^* \times \mathfrak{a}_{T_{r-2}\CC}^*,$$ let
$$
I(\lambda):=I(\one_{T_r(F)},\lambda) \quad\textrm{and}\quad I(\lambda'):=I(\one_{T_{r-2}(F)},\lambda').
$$
Here $\one_{T_r(F)}$ (resp.~$\one_{T_{r-2}(F)}$) denotes the trivial representation of $T_r(F)$ (resp.~$T_{r-2}(F)$).   Let $$(W_\lambda,W'_{\lambda'}) \in \mathcal{W}(I(\lambda),\psi)^{\GL_{r}(\OO_F)} \times \mathcal{W}(I(\lambda'),\bar{\psi})^{\GL_{r-2}(\OO_F)}$$be the unique pair satisfying $W_{\lambda}(I_{r})=W'_{\lambda'}(I_{r-2})=1.$  

 Let $v_{r}:=(r-1,r-2,\ldots, 0) \in \ZZ^r$. For $\nu=(\nu_i)\in \ZZ_{\ge 0}^r$ with $\nu_1 \geq \cdots \geq \nu_r$, define the Schur polynomial
\begin{align*} 
    \mathbb{S}_{\nu_1, \dots, \nu_r}(x_1,\dots,x_r) \in \ZZ[x_1,\dots,x_r]
\end{align*}
as usual.
 The definition directly extends to all $\nu_1 \geq \cdots \geq \nu_r$, in which case $\mathbb{S}_{\nu_1, \dots, \nu_r}$ is only a polynomial in $x_1^{\pm 1},\ldots, x_r^{\pm 1}$. By convention $\mathbb{S}_{\nu_1, \dots, \nu_r}=0$ unless $\nu_1 \geq \cdots \geq \nu_r$. We have
\begin{align} \label{CS}
\delta^{-1/2}_{B_{r}}\begin{psmatrix}\varpi^{\nu_1} & & \\ & \ddots & \\ & & \varpi^{\nu_r} \end{psmatrix}W_{\lambda}\begin{psmatrix}\varpi^{\nu_1} & & \\ & \ddots & \\ & & \varpi^{\nu_{r}} \end{psmatrix}=\mathbb{S}_{\nu_1,\dots,\nu_{r}}(q^{-\lambda})
\end{align}
by Shintani's formula (a special case of the Casselman-Shalika formula) \cite{ShintaniWhittaker} (see also \cite[\S 3.1.3]{CogdellPCMI}).  By keeping track of domains of convergence in \cite[Lemma 6.1]{GGHL}, one obtains the following lemma:

\begin{lem}\label{lem:rankin} 
 For $\mathrm{Re}(s')-\max_{i,j}(|\mathrm{Re}(\lambda_i)|+|\mathrm{Re}(\lambda_j')|)>-\tfrac{1}{2},$ one has that
\begin{align*}
 &\int_{T_{r-2}(F)}  W_{\lambda}\begin{psmatrix} t & &\\ & 1& \\ & & \varpi^{-\ell}\end{psmatrix} W'_{\lambda'}(t) \frac{|\det t|^{s'-1/2}dt}{\delta_{B_{r-2}}(t)}\\
 =&\frac{L(\tfrac{1}{2}+s',I(\lambda) \times I(\lambda'))}{q^{\ell(r-1)/2}}\sum_{n=0}^{r-2} (-1)^{n}\mathbb{S}_{(0,\dots,0, n-\ell)}(q^{-\lambda})\mathrm{Tr}(\wedge^n q^{-\lambda'})q^{-n(s'+1/2)},
\end{align*}
where the integral converges absolutely.
The identity extends as a meromorphic function of $(\lambda,\lambda',s)$.\qed
\end{lem}

\noindent We point out that if $I(\lambda)$ is unitary and irreducible (hence generic) then $|\mathrm{Re}(\lambda_i)|<\tfrac{1}{2}$ for all $i$ \cite[Corollary 2.5]{JacquetShalikaEPI}.
Applying a trivial estimate to \cite[(6.1.5)]{GGHL}, we obtain the following corollary:

\begin{cor} \label{cor:unram:bound}
Assume
 $\max_{i,j}(|\mathrm{Re}(\lambda_i)|,|\mathrm{Re}(\lambda_j')|)<\tfrac{1}{2}$  and 
 $$\mathrm{Re}(s')-\max_i(|\mathrm{Re}(\lambda_i)|)-\max_j(|\mathrm{Re}(\lambda_j')|)>-\tfrac{1}{2}.$$  
 Then for any $\epsilon>0$,
\begin{align*}
\int_{T_{r-2}(F) }  \left|W_{\lambda}\begin{psmatrix} t & &\\ & 1& \\ & & \varpi^{-\ell}\end{psmatrix} W'_{\lambda'}(t)\right| \frac{|\det t|^{\mathrm{Re}(s')-1/2}dt}{\delta_{B_{r-2}}(t)} \ll_{\epsilon} (1+q^{\epsilon})q^{-\ell((r-1)/2-1/2-\epsilon)}.
\end{align*}
The implied constant can be taken to be $1$ if the 
residual characteristic is large enough in a sense depending on $\epsilon.$ \qed
\end{cor}

Assume that $h_\ell\in \mathrm{GL}_2(\OO_F)$. Let
\begin{align*}
    \one_{\ell}(g,m):=\int_{M_{r-2,2}(F)}\one_{\GL_r(\OO_F) \times \mathcal{M}_{\ell}(\OO_F)}\left(\begin{psmatrix} I_{r-2} & z\\ & I_2 \end{psmatrix}g,m\right)\psi(\langle m,z \rangle)dz
\end{align*}
and \index{$b_{\ell}$}
\begin{align} \label{bell}
b_{\ell}(g,m):=\int_{\GL_{r-2}(F)}\one_{M_{r-2}(\OO_F)}(h)\one_{\ell}\left(\begin{psmatrix}h & \\ & \begin{psmatrix}1 & \\  & \det h^{-1}\end{psmatrix}h_\ell \end{psmatrix}g,h^{-t}m\right)\frac{dh}{|\det h|}.
\end{align}
We refer to $b_\ell$ as the \textbf{basic function}.

The following lemma implies that the basic function is $\Xi$-rapidly decreasing:

\begin{lem}\label{lem:basic:bound}
As a function of $(g,m),$ the integral 
\begin{align} \label{bound:basic}
\int_{\GL_{r-2}(F)}\one_{M_{r-2}(\OO_F)}(h)|\one_{\ell}|\left(\begin{psmatrix}h & \\ & \begin{psmatrix}1 & \\  & \det h^{-1}\end{psmatrix}h_\ell
\end{psmatrix}g,h^{-t}m\right)\frac{dh}{|\det h|}
\end{align}
has support in 
$
\{(g,m) \in X_\ell(\mathcal{O}_F) \cap X_{\ell}^{\circ}(F):\norm{e_{r}^tg}=\norm{e_{r-1}^tg\wedge e_{r}^tg}\}$
and it is bounded by $\norm{e_r^tg}^{-1}.$  Moreover, it satisfies the bound \eqref{rd:bound}. Hence the same is true of the basic function $b_\ell.$
\end{lem}

\begin{proof}
Using the Iwasawa decomposition, we may assume that 
$x=(g,m)=\left(\begin{psmatrix} g_0 &  \\ & h_\ell^{-1}\begin{psmatrix} z & y_2 \\ & t_r^{-1}\end{psmatrix}\end{psmatrix},m\right)$ for some $g_0 \in \GL_{r-2}(F),$ $z,t_r^{-1} \in F^\times,$ and $y_2 \in F.$ 
Consider 
\begin{align*}
\int_{\GL_{r-2}(F)}\one_{M_{r-2}(\mathcal{O}_F)}(h)|\one_\ell|\left( \begin{psmatrix} hg_0 & & \\ & z & y_2\\ & & (\det h)^{-1}t_{r}^{-1} \end{psmatrix},h^{-t}m\right)|\det h|^{-1}dh.
\end{align*}
This integral is nonzero only if $|y_2|\le 1, |z|=1$ and $|t_{r}^{-1}|=|\det g_0|^{-1}=|\det h|\le 1,$ so $|\det g|=1.$  This implies that the integral is supported in the set of $h$ such that 
\begin{align*}
\det \begin{psmatrix} z & y_2 \\ & t_{r}^{-1} \end{psmatrix} g_0 \in (\det h) h^{-1}\GL_{r-2}(\OO_F) \subset M_{r-2}(\OO_F).
\end{align*}
Moreover, the integral vanishes unless $m \in h^{t}\mathcal{M}_{\ell}(\OO_F)$ for some $h \in M_{r-2,r-2}(\OO_F).$ 
Using these observations, we deduce the bound on the support claimed in the lemma. Since the integral in \eqref{bound:basic} is supported on
\begin{align*}
    \{h \in \GL_{r-2}(F) : h\in \GL_{r-2}(\OO_F)g_0^{-1}\cap M_{r-2}(\OO_F)\},
\end{align*}
it is bounded by $|\det g_0|=|t_r|=\norm{e_r^tg}^{-1}$. We note that $\GL_{r-2}(\OO_F)g_0^{-1}\cap M_{r-2}(\OO_F)$ is nonempty only if $g_0^{-1}\in M_{r-2}(\OO_F).$  Moreover, for $h$ in the support of the integral $g_0^tm\in \GL_{r-2}(\OO_F)h^{-t}m\subseteq M_{r-2,2}(\mathcal{O}_F).$ Thus the integral vanishes unless
\begin{align} \label{int:cond}
g_0^{-1} \in M_{r-2}(\OO_F) \quad\textrm{and}\quad g_0^tm\in M_{r-2,2}(\mathcal{O}_F).
\end{align}

To see this implies \eqref{rd:bound}, write $x=(g,m)\in X_\ell(\mathcal{O}_F)\cap X_\ell^\circ(F)$ using the coordinates \eqref{coord}. Since the integral has support in $X_\ell(\OO_F)$ with $\norm{e_{r}^tg}=\norm{e_{r-1}^tg\wedge e_{r}^tg},$ the integral is nonvanishing only if the following holds
\begin{align} \label{int:constraints} \begin{split}
    &z,t_r^{-1}t_{r-2}^{r-2}\det g_1\in \OO_F^\times,\\
    &t_r^{-1},t_{r-2}^{-1}, t_{r-2}t_r^{-1},t_{r-2}a,y_2\in \mathcal{O}_F,\\
    &t_{r-2}^{-1}g_1^{-1}, t_{r-2}t_r^{-1}g_1\in M_{r-3}(\mathcal{O}_F),\\
    &t_{r}^{-1}t_{r-2}y_1,t_{r-2}^{-1}g_1^{-1}y_1\in \mathcal{O}_F^{r-3}. \end{split}
\end{align}
 We need to show that if $\norm{e_{r}^tg}=\norm{e_{r-1}^tg\wedge e_{r}^tg},$ then for any $d\ge 0$,
\begin{align*}
|t_r|\one_{X_\ell(\OO_F)}(g,m)\ll_d  \Xi_{X_\ell^\circ}(g,m)\sigma_{X_\ell^\circ}(g,m)^{-d}.
\end{align*}
We further assume $g_1=t't''\in T_{r-3}(F)^+$ as in Lemma \ref{lem:Levi}. By definition (see \eqref{XiXP})
    \begin{align*}  
\Xi_{X_\ell^\circ}(x)&=\frac{|t_{r-2}|^{r-2}|t_r|^{(r-2)/2}|\det g_1|^{5/4}\Xi^{1/2}_{\GL_{r-3}}(g_1)\Xi_{\GL_{r-2}}^{1/2}\begin{psmatrix}
        g_1 &y_1 \\
          & 1
    \end{psmatrix}}{|a|^{(r-2)/2}\max\left(1,|y_2|\right)^{1/2}}\\
&=|t_r|^{r/2}|\det g_1|^{1/4}\Xi^{1/2}_{\GL_{r-3}}(g_1)\Xi_{\GL_{r-2}}^{1/2}\begin{psmatrix}
        g_1 &y_1 \\
          & 1
    \end{psmatrix}|a|^{-(r-2)/2}.
\end{align*}
By Proposition \ref{prop:HCfunction}\ref{delta:asymp}, Lemma \ref{lem:Levi} and Lemma \ref{lem:parab}, there is a $d_0>0$ such that
\begin{align*}
    \Xi_{X_\ell^\circ}(g,m)\gg \frac{|t_r|^{r/2}\delta_{B_{r-3}}^{1/2}(g_1)\prod_{i=1}^j |t_i|^{1/2}}{|a|^{(r-2)/2}\max\left(1,\min\left(\norm{g_1^{-1}y_1},\norm{y_1}\right)\right)^{(r-3)/2}}\sigma_{X_\ell^\circ}(g,m)^{-d_0}.
\end{align*}
Thus it suffices to show for any $d\ge 0$
\begin{align}\label{eq:goal}
     \frac{|t_r|^{(r-2)/2}\delta_{B_{r-3}}^{1/2}(g_1)\prod_{i=1}^j |t_i|^{1/2}}{|a|^{(r-2)/2}\max\left(1,\min\left(\norm{g_1^{-1}y_1},\norm{y_1}\right)\right)^{(r-3)/2}}\sigma_{X_\ell^\circ}(g,m)^{-d}
\end{align}
is uniformly bounded below by a positive constant for $x\in X_\ell(\OO_F)$ with $\norm{e_{r}^tg}=\norm{e_{r-1}^tg\wedge e_{r}^tg}$.

The constraints \eqref{int:constraints} imply
\begin{align*}
   &|t_{r-2}^{-1}|\le |t_1|, \quad |at_{r-2}|\le 1, \quad |t_rt_{r-2}^{-1}|\ge 1, \quad \norm{g_1^{-1}y_1}\le |t_{r-2}|.
\end{align*}
We have
$
    |a|\min\left(\norm{g_1^{-1}y_1},\norm{y_1}\right)\le |at_{r-2}|\le 1$
and 
\begin{align*}
    \delta_{B_{r-3}}^{1/2}(g_1)\prod_{i=1}^j |t_i|^{1/2}&= \left(\prod_{i=1}^{r-3} |t_i|^{(r-2)/2-i}\right)\prod_{i=1}^j |t_i|^{1/2}=|\det g_1|^{-(r-4)/2}\left(\prod_{i=1}^{r-3} |t_i|^{r-3-i}\right)\prod_{i=1}^j |t_i|^{1/2}\\
    &\ge |\det g_1|^{-(r-4)/2}\prod_{i=1}^{j} |t_1|^{r-3+1/2-i}\ge |\det g_1|^{-(r-4)/2}|t_{r-2}|^{-(2r-6-j)j/2}\\
    &=|t_r|^{-(r-4)/2}|t_{r-2}|^{(r-4)(r-2)/2-(2r-6-j)j/2}\ge |t_r|^{-(r-4)/2}|t_{r-2}|^{-1/2}\ge |t_r|^{-(r-3)/2}.
\end{align*}
Thus \eqref{eq:goal} is bounded below by 
\begin{align*}
    |t_r/a|^{1/2} \sigma_{X_\ell^\circ}(g,m)^{-d}.
\end{align*}
Additionally, we have 
$
    |t_{r-2}t_{r-3}|\le |t_r|,$ $|t_{r-2}y_1|\le |t_r|,$ and $|a| \leq 1$ by \eqref{int:constraints}.
Therefore for any $1/2>\epsilon>0,$
\begin{align*}
    |t_r/a|^{1/2} \sigma_{X_\ell^\circ}(g,m)^{-d}\gg_{\epsilon,d} |t_r/a|^{1/2-\epsilon}\ge 1.
\end{align*}
\end{proof}

\begin{lem} \label{unramifiedcomp}
Assume $\max_{i,j}(|\mathrm{Re}(\lambda_i)|,|\mathrm{Re}(\lambda_j')|)<\tfrac{1}{2}.$ The integral defining $
Z(b_{\ell},W_{\lambda}\otimes (W'_{\lambda'})_{s'},\chi)$
converges absolutely for $ \mathrm{Re}(\chi) \gg 1,  \mathrm{Re}(s') \gg 1,$ where the implied constants are absolute.
One has that
$$Z(b_{\ell},W_{\lambda}\otimes (W'_{\lambda'})_{s'},\chi)=L(\tfrac{1}{2}+s',I(\lambda) \times I(\lambda'))L(\tfrac{1}{2},I(\lambda)^{\vee} \otimes \chi).$$
\end{lem}

\begin{proof}
It suffices to treat the case $\ell=(1,0)$ and $h_\ell=I_2.$ 
We claim that 
\begin{align*} 
\begin{split}
    \int_{\mathcal{P}_{r-2}(F) \backslash \GL_{r-2}(F) \times N(F) \backslash \GL_r(F)} &\int_{\GL_{r-2}(F)}\one_{M_{r-2}(\OO_F)}(h)|\one_{\ell}|\left(\begin{psmatrix}h & \\ & 1 & \\  & & \det h^{-1}
\end{psmatrix}g,h^{-t}m\right)\frac{dh}{|\det h|}\\
&\quad\times
    |Z(W_\lambda\otimes (W'_{\lambda'})_{s'},\chi)(g,g'^tm_\ell)|
    |\det g'|d\dot{g}'d\dot{g}
    \end{split}
\end{align*}
converges under the assumptions of the lemma. The claim is a consequence of \Cref{lem:ZfW2} and \Cref{lem:basic:bound} as we now explain.  In  \Cref{lem:ZfW2} we take 
$$
\Omega=\left\{(\lambda,\lambda'):\max\left(|\mathrm{Re}(\lambda_i)|,|\mathrm{Re}(\lambda_i')|\right) \leq \tfrac{1}{2}\right\}.
$$
The bound in Lemma \ref{lem:ZfW2} ultimately relies on the gauge estimate in \Cref{lem:nonArch:Whitt:fam:bound}. By \eqref{CS}, in the notation of \Cref{lem:nonArch:Whitt:fam:bound} we can take $C$ be be an absolute constant and $\Omega=\OO_F^{r-1}.$  It follows
that we can take $N_1$ and the implicit constant to be independent of the residual characteristic.

By \eqref{Z:ids} we have that
\begin{align*}
Z(b_{\ell},W_\lambda\otimes (W'_{\lambda'})_{s'},\chi)&= \int b_\ell\left(g,g'^tm_{\ell}\right) 
    Z(W_\lambda\otimes (W'_{\lambda'})_{s'},\chi)(g,g'^tm_\ell)
    |\det g'|d\dot{g}'d\dot{g}\\
    &=\int\int_{\GL_{r-2}(F)}\one_{M_{r-2}(\OO_F)}(h)\one_{\ell}\left(\begin{psmatrix}h & \\ & \begin{psmatrix}1 & \\  & \det h^{-1}\end{psmatrix}
\end{psmatrix}g,h^{-t}g'^tm\right)\frac{dh}{|\det h|}\\
&\hspace{0.5 in}\times Z(W_\lambda\otimes (W'_{\lambda'})_{s'},\chi)(g,g'^tm_\ell)
    |\det g'|d\dot{g}'d\dot{g}.
\end{align*}
Here the outer integrals are over $\mathcal{P}_{r-2}(F) \backslash \GL_{r-2}(F) \times N(F) \backslash \GL_r(F).$  By the absolute convergence statement above, we can change variables $(g,g') \mapsto \left(\begin{psmatrix}h & \\ & \begin{psmatrix}1 & \\  & \det h^{-1}\end{psmatrix}\end{psmatrix}^{-1}g,g'h\right) $
to see that the above is 
\begin{align*}
\begin{split}
 \int_{\GL_{r-2}(F)}&\int_{\mathcal{P}_{r-2}(F) \backslash \GL_{r-2}(F) \times N(F) \backslash \GL_r(F)}\one_{\ell}\left(
g,g'^tm_\ell\right)\one_{M_{r-2}(\OO_F)}(h)\\& \times Z(W_\lambda\otimes (W'_{\lambda'})_{s'},\chi)\left(\begin{psmatrix} h & & \\ &1 & \\  & &\det h^{-1}\end{psmatrix}^{-1}g,h^tg'^tm_\ell\right) |\det h|^{r}
    |\det g'| dhd\dot{g}'d\dot{g}.
\end{split}
\end{align*} By the definition of $\one_\ell,$ this is
\begin{align} \label{before:comp}
\begin{split}
   & \int_{\mathcal{P}_{r-2}(F)\backslash\GL_{r-2}(F)} \int_{\GL_{r-2}(F)}\one_{\OO_F^{r-2}}\left(g'^te_{r-2}\right)\one_{M_{r-2}(\OO_F)}(h)\\
&\quad\quad\quad\quad\times Z(W_\lambda\otimes (W'_{\lambda'})_{s'},\chi)\left(\begin{psmatrix} h & & \\ &1 & \\  & &\det h^{-1}\end{psmatrix}^{-1},h^tg'^tm_\ell\right)|\det h|^{r}
    |\det g'| dhd\dot{g}'.
\end{split}
\end{align}

Now using \eqref{ZWchi:func} and changing variables $c\mapsto c(\det h)$ in the integral \eqref{ZWchi}, the inner integral is
\begin{align*}
    \int_{\GL_{r-2}(F)}&\one_{M_{r-2}(\OO_F)}(h)\chi(\det h)|\det h|^{(r-1)/2}Z(W_\lambda\otimes \mathcal{R}(h)(W'_{\lambda'})_{s'},\chi)(I_{r},g'^tm_\ell)dh\\
    &=L(1+s',I(\lambda') \otimes \chi)Z(W_\lambda\otimes(W'_{\lambda'})_{s'},\chi)(I_r,g'^tm_\ell).
\end{align*}
Here we have used some basic facts on Godement-Jacquet integrals \cite[\S 1.6]{GodementJacquetBook}.  

Substituting this into \eqref{before:comp}, we arrive at $L(1+s',I(\lambda')\otimes\chi)$ times
\begin{align*} 
\int_{\mathcal{P}_{r-1}(F) \backslash \GL_{r-2}(F)}
\one_{\OO_F^{r-2}}\left(g'^te_{r-2}\right)Z(W_\lambda\otimes (W'_{\lambda'})_{s'},\chi)(I_r,g'^tm_\ell)|\det g'|d\dot{g}'.
\end{align*}
Using Corollary \ref{cor:unram:bound} to justify interchanges of integrals, by \eqref{ZWchi} and \eqref{ZWchi:func} this is
\begin{align*}
&\int_{U_{r-2}(F) \backslash \GL_{r-2}(F) \times F^\times} \one_{\OO_F^{r-2}}\left(g'^te_{r-2}\right)W_\lambda\begin{psmatrix}g' & & \\ &   1 & \\ && c^{-1} \end{psmatrix}W_{\lambda'}'(g')  \frac{|\det g'|^{s'-1/2}dg' \chi(c)d^\times c}{|c|^{(r-2)/2}}\\&=
    \int_{T_{r-2}(F) \times F^\times}W_\lambda\begin{psmatrix}t & & \\ &   1 & \\ && c^{-1} \end{psmatrix}W_{\lambda'}'(t)  \frac{|\det t|^{s'-1/2}dg' \chi(c)d^\times c}{\delta_{B_{r-2}}(t)|c|^{(r-2)/2}}.
\end{align*}
By Lemma \ref{lem:rankin} the above is 
\begin{align*}
    & L\left(\tfrac{1}{2}+s',I(\lambda) \times I(\lambda')\right) \sum_{k=0}^{\infty} q^{-k/2}\chi(\varpi^{k})\sum_{n=0}^{r-2} (-1)^n\mathbb{S}_{(0,\dots,0, n-k)}(q^{-\lambda})\mathrm{Tr}(\wedge^n q^{-\lambda'})q^{-n(s'+1/2)}.  
\end{align*}
The sum over $n$ is supported in $n \leq k.$  We change variables $k \mapsto k+n$ to see that this is 
\begin{align*}
   & L\left(\tfrac{1}{2}+s',I(\lambda) \times I(\lambda')\right) \sum_{k=0}^{\infty} q^{-k/2}\chi(\varpi^{k})\mathbb{S}_{(0,\dots,0, -k)}(q^{-\lambda})\sum_{n=0}^{r-2} (-1)^n\chi(\varpi^n)\mathrm{Tr}(\wedge^n q^{-\lambda'})q^{-n(s'+1)}  
  \\
    & = L\left(\tfrac{1}{2}+s',I(\lambda) \times I(\lambda')\right) L\left(\tfrac{1}{2}, I(\lambda)^{\vee}\otimes \chi\right)L\left(1+s',I(\lambda') \otimes \chi\right)^{-1}.
\end{align*}
\end{proof}

\begin{cor} \label{cor:basic:fixed}
One has that $b_\ell \in \mathcal{S}^{\mathrm{as}}(X_\ell(F),\mathcal{H})$ and $\mathcal{F}(b_\ell)=b_\ell.$
\end{cor}
\begin{proof} 
Since $b_\ell$ is $\GL_{r,r-2}(\OO_F)$-invariant, the first assertion follows from Lemma \ref{lem:basic:bound},  Lemma \ref{unramifiedcomp} and the definition of the asymptotic Schwartz space in \S \ref{sec:As}.  The second assertion follows from Lemma \ref{unramifiedcomp} and the definition of the Fourier transform in \Cref{thm:FT}.
\end{proof}

\section{The summation formula} \label{sec:the:sum}

Assume $F$ is a number field, and let $S$ be a set of places of $F$ large enough that $\beta \in \OO_F^{S \times}$ and $\ell \in \mathbb{P}^1(F)$ has image in $\mathbb{P}^1(\OO_F^S).$   Then all of the schemes defined above have natural models over  $\OO_F^S.$   We will denote these models by the same letters.

Let $\infty$ be the set of infinite places of $F.$ For an algebraic group $G$ over $F,$ let 
\begin{align}  \label{AG}
    A_G \leq  \mathrm{Res}_{F/\QQ}G(\RR)=G(F_\infty)
\end{align}
be the neutral component of the real points of the greatest $\QQ$-split torus in the center of $\mathrm{Res}_{F/\QQ}G.$ Let 
\begin{align} \begin{split}
    G(\A_F)^1:&=\bigcap_{\chi \in X^*(G)}\ker(|\cdot| \circ \chi:G(\A_F) \lto \RR),\\
[G]:&=G(F) \backslash G(\A_F) \quad\textrm{and}\quad [G]^1:=G(F) \backslash G(\A_F)^1. \end{split}
\end{align}
For an irreducible admissible representation $\pi$ of $\GL_r(\A_F)$ and $s \in \CC,$ we let $\pi_s:=\pi\otimes |\det|^s.$  We let $\mathrm{Re}(\pi)$ be the unique real number such that $\pi_{-\mathrm{Re}(\pi)}$ has unitary central character.  

\subsection{Unfolding the zeta integral}

Let $\pi$ be a cuspidal automorphic representation of $A_{\GL_r} \backslash \GL_r(\A_F).$ For any $\varphi$ in the $\pi$-isotypic subspace of $L^2(A_{\GL_r} \backslash [\GL_r])$ and $s \in \CC,$ let $\varphi_{s}(g):=|\det g|^{s}\varphi(g).$  For functions $\Phi \in \mathcal{S}(\GL_r(\A_F))$ that are finite under a maximal compact subgroup $K\leq \GL_r(\A_F),$ set
\begin{align*}
    K_{\pi_s(\Phi)}(g,h):=\sum_{\varphi \in \mathcal{B}_{\pi}}\mathcal{R}(\Phi)\varphi_s(g)\overline{\varphi}_{-s}(h),
\end{align*}
where $\mathcal{B}_{\pi}$ is an orthonormal basis of the $\pi$-isotypic subspace $V_\pi$ of $L^2(A_{\GL_r} \backslash \GL_r(\A_F))$ consisting of $K$-finite vectors.  
Because of the assumption that $\Phi$ is $K$-finite, the sum in fact has finite support.   
We assume without loss of generality that the orthonormal basis $\mathcal{B}_{\pi}$ consists of pure tensors.  

Whenever $\varphi$ is a smooth vector in $V_\pi,$ let
\begin{align} \label{global:Whitt}
    W^{\varphi}(g):=\int_{[U_r]}\varphi(ng)\overline{\psi}(n)dn \quad\textrm{and} \quad W'^{\varphi}(g):=\int_{[U_r]}\varphi(ng)\psi(n)dn
\end{align}
be the usual global Whittaker functions.  If $\varphi=\otimes_v\varphi_v$ is a pure tensor, then 
$$
W^{\varphi}(g)=\prod_vW^{\varphi_v}(g) \quad\textrm{and}\quad W'^{\varphi}(g)=\prod_vW'^{\varphi_v}(g)
$$
for suitable $W^{\varphi_v}\in \mathcal{W}(\pi_v,\psi)$ and $W'^{\varphi_v} \in \mathcal{W}(\pi_v,\overline{\psi}).$  Upon renormalizing if necessary, we can  assume that $W^{\varphi_v}(I_r)=W'^{\varphi_v}(I_r)=1$ for all but finitely many $v.$   

\begin{lem} \label{lem:rap:dec}
 If $\Phi$ has cuspidal image, then
\begin{align*}
    \int_{\mathrm{Re}(s)=t}\sum_{\pi}K_{\pi_s(\Phi)}(g,h)\frac{ds}{2\pi i}=\sum_{\gamma \in \GL_{r}(F)}\Phi(g^{-1}\gamma h)
\end{align*}
for any $t\in \RR.$  Moreover, if $\Phi$ is finite under a maximal compact subgroup on the left and right, then for any $N>0,A>0$,
$$
(1+|\mathrm{Im}(s)|)^{-N}\sum_{\pi}\sum_{\varphi}|\mathcal{R}(\Phi)\varphi_s(g)\overline{\varphi}_{-s}(h)|
$$
is bounded by a rapidly decreasing function of $(g,h) \in A_{\GL_r} \backslash [\GL_{r}],$ independent of $s $ provided $-A<\mathrm{Re}(s)<A.$
\end{lem}
\begin{proof}
This is standard.  We refer to \cite[Chapter 16]{Getz:Hahn} for the analogue when $\Phi \in \mathcal{S}(\GL_r(\A_F))$ is replaced by $\Phi \in C^\infty_c(A_{\GL_r} \backslash \GL_r(\A_F)).$
\end{proof}

We let 
$$
\mathcal{S}^{\mathrm{as}}(X_{\ell}(\A_F),\mathcal{H}):=\otimes_{v \nmid \infty}\mathcal{S}^{\mathrm{as}}(X_{\ell}(F_v),\mathcal{H}) \otimes \widehat{\otimes}_{v|\infty}\mathcal{S}^{\mathrm{as}}(X_{\ell}(F_v),\mathcal{H}),
$$
where the restricted tensor product is taken with respect to $b_{\ell,v}.$ Here and below $\widehat{\otimes}$ denotes the  completed projective tensor product. Note that $b_{\ell,v}$ is $\Xi$-rapidly decreasing by \Cref{lem:basic:bound}. We say $f=\otimes_v f_v\in \mathcal{S}^{\mathrm{as}}(X_{\ell}(\A_F),\mathcal{H})$ is $\Xi$-rapidly decreasing if each $f_v$ is.   We write $\norm{\,\cdot\,}:=\prod_{v}\norm{\,\cdot\,}_v$ for any of the various norms that have appeared earlier in the paper.  If $Z$ is any affine scheme of finite type over $F,$ we set
\begin{align*}
\mathcal{S}_{\mathrm{ES}}(Z(\A_F)):=\mathcal{S}_{\mathrm{ES}}(Z(F_\infty))\otimes C_c^\infty(Z(\A_F^\infty)) 
\end{align*}
(see \S \ref{sec:HC} for notation).

\begin{lem}\label{lem:adelicbound}
    Suppose $f=\otimes_v f_v\in \mathcal{S}^{\mathrm{as}}(X_{\ell}(\A_F),\mathcal{H})$ is $\Xi$-rapidly decreasing.  There is a nonnegative $\widetilde{f} \in \mathcal{S}_{\mathrm{ES}}(X_\ell(\A_F))$ such that for any $N\in \RR$
     \begin{align*}
        |f\left(g, g'^tm_\ell\right)|\ll_N \frac{\widetilde{f}\left(g, g'^tm\right)}{\norm{e_r^t g}^{r/2}\norm{g'^tm_\ell}^{(r-2)/2}}\left(\frac{\norm{e_r^t g\wedge e_{r-1}^tg}}{\norm{e_r^tg}}\right)^{-N}.
    \end{align*}
\end{lem}
\begin{proof}
Let $S$ be a finite set of places including $\infty$ such that $f_v=b_{\ell,v}$ for $v\not\in S$. By \Cref{lem:basic:bound} and the definition of $\Xi$-rapidly decreasing function, we have for any $N\in \RR,$ $N'\ge 0$ and $d'\ge 0$ 
    \begin{align*}
        |f\left(g, g'^tm\right)|\ll_{N',d'}     \frac{\one_{X_\ell(\widehat{\OO}_F^S)}\left(g, g'^tm\right)}{\norm{e_r^tg}^{S}} \left(\frac{\norm{e_r^t g\wedge e_{r-1}^tg}^S}{\norm{e_r^tg}^S}\right)^{-N}\prod_{v\in S} \frac{\Xi_{X_\ell^\circ,v}(g,g'^tm)\sigma_{X_\ell^\circ,v}^{-d'}(g,g'^tm)}{\max\left(1,\norm{g}_{\ell,v},\norm{g'^{t}m}_v\right)^{N'}}.
    \end{align*}
    Write $(g,g'^tm_\ell)$ using the coordinates \eqref{coord}. Then by \Cref{cor:xi:bounded}, \Cref{lem:Levi} and \Cref{lem:parab}, there is a $d_0>0$ such that
    \begin{align*}
        \Xi_{X_\ell^\circ,v}(g,g'^tm)\ll\frac{|t_r|_v^{r/2}|t_{r}^{-1}t_{r-2}^{r-2}\det g_1|_v \sigma_v\begin{psmatrix}
            g_1 & y_1\\
            & 1
        \end{psmatrix}}{|a|^{(r-2)/2}_v}^{d_0}.
    \end{align*}
    Therefore, for any $N''\ge 0,$ by choosing $N'$ large and $d'=d_0$, we have
    \begin{align*}
        f\left(g, g'^tm\right)\ll_{N'} \frac{\one_{X_\ell(\widehat{\OO}_F^S)}\left(g, g'^tm\right)|t_r|^{r/2}_S|t_r|^S}{|a|^{(r-2)/2}_S\prod_{v\in S}\max\left(1,\norm{g}_{\ell,v},\norm{g'^{t}m}_v\right)^{N''}}\left(\frac{\norm{e_r^t g\wedge e_{r-1}^tg}^S}{\norm{e_r^tg}^S}\right)^{-N}.
    \end{align*}
    Since $|t_r/a|_v \geq 1$ in the support of $b_{\ell,v}$ as explained in the proof of Lemma \ref{lem:basic:bound}, we deduce the bound in the lemma.
\end{proof}

Let $f \in \mathcal{S}^{\mathrm{as}}(X_{\ell}(\A_F),\mathcal{H})$ be $\Xi$-rapidly decreasing. For $(g,g') \in \GL_{r,r-2}(\A_F),$ let \index{$\Theta_f$}
\begin{align} \label{Thetaf2} \begin{split}
    \Theta_f(g,g'):&=\sum_{(x,m) \in X_\ell^\circ(F)}\mathcal{R}_u(g,g')f(x,m)\\&=\sum_{(x,m) \in X_\ell^\circ(F) }|\det g'|^{3/2}f\left( \begin{psmatrix} g' & \\ & I_2\end{psmatrix}^{-1}xg,g'^{t}m\right). \end{split}
\end{align}
We define $\Theta_{|f|}(g,g')$ analogously.

\begin{prop} \label{prop:mod:growth}
The function
\begin{align*}
[\GL_{r,r-2}]^1 &\lto \CC\\
    (g,g') &\longmapsto \int_{A_{\GG_m} \times A_{\GG_m}}\Theta_{|f|}(ag,a'g')|a'|^{s'}d^\times a' d^\times a 
\end{align*}
is of moderate growth provided $\mathrm{Re}(s') \gg 1.$  Hence the same is true if we replace $\Theta_{|f|}$ by $\Theta_{f}.$
\end{prop}
\begin{proof}
 Without loss of generality we may assume $s' \in \RR.$   By \Cref{lem:adelicbound} 
\begin{align} \label{theta:int}
\int_{A_{\GG_m} \times A_{\GG_m}}\Theta_{|f|}(ag,a'g')|a'|^{s'}d^\times a' d^\times a
\end{align}
is dominated by  
\begin{align*}
\int_{A_{\GG_m} \times A_{\GG_m}} \sum_{(x,m)\in X_\ell^\circ(F)}&\widetilde{f}\left( \begin{psmatrix} a'g' & \\ & I_2\end{psmatrix}^{-1}xag,a'g'^{t}m\right)\frac{|a'|^{s'+3(r-2)/2}|\det g'|^{3/2}d^\times a'd^\times a}{\norm{e_r^txag}^{r/2}\norm{ a'g'^tm}^{(r-2)/2}} 
\end{align*}
for some non-negative $\widetilde{f} \in \mathcal{S}_{\mathrm{ES}}(X_\ell(\A_F)).$

We bound the sum over $\GL_{r-2,2}(F)$-orbits in $X_r^\circ(F)$ by an integral over 
$$
 \left\{(g,h) \in \GL_{r-2,2}(\A_F): |\det g\det h|=1\right\}
$$ 
and $F^\times$-orbits in $\mathcal{M}_{\ell}^\circ(F)$ by an integral over $\A_F^\times$ using \cite[Lemma 3.3]{FinisLapid2011}.  Then there is a nonnegative $\Phi \in \mathcal{S}_{\mathrm{ES}}(X_r(\A_F) \times \mathcal{M}_{\ell}(\A_F))$ such that the integral \eqref{theta:int} is dominated by 
\begin{align*}
    \sum_{(x,m)}\int_{\GL_{r-2,2}(\mathbb{A}_F)\times \mathbb{A}_F^\times}&\Phi
    \left( \begin{psmatrix} g''a'g' & \\ & h\end{psmatrix}^{-1}xg,a'g'^{t}m\right)\frac{|\det g''|^t}{|\det h|^{t(r-2)/2}} \frac{|a'|^{s'+3(r-2)/2}|\det g'|^{3/2}dg''dhd^\times a'}{\norm{e_r^t\begin{psmatrix} g''a'g' & \\ & h\end{psmatrix}^{-1}xg}^{r/2}\norm{ a'g'^tm}^{(r-2)/2}} 
\end{align*}
for any $t \in \RR.$ 
Here the sum is over $(x,m) \in P(F) \backslash \GL_r(F) \times \mathbb{P}(\mathcal{M}_{\ell})(F)$ and $P$ is the standard parabolic subgroup with Levi subgroup $\GL_{r-2,2}.$ Changing variables $g'' \mapsto g''(a'g')^{-1}$ gives 
\begin{align*}
    \sum_{(x,m)}\int_{\GL_{r-2,2}(\mathbb{A}_F)\times \mathbb{A}_F^\times}\Phi\left( \begin{psmatrix} g'' & \\ & h\end{psmatrix}^{-1}xg,a'g'^{t}m\right)\frac{|\det g''|^{t}}{|\det h|^{t(r-2)/2}}\frac{|a'|^{s'+(r-2)(3/2-t)}|\det g'|^{-t+3/2}dg''dhd^\times a'}{\norm{e_r^t\begin{psmatrix} g'' & \\ & h\end{psmatrix}^{-1}xg}^{r/2}\norm{ a'g'^tm}^{(r-2)/2}} .
\end{align*}
This is a product of two Eisenstein series.  When they are absolutely convergent, they are known to be of moderate growth. As a function of $g',$ it is a mirabolic Eisenstein series.  It converges absolutely for $s'-(r-2)t \gg 1.$  By \cite[\S II.1.5]{MW:Spectral:Decomp:ES},  as a function of $g$ it is convergent and of moderate growth if $t \gg 1$ provided that the section
\begin{align*}
g \longmapsto \int_{\GL_{r-2,2}(\A_F)} \norm{e_r^t\begin{psmatrix} g'' & \\ & h\end{psmatrix}^{-1}g}^{-(r/2)}\Phi\left(\begin{psmatrix} g'' & \\ & h\end{psmatrix}^{-1} g,m\right)\frac{|\det g''|^t}{|\det h|^{t(r-2)/2}}dg''dh
\end{align*}
converges  for any $m \in \mathcal{M}_\ell^\circ(\A_F).$

Changing variables $(g'',h) \mapsto (g''^{-1},h^{-1})$ and using \eqref{embed0}, we see that the section above is bounded by 
\begin{align*}
&\int_{\GL_{r-2,2}(\A_F)} f_1\left( (\det g'')(\det h), g''\det h,h\right)\frac{|\det h|^{t(r-2)/2}}{|\det g''|^t}dg''dh\\
&=\int_{\GL_{r-2,2}(\A_F)} f_1\left( (\det g'')(\det h)^{3-r}, g'',h\right)\frac{|\det h|^{3t(r-2)/2}}{|\det g''|^t}dg''dh
\end{align*} 
for some $f_1\in \mathcal{S}(\A_F^\times \times M_{r-2}(\A_F) \times M_{2}(\A_F)).$ 
For any $t',$ we can find $f_2 \in \mathcal{S}(M_{r-2}(\A_F) \times M_{2}(\A_F))$ such that the above is bounded by 
$$
\int_{\GL_{r-2,2}(\A_F)} f_2\left( g'',h\right)|\det g''|^{t'-t}|\det h|^{3t(r-2)/2-(r-3)t'}dg''dh.
$$
This converges if we choose $t$ and $t'$ so that $t'-t \gg 1$ and $3t(r-2)/2-(r-3)t' \gg 1.$ 
\end{proof}

Set \index{$Z(f,\varphi,\varphi'_{s'})$}
\begin{align}
    Z(f,\varphi,\varphi'_{s'}):=\int_{[\GL_{r-2} \times \GL_r]} \Theta_f(g,g')\varphi(g)\varphi'_{s'}(g')dg'.
\end{align}
This converges for $\mathrm{Re}(s') \gg 1$ by Proposition \ref{prop:mod:growth}.  

\begin{lem} \label{sumexpansion}
Under assumptions \eqref{a} and \eqref{b}, for $t' \gg 1$  one has
\begin{align*}
    &\sum_{x \in X^\circ_{\ell}(F)}f\left(x\right)=\sum_{\pi,\pi'}\sum_{(\varphi,\varphi') \in \mathcal{B}_{\pi} \times \mathcal{B}_{\pi'}}\int_{(\mathrm{Re}(s_0),\mathrm{Re}(s'))=(0,t')}\overline{\varphi}(I_r)\overline{\varphi}'(I_{r-2})Z(f,\varphi_{s_0},\varphi'_{s'})\frac{ds'ds_0}{(2\pi i)^2}.
\end{align*}
Here the sum in $\pi$ (resp.~$\pi'$) is over the set of isomorphism classes of cuspidal automorphic representations of $A_{\GL_r} \backslash \GL_r(\A_F)$ (resp.~$A_{\GL_{r-2}} \backslash \GL_{r-2}(\A_F)$).
\end{lem}

\begin{proof}
By \eqref{a}, $f=\mathcal{R}_u(\Phi',\Phi)f_0.$ We have
\begin{align*}
&\sum_{(x,m) \in X_\ell^\circ(F)}\mathcal{R}_u(\Phi',\Phi)f_0\left(x,m\right)\\
&=\sum_{(x,m)\in X_\ell^\circ(F) }\int_{\GL_{r-2}(\A_F) \times \GL_r(\A_F)}f_0\left(\begin{psmatrix} g' & \\ & I_2 \end{psmatrix}^{-1}xg,g'^{t}m \right)\Phi'^\vee(g'^{-1})\Phi^\vee(g^{-1})|\det g'|^{3/2}dg'dg\\
&=\int_{[\GL_{r-2} \times \GL_r]}\Theta_{f_0}(g,g')\sum_{(\gamma',\gamma) \in \GL_{r-2}(F) \times \GL_{r}(F)}\Phi'^\vee(g'^{-1}\gamma')\Phi^\vee(g^{-1}\gamma )dg'dg.
\end{align*}
Expanding the two automorphic kernel functions using Lemma \ref{lem:rap:dec}, this is 
\begin{align*}
&\int_{[\GL_{r,r-2}]}
\Theta_{f_0}(g,g')
\sum_{\pi,\pi'}\int_{(\mathrm{Re}(s_0),\mathrm{Re}(s'))=(0,t')}K_{\pi_{s_0}(\Phi^\vee)}(g,I_r)K_{\pi'_{s'}(\Phi'^\vee)}(g',I_{r-2}) \frac{ds_0 ds'}{(2\pi i)^2}dg'dg.
\end{align*}
Provided $t' \gg 1,$ by Lemma \ref{lem:rap:dec} and Proposition \ref{prop:mod:growth} we can apply Fubini-Tonelli Theorem to see that this is 
\begin{align*}
&\int_{(\mathrm{Re}(s_0),\mathrm{Re}(s'))=(0,t')}\sum_{\pi,\pi'}\int_{[\GL_{r,r-2}] }\Theta_{f_0}(g,g') K_{\pi_{s_0}(\Phi^\vee)}(g,I_r) K_{\pi'_{s'}(\Phi'^{\vee})}(g',I_{r-2})dg'dg\frac{ds_0 ds'}{(2\pi i)^2}.
\end{align*}
 We have
\begin{align*}
   & \int_{[\GL_{r,r-2}]}\Theta_{f_0}(g,g') K_{\pi_{s_0}(\Phi^\vee)}(g,I_r) K_{\pi'_{s'}(\Phi'^{\vee})}(g',I_{r-2})dg'dg\\
    &=\int_{[\GL_{r,r-2}]}\sum_{(x,m) \in X_\ell^\circ(F)}f_0\left(\begin{psmatrix} g' & \\ & I_2 \end{psmatrix}^{-1}xg,g'^{t}m \right)\\& \times \sum_{(\varphi,\varphi') \in \mathcal{B}_{\pi} \times \mathcal{B}_{\pi'}}\overline{\varphi}(I_r)\overline{\varphi}(I_{r-2})\mathcal{R}(\Phi^\vee)\varphi_{s_0}(g)
    \mathcal{R}(\Phi'^\vee)\varphi'_{s'}(g')|\det g'|^{3/2}
   dg'dg\\
   &=\sum_{(\varphi,\varphi') \in \mathcal{B}_{\pi} \times \mathcal{B}_{\pi'}}\overline{\varphi}(I_r)\overline{\varphi}(I_{r-2}) Z(f,\varphi_{s_0},\varphi_{s'}).
\end{align*} 
Here one can again use 
 Lemma \ref{lem:rap:dec} and Proposition \ref{prop:mod:growth} to justify the rearrangement of the integrals.
\end{proof}
\begin{lem} \label{lem:unfold:prep} Assume $\varphi=\otimes_v\varphi_v$ and $\varphi'=\otimes_v\varphi_v'$ are pure tensors and that $\mathrm{Re}(s_0)=0.$ For $\Xi$-rapidly decreasing $f=\otimes_v f_v\in \mathcal{S}^{\mathrm{as}}(X_\ell(\A_F),\mathcal{H}),$ the integral
\begin{align*}
&\int|f|\left(g,g'^{t}m_\ell\right)
\left|\sum_{\alpha \in F^\times}\int_{U_{r-2}(F)\backslash \mathcal{P}_{r-2}(\A_F)}  W^\varphi_{s_0}\left(\begin{psmatrix} pg' & \\ & \begin{psmatrix}
    1 & \\
    & \alpha
\end{psmatrix}h_\ell  \end{psmatrix}g \right) \varphi'_{s'}(pg')\frac{d_r p}{|\det pg'|^{3/2}}\right| |\det g'|dg'dg
\end{align*}
is finite if $\mathrm{Re}(s')\gg 1.$ Furthermore, if  $\mathrm{Re}(s')\gg 1$ and $t \gg 1,$ then
\begin{align*}
&\int f\left(g,g'^{t}m_\ell\right)\sum_{\alpha \in F^\times}\int_{U_{r-2}(F)\backslash \mathcal{P}_{r-2}(\A_F)}  W^\varphi_{s_0}\left(\begin{psmatrix} pg' & \\ & \begin{psmatrix}
    1 & \\
    & \alpha
\end{psmatrix}h_\ell  \end{psmatrix}g \right)\varphi'_{s'}(pg') \frac{|\det g'|d_r pdg'dg}{|\det pg'|^{3/2}} \\
   &= |\det h_\ell|^{2-r}\sum_{\chi}\int_{\mathrm{Re}(s)=t}\prod_{v}Z(f_v,W^{\varphi_v}_{s_0}\otimes W'^{\varphi'_v}_{s'},\chi_s)\frac{c_F ds}{2\pi i}.
\end{align*}
In both assertions, the outer integrals are over $\mathcal{P}_{r-2}(\A_F) \backslash \GL_{r-2}(\A_F) \times N(\A_F) \backslash \GL_r(\A_F).$ Moreover, the sum is over characters $\chi:A_{\GG_m} \backslash [\GG_m] \to \CC^\times,$ and $c_F \in \RR_{>0}$ depends only on $F.$
\end{lem}

\begin{proof}
Since $\varphi'$ is rapidly decreasing, by an adelic analogue of gauge estimates \eqref{gauge} and \eqref{CS} for $W,$ for $\mathrm{Re}(s')\gg 1$
\begin{align*} 
    \sum_{\alpha \in F^\times}\int_{U_{r-2}(F)\backslash \mathcal{P}_{r-2}(\A_F)}|W^\varphi_{s_0}|\left(\begin{psmatrix} pg' & \\ & \begin{psmatrix}
    1 & \\
    & \alpha
\end{psmatrix}h_\ell  \end{psmatrix}g \right)  |\varphi'_{s'}|(pg') \frac{d_r p}{|\det pg'|^{3/2}}<\infty.
\end{align*}
Using the definition \eqref{global:Whitt}, for $\mathrm{Re}(s')$ large we have
\begin{align*}
    &\sum_{\alpha \in F^\times}\int_{U_{r-2}(F)\backslash \mathcal{P}_{r-2}(\A_F)}  W^{\varphi}_{s_0}\left(\begin{psmatrix} pg' & &\\ & \begin{psmatrix}
    1 & \\
    & \alpha
\end{psmatrix}h_\ell \end{psmatrix}g \right) \varphi'_{s'}(pg')\frac{d_r p}{|\det pg'|^{3/2}}\\
    &=\sum_{\alpha\in F^\times}\int_{U_{r-2}(\A_F)\backslash \mathcal{P}_{r-2}(\A_F)}W^{\varphi}_{s_0}\left(\begin{psmatrix} pg' & &\\ &\begin{psmatrix}
    1 & \\
    & \alpha
\end{psmatrix}h_\ell \end{psmatrix}g \right)W'^{\varphi'}_{s'}(p'g')\frac{d_r p}{|\det pg'|^{3/2}}.
\end{align*}
By applying Mellin inversion to the sum over $\alpha,$ if $ \mathrm{Re}(s') \gg 1$ and $t \gg 1$ this is
\begin{align*}
&\sum_{\chi}\int_{\mathrm{Re}(s)=t} \int_{\A_F^\times} \left(\prod_v J(W^{\varphi_v}_{s_0}\otimes W'^{\varphi'_v}_{s'})\right)\left(\begin{psmatrix}
    I_{r-2}& \\
    & \begin{psmatrix}
    1 & \\
    & c^{-1}
\end{psmatrix}h_\ell 
\end{psmatrix}g,g'^tm_\ell\right)\frac{\chi_s(c)}{|c|^{(r-2)/2}} d^\times c \frac{c_F ds}{2\pi i}\\
&=\sum_{\chi}\int_{\mathrm{Re}(s)=t} \left(\prod_{v}Z(W^{\varphi_v}_{s_0}\otimes W'^{\varphi_v}_{s'},\chi_s)\right)(g,g'^tm_\ell)\frac{c_F ds}{2\pi i}.
\end{align*}
By \cite[\S 2]{Blomer_Brumley_Ramanujan_Annals} this is justified by the natural adelic analogue of \Cref{lem:ZWlambda}, which follows from \Cref{lem:ZWlambda} and \Cref{cor:unram:bound}.  We point out that one has to use an integration by parts argument at the infinite places to control the sum over $\chi$ and integral over $\mathrm{Re}(s)=t.$

Thus to prove the lemma, it suffices to show that
\begin{align*}
    \sum_{\chi}\int_{\mathrm{Re}(s)=t} \int|f|\left(g,g'^{t}m_\ell\right)|Z(W_{s_0}\otimes W'_{s'},\chi_s)(g,g'^tm_\ell)||\det g'|dg'dgds
\end{align*}
converges absolutely for $\mathrm{Re}(s') \gg 1$ and $t \gg 1.$ Here the integral is over $\mathcal{P}_{r-2}(\A_F) \backslash \GL_{r-2}(\A_F) \times N(\A_F) \backslash \GL_r(\A_F).$ This follows from \Cref{lem:ZfW2}, Lemma \ref{lem:uniform:ZfW} and the proof of \Cref{unramifiedcomp}.
\end{proof}

\begin{prop} \label{prop:unfolding}
Under the same assumptions as \Cref{lem:unfold:prep}, one has
\begin{align*}
    Z(f,\varphi_{s_0},\varphi'_{s'})=\sum_{\chi}\int_{\mathrm{Re}(s)=t}\prod_{v}Z(f_v,W^{\varphi_v}_{s_0} \otimes W^{\varphi'_v}_{s'},\chi_s)\frac{c_F ds}{2\pi i},
\end{align*}
where the sum is over characters $\chi:A_{\GG_m} \backslash [\GG_m] \to \CC^\times$ and $c_F \in \RR_{>0}$ depends only on $F.$
\end{prop}

\begin{proof}
Every element of $X_\ell^\circ(F)$ is in the $\GL_{r,r-2}(F)$-orbit of
$
\left(\begin{psmatrix}I_{r-2} & \\ & h_{\ell}^{-1}\end{psmatrix},m_\ell\right).
$
The set of $R$-points of the stabilizer of this element  in $\GL_{r,r-2}$
is
\begin{align*}
    \left\{ \left(\begin{psmatrix} p & z\\ & I_2 \end{psmatrix},p\right):(p,z) \in \mathcal{P}_{r-2}(R) \times M_{r-2,2}(R)\right\}.
\end{align*}
Thus $Z(f,\varphi_{s_0},\varphi'_{s'})$ unfolds to 
\begin{align} \label{before:Whitt} \begin{split}
    &\int_{N(\A_F) \backslash \GL_r(\A_F)}\bigg(\int_{\mathcal{P}_{r-2}(F) \backslash \GL_{r-2}(\A_F)} 
 f\left(\begin{psmatrix} g' & \\ & h_{\ell} \end{psmatrix}^{-1}g,g'^{t}m_\ell\right) \\ &\hspace{0.4in}\times  \int_{[M_{r-2,2}]}\overline{\psi}\left(\langle (e_{r-2},0),z \rangle \right)\varphi_{s_0}\left(\begin{psmatrix} I_{r-2} & z\\ & I_2 \end{psmatrix}g\right)dz \varphi'_{s'}(g')|\det g'|^{3/2}dg'\bigg)dg.
 \end{split}
\end{align}
This unfolding may be justified using Proposition \ref{prop:mod:growth} and the fact that cusp forms are rapidly decreasing.  

The Whittaker expansion of $\varphi$ is absolutely convergent, uniformly on compact subsets \cite[Theorem 1.1]{CogdellPCMI}.  Applying this expansion we obtain that
\begin{align*}
&\int_{[M_{r-2,2}]}\overline{\psi}\left(\langle (e_{r-2},0),z \rangle \right)\varphi_{s_0}\left(\begin{psmatrix} I_{r-2} & z\\ & I_2 \end{psmatrix}g\right)dz\\
   &= \int_{[M_{r-2,2}]} \overline{\psi}\left(\langle (e_{r-2},0),z \rangle \right)\sum_{\gamma \in U_{r-1}(F) \backslash \GL_{r-1}(F)}W^{\varphi}_{s_0}\left(\begin{psmatrix} \gamma & \\ & 1 \end{psmatrix}\begin{psmatrix} I_{r-2} & z\\ & I_2 \end{psmatrix}g\right)dn\\
   &= \sum_{\alpha \in F^\times}\sum_{\gamma \in U_{r-2}(F) \backslash \mathcal{P}_{r-2}(F)}W^\varphi_{s_0}\left(\begin{psmatrix} \alpha \gamma & & \\ & \alpha & \\ && 1 \end{psmatrix}g \right)\\
   &= \sum_{\alpha \in F^\times}\sum_{\gamma \in U_{r-2}(F) \backslash \mathcal{P}_{r-2}(F)}W^\varphi_{s_0}\left(\begin{psmatrix} \gamma & & \\ & 1 & \\ && \alpha \end{psmatrix}g \right).
\end{align*}
Substituting this into \eqref{before:Whitt} and unfolding, we obtain
\begin{align*}
    & \int_{N(\A_F) \backslash \GL_r(\A_F)\times \mathcal{P}_{r-2}(\A_F) \backslash \GL_{r-2}(\A_F)} \bigg( \int_{[\mathcal{P}_{r-2}]} 
 f\left(\begin{psmatrix} p'g' & \\ & h_{\ell} \end{psmatrix}^{-1}g,g'^{t}m_\ell\right) \\&\hspace{0.3in} \times  \sum_{\alpha \in F^\times}\sum_{\gamma \in U_{r-2}(F) \backslash \mathcal{P}_{r-2}(F)} W^\varphi_{s_0}\left(\begin{psmatrix} \gamma & & \\ & 1 & \\ && \alpha \end{psmatrix}g \right)\varphi'_{s'}(p'g')|\det p'g'|^{1/2}d_rp'|\det g'|dg'\bigg)dg\\
 &= |\det h_\ell|^{r-2}\int_{N(\A_F) \backslash \GL_r(\A_F)\times \mathcal{P}_{r-2}(\A_F) \backslash \GL_{r-2}(\A_F)}  f\left(g,g'^{t}m_\ell\right) \\& \hspace{0.3in}\times \bigg(\sum_{\alpha \in F^\times} \int_{U_{r-2}(F)\backslash \mathcal{P}_{r-2}(\A_F)}W^\varphi_{s_0}\left(\begin{psmatrix} p'g' &  \\ & \begin{psmatrix}
     1 & \\ & \alpha 
 \end{psmatrix}h_\ell\end{psmatrix}g \right)\varphi'_{s'}(p'g')\frac{d_rp'}{|\det p'g'|^{3/2}}|\det g'|dg'\bigg)dg.
\end{align*}
Here the unfolding is justified by the bound in \Cref{lem:unfold:prep}. The identity follows from the identity in 
\Cref{lem:unfold:prep}.
\end{proof}

\subsection{The summation formula}

\begin{thm}
    For $f\in \mathcal{S}^{\mathrm{as}}(X_{\ell}(\A_F),\mathcal{H})$ satisfying assumptions \eqref{a} and \eqref{b}, we have
    \[
    \sum_{x \in X^\circ_{\ell}(F)}f\left(x\right) = \sum_{x \in X^\circ_{\ell}(F)}\mathcal{F}(f)\left(x\right).
    \]
\end{thm}

\begin{proof} 
    Combining \Cref{sumexpansion} and \Cref{prop:unfolding}, for $t\gg 1$ and  $t'\gg 1$ we have 
    \begin{align*}
        \sum_{x \in X_{\ell}^{\circ}(F)}\mathcal{F}(f)\left(x\right)         &=\sum_{\pi,\pi'}\sum_{(\varphi,\varphi') \in \mathcal{B}_{\pi} \times \mathcal{B}_{\pi'}}\int_{\mathrm{Re}(s_0,s')=(0,t')}\overline{\varphi}(I_r)\overline{\varphi}'(I_{r-2})Z(\mathcal{F}(f),\varphi_{s_0},\varphi'_{s'}) \frac{ ds'ds_0}{(2\pi i)^2}\\
        &=\sum_{\pi,\pi'}\sum_{(\varphi,\varphi') \in \mathcal{B}_{\pi} \times \mathcal{B}_{\pi'}}\int_{\mathrm{Re}(s_0,t')=(0,t')}\overline{\varphi}(I_r)\overline{\varphi}'(I_{r-2})\\
        &\hspace{0.3in} \times \sum_{\chi}\int_{\mathrm{Re}(s)=t}\prod_{v}Z\big(\mathcal{F}(f_v),W^{\varphi_v}_{s_0}\otimes W^{\varphi'_v}_{s'},\chi_{s}\big) \frac{c_F ds ds'ds_0}{(2\pi i)^3}. 
    \end{align*}
By \Cref{unramifiedcomp}, the above is 
\begin{align} \label{before:fix:pi} \begin{split}
\sum_{\pi,\pi'}\sum_{(\varphi,\varphi') \in \mathcal{B}_{\pi} \times \mathcal{B}_{\pi'}}&\int_{\mathrm{Re}(s_0,s')=(0,t')}\overline{\varphi}(I_r)\overline{\varphi}'(I_{r-2})\\
   \times & \sum_{\chi}\int_{\mathrm{Re}(s)=t}\prod_{v\in S}Z\big(\mathcal{F}(f_v),W^{\varphi_v}_{s_0} \otimes W^{\varphi'_v}_{s'},\chi_{s}\big)\\
       \times &  L^S(\tfrac{1}{2}+s_0+s',\pi\times \pi') L^S(\tfrac{1}{2}-s_0+s,\pi^{\vee}\otimes \chi )\frac{c_F ds ds'ds_0}{(2\pi i)^3}, \end{split}
\end{align}
where $S$ is a finite set of places of $F$ including the infinite places.  By Theorem \ref{thm:FT} and \Cref{zetafuncequa}, 
\begin{align*}
&L^S(\tfrac{1}{2}+s_0+s',\pi\times \pi') L^S(\tfrac{1}{2}-s_0+s,\pi^{\vee}\otimes \chi)\prod_{v\in S}Z\big(\mathcal{F}(f_v),(W^{\varphi_v}_{s_0} \otimes W^{\varphi'_v}_{s'}),\chi_{s}\big)\\
&=L^S(\tfrac{1}{2}+s_0+s',\pi \times \pi') 
        L^S(\tfrac{1}{2}-s_0+s,\pi^{\vee}\otimes \chi )\\
        &\quad\times\prod_{v\in S}Z\big(f_v,\widetilde{W}^{\varphi_v}_{-s_0} \otimes \widetilde{W}^{\varphi'_v}_{-s'},(\chi_{s})^{-1}\big) \gamma(\tfrac{1}{2}-s_0-s',\pi^{\vee}_v\times \pi'^{\vee}_v,\psi_v)\gamma(\tfrac{1}{2}+s_0-s,\pi_{v}\otimes\chi^{-1}_v,\psi_v).
\end{align*}
Using the functional equation of global Rankin-Selberg $L$-functions \cite[Theorem 4.1]{CogdellPCMI},  this is 
\begin{align*}
L^S(\tfrac{1}{2}-s_0-s',\pi^\vee \times \pi'^\vee) 
        L^S(\tfrac{1}{2}+s_0-s,\pi\otimes \chi^{-1} )\prod_{v\in S}Z\big(f_v,\widetilde{W}^{\varphi_v}_{-s_0} \otimes \widetilde{W}^{\varphi'_v}_{-s'},(\chi_{s})^{-1}\big).
\end{align*}
Substituting into \eqref{before:fix:pi}, we see that $\sum_{x\in X_\ell^\circ(F)} \mathcal{F}(f)(x)$ is
\begin{align*}
\sum_{\pi,\pi'}&\sum_{(\varphi,\varphi') \in \mathcal{B}_{\pi} \times \mathcal{B}_{\pi'}}\int_{\mathrm{Re}(s_0,s')=(0,t')}\overline{\varphi}(I_r)\overline{\varphi}'(I_{r-2}) \sum_{\chi}\int_{\mathrm{Re}(s)=t}\Bigg(\prod_{v\in S}Z\big(f_v,\widetilde{W}^{\varphi_v}_{-s_0} \otimes \widetilde{W}^{\varphi'_v}_{-s'},(\chi_{s})^{-1}\big) \Bigg)\\
        & \times L^S(\tfrac{1}{2}-s_0-s',\pi^\vee \times \pi'^\vee) 
        L^S(\tfrac{1}{2}+s_0-s,\pi\otimes \chi^{-1} )\frac{c_F ds ds'ds_0}{(2\pi i)^3}.
\end{align*}
We now shift the contour to $\mathrm{Re}(s',s)=(-t',-t)$ and then change variables to see that the above is  
\begin{align*}
\sum_{\pi,\pi'}&\sum_{(\varphi,\varphi') \in \mathcal{B}_{\pi} \times \mathcal{B}_{\pi'}}\int_{\mathrm{Re}(s_0,s')=(0,t')}\overline{\varphi}(I_r)\overline{\varphi}'(I_{r-2}) \sum_{\chi}\int_{\mathrm{Re}(s)=t}\Bigg(\prod_{v\in S}Z\big(f_v,\widetilde{W}^{\varphi_v}_{s_0} \otimes \widetilde{W}^{\varphi'_v}_{s'},(\chi^{-1})_s)\big) \Bigg)\\
        & \times L^S(\tfrac{1}{2}+s_0+s',\pi^\vee \times \pi'^\vee) 
        L^S(\tfrac{1}{2}-s_0+s,\pi\otimes \chi^{-1} )\frac{c_F ds ds'ds_0}{(2\pi i)^3}.
\end{align*}
The contour shift is permissible by the definition of $\mathcal{S}^{\mathrm{as}}(X_\ell(\A_F),\mathcal{H})$ and the fact that the partial $L$-function $L^S(\tfrac{1}{2}+s_0+s',\pi^\vee \times \pi'^\vee) 
        L^S(\tfrac{1}{2}-s_0+s,\pi\otimes \chi^{-1} )$ is holomorphic on the plane and admits a standard (pre)convexity bound \cite[\S 1]{BrumleyNarrow}.
Arguing as above, but in reverse, this is 
\begin{align*}
\sum_{\pi,\pi'}\sum_{(\varphi,\varphi') \in \mathcal{B}_{\pi} \times \mathcal{B}_{\pi'}}\int_{\mathrm{Re}(s_0,s')=(0,t')}\overline{\varphi}(I_r)\overline{\varphi}'(I_{r-2}) Z(f,\varphi^\vee_{s_0},\varphi'^\vee_{s'})\frac{ds'ds_0}{(2\pi i)^2}.
\end{align*}
Changing variables $(\pi,\pi',\varphi,\varphi') \mapsto (\pi^\vee,\pi'^\vee,\varphi^\vee,\varphi'^\vee)$ and using the fact that $\varphi^\vee(I_r)=\varphi(I_r),$ $\varphi'^\vee(I_{r-2})=\varphi'(I_{r-2}),$ we have shown that 
\begin{align*}
\sum_{x \in X_\ell^\circ(F)}\mathcal{F}(f)(x)
&=\sum_{\pi,\pi'}\sum_{(\varphi,\varphi') \in \mathcal{B}_{\pi} \times \mathcal{B}_{\pi'}}\int_{\mathrm{Re}(s_0,s')=(t_0,t')}\overline{\varphi}(I_r)\overline{\varphi}'(I_{r-2}) Z(f,\varphi_{s_0},\varphi'_{s'})\frac{ds'ds_0}{(2\pi i)^2}.
\end{align*}
This is 
$    \sum_{x \in X^\circ_{\ell}(F)}f\left(x\right) $ by  \Cref{sumexpansion}.
\end{proof}

\bibliography{refs}{}

@article {Blomer_Brumley_Ramanujan_Annals,
    AUTHOR = {Blomer, Valentin and Brumley, Farrell},
     TITLE = {On the {R}amanujan conjecture over number fields},
   JOURNAL = {Ann. of Math. (2)},
  FJOURNAL = {Annals of Mathematics. Second Series},
    VOLUME = {174},
      YEAR = {2011},
    NUMBER = {1},
     PAGES = {581--605},
      ISSN = {0003-486X},
     CODEN = {ANMAAH},
   MRCLASS = {11F70 (22E55)},
  MRNUMBER = {2811610},
MRREVIEWER = {Luis Alberto Lomel{\'{\i}}},
       DOI = {10.4007/annals.2011.174.1.18},
       URL = {http://dx.doi.org/10.4007/annals.2011.174.1.18},
}

@article {BK-lifting,
    AUTHOR = {Braverman, A. and Kazhdan, D.},
     TITLE = {{$\gamma$}-functions of representations and lifting},
      NOTE = {With an appendix by V. Vologodsky,
              GAFA 2000 (Tel Aviv, 1999)},
   JOURNAL = {Geom. Funct. Anal.},
  FJOURNAL = {Geometric and Functional Analysis},
      YEAR = {2000},
    NUMBER = {Special Volume, Part I},
     PAGES = {237--278},
      ISSN = {1016-443X},
     CODEN = {GFANFB},
   MRCLASS = {11F70 (11R39 11S37 11S40 22E55)},
  MRNUMBER = {1826255 (2002g:11064)},
MRREVIEWER = {Volker J. Heiermann},
       DOI = {10.1007/978-3-0346-0422-2_9},
       URL = {http://dx.doi.org/10.1007/978-3-0346-0422-2_9},
}

@article {BK:normalized,
    AUTHOR = {Braverman, Alexander and Kazhdan, David},
     TITLE = {Normalized intertwining operators and nilpotent elements in
              the {L}anglands dual group},
      NOTE = {Dedicated to Yuri I. Manin on the occasion of his 65th
              birthday},
   JOURNAL = {Mosc. Math. J.},
  FJOURNAL = {Moscow Mathematical Journal},
    VOLUME = {2},
      YEAR = {2002},
    NUMBER = {3},
     PAGES = {533--553},
      ISSN = {1609-3321},
   MRCLASS = {22E50 (11F70 22E55)},
  MRNUMBER = {1988971},
MRREVIEWER = {Goran Mui\"A},
}

@article {BrumleyNarrow,
    AUTHOR = {Brumley, Farrell},
     TITLE = {Effective multiplicity one on {${\rm GL}\sb N$} and narrow
              zero-free regions for {R}ankin-{S}elberg {$L$}-functions},
   JOURNAL = {Amer. J. Math.},
  FJOURNAL = {American Journal of Mathematics},
    VOLUME = {128},
      YEAR = {2006},
    NUMBER = {6},
     PAGES = {1455--1474},
      ISSN = {0002-9327},
     CODEN = {AJMAAN},
   MRCLASS = {11F70 (11F67 11M20 11M36 22E55)},
  MRNUMBER = {2275908 (2007h:11062)},
MRREVIEWER = {Daniel Bump},
       URL =
              {http://muse.jhu.edu/journals/american_journal_of_mathematics/v128/128.6brumley.pdf},
}

@article {CasselmanShalika,
    AUTHOR = {Casselman, W. and Shalika, J.},
     TITLE = {The unramified principal series of {$p$}-adic groups. {II}.
              {T}he {W}hittaker function},
   JOURNAL = {Compositio Math.},
  FJOURNAL = {Compositio Mathematica},
    VOLUME = {41},
      YEAR = {1980},
    NUMBER = {2},
     PAGES = {207--231},
      ISSN = {0010-437X},
     CODEN = {CMPMAF},
   MRCLASS = {22E50},
  MRNUMBER = {581582 (83i:22027)},
MRREVIEWER = {J. Szmidt},
       URL = {http://www.numdam.org/item?id=CM_1980__41_2_207_0},
}

@incollection {CogdellPCMI,
    AUTHOR = {Cogdell, James W.},
     TITLE = {{$L$}-functions and converse theorems for {${\rm GL}\sb n$}},
 BOOKTITLE = {Automorphic forms and applications},
    SERIES = {IAS/Park City Math. Ser.},
    VOLUME = {12},
     PAGES = {97--177},
 PUBLISHER = {Amer. Math. Soc., Providence, RI},
      YEAR = {2007},
   MRCLASS = {11F70 (22E50 22E55)},
  MRNUMBER = {2331345 (2008e:11060)},
MRREVIEWER = {Ravi Raghunathan},
}

@preamble{
   "\def\polhk#1{\setbox0=\hbox{#1}{\ooalign{\hidewidth
    \lower1.5ex\hbox{`}\hidewidth\crcr\unhbox0}}} "
}

@article {Elazar:Shaviv,
    AUTHOR = {Elazar, Boaz and Shaviv, Ary},
     TITLE = {Schwartz functions on real algebraic varieties},
   JOURNAL = {Canad. J. Math.},
  FJOURNAL = {Canadian Journal of Mathematics. Journal Canadien de
              Math\'{e}matiques},
    VOLUME = {70},
      YEAR = {2018},
    NUMBER = {5},
     PAGES = {1008--1037},
      ISSN = {0008-414X},
   MRCLASS = {14P05 (14P20 22E45 46A11 46F05)},
  MRNUMBER = {3831913},
MRREVIEWER = {Zbigniew Szafraniec},
       DOI = {10.4153/CJM-2017-042-6},
       URL = {https://doi.org/10.4153/CJM-2017-042-6},
}

@article {FinisLapid2011,
    AUTHOR = {Finis, Tobias and Lapid, Erez},
     TITLE = {On the continuity of {A}rthur's trace formula: the semisimple
              terms},
   JOURNAL = {Compos. Math.},
  FJOURNAL = {Compositio Mathematica},
    VOLUME = {147},
      YEAR = {2011},
    NUMBER = {3},
     PAGES = {784--802},
      ISSN = {0010-437X},
   MRCLASS = {11F72},
  MRNUMBER = {2801400},
MRREVIEWER = {Wen-Wei Li},
       DOI = {10.1112/S0010437X11004891},
       URL = {http://dx.doi.org/10.1112/S0010437X11004891},
}

@book {Getz:Hahn,
    AUTHOR = {Getz, Jayce R. and Hahn, Heekyoung},
     TITLE = {An introduction to automorphic representations---with a view
              toward trace formulae},
    SERIES = {Graduate Texts in Mathematics},
    VOLUME = {300},
 PUBLISHER = {Springer, Cham},
      YEAR = {[2024] \copyright 2024},
     PAGES = {xviii+609},
      ISBN = {978-3-031-41151-9; 978-3-031-41153-3},
   MRCLASS = {11Fxx (11-02 14G35 14Lxx 22-02 22Exx)},
  MRNUMBER = {4738301},
MRREVIEWER = {F\'{e}licien\ Comtat},
       DOI = {10.1007/978-3-031-41153-3},
       URL = {https://doi.org/10.1007/978-3-031-41153-3},
}

@book {GodementJacquetBook,
    AUTHOR = {Godement, Roger and Jacquet, Herv{\'e}},
     TITLE = {Zeta functions of simple algebras},
    SERIES = {Lecture Notes in Mathematics, Vol. 260},
 PUBLISHER = {Springer-Verlag, Berlin-New York},
      YEAR = {1972},
     PAGES = {ix+188},
   MRCLASS = {12A80 (12A65 12A70 12B35 22E55)},
  MRNUMBER = {0342495},
MRREVIEWER = {L. Corwin},
}

@book {Gortz_Wedhorn,
    AUTHOR = {G{\"o}rtz, U. and Wedhorn, T.},
     TITLE = {Algebraic geometry {I}},
    SERIES = {Advanced Lectures in Mathematics},
      NOTE = {Schemes with examples and exercises},
 PUBLISHER = {Vieweg + Teubner, Wiesbaden},
      YEAR = {2010},
     PAGES = {viii+615},
      ISBN = {978-3-8348-0676-5},
   MRCLASS = {14-01},
  MRNUMBER = {2675155 (2011f:14001)},
MRREVIEWER = {C{\'{\i}}cero Carvalho},
       DOI = {10.1007/978-3-8348-9722-0},
       URL = {http://dx.doi.org/10.1007/978-3-8348-9722-0},
}

@incollection {JacquetPerfectRS,
    AUTHOR = {Jacquet, Herv{\'e}},
     TITLE = {Archimedean {R}ankin-{S}elberg integrals},
 BOOKTITLE = {Automorphic forms and {$L$}-functions {II}. {L}ocal aspects},
    SERIES = {Contemp. Math.},
    VOLUME = {489},
     PAGES = {57--172},
 PUBLISHER = {Amer. Math. Soc., Providence, RI},
      YEAR = {2009},
   MRCLASS = {11F70 (11F66 22E46)},
  MRNUMBER = {2533003 (2011a:11103)},
MRREVIEWER = {Wee Teck Gan},
       DOI = {10.1090/conm/489/09547},
       URL = {http://dx.doi.org/10.1090/conm/489/09547},
}

@article {JPSSGL3I,
    AUTHOR = {Jacquet, Herv{\'e} and Piatetski-Shapiro, Ilja Iosifovitch and
              Shalika, Joseph},
     TITLE = {Automorphic forms on {${\rm GL}(3)$}. {I}},
   JOURNAL = {Ann. of Math. (2)},
  FJOURNAL = {Annals of Mathematics. Second Series},
    VOLUME = {109},
      YEAR = {1979},
    NUMBER = {1},
     PAGES = {169--212},
      ISSN = {0003-486X},
     CODEN = {ANMAAH},
   MRCLASS = {10D40 (22E50 22E55)},
  MRNUMBER = {519356},
MRREVIEWER = {Stephen Gelbart},
       DOI = {10.2307/1971270},
       URL = {http://dx.doi.org/10.2307/1971270},
}

@article {JPSSGL3II,
    AUTHOR = {Jacquet, Herv{\'e} and Piatetski-Shapiro, Ilja Iosifovitch and
              Shalika, Joseph},
     TITLE = {Automorphic forms on {${\rm GL}(3)$}. {II}},
   JOURNAL = {Ann. of Math. (2)},
  FJOURNAL = {Annals of Mathematics. Second Series},
    VOLUME = {109},
      YEAR = {1979},
    NUMBER = {2},
     PAGES = {213--258},
      ISSN = {0003-486X},
   MRCLASS = {10D40 (22E50 22E55)},
  MRNUMBER = {528964},
MRREVIEWER = {Stephen Gelbart},
       DOI = {10.2307/1971112},
       URL = {http://dx.doi.org/10.2307/1971112},
}

@article {JacquetShalikaEPI,
    AUTHOR = {Jacquet, H. and Shalika, J. A.},
     TITLE = {On {E}uler products and the classification of automorphic
              representations. {I}},
   JOURNAL = {Amer. J. Math.},
  FJOURNAL = {American Journal of Mathematics},
    VOLUME = {103},
      YEAR = {1981},
    NUMBER = {3},
     PAGES = {499--558},
      ISSN = {0002-9327},
     CODEN = {AJMAAN},
   MRCLASS = {10D40 (12A67 22E55)},
  MRNUMBER = {618323},
MRREVIEWER = {Freydoon Shahidi},
       DOI = {10.2307/2374103},
       URL = {http://dx.doi.org/10.2307/2374103},
}

@article {JPSS:Conv,
    AUTHOR = {Jacquet, H. and Piatetskii-Shapiro, I. I. and Shalika, J. A.},
     TITLE = {Rankin-{S}elberg convolutions},
   JOURNAL = {Amer. J. Math.},
  FJOURNAL = {American Journal of Mathematics},
    VOLUME = {105},
      YEAR = {1983},
    NUMBER = {2},
     PAGES = {367--464},
      ISSN = {0002-9327},
   MRCLASS = {11F67 (11F70 11R39 22E55)},
  MRNUMBER = {701565},
MRREVIEWER = {Freydoon Shahidi},
       DOI = {10.2307/2374264},
       URL = {https://doi.org/10.2307/2374264},
}

@article {JacquetShalika:Whittaker,
    AUTHOR = {Jacquet, Herv\'e and Shalika, Joseph},
     TITLE = {The {W}hittaker models of induced representations},
   JOURNAL = {Pacific J. Math.},
  FJOURNAL = {Pacific Journal of Mathematics},
    VOLUME = {109},
      YEAR = {1983},
    NUMBER = {1},
     PAGES = {107--120},
      ISSN = {0030-8730},
   MRCLASS = {22E50},
  MRNUMBER = {716292},
MRREVIEWER = {Stephen Gelbart},
       URL = {http://projecteuclid.org/euclid.pjm/1102720206},
}

@article {Jacquet:quasi-split,
    AUTHOR = {Jacquet, Herv\'{e}},
     TITLE = {Distinction by the quasi-split unitary group},
   JOURNAL = {Israel J. Math.},
  FJOURNAL = {Israel Journal of Mathematics},
    VOLUME = {178},
      YEAR = {2010},
     PAGES = {269--324},
      ISSN = {0021-2172},
   MRCLASS = {11F70 (11F55 22E50 22E55)},
  MRNUMBER = {2733072},
MRREVIEWER = {Solomon Friedberg},
       DOI = {10.1007/s11856-010-0066-1},
       URL = {https://doi.org/10.1007/s11856-010-0066-1},
}

@article {Kemarsky,
    AUTHOR = {Kemarsky, Alexander},
     TITLE = {A note on the {K}irillov model for representations of {${\rm
              GL}_n(\Bbb{C})$}},
   JOURNAL = {C. R. Math. Acad. Sci. Paris},
  FJOURNAL = {Comptes Rendus Math\'{e}matique. Acad\'{e}mie des Sciences. Paris},
    VOLUME = {353},
      YEAR = {2015},
    NUMBER = {7},
     PAGES = {579--582},
      ISSN = {1631-073X},
   MRCLASS = {22E46},
  MRNUMBER = {3352025},
       DOI = {10.1016/j.crma.2015.04.002},
       URL = {https://doi.org/10.1016/j.crma.2015.04.002},
}

@book {KnappSS,
    AUTHOR = {Knapp, Anthony W.},
     TITLE = {Representation theory of semisimple groups},
    SERIES = {Princeton Landmarks in Mathematics},
      NOTE = {An overview based on examples,
              Reprint of the 1986 original},
 PUBLISHER = {Princeton University Press, Princeton, NJ},
      YEAR = {2001},
     PAGES = {xx+773},
      ISBN = {0-691-09089-0},
   MRCLASS = {22E46 (22-01 22E30)},
  MRNUMBER = {1880691},
}

@incollection {Kottwitz:Clay,
    AUTHOR = {Kottwitz, Robert E.},
     TITLE = {Harmonic analysis on reductive {$p$}-adic groups and {L}ie
              algebras},
 BOOKTITLE = {Harmonic analysis, the trace formula, and {S}himura varieties},
    SERIES = {Clay Math. Proc.},
    VOLUME = {4},
     PAGES = {393--522},
 PUBLISHER = {Amer. Math. Soc., Providence, RI},
      YEAR = {2005},
   MRCLASS = {22E35 (17B99)},
  MRNUMBER = {2192014},
MRREVIEWER = {David A. Renard},
}

@article {LafforgueJJM,
    AUTHOR = {Lafforgue, Laurent},
     TITLE = {Noyaux du transfert automorphe de {L}anglands et formules de
              {P}oisson non lin\'eaires},
   JOURNAL = {Jpn. J. Math.},
  FJOURNAL = {Japanese Journal of Mathematics},
    VOLUME = {9},
      YEAR = {2014},
    NUMBER = {1},
     PAGES = {1--68},
      ISSN = {0289-2316},
   MRCLASS = {11F66 (11F03 11F70)},
  MRNUMBER = {3173438},
MRREVIEWER = {Ravi Raghunathan},
       DOI = {10.1007/s11537-014-1274-y},
       URL = {http://dx.doi.org/10.1007/s11537-014-1274-y},
}

@book {MW:Spectral:Decomp:ES,
    AUTHOR = {M{\oe}glin, C. and Waldspurger, J.-L.},
     TITLE = {Spectral decomposition and {E}isenstein series},
    SERIES = {Cambridge Tracts in Mathematics},
    VOLUME = {113},
      NOTE = {Une paraphrase de l'{\'E}criture [A paraphrase of Scripture]},
 PUBLISHER = {Cambridge University Press, Cambridge},
      YEAR = {1995},
     PAGES = {xxviii+338},
      ISBN = {0-521-41893-3},
   MRCLASS = {11F70 (22E55)},
  MRNUMBER = {1361168},
MRREVIEWER = {Joe Repka},
       DOI = {10.1017/CBO9780511470905},
       URL = {http://dx.doi.org/10.1017/CBO9780511470905},
}

@incollection {NgoSums,
    AUTHOR = {Ng{\^o}, Bao Ch{\^a}u},
     TITLE = {On a certain sum of automorphic {$L$}-functions},
 BOOKTITLE = {Automorphic forms and related geometry: assessing the legacy
              of {I}. {I}. {P}iatetski-{S}hapiro},
    SERIES = {Contemp. Math.},
    VOLUME = {614},
     PAGES = {337--343},
 PUBLISHER = {Amer. Math. Soc., Providence, RI},
      YEAR = {2014},
   MRCLASS = {11F70},
  MRNUMBER = {3220933},
MRREVIEWER = {Fan Gao},
       DOI = {10.1090/conm/614/12270},
       URL = {http://dx.doi.org/10.1090/conm/614/12270},
}

@Book{Renard,
 Author = {David {Renard}},
 Title = {{Repr\'esentations des groupes r\'eductifs \(p\)-adiques}},
 FJournal = {{Cours Sp\'ecialis\'es (Paris)}},
 Journal = {{Cours Sp\'ec. (Paris)}},
 ISSN = {1284-6090},
 Volume = {17},
 ISBN = {978-2-85629-278-5/hbk},
 Pages = {vi + 332},
 Year = {2010},
 Publisher = {Paris: Soci\'et\'e Math\'ematique de France},
 Language = {French},
 MSC2010 = {22E50 22-02 11F70},
 Zbl = {1186.22020}
}

@article {SakellaridisSph,
    AUTHOR = {Sakellaridis, Yiannis},
     TITLE = {Spherical varieties and integral representations of
              {$L$}-functions},
   JOURNAL = {Algebra Number Theory},
  FJOURNAL = {Algebra \& Number Theory},
    VOLUME = {6},
      YEAR = {2012},
    NUMBER = {4},
     PAGES = {611--667},
      ISSN = {1937-0652},
   MRCLASS = {11F67 (11F66 11F70 14M27 22E55)},
  MRNUMBER = {2966713},
MRREVIEWER = {Farrell Brumley},
       DOI = {10.2140/ant.2012.6.611},
       URL = {http://dx.doi.org/10.2140/ant.2012.6.611},
}

@article {SV,
    AUTHOR = {Sakellaridis, Yiannis and Venkatesh, Akshay},
     TITLE = {Periods and harmonic analysis on spherical varieties},
   JOURNAL = {Ast\'{e}risque},
  FJOURNAL = {Ast\'{e}risque},
    NUMBER = {396},
      YEAR = {2017},
     PAGES = {viii+360},
      ISSN = {0303-1179,2492-5926},
      ISBN = {978-2-85629-871-8},
   MRCLASS = {22E50 (11F67)},
  MRNUMBER = {3764130},
MRREVIEWER = {Arnab\ Mitra},
}

@article {ShintaniWhittaker,
    AUTHOR = {Shintani, Takuro},
     TITLE = {On an explicit formula for class-{$1$} ``{W}hittaker
              functions'' on {$\mathrm{GL}\sb{n}$} over {$P$}-adic fields},
   JOURNAL = {Proc. Japan Acad.},
  FJOURNAL = {Proceedings of the Japan Academy},
    VOLUME = {52},
      YEAR = {1976},
    NUMBER = {4},
     PAGES = {180--182},
      ISSN = {0021-4280},
   MRCLASS = {22E50 (10D20)},
  MRNUMBER = {0407208 (53 \#10991)},
MRREVIEWER = {Stephen Gelbart},
}

@incollection {Tate_NT,
    AUTHOR = {Tate, J.},
     TITLE = {Number theoretic background},
 BOOKTITLE = {Automorphic forms, representations and {$L$}-functions
              ({P}roc. {S}ympos. {P}ure {M}ath., {O}regon {S}tate {U}niv.,
              {C}orvallis, {O}re., 1977), {P}art 2},
    SERIES = {Proc. Sympos. Pure Math., XXXIII},
     PAGES = {3--26},
 PUBLISHER = {Amer. Math. Soc., Providence, R.I.},
      YEAR = {1979},
   MRCLASS = {12A67},
  MRNUMBER = {546607 (80m:12009)},
MRREVIEWER = {A. I. Vinogradov},
}

@book {Wallach:RGII,
    AUTHOR = {Wallach, Nolan R.},
     TITLE = {Real reductive groups. {II}},
    SERIES = {Pure and Applied Mathematics},
    VOLUME = {132},
 PUBLISHER = {Academic Press, Inc., Boston, MA},
      YEAR = {1992},
     PAGES = {xiv+454},
      ISBN = {0-12-732961-7},
   MRCLASS = {22E46 (22D25 22E30 46L99)},
  MRNUMBER = {1170566},
MRREVIEWER = {Jorge A. Vargas},
}

@article {Ngo:Hankel,
    AUTHOR = {Ng{\^{o}}, B. C.},
     TITLE = {Hankel transform, {L}anglands functoriality and functional
              equation of automorphic {$L$}-functions},
   JOURNAL = {Jpn. J. Math.},
  FJOURNAL = {Japanese Journal of Mathematics},
    VOLUME = {15},
      YEAR = {2020},
    NUMBER = {1},
     PAGES = {121--167},
      ISSN = {0289-2316},
   MRCLASS = {22E35 (11F66)},
  MRNUMBER = {4068833},
       DOI = {10.1007/s11537-019-1650-8},
       URL = {https://doi.org/10.1007/s11537-019-1650-8},
}

@article {BP:GLn,
    AUTHOR = {Beuzart-Plessis, R.~},
     TITLE = {Plancherel formula for {${\rm GL}_n(F)\backslash {\rm
              GL}_n(E)$} and applications to the {I}chino-{I}keda and formal
              degree conjectures for unitary groups},
   JOURNAL = {Invent. Math.},
  FJOURNAL = {Inventiones Mathematicae},
    VOLUME = {225},
      YEAR = {2021},
    NUMBER = {1},
     PAGES = {159--297},
      ISSN = {0020-9910},
   MRCLASS = {22E50 (11F70)},
  MRNUMBER = {4270666},
MRREVIEWER = {Alexandre Afgoustidis},
       DOI = {10.1007/s00222-021-01032-6},
       URL = {https://doi.org/10.1007/s00222-021-01032-6},
}

@article {BP:Ast,
    AUTHOR = {Beuzart-Plessis, R.~},
     TITLE = {A local trace formula for the {G}an-{G}ross-{P}rasad
              conjecture for unitary groups: the {A}rchimedean case},
   JOURNAL = {Ast\'{e}risque},
  FJOURNAL = {Ast\'{e}risque},
    NUMBER = {418},
      YEAR = {2020},
     PAGES = {ix+305},
      ISSN = {0303-1179},
      ISBN = {978-2-85629-919-7},
   MRCLASS = {22E50 (11F72 11F85 20G05)},
  MRNUMBER = {4146145},
MRREVIEWER = {Wen-Wei Li},
       DOI = {10.24033/ast},
       URL = {https://doi.org/10.24033/ast},
}

@article {Baruch,
    AUTHOR = {Baruch, Ehud Moshe},
     TITLE = {A proof of {K}irillov's conjecture},
   JOURNAL = {Ann. of Math. (2)},
  FJOURNAL = {Annals of Mathematics. Second Series},
    VOLUME = {158},
      YEAR = {2003},
    NUMBER = {1},
     PAGES = {207--252},
      ISSN = {0003-486X},
   MRCLASS = {22E46},
  MRNUMBER = {1999922},
MRREVIEWER = {Zhengyu Mao},
       DOI = {10.4007/annals.2003.158.207},
       URL = {https://doi.org/10.4007/annals.2003.158.207},
}

@article {Waldspurger:plancherel,
    AUTHOR = {Waldspurger, J.-L.},
     TITLE = {La formule de {P}lancherel pour les groupes {$p$}-adiques
              (d'apr\`es {H}arish-{C}handra)},
   JOURNAL = {J. Inst. Math. Jussieu},
  FJOURNAL = {Journal of the Institute of Mathematics of Jussieu. JIMJ.
              Journal de l'Institut de Math\'{e}matiques de Jussieu},
    VOLUME = {2},
      YEAR = {2003},
    NUMBER = {2},
     PAGES = {235--333},
      ISSN = {1474-7480,1475-3030},
   MRCLASS = {22E35 (22E50)},
  MRNUMBER = {1989693},
MRREVIEWER = {Rebecca\ Herb},
       DOI = {10.1017/S1474748003000082},
       URL = {https://doi.org/10.1017/S1474748003000082},
}

@book {Wallach,
    AUTHOR = {Wallach, Nolan R.},
     TITLE = {Real reductive groups. {II}},
    SERIES = {Pure and Applied Mathematics},
    VOLUME = {132-II},
 PUBLISHER = {Academic Press, Inc., Boston, MA},
      YEAR = {1992},
     PAGES = {xiv+454},
      ISBN = {0-12-732961-7},
   MRCLASS = {22E46 (22D25 22E30 46L99)},
  MRNUMBER = {1170566},
MRREVIEWER = {Jorge\ A.\ Vargas},
}

@incollection {Bernstein:P,
    AUTHOR = {Bernstein, Joseph N.},
     TITLE = {{$P$}-invariant distributions on {${\rm GL}(N)$} and the
              classification of unitary representations of {${\rm GL}(N)$}
              (non-{A}rchimedean case)},
 BOOKTITLE = {Lie group representations, {II} ({C}ollege {P}ark, {M}d.,
              1982/1983)},
    SERIES = {Lecture Notes in Math.},
    VOLUME = {1041},
     PAGES = {50--102},
 PUBLISHER = {Springer, Berlin},
      YEAR = {1984},
   MRCLASS = {22E50},
  MRNUMBER = {748505},
MRREVIEWER = {Stephen Gelbart},
       DOI = {10.1007/BFb0073145},
       URL = {https://doi.org/10.1007/BFb0073145},
}

@article {Sahi,
    AUTHOR = {Sahi, Siddhartha},
     TITLE = {On {K}irillov's conjecture for {A}rchimedean fields},
   JOURNAL = {Compositio Math.},
  FJOURNAL = {Compositio Mathematica},
    VOLUME = {72},
      YEAR = {1989},
    NUMBER = {1},
     PAGES = {67--86},
      ISSN = {0010-437X},
   MRCLASS = {22E45 (22E50)},
  MRNUMBER = {1026329},
MRREVIEWER = {Marko Tadi\'{c}},
       URL = {http://www.numdam.org/item?id=CM_1989__72_1_67_0},
}

@article {BPLZZ,
    AUTHOR = {Beuzart-Plessis, Rapha\"{e}l and Liu, Yifeng and Zhang, Wei and
              Zhu, Xinwen},
     TITLE = {Isolation of cuspidal spectrum, with application to the
              {G}an-{G}ross-{P}rasad conjecture},
   JOURNAL = {Ann. of Math. (2)},
  FJOURNAL = {Annals of Mathematics. Second Series},
    VOLUME = {194},
      YEAR = {2021},
    NUMBER = {2},
     PAGES = {519--584},
      ISSN = {0003-486X},
   MRCLASS = {11F67 (11F70 11F72)},
  MRNUMBER = {4298750},
MRREVIEWER = {Bin Xu},
       DOI = {10.4007/annals.2021.194.2.5},
       URL = {https://doi.org/10.4007/annals.2021.194.2.5},
}

@misc{Rodier:Whitt,
author = {Rodier, F.},
title = {Mod\`eles de Whittaker des repr\'esentations admissibles des groupes r\'eductifs $p$-adiques quasi-d\'eploy\'es},
url ={https://normalesup.org/~rodier/Whittaker.pdf}}

@article {Shahidi:On:certain,
    AUTHOR = {Shahidi, Freydoon},
     TITLE = {On certain {$L$}-functions},
   JOURNAL = {Amer. J. Math.},
  FJOURNAL = {American Journal of Mathematics},
    VOLUME = {103},
      YEAR = {1981},
    NUMBER = {2},
     PAGES = {297--355},
      ISSN = {0002-9327},
   MRCLASS = {10D15 (10D40 22E45 22E55)},
  MRNUMBER = {610479},
MRREVIEWER = {Stephen Gelbart},
       DOI = {10.2307/2374219},
       URL = {https://doi.org/10.2307/2374219},
}

@article {FLO,
    AUTHOR = {Feigon, Brooke and Lapid, Erez and Offen, Omer},
     TITLE = {On representations distinguished by unitary groups},
   JOURNAL = {Publ. Math. Inst. Hautes \'{E}tudes Sci.},
  FJOURNAL = {Publications Math\'{e}matiques. Institut de Hautes \'{E}tudes
              Scientifiques},
    VOLUME = {115},
      YEAR = {2012},
     PAGES = {185--323},
      ISSN = {0073-8301},
   MRCLASS = {22E55 (11F03)},
  MRNUMBER = {2930996},
MRREVIEWER = {Neven Grbac},
       DOI = {10.1007/s10240-012-0040-z},
       URL = {https://doi.org/10.1007/s10240-012-0040-z},
}

@book {HT,
    AUTHOR = {Harris, Michael and Taylor, Richard},
     TITLE = {The geometry and cohomology of some simple {S}himura
              varieties},
    SERIES = {Annals of Mathematics Studies},
    VOLUME = {151},
      NOTE = {With an appendix by Vladimir G. Berkovich},
 PUBLISHER = {Princeton University Press, Princeton, NJ},
      YEAR = {2001},
     PAGES = {viii+276},
      ISBN = {0-691-09090-4},
   MRCLASS = {11G18 (11F70 11S37 14G35 22E45)},
  MRNUMBER = {1876802},
MRREVIEWER = {James Milne},
}

@ARTICLE{GGHL,
       author = {{Getz}, Jayce R. and {Gu}, Miao Pam and {Hsu}, Chun-Hsien and {Leslie}, Spencer},
        title = "{On triple product $L$-functions and the fiber bundle method}",
      journal = {arXiv e-prints},
         year = 2025,
        month = mar,
          eid = {arXiv:2503.21648},
        pages = {arXiv:2503.21648},
          doi = {10.48550/arXiv.2503.21648},
archivePrefix = {arXiv},
       eprint = {2503.21648},
 primaryClass = {math.NT},
       adsurl = {https://ui.adsabs.harvard.edu/abs/2025arXiv250321648G}
}

@article {Lapid:Mao:Asymp,
    AUTHOR = {Lapid, Erez and Mao, Zhengyu},
     TITLE = {On the asymptotics of {W}hittaker functions},
   JOURNAL = {Represent. Theory},
  FJOURNAL = {Representation Theory. An Electronic Journal of the American
              Mathematical Society},
    VOLUME = {13},
      YEAR = {2009},
     PAGES = {63--81},
   MRCLASS = {22E50 (11F70)},
  MRNUMBER = {2495561},
MRREVIEWER = {Wee Teck Gan},
       DOI = {10.1090/S1088-4165-09-00343-4},
       URL = {https://doi.org/10.1090/S1088-4165-09-00343-4},
}

@book {Treves,
    AUTHOR = {Tr{\`{e}}ves, Fran\c{c}ois},
     TITLE = {Topological vector spaces, distributions and kernels},
      NOTE = {Unabridged republication of the 1967 original},
 PUBLISHER = {Dover Publications, Inc., Mineola, NY},
      YEAR = {2006},
     PAGES = {xvi+565},
      ISBN = {0-486-45352-9},
   MRCLASS = {46-01 (46-02 46Axx 46Fxx)},
  MRNUMBER = {2296978},
}

@ARTICLE{vandenBan:uniformtempered,
       author = {{van den Ban}, E.~P.},
        title = "{Uniform temperedness of Whittaker integrals for a real reductive group}",
      journal = {arXiv e-prints},
         year = 2023,
        month = apr,
          eid = {arXiv:2304.11044},
        pages = {arXiv:2304.11044},
          doi = {10.48550/arXiv.2304.11044},
archivePrefix = {arXiv},
       eprint = {2304.11044},
 primaryClass = {math.RT},
       adsurl = {https://ui.adsabs.harvard.edu/abs/2023arXiv230411044V}
}

@incollection {CPS:derivatives,
    AUTHOR = {Cogdell, J. W. and Piatetski-Shapiro, I. I.},
     TITLE = {Derivatives and {L}-functions for {$\mathrm{GL}_n$}},
 BOOKTITLE = {Representation theory, number theory, and invariant theory},
    SERIES = {Progr. Math.},
    VOLUME = {323},
     PAGES = {115--173},
 PUBLISHER = {Birkh\"{a}user/Springer, Cham},
      YEAR = {2017},
   MRCLASS = {11F70 (11F55 22E55)},
  MRNUMBER = {3753910},
MRREVIEWER = {Hang Xue},
       DOI = {10.1007/978-3-319-59728-7\_5},
       URL = {https://doi.org/10.1007/978-3-319-59728-7_5},
}

@article{Arthur:HP:realreductive,
  title={Harmonic analysis of the {S}chwartz space on a reductive {L}ie group {I} \& {II}},
  author={Arthur, J.},
  url={https://www.math.toronto.edu/arthur/pdf/20160115093856221.pdf},}

@article {Delorme:Whitt,
    AUTHOR = {Delorme, Patrick},
     TITLE = {Formule de {P}lancherel pour les fonctions de {W}hittaker sur
              un groupe r\'{e}ductif {$p$}-adique},
   JOURNAL = {Ann. Inst. Fourier (Grenoble)},
  FJOURNAL = {Universit\'{e} de Grenoble. Annales de l'Institut Fourier},
    VOLUME = {63},
      YEAR = {2013},
    NUMBER = {1},
     PAGES = {155--217},
      ISSN = {0373-0956,1777-5310},
   MRCLASS = {22E35 (22E50)},
  MRNUMBER = {3097945},
MRREVIEWER = {Jonathan\ M.\ Rosenberg},
       DOI = {10.5802/aif.2758},
       URL = {https://doi.org/10.5802/aif.2758},
}

@ARTICLE{vandenBan:MS,
       author = {{van den Ban}, Erik P.},
        title = "{Maass-Selberg relations for Whittaker functions on a real reductive group}",
      journal = {arXiv e-prints},
         year = 2025,
        month = nov,
          eid = {arXiv:2511.19224},
        pages = {arXiv:2511.19224},
          doi = {10.48550/arXiv.2511.19224},
archivePrefix = {arXiv},
       eprint = {2511.19224},
 primaryClass = {math.RT},
       adsurl = {https://ui.adsabs.harvard.edu/abs/2025arXiv251119224V}
}

@article {Bernstein:Planch,
    AUTHOR = {Bernstein, Joseph N.},
     TITLE = {On the support of {P}lancherel measure},
   JOURNAL = {J. Geom. Phys.},
  FJOURNAL = {Journal of Geometry and Physics},
    VOLUME = {5},
      YEAR = {1988},
    NUMBER = {4},
     PAGES = {663--710 (1989)},
      ISSN = {0393-0440},
   MRCLASS = {22E45 (22E55 43A85)},
  MRNUMBER = {1075727},
MRREVIEWER = {Gestur\ \'{O}lafsson},
       DOI = {10.1016/0393-0440(88)90024-1},
       URL = {https://doi.org/10.1016/0393-0440(88)90024-1},
}

@MISC {RB:quest,
    TITLE = {Vector-Valued Stone-Weierstrass Theorem?},
    AUTHOR = {R. Bryant},
    HOWPUBLISHED = {MathOverflow},
    NOTE = {URL:https://mathoverflow.net/q/371618 (version: 2020-09-14)},
    EPRINT = {https://mathoverflow.net/q/371618},
    URL = {https://mathoverflow.net/q/371618}
}

@article {Delorme:Harinck:Sakellaridis,
    AUTHOR = {Delorme, Patrick and Harinck, Pascale and Sakellaridis,
              Yiannis},
     TITLE = {Paley-{W}iener theorems for a {$p$}-adic spherical variety},
   JOURNAL = {Mem. Amer. Math. Soc.},
  FJOURNAL = {Memoirs of the American Mathematical Society},
    VOLUME = {269},
      YEAR = {2021},
    NUMBER = {1312},
     PAGES = {v+102},
      ISSN = {0065-9266,1947-6221},
      ISBN = {978-1-4704-4402-0; 978-1-4704-6462-2},
   MRCLASS = {22E35 (43A85)},
  MRNUMBER = {4226181},
MRREVIEWER = {Volker\ J.\ Heiermann},
       DOI = {10.1090/memo/1312},
       URL = {https://doi.org/10.1090/memo/1312},
}

@book {Varadarajan,
    AUTHOR = {Varadarajan, V. S.},
     TITLE = {Harmonic analysis on real reductive groups},
    SERIES = {Lecture Notes in Mathematics, Vol. 576},
 PUBLISHER = {Springer-Verlag, Berlin-New York},
      YEAR = {1977},
     PAGES = {v+521},
   MRCLASS = {22E45},
  MRNUMBER = {473111},
MRREVIEWER = {A.\ U.\ Klimyk},
}

@article {Queiro:Sa,
    AUTHOR = {Queir\'{o}, Jo\~{a}o F. and S\'{a}, Eduardo M.},
     TITLE = {Singular values and invariant factors of matrix sums and
              products},
   JOURNAL = {Linear Algebra Appl.},
  FJOURNAL = {Linear Algebra and its Applications},
    VOLUME = {225},
      YEAR = {1995},
     PAGES = {43--56},
      ISSN = {0024-3795,1873-1856},
   MRCLASS = {15A18},
  MRNUMBER = {1341069},
MRREVIEWER = {A.\ R.\ Amir-Mo\'{e}z},
       DOI = {10.1016/0024-3795(93)00317-S},
       URL = {https://doi.org/10.1016/0024-3795(93)00317-S},
}

@article {Robertson:Robertson,
    AUTHOR = {Robertson, Alex P. and Robertson, Wendy},
     TITLE = {On the closed graph theorem},
   JOURNAL = {Proc. Glasgow Math. Assoc.},
  FJOURNAL = {Proceedings of the Glasgow Mathematical Association},
    VOLUME = {3},
      YEAR = {1956},
     PAGES = {9--12},
      ISSN = {2040-6185,2051-2104},
   MRCLASS = {46.1X},
  MRNUMBER = {84108},
MRREVIEWER = {H.\ D.\ Block},
}

@incollection {Tadic,
    AUTHOR = {Tadi\'{c}, Marko},
     TITLE = {{${\mathrm{GL}}(n,\mathbb{C})$\^{}} and {${\mathrm{GL}}(n,\mathbb{R})$}\^{}},
 BOOKTITLE = {Automorphic forms and {$L$}-functions {II}. {L}ocal aspects},
    SERIES = {Contemp. Math.},
    VOLUME = {489},
     PAGES = {285--313},
 PUBLISHER = {Amer. Math. Soc., Providence, RI},
      YEAR = {2009},
      ISBN = {978-0-8218-4708-4},
   MRCLASS = {22E50 (22E46)},
  MRNUMBER = {2537046},
MRREVIEWER = {Hadi\ Salmasian},
       DOI = {10.1090/conm/489/09551},
       URL = {https://doi.org/10.1090/conm/489/09551},
}

@incollection {BP:Asai,
    AUTHOR = {Beuzart-Plessis, Rapha\"{e}l},
     TITLE = {Archimedean theory and {$\epsilon$}-factors for the {A}sai
              {R}ankin-{S}elberg integrals},
 BOOKTITLE = {Relative trace formulas},
    SERIES = {Simons Symp.},
     PAGES = {1--50},
 PUBLISHER = {Springer, Cham},
      YEAR = {[2021] \copyright 2021},
      ISBN = {978-3-030-68505-8; 978-3-030-68506-5},
   MRCLASS = {11F70 (11F66 22E50 22E55)},
  MRNUMBER = {4611942},
       DOI = {10.1007/978-3-030-68506-5\{_}1}

@ARTICLE{DRS:Schwartz,
       author = {{Getz}, J.~R. and {Guti{\'e}rrez Terradillos}, A. and {Hosseinijafari}, F. and {Slipper}, A. and {Xi}, G. and {Yao}, H-Y. and {Zhao}, A.},
        title = "{The $\rho$-Fourier transform}",
      journal = {arXiv e-prints},
         year = 2025,
        month = nov,
          eid = {arXiv:2512.00182},
        pages = {arXiv:2512.00182},
          doi = {10.48550/arXiv.2512.00182},
archivePrefix = {arXiv},
       eprint = {2512.00182},
 primaryClass = {math.NT},
       adsurl = {https://ui.adsabs.harvard.edu/abs/2025arXiv251200182G}
}
\bibliographystyle{alpha-inits}
\printindex
\end{document}